\documentclass[reqno]{amsart}
\usepackage[margin = 1.1in]{geometry}
\usepackage{amsmath, amssymb, amsthm, fancyhdr, verbatim, graphicx}
\usepackage{enumerate}
\usepackage[cmtip, all, 2cell]{xy}
\UseAllTwocells
\usepackage[usenames,dvipsnames]{xcolor}
\usepackage{mathrsfs}
\usepackage{tikz-cd}
\usetikzlibrary{decorations.pathmorphing,decorations.pathreplacing}
\usepackage{framed, hyperref}
\usepackage[titletoc]{appendix}
\usepackage{bbm}
\usepackage{lipsum}
\usepackage{adjustbox}
\usepackage{array}

\usepackage{amssymb}

\usepackage{stmaryrd}
\newcommand{\shear}{\mathbin{\mkern-6mu\fatslash}}

\numberwithin{equation}{subsection}

\DeclareFontFamily{U}{matha}{\hyphenchar\font45}
\DeclareFontShape{U}{matha}{m}{n}{
      <5> <6> <7> <8> <9> <10> gen * matha
      <10.95> matha10 <12> <14.4> <17.28> <20.74> <24.88> matha12
      }{}
\DeclareSymbolFont{matha}{U}{matha}{m}{n}
\DeclareFontFamily{U}{mathx}{\hyphenchar\font45}
\DeclareFontShape{U}{mathx}{m}{n}{
      <5> <6> <7> <8> <9> <10>
      <10.95> <12> <14.4> <17.28> <20.74> <24.88>
      mathx10
      }{}
\DeclareSymbolFont{mathx}{U}{mathx}{m}{n}

\DeclareMathSymbol{\obot}{2}{matha}{"6B}

\RequirePackage{xspace}

\newcommand{\sG}{\ensuremath{\mathscr{G}}\xspace}

\newcommand{\rH}{\ensuremath{\mathrm{H}}\xspace}

\usepackage{color}

\newcommand{\F}{\mathbf{F}}

\newcommand{\G}{\mathbf{G}}

\newcommand{\wt}[1]{\widetilde{#1}}

\newcommand{\Q}{\mathbf{Q}}
\newcommand{\Z}{\mathbf{Z}}

\newcommand{\cl}{\overline}

\newcommand{\ul}[1]{\underline{#1}}
\newcommand{\ol}[1]{\overline{#1}}
\newcommand{\wh}[1]{\widehat{#1}}

\newcommand{\Cal}[1]{\mathcal{#1}}

\newcommand{\mbf}[1]{\mathbf{#1}}

\newcommand{\co}{\colon}
\newcommand{\mrm}[1]{\mathrm{#1}}

\newcommand{\bbm}[1]{\mathbbm{#1}}

\newcommand{\DD}{\mathbf{D}}

\newcommand{\TT}{\mathbb{T}}
\newcommand{\inj}{\hookrightarrow}
\newcommand{\surj}{\twoheadrightarrow}

\newcommand{\tw}[1]{\sm{\langle #1 \rangle}}

\DeclareMathOperator{\Tot}{Tot}
\DeclareMathOperator{\Perf}{Perf}

\DeclareMathOperator{\GL}{GL}

\DeclareMathOperator{\Frob}{Frob}

\DeclareMathOperator{\Tr}{Tr}

\DeclareMathOperator{\Hom}{Hom}

\DeclareMathOperator{\ext}{ext}
\newcommand{\cHom}{\Cal{H}om}
\newcommand{\cExt}{\Cal{E}xt}

\DeclareMathOperator{\rank}{rank}

\DeclareMathOperator{\Aut}{Aut}

\DeclareMathOperator{\Nm}{Nm}
\DeclareMathOperator{\Spec}{Spec\,}

\DeclareMathOperator{\End}{End}
\DeclareMathOperator{\Isom}{Isom}

\DeclareMathOperator{\Bun}{Bun}
\DeclareMathOperator{\Ext}{Ext}
\DeclareMathOperator{\Pic}{Pic}

\DeclareMathOperator{\Id}{Id}

\DeclareMathOperator{\QCoh}{QCoh}

\DeclareMathOperator{\Sym}{Sym}

\DeclareMathOperator{\Proj}{Proj\,}

\DeclareMathOperator{\pt}{pt}

\DeclareMathOperator{\Fix}{Fix}

\DeclareMathOperator{\Sht}{Sht}
\DeclareMathOperator{\length}{length}

\DeclareMathOperator{\Sat}{Sat}

\DeclareMathOperator{\pr}{pr}

\DeclareMathOperator{\Hk}{Hk}

\DeclareMathOperator{\CH}{CH}

\DeclareMathOperator{\Map}{Map}

\DeclareMathOperator{\AS}{AS}

\DeclareMathOperator{\RHom}{RHom}
\DeclareMathOperator{\FT}{FT}

\DeclareMathOperator{\cc}{\mathfrak{c}}

\newcommand{\sm}[1]{{\scriptstyle  #1}}

\newcommand{\limit}{\varprojlim}

\DeclareMathOperator{\can}{can}

\DeclareMathOperator{\ev}{ev}

\renewcommand{\cl}{\mathrm{cl}}

\newcommand{\bt}{\boxtimes}

\newcommand{\incl}{\hookrightarrow}

\newcommand{\Qsh}[1]{\ol{\Q}_{#1}}

\newcommand{\bu}{\bullet}

\newcommand\sss{\subsubsection}

\newcommand\ot{\otimes}

\newcommand{\mbeta}{\beta}
\newcommand\g\gamma

\newcommand{\s}{\sigma}

\newcommand{\y}{\eta}

\newcommand\sh{\sharp}

\newcommand\frL{\mathfrak{L}}

\newcommand\frp{\mathfrak{p}}
\newcommand\frq{\mathfrak{q}}

\newcommand\Corr{\mathrm{Corr}}
\newcommand\CoCorr{\mathrm{CoCorr}}
\newcommand\univ{\mathrm{univ}}

\newcommand{\Dmot}[1]{\cD_{\mrm{mot}}({#1};\ol{\Q})}
\newcommand{\Dmotg}[1]{\cD_{\mrm{mot},\mrm{gm}}({#1};\ol{\Q})}

\DeclareMathOperator{\red}{red}

\newcommand\nc{\newcommand}
\nc\renc{\renewcommand}

\nc{\A}{\bA}
\nc{\V}{\Tot}
\nc{\Ex}{\mrm{Ex}}
\nc{\gys}{\mrm{gys}}
\nc{\fsupp}{\mrm{fsupp}}
\nc{\un}{\mbf{1}}
\nc{\vb}[1]{\langle{#1}\rangle}
\nc{\unit}{\mrm{unit}}
\nc{\counit}{\mrm{counit}}
\nc{\DVect}{\on{DVect}}
\nc{\MGL}{\mrm{MGL}}
\nc{\MGLmod}{\on{\mbf{D}_\MGL}}
\nc{\Nis}{\mrm{Nis}}

\nc{\Gm}{{\bG_{m}}}
\nc{\act}{\mrm{act}}
\nc{\Chom}{\on{C}_\bullet}
\nc{\oH}{{\mrm{H}}}

\DeclareMathOperator{\add}{add}

\usepackage{xparse}

\usepackage{xspace}

\nc{\ssec}{\subsection}
\nc{\sssec}{\subsubsection}
\nc{\on}{\operatorname}
\nc{\term}[1]{#1\xspace}

\usepackage{tikz}
\usetikzlibrary{matrix}
\usepackage{tikz-cd}

\tikzset{
  commutative diagrams/.cd,
  arrow style=tikz,
  diagrams={>=latex}}
\makeatletter
\tikzset{
  column sep/.code=\def\pgfmatrixcolumnsep{\pgf@matrix@xscale*(#1)},
  row sep/.code   =\def\pgfmatrixrowsep{\pgf@matrix@yscale*(#1)},
  matrix xscale/.code=%
    \pgfmathsetmacro\pgf@matrix@xscale{\pgf@matrix@xscale*(#1)},
  matrix yscale/.code=%
    \pgfmathsetmacro\pgf@matrix@yscale{\pgf@matrix@yscale*(#1)},
  matrix scale/.style={/tikz/matrix xscale={#1},/tikz/matrix yscale={#1}}}
\def\pgf@matrix@xscale{1}
\def\pgf@matrix@yscale{1}
\makeatother

\nc{\cA}{\ensuremath{\mathcal{A}}\xspace}
\nc{\cB}{\ensuremath{\mathcal{B}}\xspace}
\nc{\cC}{\ensuremath{\mathcal{C}}\xspace}
\nc{\cD}{\ensuremath{\mathcal{D}}\xspace}
\nc{\cE}{\ensuremath{\mathcal{E}}\xspace}
\nc{\cF}{\ensuremath{\mathcal{F}}\xspace}
\nc{\cG}{\ensuremath{\mathcal{G}}\xspace}
\nc{\cH}{\ensuremath{\mathcal{H}}\xspace}
\nc{\cI}{\ensuremath{\mathcal{I}}\xspace}
\nc{\cJ}{\ensuremath{\mathcal{J}}\xspace}
\nc{\cK}{\ensuremath{\mathcal{K}}\xspace}
\nc{\cL}{\ensuremath{\mathcal{L}}\xspace}
\nc{\cM}{\ensuremath{\mathcal{M}}\xspace}
\nc{\cN}{\ensuremath{\mathcal{N}}\xspace}
\nc{\cO}{\ensuremath{\mathcal{O}}\xspace}
\nc{\cP}{\ensuremath{\mathcal{P}}\xspace}
\nc{\cQ}{\ensuremath{\mathcal{Q}}\xspace}
\nc{\cR}{\ensuremath{\mathcal{R}}\xspace}
\nc{\cS}{\ensuremath{\mathcal{S}}\xspace}
\nc{\cT}{\ensuremath{\mathcal{T}}\xspace}
\nc{\cU}{\ensuremath{\mathcal{U}}\xspace}
\nc{\cV}{\ensuremath{\mathcal{V}}\xspace}
\nc{\cW}{\ensuremath{\mathcal{W}}\xspace}
\nc{\cX}{\ensuremath{\mathcal{X}}\xspace}
\nc{\cY}{\ensuremath{\mathcal{Y}}\xspace}
\nc{\cZ}{\ensuremath{\mathcal{Z}}\xspace}

\nc{\bA}{\ensuremath{\mathbf{A}}\xspace}
\nc{\bB}{\ensuremath{\mathbf{B}}\xspace}
\nc{\bC}{\ensuremath{\mathbf{C}}\xspace}
\nc{\bD}{\ensuremath{\mathbf{D}}\xspace}
\nc{\bE}{\ensuremath{\mathbf{E}}\xspace}
\nc{\bF}{\ensuremath{\mathbf{F}}\xspace}
\nc{\bG}{\ensuremath{\mathbf{G}}\xspace}
\nc{\bH}{\ensuremath{\mathbf{H}}\xspace}
\nc{\bI}{\ensuremath{\mathbf{I}}\xspace}
\nc{\bJ}{\ensuremath{\mathbf{J}}\xspace}
\nc{\bK}{\ensuremath{\mathbf{K}}\xspace}
\nc{\bL}{\ensuremath{\mathbf{L}}\xspace}
\nc{\bM}{\ensuremath{\mathbf{M}}\xspace}
\nc{\bN}{\ensuremath{\mathbf{N}}\xspace}
\nc{\bO}{\ensuremath{\mathbf{O}}\xspace}
\nc{\bP}{\ensuremath{\mathbf{P}}\xspace}
\nc{\bQ}{\ensuremath{\mathbf{Q}}\xspace}
\nc{\bR}{\ensuremath{\mathbf{R}}\xspace}
\nc{\bS}{\ensuremath{\mathbf{S}}\xspace}
\nc{\bT}{\ensuremath{\mathbf{T}}\xspace}
\nc{\bU}{\ensuremath{\mathbf{U}}\xspace}
\nc{\bV}{\ensuremath{\mathbf{V}}\xspace}
\nc{\bW}{\ensuremath{\mathbf{W}}\xspace}
\nc{\bX}{\ensuremath{\mathbf{X}}\xspace}
\nc{\bY}{\ensuremath{\mathbf{Y}}\xspace}
\nc{\bZ}{\ensuremath{\mathbf{Z}}\xspace}

\nc{\bbA}{\ensuremath{\mathbb{A}}\xspace}
\nc{\bbB}{\ensuremath{\mathbb{B}}\xspace}
\nc{\bbC}{\ensuremath{\mathbb{C}}\xspace}
\nc{\bbD}{\ensuremath{\mathbb{D}}\xspace}
\nc{\bbE}{\ensuremath{\mathbb{E}}\xspace}
\nc{\bbF}{\ensuremath{\mathbb{F}}\xspace}
\nc{\bbG}{\ensuremath{\mathbb{G}}\xspace}
\nc{\bbH}{\ensuremath{\mathbb{H}}\xspace}
\nc{\bbI}{\ensuremath{\mathbb{I}}\xspace}
\nc{\bbJ}{\ensuremath{\mathbb{J}}\xspace}
\nc{\bbK}{\ensuremath{\mathbb{K}}\xspace}
\nc{\bbL}{\ensuremath{\mathbb{L}}\xspace}
\nc{\bbM}{\ensuremath{\mathbb{M}}\xspace}
\nc{\bbN}{\ensuremath{\mathbb{N}}\xspace}
\nc{\bbO}{\ensuremath{\mathbb{O}}\xspace}
\nc{\bbP}{\ensuremath{\mathbb{P}}\xspace}
\nc{\bbQ}{\ensuremath{\mathbb{Q}}\xspace}
\nc{\bbR}{\ensuremath{\mathbb{R}}\xspace}
\nc{\bbS}{\ensuremath{\mathbb{S}}\xspace}
\nc{\bbT}{\ensuremath{\mathbb{T}}\xspace}
\nc{\bbU}{\ensuremath{\mathbb{U}}\xspace}
\nc{\bbV}{\ensuremath{\mathbb{V}}\xspace}
\nc{\bbW}{\ensuremath{\mathbb{W}}\xspace}
\nc{\bbX}{\ensuremath{\mathbb{X}}\xspace}
\nc{\bbY}{\ensuremath{\mathbb{Y}}\xspace}
\nc{\bbZ}{\ensuremath{\mathbb{Z}}\xspace}

\newtheorem{thm}{Theorem}[subsection]
\newtheorem{lemma}[thm]{Lemma}
\newtheorem{prop}[thm]{Proposition}
\newtheorem{cor}[thm]{Corollary}

\newtheorem{conj}[thm]{Conjecture}
\newtheorem{defn-prop}[thm]{Definition-Proposition}

\theoremstyle{remark}
\newtheorem{remark}[thm]{Remark} 
\newtheorem{defn}[thm]{Definition}
\newtheorem{constr}[thm]{Construction}
\newtheorem{example}[thm]{Example}

\makeatletter
\def\th@remark{%
  \thm@headfont{\bfseries}%
  \normalfont % body font
  \thm@preskip \thm@preskip 
  \thm@postskip\thm@preskip
}
\def\imod#1{\allowbreak\mkern5mu({\operator@font mod}\,\,#1)}
\makeatother

\title[Proof of the Modularity Conjecture for Higher Theta Series]{Modularity of Higher Theta Series III: \\
Proof of the Modularity Conjecture}

\author{Tony Feng}
\address{University of California Berkeley, Department of Mathematics, Berkeley, CA 94720, USA}
\email{fengt@berkeley.edu}

\author{Zhiwei Yun}
\address{Massachusetts Institute of Technology, Department of Mathematics, 77 Massachusetts Avenue, Cambridge, MA 02139, USA}
\email{zyun@mit.edu}

\author{Wei Zhang}
\address{Massachusetts Institute of Technology, Department of Mathematics, 77 Massachusetts Avenue, Cambridge, MA 02139, USA}
\email{weizhang@mit.edu}

\begin{document}

\begin{abstract}
We prove the Modularity Conjecture for higher theta series on moduli stacks of Hermitian shtukas. For general linear shtukas, we establish a more refined phenomenon that we call \emph{supermodularity}. As a key input, we prove the Trace Conjecture for Hitchin stacks of low corank, realizing virtual fundamental classes of special cycles as categorical traces. 
\end{abstract}

\maketitle

\tableofcontents

\section{Introduction}

\subsection{The main theorem}
Let $k=\F_q$ be a finite field of characteristic $p>2$, let $X$ be a smooth, proper, geometrically connected curve over $k$, and let $\nu\co X'\to X$ be an \'etale double cover. In \cite{FYZ2}, the authors constructed \emph{higher theta series} attached to this data. More precisely, for every \emph{rank} $n \geq 1$, every \emph{corank} $m \in [1, n]$, and every $r = 0, 1, \ldots$, the higher theta series $\wt Z_m^{n,r}$ is a certain Fourier series with coefficients in $\CH_{r(n-m)}(\Sht^r_{U(n)})$, the Chow group of $r(n-m)$-dimensional cycle classes on the moduli stack $\Sht_{U(n)}^r$ of rank $n$ Hermitian shtukas with $r$ legs. 

When $r=0$, $\wt Z_m^{n,r}$ specializes to the classical theta functions. When $r=1$, the construction may be viewed as a function field analogue of arithmetic theta series in the sense of Kudla \cite{Kud04}. The adjective ``higher'' refers to the phenomena established in \cite{FYZ, FHM} that for general $r$, the higher theta series are connected to the $r$th derivatives of automorphic $L$-functions. While the $r=0$ and $r=1$ cases of this story have parallels over number fields, the existence of a story for higher derivatives is a distinctive feature of the function field setting. 

Classical theta series and arithmetic theta series enjoy more symmetries than are evident in their definition. This is formulated precisely as \emph{modularity} of the theta series: they descend to automorphic functions. For arithmetic theta series, the strongest formulations of modularity are still conjectural in most cases (and even the formulations may depend on conjectural ingredients). The main result of the earlier work \cite{FYZ2} was the formulation of an analogous Modularity Conjecture for higher theta series. A priori, $\wt Z_m^{n,r}$ is a function on pairs $(\cG, \cE)$ where $\cG$ is a rank $2m$ skew-Hermitian bundle on $X'$ and $\cE\subset\cG$ is a Lagrangian subbundle. The Modularity Conjecture predicts that $\wt Z_m^{n,r}$ is independent of the choice of $\cE \subset \cG$, and our main result is a proof of this conjecture. 

\begin{thm}\label{thm:main}
        The Modularity Conjecture \cite[Conjecture 4.15]{FYZ2} holds true.
\end{thm}

While this result is formulated for unitary groups and ``signature $(1,n-1)$'', the method of proof is very general and we expect it to work uniformly for all dual reductive pairs; see Zeff's thesis \cite{ZeffThesis} for initial steps toward such a general formalism.

\subsection{Prior work}\label{ssec:prior-work} We may contextualize the Modularity Conjecture of \cite{FYZ2} in terms of the analogous problem of \emph{modularity of arithmetic theta series}, laid out in Kudla's MSRI article \cite{Kud04}. It is useful to arrange the expectations of \cite{Kud04} into a hierarchy of
successively stronger modularity statements, summarized in Table \ref{tab:prior-work-hierarchy}.
\begin{center}
\refstepcounter{table}\label{tab:prior-work-hierarchy}
\small
\renewcommand{\arraystretch}{1.2}
\begin{adjustbox}{max width=\textwidth}
\begin{tabular}{|>{\raggedright\arraybackslash}p{0.22\textwidth}|>{\raggedright\arraybackslash}p{0.31\textwidth}|>{\raggedright\arraybackslash}p{0.37\textwidth}|}
\hline
\textbf{Tier} & \textbf{Number fields} & \textbf{Function fields} \\
\hline
Generic-fiber cohomology & Kudla--Millson \cite{KM90}. & Unitary groups \cite{FYZ3}. \\
\hline
Generic-fiber Chow groups & Orthogonal Shimura varieties: Borcherds in codimension one \cite{Bor99} (Yuan--Zhang--Zhang \cite{YZZ} over totally real fields), in general Zhang \cite{Zh09}, Bruinier--Westerholt-Raum \cite{BWR15}. & Unitary groups \cite{FK}. \\
\hline
Integral/arithmetic Chow groups & Unitary divisor case \cite{BHKRY1}; Orthogonal Shimura varieties \cite{HM22}. & Unitary groups (this paper). \\
\hline
\end{tabular}
\end{adjustbox}
\smallskip
\textsc{Table \thetable.} Three tiers of modularity problems for (arithmetic) theta series and their prior progress.
\end{center}
\begin{enumerate}
        \item \emph{Modularity in (Betti) cohomology of the generic fiber} was established for arithmetic theta series in the work of Kudla--Millson \cite{KM90}. The work \cite{FYZ3} established modularity of higher theta series in the $\ell$-adic cohomology of the generic fiber, for unitary groups. Work in progress of Xu--Zeff will generalize this to more groups.
        \item \emph{Modularity in the Chow group of the generic fiber} is subtler. For orthogonal Shimura varieties, this problem was raised by Kudla in \cite{Kudla1997a,Kud04}. Borcherds proved the divisor case over $\Q$ using Borcherds products \cite{Bor99} and Yuan--Zhang--Zhang \cite{YZZ} proved the divisor case over any totally real number field; Zhang proved the higher-codimension case conditionally on a convergence assertion \cite{Zh09}; and Bruinier--Westerholt-Raum supplied the convergence, completing the proof of Kudla's conjecture for orthogonal Shimura varieties \cite{BWR15,BR25}. Pollack has given an alternative proof of the automatic convergence theorem \cite{Pollack}; Raum has recently proved the analog in the unitary case over a general totally real number field in \cite{Raum}. There are also conditional extensions, assuming Beilinson--Bloch type conjectures, in work of Kudla and Maeda \cite{Kud21,Mae21}. In related compactified or integral settings, Chow-valued modularity has been proved for divisors on unitary Shimura varieties by Bruinier--Howard--Kudla--Rapoport--Yang \cite{BHKRY1} and for special zero-cycles on toroidal compactifications by Bruinier--Rosu--Zemel \cite{BRZ24}. In the function field setting, the corresponding generic-fiber Chow result for unitary groups was proved in \cite{FK}.
        \item \emph{Modularity in the arithmetic Chow group of an integral model} is the strongest form envisioned in the Kudla program. In the number field setting, it requires not only integral extensions of the special cycles, but also Green currents. Kudla--Rapoport constructed the basic unitary special cycles for nonsingular coefficients \cite{KRI,KRII}; see also the book of Kudla--Rapoport--Yang for low-dimensional cases \cite{KRY}. The divisor case for unitary Shimura varieties was proved in \cite{BHKRY1}. For integral models of orthogonal Shimura varieties, Howard--Madapusi proved Chow-valued modularity, though not Arakelov Chow-valued modularity \cite{HM22}. In the higher codimension case, a detailed construction of the generating series valued in the arithmetic Chow group of an integral model of unitary Shimura varieties can be found in \cite[\S3]{Chen1} and \cite[\S3]{Chen2}; the modularity of these generating series is still conjectural. For function fields, \cite{FYZ2} proposes the formulation of the Modularity Conjecture in this strength, and this paper provides the proof.
\end{enumerate}

Despite the resemblances of the statements, the proofs of modularity (outside the case $r=0$) are completely different in the number field and function field settings. We refer to the Introductions of \cite{FYZ3} and \cite{FK} for discussion and comparison. Our proof of Theorem \ref{thm:main} can be regarded as an improvement on the strategy of \cite{FYZ3, FK}, and will be discussed further in \S \ref{ssec:proof}.

\subsection{Supermodularity for general linear groups}
When the double cover is split, $X'=X\sqcup X$, the unitary groups become
general linear groups, and we prove a more refined phenomenon that we call \emph{supermodularity}. 

To formulate it, we first recall that \cite{FYZ2} already extended the Modularity Conjecture to general linear groups. In the unitary case, we may view the higher theta series $\wt Z_m^{n,r}$ as a function on $\Bun_{P_m}(\F_q)$, where $P_m$ is the Siegel parabolic subgroup of a rank $2m$ unitary group. In the general linear case, the Modularity Conjecture admits parabolic refinements as formulated in \cite[\S 4.9]{FYZ2}. Fix a decomposition $m=m_1+m_2$, and
let $P_{(m_1,m_2)}\subset\GL(m)$ be the corresponding standard parabolic (with the convention that $P_{(m,0)}=P_{(0,m)}=\GL(m)$). We can define a higher theta series $\wt Z_{m_1,m_2}^{\mu}$ (where $\mu$ is a sequence of signs $\pm 1$ recording the leg types; see \S\ref{ssec:main-protagonists}) on
$\Bun_{P_{(m_1,m_2)}}(\F_q)$, so that it depends a priori on a rank $m$ bundle
$\cG$, a rank $m_1$ subbundle $\cE_1\subset\cG$, and the resulting
extension class.  The Modularity Conjecture for general linear groups \cite[Conjecture 4.24]{FYZ2}
asserts that the value depends on the subbundle $\cE_1$ only through $\cG$,
i.e., that $\wt Z_{m_1,m_2}^{\mu}$ descends to $\Bun_m(\F_q)$.

\emph{Supermodularity} 
expresses an even stronger modularity property of the higher theta series: it is essentially independent \emph{even of the parabolic} $P_{(m_1,m_2)}$. Since in the extreme case $(m_1,m_2)=(m,0)$ the higher theta series visibly reduces to the special cycle class $[\cZ_{\cG,0}^{\mu}]$ attached to $\cG$ alone, supermodularity expresses all the $\wt Z_{m_1,m_2}^{\mu}$ directly in terms of $[\cZ_{\cG,0}^{\mu}]$. The precise result is formulated in Theorem \ref{thm: intro supermodularity}.

Supermodularity is a phenomenon special to general linear groups, and therefore has no precedent for Shimura varieties over number fields. We nonetheless expect it to have a manifestation over $p$-adic local fields.

\subsection{Applications} We list various applications (some potential, some already realized) of Theorem \ref{thm:main}, for which the earlier generic modularity results of \cite{FYZ3, FK} are not sufficient. 

\begin{itemize}
\item (\emph{Arithmetic theta lifting}) In Kudla's program, the modularity of (arithmetic) theta series is needed for (arithmetic) theta lifting. 
In \cite{FHM}, Theorem \ref{thm:main} is used to construct a higher theta lifting, and to prove an incarnation of a \emph{higher arithmetic inner product formula}. Specifically, \cite[Corollary 1.2.1]{FHM} is conditional on Theorem \ref{thm:main}. 

\item (\emph{Arithmetic intersection theory}) Modularity of arithmetic theta series was a key input to the results of \cite{SSTT22}, which established exceptional jumps of Picard ranks of K3 surfaces over number fields, using arithmetic intersection theory. Ruofan Jiang and Wenqing Wei are pursuing arithmetic intersection problems that rely on Theorem \ref{thm:main}. In order to control the intersection theory, it is crucial to have the full integral modularity.

\item (\emph{Local modularity}) Yet another direction of application is towards local problems on special cycles. For this, it is again necessary to have modularity that holds integrally (not just over the generic fiber). Zhang \cite{Zh21} applied modularity of arithmetic theta series to prove the Arithmetic Fundamental Lemma (AFL), and Luo \cite{luo2025kudlarapoportconjectureunramifiedmaximal} used it to prove the Kudla--Rapoport Conjecture at places of maximal parahoric level structure.  %Relatedly, the first author, together with He, Mihatsch, Li, and Zhang, is studying \emph{local modularity} in the sense of Li--Zhang \cite{LZ1, LZ2} conditional on an integral version of global modularity, as proved here. 
\end{itemize}

\subsection{On the proof}\label{ssec:proof}
At a high level, the proof is illustrated in Figure \ref{fig:cartoon}. The philosophy is to find the ``right'' proof of modularity for the classical theta series ($r=0$). Such a proof can be seen as an incarnation of the Poisson summation argument. It should then be promoted to the sheaf-theoretic level via the classical sheaf-function correspondence. We then apply an elaboration of that correspondence, called the \emph{sheaf-cycle correspondence} in \cite{FYZ3}, to extract modularity of higher theta series.

\begin{figure}[t]
\centering
\includegraphics[width=0.95\textwidth]{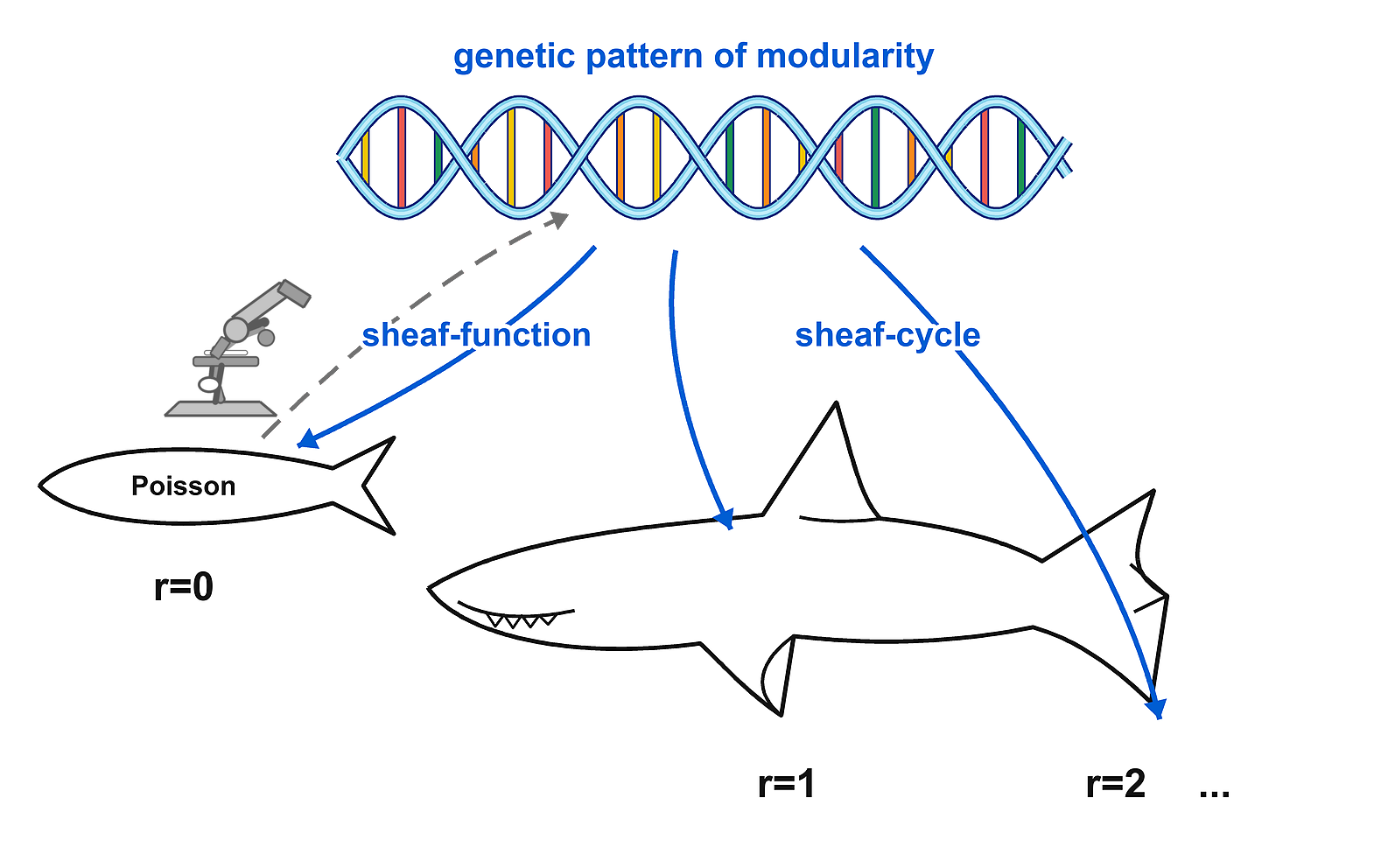}
\caption{Cartoon of the proof of the Modularity Conjecture. By studying the $r=0$ specimen (Poisson), we sequence the ``genetic pattern of modularity''. We promote this to a sheaf-theoretic level using (motivic) \emph{derived Fourier analysis}, and then extract the modularity of all higher theta series via the \emph{sheaf-cycle correspondence}.}\label{fig:cartoon}
\end{figure}

Carrying out this strategy requires much foundational technology developed in \cite{FYZ3} and
\cite{FK}, including the \emph{derived algebraic geometry} approach to virtual fundamental classes of special cycles, the (motivic) \emph{derived Fourier transform} to promote finite Fourier analysis to the sheaf-theoretic level, and the (motivic) sheaf-cycle correspondence to extract Chow cycles via categorical traces. These are described in more detail in the
Introductions of \cite{FYZ3, FK}. We focus on explaining the novelties of the present paper, which allow us to prove the full Modularity Conjecture.

\subsubsection{The Trace Conjecture} The \emph{Trace Conjecture} of \cite[Conjecture 9.4.2]{FH} is a fundamental formula that unites three different constructions of virtual fundamental classes for special cycles:
\begin{itemize}
\item the classical explicit approach present in \cite{FYZ, FYZ2} and all classical works in the Kudla program,
\item the derived algebraic geometry approach introduced in \cite{FYZ2} (with Shimura variety counterpart appearing in \cite{Mad22}), and 
\item the (categorical) trace approach introduced in \cite{FYZ3, FK}.
\end{itemize}
Each of these approaches possesses unique advantages, and all of them are exploited in an essential way here for the proof of the Modularity Conjecture. 

A key new input in this paper is a proof of the Trace Conjecture in the low-corank range $m\le n/3$ (Theorem \ref{thm:trace-conjecture-n/3}). We use this to prove the Modularity Conjecture in the same range, and then deduce modularity for all coranks $m\le n$ by an embedding trick. Both the proof of the Trace Conjecture in low corank and
the embedding trick are insensitive to the group-theoretic setup, and should
apply with little change to orthogonal and symplectic dual pairs once the
relevant statements are formulated.

The previous generic modularity results implicitly used that the Trace Conjecture is easy to prove after restricting away from a point of $X'$, by adding level structure to ``resolve'' the Hitchin stacks. In fact, at several junctures the earlier works \cite{FYZ3, FK} use this trick of adding level structure away from the legs to kill geometric difficulties. This trick is not available if one wants to prove the full Modularity Conjecture, where the legs are unrestricted. 

Roughly speaking, \cite{FYZ2} gave an elementary computation of the virtual fundamental classes of special cycles, expressing them as a linear combination of terms that are products of a ``core'' fundamental class and a Chern class (representing excess intersection). Part of the proof of the Trace Conjecture in the low corank range is a new dimension estimate of the ``core'' part (Theorem \ref{thm:special-cycle-dimension-bound}), showing that their dimensions do not exceed their ``expected dimensions'' when $m \leq n/3$. This principle was not anticipated in the analogous theory for Shimura varieties when the corank $m$ exceeds 1, because it concerns a part of the special cycles which have no analogous counterpart in the context of Shimura varieties.

\subsubsection{The trace formula for Hitchin stacks}
Another source of progress comes from an improved understanding of the sheaf-cycle correspondence in the geometric situations appearing in this problem. The sheaf-cycle correspondence can be viewed as a ``higher'' trace operation producing Chow-valued traces. However, functoriality of the formation of trace requires ``nice'' geometric properties. This is already visible for the usual sheaf-function correspondence, where, for example, the Grothendieck--Lefschetz trace formula requires properness assumptions in general, which are not satisfied by the Hitchin stacks that govern the special cycles. Yet in \S \ref{sec: trace formula} we establish a certain ``trace formula'' which does apply to these Hitchin stacks.

\subsubsection{Fourier duality} Let us also mention a major departure of the Fourier duality argument in this paper from earlier works. In \cite{FYZ3} and \cite{FK}, it was necessary to add level structure
along an uncontrolled finite subset of $X$. For example, this was necessary in order to define the \emph{arithmetic Fourier transform} in \cite{FYZ3, FK}. Adding such level structure is useful for ``taming'' the derived aspect of the geometry, but forces one to restrict to the generic locus. In order to avoid such a restriction, we develop a new Fourier duality approach that interfaces well with derived geometry, and in particular bypasses the arithmetic Fourier transform altogether. A more substantive description of how our approach improves on the one in \cite{FYZ3} is given in \S \ref{ssec:departure}, after the necessary notation has been introduced. 
  
\subsection{AI Methodology} The proof of Theorem \ref{thm:special-cycle-dimension-bound} was found with assistance from GPT 5.5 Pro. Specifically, GPT 5.5 Pro was prompted with the statement of the theorem and produced essentially the same line of argument documented here (which in turn closely follows the arguments in \cite[\S 9]{FYZ} and \cite[\S 8]{FYZ2}). A more detailed technical documentation of this contribution is given in Remark \ref{rem:AI-help}. No other results here were found with AI assistance. The paper was proofread and revised with assistance from various AI tools including Codex, Claude Code, Refine.ink, and Gemini via Google's Paper Assistant Tool (PAT). The figures and tables were produced by Codex.

\subsection{Acknowledgments} T.F. was supported by the NSF (grants DMS-2302520 and DMS-2441922), the Simons Foundation,
and the Alfred P. Sloan Foundation.  Z.Y.
was supported by the Simons Foundation. W.Z. was supported by the NSF
(grant DMS-2401548) and the Simons Foundation.

\section{Notation} Notation is as in \cite{FK} and \cite{FYZ3}.

\subsection{Main protagonists}\label{ssec:main-protagonists}
Let $k=\F_q$ be a finite field of characteristic $p>2$, let $X$ be a smooth, projective, geometrically connected curve over $k$ of genus $g$, and let $\nu\co X'\to X$ be an \'etale double cover, possibly split, with involution $\sigma$.

For a non-negative integer $n$, we denote by $\Bun_n$ the moduli stack of rank $n$ vector bundles on $X$. For a non-negative integer $r$ and $\mu = (\mu_1, \ldots, \mu_r)$ a sequence with $\mu_i \in \{\pm 1\}$ and $\sum \mu_i = 0$, we denote by $\Sht_n^\mu$ the moduli stack of rank $n$ shtukas on $X$ with simple modifications dictated by $\mu$, defined in \cite[\S 5.1]{YZ}. 

Fix a non-trivial additive character $\psi_0\co\F_p\to\ol{\Q}^{\times}$, and write
\[
\psi:=\psi_0\circ\Tr_{\F_q/\F_p}\co\F_q\to\ol{\Q}^{\times}.
\]

 \subsection{Derived Artin stacks}
   Unless noted otherwise, we always work in the category of derived Artin stacks. Hence when we say ``Cartesian square'' we mean what might be called ``derived Cartesian square'' (sometimes we keep the adjective ``derived'' for emphasis). 

    For a derived Artin stack $A$, we denote by $A_{\cl}$ its classical truncation.
    
    By definition, a map of derived Artin stacks is a \emph{closed embedding} or \emph{proper} if the induced map of classical truncations has this property. For example, the inclusion of the classical truncation $A_{\cl} \inj A$ is a closed embedding.

We say a map of derived Artin stacks is \emph{schematic} if it is representable in derived schemes.

  \subsection{Cotangent complexes}

    For a map $f \co A \to B$ of derived Artin stacks, we denote by $\bL_f :=  \bL_{A/B} \in \Perf(A)$ the relative cotangent complex.

    Let $f \co A \to B$ be a map of derived Artin stacks that is locally finitely presented on classical truncations.
    The map $f$ is \emph{étale} if the relative cotangent complex $\bL_f$ vanishes (i.e., is isomorphic to $0 \in \Perf(A)$).
    The map $f$ is \emph{smooth} (resp. \emph{quasi-smooth}) if the relative cotangent complex $\bL_f$ is perfect of tor-amplitude $[0, \infty)$ (resp. $[-1, \infty)$).
    We use cohomological indexing for these tor-amplitude intervals; thus vector bundles lie in degree $0$, positive degrees record ``stacky directions'', and negative directions record ``derived directions''.
    
    Note that unlike properness, these properties cannot in general be detected on classical truncations.
    Moreover, while a smooth map is also smooth on classical truncations, quasi-smoothness is typically destroyed by classical truncation. Informally speaking, the ubiquity of quasi-smoothness\footnote{This is sometimes called the ``Hidden smoothness philosophy'', and was anticipated by Deligne, Drinfeld, and Kontsevich.} (in contrast to its classical counterpart, the property of being a local complete intersection) underpins the significance of derived algebraic geometry in our work. 
    
    When $\bL_f$ is perfect, we write $d(f)$ for its virtual rank (or Euler characteristic), and call it the \emph{relative dimension} of $f$. This integer is locally constant on $A$; for disconnected sources, all Chow degrees, shifts, and Tate twists involving $d(f)$ are read componentwise.
    
  \subsection{Derived motives}
For a derived Artin stack $A$, we write $\Dmot{A}$ for the derived category of motives on $A$ with $\ol{\Q}$-coefficients in the sense of \cite[\S 3.1]{FK}, and $\Dmotg{A} \subset \Dmot{A}$ for the full subcategory of geometric motives in the sense of \cite[\S 3.3]{FK}. 

We denote the monoidal unit of $\Dmot{A}$ by $\Qsh{A}$.

We use cohomological indexing. For $\cK \in \Dmot{A}$ and $i \in \Z$, we write $\cK\sm{[i]}$ for the cohomological shift, $\cK\sm{(i)}$ for Tate twist, and $\cK\tw{i} := \cK\sm{[2i](i)} \in \Dmot{A}$ for the indicated shift and Tate twist. 

\subsection{Chow groups}\label{ssec:chow-groups-notation}
Let $A$ be a derived Artin stack. We define the \emph{Chow cohomology groups} of $A$ (with $\ol{\Q}$-coefficients) by
\[
\CH^i(A) := \rH^{2i}(A; \Qsh{A}\sm{(i)}) \cong \rH^0(A; \Qsh{A}\tw{i}), \quad \text{for } i \in \bZ.
\]

Now suppose that $A$ is locally of finite type over a field $\F$, and let $\pi \co A \rightarrow \Spec(\F)$ be the structural morphism. We define the \emph{Chow groups} of $A$ (with $\ol{\Q}$-coefficients) by
\[
\CH_{i}(A) := \rH^{-2i}(A; \pi^! \Qsh{\Spec(\F)}\sm{(-i)} )
\cong \rH^0(A; \pi^!\Qsh{\Spec(\F)}\tw{-i}), \quad \text{ for } i \in \bZ
\]
This definition agrees with the classical definition of Chow groups with $\ol{\Q}$-coefficients \emph{under assumptions} that $A$ is ``reasonable'': see the discussion of \cite[\S 3.5]{FK}. If $A$ is a derived Artin stack, then by the derived invariance of $\Dmot{A}$, the inclusion of the classical truncation $A_{\cl} \inj A$ induces isomorphisms $\CH_i(A_{\cl}) \cong \CH_{i}(A)$.

More generally, for a morphism $f \co A \rightarrow B$ of derived Artin stacks over a field $\F$ that is locally of finite type, we define the \emph{relative Chow groups of $f$} to be
\[
\CH_{i}(A/B) := \rH^{-2i}(A; f^! \Qsh{B}\sm{(-i)})
\cong \rH^0(A; f^!\Qsh{B}\tw{-i}) , \quad \text{for } i \in \bZ.
\]
For the identity morphism of $A$, these definitions give a canonical identification
\[
\CH_i(A/A) \cong \CH^{-i}(A).
\]

\noindent \textbf{Functorial properties.}
\begin{itemize}
\item The Chow groups are covariantly functorial with respect to proper morphisms.
That is, if $f \co A \to B$ is a proper morphism of derived Artin stacks, then we have pushforward maps
\begin{equation}
f_! \co \CH_i(A) \to \CH_i(B)
\end{equation}
and more generally $f_! \co \CH_i(A/C) \to \CH_i(B/C)$ if $f$ is defined over some $C$.

\item Let $f \co A \rightarrow B$ be a quasi-smooth map of derived Artin stacks, of relative dimension $d(f)$.
There are (virtual) Gysin pullback maps
\begin{equation}\label{eq: Gysin pullback}
f^* \co \CH_i(B) \to \CH_{i+d(f)}(A),
\end{equation}
and more generally $f^* : \CH_i(B/C) \to \CH_{i+d(f)}(A/C)$ if $B$ is defined over $C$. These are functorial and satisfy a base change formula with respect to proper pushforwards.
\end{itemize}
The maps \eqref{eq: Gysin pullback} are induced by a natural transformation
\begin{equation}\label{eq: gys}
[f] \co f^* \rightarrow f^! \tw{-d(f)}
\end{equation}
called the \emph{relative fundamental class} (or \emph{Gysin}  natural transformation), constructed in \cite[\S 3]{KhanI}. It satisfies various natural compatibilities detailed in \cite[\S 3.2]{KhanI} or \cite[\S 3.4]{FYZ3}. 

\begin{example}[Derived fundamental classes] In particular, one has a relative (virtual) fundamental class
\begin{equation}
[A/B] \in \CH_{d(f)}(A/B)
\end{equation}
defined as the Gysin pullback of the unit in $\CH_0(B/B) \cong \CH^0(B)$.
Equivalently, it is determined by the morphism
\begin{equation}\label{eq: Gysin}
\Qsh{A} \cong f^* \Qsh{B} \xrightarrow{\gys_f} f^! \Qsh{B} \tw{-d(f)}
\end{equation}
obtained by evaluating the Gysin transformation \eqref{eq: gys} on $\Qsh{B}$. When $B = \pt := \Spec \F$, we abbreviate $[A] := [A/B]$ and call it the (derived) \emph{fundamental class} of $A$. 
\end{example}

  \subsection{Derived vector bundles}\label{ssec: notate dvb}

    Let $A$ be a derived Artin stack.
    
    Given a perfect complex $\cE$ on $A$, we denote by $\V(\cE)$ the derived stack of sections of $\cE$, as in \cite[\S 6.1.1]{FYZ3}.
    We refer to \emph{$\Tot(\cE)$} as the \emph{derived vector bundle} associated with $\cE$. In terms of the functor of points, $\Tot(\cE)$ is the derived stack over $A$ sending an $A$-scheme $u : T \to A$ to the mapping space $\Map_{\QCoh(T)}(\cO_T, u^*\cE)$.

    Throughout we use calligraphic letters such as $\cE$ for perfect complexes, and Roman letters such as $E$ for the corresponding total spaces. 
    
    We denote by $\cE^*$ the dual perfect complex to $\cE$. When $\cE$ is a perfect complex on $X\times A$ (or on $X'\times A$, with $\omega_{X'}$ in place of $\omega_X$), we denote by $\cE^\vee:=\cE^*\otimes\omega_X$ the \emph{Serre dual} perfect complex, with $\omega_X$ implicitly pulled back from $X$. We will denote the dual derived vector bundle to $E=\Tot(\cE)$ by $\wh{E} = \Tot(\cE^*)$.

\part{The Trace Conjecture}

\section{Dimension bounds on special cycles in low corank}
In this section we prove a key dimension bound on the ``injective part'' of the special cycles defined in \cite{FYZ2}, in low corank. This will be a key input to the proof of the Trace Conjecture in low corank.

\subsection{Setup and statement}
Let $k=\mathbb F_q$ have odd characteristic, let $X$ be a smooth proper
geometrically connected curve over $k$, and let
$\nu:X'\to X$ be an \'etale quadratic cover with involution $\sigma$.
Throughout Part I, the cover $X'\to X$ is assumed connected;
the split case is treated separately in Theorem
\ref{thm:split-trace-conjecture-n/3} of Part
\ref{part: supermodularity}.
Fix integers $1\le m\le n$, an integer $r\ge0$, and a rank-$m$
vector bundle $\cE$ on $X'$.

% {\color{red}
% We have the usual definitions of the unitary Hecke stack
% $\Hk_{U(n)}^r$, the unitary shtuka stack $\Sht_{U(n)}^r$, and the
% special cycle $\cZ_{\cE}^r$.  Thus a point of
% $\cZ_{\cE}^r$ consists of a Hermitian shtuka
% \[
%  \cF_0 \dashrightarrow\cdots\dashrightarrow \cF_r
%       \cong \Frob^* \cF_0
% \]
% with legs $x'_1,\ldots,x'_r$, together with compatible maps
% \[
%         t_i:\cE\longrightarrow \cF_i,\qquad 0\le i\le r.
% \]
% The open substack
% $\cZ_{\cE}^{r,\circ} \subset \cZ_{\cE}^r $ is defined by requiring every $t_i$ to be
% injective on every geometric fiber of the test scheme. The main result of this section is the following theorem. 
% }

\subsubsection{Hermitian bundles} We recall the definitions needed in this section. For a vector bundle $\cF$ on $X'\times S$, write $\cF^\vee:=\cHom_{X'\times S}(\cF,\omega_{X'}\boxtimes\cO_S)$. A rank $n$ Hermitian, or unitary, bundle is a pair $(\cF,h)$ where $\cF$ is a rank $n$ vector bundle on $X'\times S$ and $h\co \cF\xrightarrow{\sim}\sigma^*\cF^\vee$ satisfies $\sigma^*h^\vee=h$; see \cite[Definition 6.1]{FYZ}. The stack of such bundles is denoted $\Bun_{U(n)}$.

\subsubsection{Hecke correspondences} The Hecke stack $\Hk_{U(n)}^r$ is the stack of $r$ successive unitary modifications from \cite[Definition 6.3]{FYZ}. Its $S$-points consist of legs $x_1',\ldots,x_r'\in X'(S)$, Hermitian bundles $(\cF_i,h_i)$ for $0\le i\le r$, and isomorphisms
\begin{equation}\label{eq:setup-Hecke-generic-isomorphisms}
f_i\co \cF_{i-1}|_{X'\times S-(\Gamma_{x_i'}\cup\Gamma_{\sigma x_i'})}
\xrightarrow{\sim}
\cF_i|_{X'\times S-(\Gamma_{x_i'}\cup\Gamma_{\sigma x_i'})}
\quad (1\le i\le r)
\end{equation}
compatible with the Hermitian structures. The modification $f_i$ is required to be lower of length $1$ at $x_i'$ and upper of length $1$ at $\sigma(x_i')$ in the terminology of \cite[Definition 6.5]{FYZ}: equivalently, there is an intermediate vector bundle $\cF_{i-1/2}^{\flat}$ with injections $\cF_{i-1/2}^{\flat}\hookrightarrow\cF_{i-1}$ and $\cF_{i-1/2}^{\flat}\hookrightarrow\cF_i$ whose cokernels are line bundles supported on the graphs of $x_i'$ and $\sigma(x_i')$, respectively.

\subsubsection{Hermitian shtukas} The stack of rank $n$ Hermitian shtukas is defined as in \cite[Definition 6.6]{FYZ} by the Cartesian square
\begin{equation}\label{eq:setup-unitary-shtuka-square}
\begin{tikzcd}[ampersand replacement=\&]
\Sht_{U(n)}^r \ar[r] \ar[d] \&
\Hk_{U(n)}^r \ar[d, "{(\pr_0,\pr_r)}"] \\
\Bun_{U(n)} \ar[r, "{(\Id,\Frob)}"'] \&
\Bun_{U(n)}\times \Bun_{U(n)} .
\end{tikzcd}
\end{equation}
We abbreviate an $S$-point of $\Sht_{U(n)}^r$ as 
\begin{equation}\label{eq:setup-shtuka-chain}
(\cF_0,h_0)\dashrightarrow(\cF_1,h_1)\dashrightarrow\cdots\dashrightarrow
(\cF_r,h_r)\xrightarrow{\varphi}(\Frob_S^*\cF_0,\Frob_S^*h_0).
\end{equation}

\subsubsection{Special cycles} The special cycle $\cZ_{\cE}^r$ is defined in \cite[Definition 7.1]{FYZ}. Its $S$-points are points of $\Sht_{U(n)}^r(S)$ together with maps
\begin{equation}\label{eq:setup-special-cycle-maps}
\begin{tikzcd}[ampersand replacement=\&]
\cE\boxtimes\cO_S \ar[d, "t_0"'] \ar[r, equals] \&
\cE\boxtimes\cO_S \ar[d, "t_1"] \ar[r, equals] \&
\cdots \ar[d] \ar[r, equals] \&
\cE\boxtimes\cO_S \ar[d, "t_r"] \ar[r, "\sim"] \&
\Frob_S^*(\cE\boxtimes\cO_S) \ar[d, "\Frob_S^*t_0"] \\
\cF_0 \ar[r, dashed, "f_1"'] \&
\cF_1 \ar[r, dashed] \&
\cdots \ar[r, dashed, "f_r"'] \&
\cF_r \ar[r, "\varphi"'] \&
\Frob_S^*\cF_0 .
\end{tikzcd}
\end{equation}
such that $\varphi$ identifies $t_r$ with $\Frob_S^*t_0$, and for each $1\le i\le r$ the maps $t_{i-1}$ and $t_i$ agree under the generic isomorphism $f_i$ in \eqref{eq:setup-Hecke-generic-isomorphisms}. For this compatibility, the composite maps
\begin{equation}\label{eq:setup-Hermitian-coefficient}
a_t\co
\cE\boxtimes\cO_S
\xrightarrow{t_i}
\cF_i
\xrightarrow{h_i}
\sigma^*\cF_i^\vee
\xrightarrow{\sigma^*t_i^\vee}
\sigma^*\cE^\vee\boxtimes\cO_S ,
\end{equation}
are independent of $i$, and invariant under $\Frob_S$. This implies that they must lie in the set $\cA_{\cE}(k)$ of Hermitian maps $\cE\to\sigma^*\cE^\vee$. This induces an open-closed decomposition
\[
\cZ_\cE^r=\coprod_{a\in\cA_{\cE}(k)}\cZ_\cE^r(a).
\]

Finally, $\cZ_{\cE}^{r,\circ}\subset\cZ_{\cE}^r$ denotes the open substack where the maps in \eqref{eq:setup-special-cycle-maps} are injective after pullback to every geometric fiber of the test scheme. In the notation of \cite[Definition 7.4]{FYZ}, it restricts on the coefficient-$a$ summand to $\cZ_\cE^r(a)^\circ$. The virtual class $[\cZ_{\cE}^{r,\circ}]$ is defined in \cite[Definition 4.4]{FYZ2}, with an equivalent quotient formulation in \cite[Definition 4.5]{FYZ2}, but the present section only uses the underlying geometric locus $\cZ_{\cE}^{r,\circ}$.

The main result of this section is the following theorem.

\begin{thm}\label{thm:special-cycle-dimension-bound}
If $m \le n/3 $, then for all $\cE$ and $r$ we have 
\begin{equation}\label{eq:main-bound}
        \dim \cZ_{\cE}^{r, \circ} \leq (n-m)r .
\end{equation}
\end{thm}

\begin{remark}[Description of AI assistance]\label{rem:AI-help}
Theorem \ref{thm:special-cycle-dimension-bound} is a similar result to \cite[Proposition 9.5]{FYZ} which treated the case $\rank \cE = 1$, but the proof --- which was obtained with the aid of GPT 5.5 Pro --- carries some additional subtleties. The main technical difference with the proof of \emph{loc. cit.} is that when $\rank \cE >1$, the saturation of $\cE$ in an ambient bundle $\cF_i$ not necessarily given by twisting with a line bundle. 

Now we describe the contribution of GPT 5.5 Pro in more detail. The first author queried GPT 5.5 Pro whether a bound of the form \eqref{eq:main-bound} was true for $m \leq \epsilon n$ for some $\epsilon>0$. We knew from examples coming from parabolic shtukas that necessarily $\epsilon \leq 1/3$. GPT 5.5 Pro returned a response concluding that $\epsilon = 1/3$, which gave us confidence that it might be correct. We partially read this response, and explain how it relates to our eventual argument, as written below. 
\begin{itemize}
\item GPT 5.5 Pro's response contained a passage very similar to the material of \S \ref{ssec:saturation-stratification}, although it gave a different argument based on Quot schemes, which we did not read. 

\item GPT 5.5 Pro's response contained the material of \ref{ssec:framed-hitching-stacks}, and in fact we rewrote that section starting from its output. This material was already known to the authors, and is a straightforward generalization of the analysis in \cite[\S 9.2]{FYZ}. 

\item The key new trick in GPT 5.5 Pro's response (from the authors' perspective) was to use Lemma \ref{lem:quasi-finite} to kill the dimensions coming from the possible saturations of $\cE$ in the $\cF_i$. 
\end{itemize}
The argument below was written by the authors, except for the figures, which were generated by Codex. 
\end{remark}

\subsection{The saturation stratification}\label{ssec:saturation-stratification}

\begin{defn}\label{def:saturation-subbundle}
For a geometric point $(\cE, \cF_{\bu}, t_{\bu})$ of $\cZ_{\cE}^{r, \circ}$, for each $i=0, \ldots, r$ let
\begin{equation}\label{eq:Pi-definition}
        \cP_i:=\Sat_{\cF_i}(t_i(\cE)) \quad \text{be the saturation of $t_i(\cE)$ in $\cF_i$.}
\end{equation}
 Then $\cP_i\subset\cF_i$ is a rank-$m$ subbundle, and
$\cF_i/\cP_i$ is locally free.
\end{defn}

\begin{lemma}\label{lem:stratify-saturation}
Let $X'$ be a smooth projective curve over a field $k$.  Let $K/k$ be a
field extension.  Let $T$ be a reduced finite-type
$K$-scheme. Let $\cE$ and $\cF$ be vector bundles of ranks $m$ and $n$,
respectively, on $X_T':=X'\times_k T$, and let
\[
    t:\cE \longrightarrow \cF
\]
be a morphism which is injective after pullback to all geometric points of $T$. Then there is a finite locally closed stratification
\[
    T=\coprod_{\alpha=1}^{N}T_\alpha
\]
such that after pullback to each $T_\alpha$, the fiberwise saturation of $t(\cE|_{T_\alpha})$ in $\cF|_{T_\alpha}$ is represented by a rank-$m$ subbundle $\cP_\alpha\subset \cF|_{T_\alpha}$
with locally free quotient. Moreover, $\cP_\alpha/t(\cE|_{T_\alpha})$
is finite flat over $T_\alpha$.
\end{lemma}

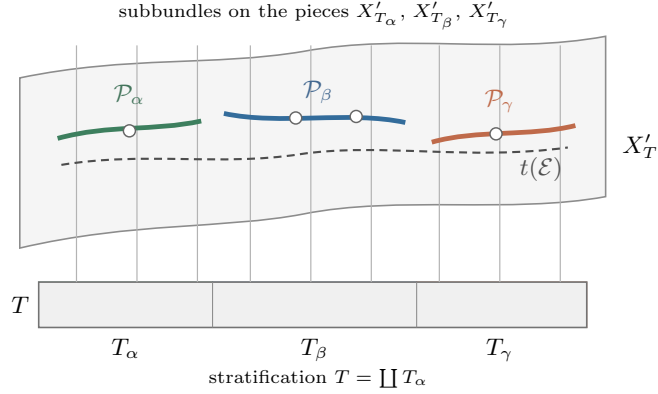
\begin{figure}[htbp]
\centering
\begin{tikzpicture}[
        x=1cm,y=1cm,
        font=\small,
        surface/.style={draw=black!45, fill=black!4, line width=.6pt},
        base/.style={draw=black!55, fill=black!6, line width=.6pt},
        stratumA/.style={fill=ForestGreen!18, draw=ForestGreen!55!black},
        stratumB/.style={fill=RoyalBlue!14, draw=RoyalBlue!55!black},
        stratumC/.style={fill=RedOrange!15, draw=RedOrange!65!black},
        fiber/.style={draw=black!30, line width=.45pt},
        eband/.style={draw=black!70, line width=.8pt, densely dashed},
        pA/.style={draw=ForestGreen!65!black, line width=1.8pt},
        pB/.style={draw=RoyalBlue!70!black, line width=1.8pt},
        pC/.style={draw=RedOrange!75!black, line width=1.8pt},
        torsion/.style={fill=white, draw=black!60, line width=.55pt}]
    \fill[surface]
        (-.2,1.1) .. controls (1.3,1.45) and (2.5,1.35) .. (3.55,1.55)
        .. controls (4.8,1.8) and (6.3,1.55) .. (7.55,1.78)
        -- (7.55,3.9)
        .. controls (6.2,3.65) and (4.9,3.95) .. (3.55,3.7)
        .. controls (2.35,3.5) and (1.25,3.65) .. (-.2,3.35)
        -- cycle;

    \fill[stratumA] (.05,.05) rectangle (2.35,.65);
    \fill[stratumB] (2.35,.05) rectangle (5.05,.65);
    \fill[stratumC] (5.05,.05) rectangle (7.3,.65);
    \draw[base] (.05,.05) rectangle (7.3,.65);
    \draw[black!45] (2.35,.05) -- (2.35,.65);
    \draw[black!45] (5.05,.05) -- (5.05,.65);

    \foreach \X in {.55,1.35,2.15,2.85,3.65,4.45,5.35,6.15,6.95} {
        \draw[fiber] (\X,.65) -- (\X,3.78);
    }

    \draw[eband]
        (.35,2.2) .. controls (1.45,2.38) and (2.45,2.18) .. (3.5,2.34)
        .. controls (4.55,2.5) and (5.75,2.25) .. (7.05,2.43);

    \draw[pA]
        (.3,2.55) .. controls (1.0,2.74) and (1.55,2.64) .. (2.2,2.78);
    \draw[pB]
        (2.5,2.88) .. controls (3.25,2.72) and (4.15,2.92) .. (4.9,2.76);
    \draw[pC]
        (5.25,2.5) .. controls (5.85,2.67) and (6.55,2.57) .. (7.15,2.72);

    \node[torsion, circle, inner sep=1.6pt] at (1.25,2.65) {};
    \node[torsion, circle, inner sep=1.6pt] at (3.45,2.82) {};
    \node[torsion, circle, inner sep=1.6pt] at (4.25,2.84) {};
    \node[torsion, circle, inner sep=1.6pt] at (6.1,2.61) {};

    \node[below] at (1.2,.02) {$T_\alpha$};
    \node[below] at (3.7,.02) {$T_\beta$};
    \node[below] at (6.15,.02) {$T_\gamma$};
    \node[left] at (.05,.35) {$T$};
    \node[right] at (7.65,2.45) {$X'_T$};

    \node[text=ForestGreen!55!black] at (1.25,3.13) {$\cP_\alpha$};
    \node[text=RoyalBlue!60!black] at (3.75,3.16) {$\cP_\beta$};
    \node[text=RedOrange!75!black] at (6.15,3.05) {$\cP_\gamma$};
    \node[black!70] at (6.7,2.15) {$t(\cE)$};
    \node[font=\scriptsize, align=center] at (3.7,4.18)
        {subbundles on the pieces $X'_{T_\alpha}$, $X'_{T_\beta}$, $X'_{T_\gamma}$};
    \node[font=\scriptsize, align=center] at (3.75,-.62)
        {stratification $T=\coprod T_\alpha$};
\end{tikzpicture}
\caption{Schematic picture of Lemma \ref{lem:stratify-saturation}.  After stratifying the base $T$, the fiberwise saturation of $t(\cE)$ in $\cF$ is represented over each stratum $T_\alpha$ by a rank-$m$ subbundle $\cP_\alpha\subset \cF|_{X'_{T_\alpha}}$.}
\label{fig:saturation-stratification}
\end{figure}

\begin{proof}By Noetherian induction, it suffices to produce a non-empty open subset $U \subset T$ and a rank $m$ saturated sub-bundle $\cP_U \subset \cF|_U$ such that the following properties hold: 
\begin{enumerate}
\item[(i)] $t(\cE|_U) \subset \cP_U \subset \cF|_U$,
\item[(ii)] After pullback to any geometric point of $U$, $\cP_U$ is the saturation of $t(\cE|_U)$.
\item[(iii)] $\cP_U/t(\cE|_U)$ is finite flat over $U$.
\end{enumerate}

\emph{Step 1.} Choose a generic point $\eta$ of an irreducible component of $T$. The saturation of $t(\cE|_\eta)$ in $\cF|_\eta$ is a rank $m$ sub-bundle that we denote $\cP_\eta$. Choose an open neighborhood $U$ of $\eta$ contained in that irreducible component. Since $X'_\eta$ is a smooth curve and the quotient $\cF_\eta/\cP_\eta$ is torsion-free on $X'_\eta$, it is locally free of rank $n-m$.  Spreading out the exact
sequence
\[
        0\to\cP_\eta\to\cF_\eta\to (\cF_\eta/\cP_\eta) \to 0
\]
and shrinking $U$ if necessary, we obtain a rank-$m$ subbundle
$\cP_U\subset\cF|_U$ such that $\cF|_U/\cP_U$ is locally free.

We next argue that it is possible to shrink $U$ to arrange that $t(\cE|_U) \subset \cP_U$. The composite map $\cE|_U \rightarrow \cF|_U/\cP_U$ vanishes at $\eta$. The support of its image is therefore a closed subset of $X'_U$ missing the fiber over $\eta$. By properness of $X'$, the image of this closed subset in $U$ is a closed subset missing $\eta$. Therefore, after replacing $U$ by an open subset, still containing $\eta$, we may assume that $\cE|_U \rightarrow \cF|_U/\cP_U$ vanishes, so $\cP_U$ contains $t(\cE|_U)$. 

\emph{Step 2.} Consider the quotient $\cQ_U := \cP_U/t(\cE|_U)$. Its support is closed in $X'_U$, hence is proper over $U$. Since $\cP_\eta/t_\eta(\cE_\eta)$ has finite support
on $X'_\eta$, after shrinking $U$ to a further open subset containing $\eta$, we may further assume that $\cQ_U$ has support which is finite over $U$. 

We claim that $\cP_U$ represents the fiberwise saturated image over any geometric point $u \rightarrow U$. Since
$\cF_U/\cP_U$ and $\cP_U$ are locally free, $(\cF_U/\cP_U)_u \cong \cF_u/\cP_u$ is torsion-free on the smooth curve $X'_{u}$. Since $\cQ_U$ is finite over $U$, $\cQ_u = (\cP_U/t(\cE|_U))_u \cong (\cP_u/t_u(\cE_u)) $ has finite support on
$X'_{u}$, hence is torsion. Therefore
$\cP_{u}/t_u(\cE_u)$ is precisely the torsion subsheaf of
$\cF_{u}/t_u(\cE_u)$, so $\cP_u$ is the saturation of $t_u(\cE_u)$.

\emph{Step 3.} We have an exact sequence
\[
        0\to \cP_U/t(\cE|_U)\to  (\cF|_U)/t(\cE|_U)  \to (\cF|_U)/\cP_U \to0 .
\]
Thanks to the assumption that the map $\cE|_U \rightarrow \cF|_U$ is injective after restriction to all geometric points of $U$, \cite[Tag 00ME]{stacks-project} applies to show that $(\cF|_U/t(\cE|_U))$ is flat over $U$.
The right-hand term is locally free on $X'_U$, hence flat over $U$. Therefore
$\cP_U/t(\cE|_U)$ is flat over $U$. Since its support is finite over $U$ by
Step 2, $\cP_U/t(\cE|_U)$ is finite flat over $U$, as required.
\end{proof}

\subsection{Framed Hitchin stacks}\label{ssec:framed-hitching-stacks}

Fix $1 \leq m \leq n$. Let $\cN$ be the stack over $k$ whose $R$-points for a commutative $k$-algebra $R$ form the groupoid of triples $(\cP, \cF, h)$, where:
\begin{itemize}
        \item $(\cF, h)$ is a rank $n$ Hermitian bundle on $X'_R$, and
        \item $\cP \subset \cF$ is a rank $m$ sub-bundle such that $\cF/\cP$ is locally free over $X'_R$. 
\end{itemize} 
This is a framed version of the Hitchin stack recalled in \S \ref{ssec: trace conjecture formulation}.

Let $\Hk_{\cN}^r$ be the Hecke stack with $R$-points the groupoid of $\cF_{\bu} \in \Hk_{U(n)}^r(R)$ together with saturated sub-bundles $\cP_i \subset \cF_i$ such that the isomorphisms 
\[
\cF_i|_{X'_R \setminus (\Gamma_{x_i'} \cup \Gamma_{\sigma x'_i})} \cong \cF_{i-1}|_{X'_R \setminus (\Gamma_{x_i'} \cup \Gamma_{\sigma x'_i})} \qquad i = 0, \ldots, r
\]
identify 
\[
\cP_i|_{X'_R \setminus (\Gamma_{x_i'} \cup \Gamma_{\sigma x'_i})} = \cP_{i-1}|_{X'_R \setminus (\Gamma_{x_i'} \cup \Gamma_{\sigma x'_i})} \qquad i = 0, \ldots, r.
\]
For $0\le i\le r$, let $\pr_i\co\Hk_{\cN}^r\to\cN$ be the map recording $(\cP_i,\cF_i,h_i)$.
Let $\pr_{\mathrm{legs}}\co\Hk_{\cN}^r\to(X')^r$ be the map recording the legs.

\begin{prop}\label{prop:one-step-sat}
Over $\cN \times X'$, let $x$ be the universal point of $X'$ and set $y := \sigma(x)$. Over the universal projective bundle $\bbP(\cF_x^{\univ})\cong \bbP((\cF_y^{\univ})^\vee) \rightarrow \cN \times X'$ let $L \subset \cF_x^{\univ}$ be the universal line; $L$ is equivalent to the datum of the universal hyperplane $H=L^\perp\subset \cF_y^{\univ}$, where orthogonal complements are taken with respect to the given Hermitian form. On each ambient projective bundle in the table, $L$ and $H$ denote the pullbacks of these universal objects. Define the four relative
parameter spaces over $\cN \times X'$ as follows:
\[
\begin{array}{c|c|c}
\textup{type} & \textup{ambient projective bundle} & \textup{parameter space}\\
\hline
0 & \bbP((\cF_y^{\univ}/\cP_y^{\univ})^\vee) &
      H^\perp\not\subset \cP_x^{\univ}\\
+ & \bbP(\cF_y^{\univ, \vee}) &
      H\not\supset \cP_y^{\univ},\ H^\perp\not\subset \cP_x^{\univ}\\
- & \bbP(\cP_x^{\univ}) &
      L\subset (\cP_y^{\univ})^\perp\\
\pm & \bbP(\cP_x^{\univ}) &
      L\not\subset (\cP_y^{\univ})^\perp
\end{array}
\]
(In types $0,-,\pm$, the ambient projective bundle already enforces one of the closed fiber conditions, so the parameter space column records the non-automatic condition.) Note that the loci in type $0,+,\pm$ are open
in their ambient projective bundles, while the type $-$ locus is closed.

Then there is a stratification of $\Hk_{\cN}^1$ into loci $_?\Hk_{\cN}^1$ with $? \in \{0,+,-,\pm\}$ such that the morphism
\[
        (\pr_1,\pr_{X'})\co {}_?\Hk_{\cN}^1\longrightarrow \cN\times X',
\]
where $\pr_{X'}$ records the leg, realizes the stratum $_?\Hk_{\cN}^1$ as the corresponding parameter space.

\end{prop}

\begin{proof}
To complete $(\cP, \cF, h, x) \in \cN \times X'(R)$ to a point of $\Hk_{\cN}^1(R)$, we need to construct the diagram
\[
\begin{tikzcd}
 \cF' & \cF^\flat \ar[l, hook'] \ar[r, hook] & \cF  \\
\cP'  \ar[u, hook] &  & \cP \ar[u, hook]
\end{tikzcd}        
\]
where $\cF'$ is related to $\cF$ by a lower modification of $\cF$ along $y$, followed by the Hermitian-dual
upper modification at $x$. Let $H\subset\cF_y$ be the hyperplane defining the lower modification, so that 
\[
        \cF^\flat=\ker(\cF\to \cF_y/H),\qquad
        \cF'=\cF^\flat+\mathcal I_{x}^{-1}L.
\]
Then $\cF'$ carries the Hermitian structure induced by $h$.

Given $(\cP,\cF,h,x)$ and $H$, or equivalently $L=H^\perp$, at most one saturated rank-$m$ subbundle $\cP'\subset\cF'$ agrees with $\cP$ away from $\Gamma_x\cup\Gamma_y$. Indeed, if $\cP'$ and $\cP''$ both satisfy this property, then $\cP''\to\cF'\to\cF'/\cP'$ vanishes off the relative Cartier divisor $\Gamma_x\cup\Gamma_y$. Thus $\cP''\subset\cP'$, and symmetry gives equality.

To classify the modification at geometric points, assume for the moment that $R$ is a field. The Hermitian structure gives the perfect pairing
\begin{equation}\label{eq:fiber-pairing}
        \langle\ ,\ \rangle_{x,y}:\cF_x\times \cF_y\longrightarrow
        \omega_{X'}|_y .
\end{equation}
A hyperplane $H\subset\cF_y$ determines the lower modification at $y$, and the orthogonal line $H^\perp\subset\cF_x$ determines its Hermitian-dual upper modification at $x$. Thus $\cP$ and $\cP'$ agree away from $x$ and $y$, and at each of these points the change has length zero or one. At $y$ we have
\[
\cF^\flat = \ker(\cF \surj \cF_y/H).
\]
while the generic identification of $\cF'$ with $\cF^\flat$ gives
\[
\cF' = \cF^\flat + \cI_x^{-1} H^\perp
\]
where $\cI_x$ is the ideal sheaf of the graph of $x$. The two conditions $H\supset\cP_y$ and $H^\perp\subset\cP_x$ therefore determine the modification type, with the precise fiber changes recorded in \eqref{eq:fiber-change-tests}.
\begin{equation}\label{eq:fiber-change-tests}
\begin{array}{c|c}
\textup{fiber condition} & \textup{sub-bundle modification}\\
\hline
H \supset \cP_y  & \cP'_y=\cP_y\\
H \not\supset \cP_y  & \length_y(\cP/\cP')=1\\
H^\perp\not\subset \cP_x & \cP'_x=\cP_x\\
H^\perp\subset \cP_x & \length_x(\cP'/\cP)=1 .
\end{array}
\end{equation}
\begin{figure}[htbp]
 \centering
\vspace{1.2ex}
\begin{minipage}{.78\textwidth}
\centering
\begin{adjustbox}{max width=\linewidth}
\begin{tikzcd}[ampersand replacement=\&, cells={nodes={inner sep=0pt}}]
{\begin{tikzpicture}[
        x=1cm,y=1cm,
        font=\small,
        line cap=round,
        line join=round,
        pband/.style={fill=black!5},
        pboundary/.style={draw=black!65, line width=.8pt},
        pprime/.style={draw=black!48, line width=1.35pt},
        neutral cut/.style={fill=white},
        common p/.style={draw=black!80, line width=.7pt},
        base tick/.style={draw=black!45, line width=.65pt},
        brace/.style={decorate, decoration={brace, amplitude=4pt}, line width=.55pt}]
    \fill[pband] (0,0) -- (9,0) -- (9,1.55) -- (0,1.55) -- cycle;
    \draw[pboundary] (0,1.55) -- (0,0) -- (9,0) -- (9,1.55);

    \fill[pband] (1.18,1.55) rectangle (2.22,2.08);
    \fill[neutral cut] (5.72,.92) rectangle (6.76,1.55);

    \draw[pprime]
        (0,1.55) -- (1.18,1.55) -- (1.18,2.08) -- (2.22,2.08)
        -- (2.22,1.55) -- (5.72,1.55) -- (5.72,.92) -- (6.76,.92)
        -- (6.76,1.55) -- (9,1.55);
    \draw[common p] (0,1.47) -- (9,1.47);

    \draw[base tick] (1.7,0) -- (1.7,-.17);
    \draw[base tick] (6.24,0) -- (6.24,-.17);
    \node[below] at (1.7,-.2) {$x$};
    \node[below] at (6.24,-.2) {$y$};

    \node[right] at (9.16,.78) {$\cP$};
    \node[above, text=black!48] at (4.5,2.24) {$\cP'$};

    \draw[brace] (1.18,2.22) -- (2.22,2.22)
        node[midway, above=4pt] {$\length_x(\cP'/\cP)=1$};
    \draw[brace] (6.76,.82) -- (5.72,.82)
        node[midway, below=4pt] {$\length_y(\cP/\cP')=1$};
\end{tikzpicture}}
\end{tikzcd}
\end{adjustbox}
\end{minipage}
\caption{Cartoon of a modification between $\cP$ and $\cP'$. The bump at $x$ indicates $\length_x(\cP'/\cP)=1$, while the dip at $y$ indicates $\length_y(\cP/\cP')=1$.}
\label{fig:length-one-modifications}
\end{figure}
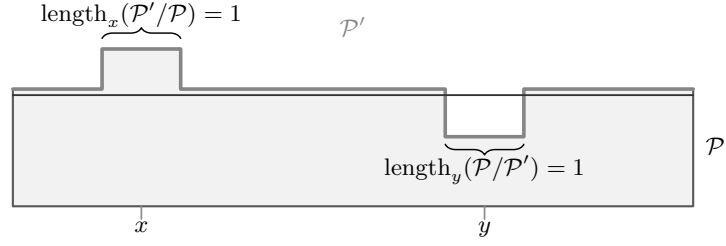
The four combinations give the following modification types.
\newcommand{\satmodpic}[2]{%
\vcenter{\hbox{%
\begin{tikzpicture}[
        x=.36cm,y=.36cm,
        line cap=round,
        line join=round,
        pband/.style={fill=black!5},
        pboundary/.style={draw=black!65, line width=.45pt},
        pprime/.style={draw=black!48, line width=.85pt},
        common p/.style={draw=black!80, line width=.4pt},
        tick/.style={draw=black!45, line width=.35pt}]
    \fill[pband] (0,0) -- (6,0) -- (6,.8)
        \ifnum#1=1
            -- (4.95,.8) -- (4.95,.25) -- (3.95,.25) -- (3.95,.8)
        \fi
        \ifnum#2=1
            -- (2.05,.8) -- (2.05,1.4) -- (1.05,1.4) -- (1.05,.8)
        \fi
        -- (0,.8) -- cycle;
    \draw[pboundary] (0,.8) -- (0,0) -- (6,0) -- (6,.8);
    \draw[pprime] (0,.8)
        \ifnum#2=1
            -- (1.05,.8) -- (1.05,1.4) -- (2.05,1.4) -- (2.05,.8)
        \fi
        -- (3.95,.8)
        \ifnum#1=1
            -- (3.95,.25) -- (4.95,.25) -- (4.95,.8)
        \fi
        -- (6,.8);
    \draw[common p] (0,.68) -- (6,.68);
    \draw[tick] (1.55,0) -- (1.55,-.12);
    \draw[tick] (4.45,0) -- (4.45,-.12);
    \node[font=\scriptsize, below] at (1.55,-.12) {$x$};
    \node[font=\scriptsize, below] at (4.45,-.12) {$y$};
\end{tikzpicture}}}}
\[
\renewcommand{\arraystretch}{1.35}
\begin{array}{@{}c|c|c|c@{}}
\textup{type} & \textup{condition at }y & \textup{condition at }x
& \begin{gathered}\textup{sub-bundle modification}\end{gathered}\\
\hline
0 & H\supset \cP_y & H^\perp\not\subset \cP_x
& \vspace{.1cm}

\satmodpic{0}{0}
\\
+ & H\not\supset \cP_y & H^\perp\not\subset \cP_x
& \vspace{.1cm}

\satmodpic{1}{0}
\\
- & H\supset \cP_y & H^\perp\subset \cP_x
& \vspace{.1cm}

\satmodpic{0}{1}
\\
\pm & H\not\supset \cP_y & H^\perp\subset \cP_x
& \vspace{.1cm}

\satmodpic{1}{1}
\end{array}
\]
These are illustrated in Figure \ref{fig:move-types}. 

For an arbitrary test ring $R$, take $?\in\{0,+,-,\pm\}$ and let $\cP$ and $H=L^\perp$ satisfy the corresponding parameter-space conditions. Define $\cP'$ by the following formula:
\[
\begin{array}{c|c}
\textup{type} & \cP'\\
\hline
0 & \cP\\
+ & \ker(\cP\to \cF_y/H)\\
- & \cP+\mathcal I_{x}^{-1}L\\
\pm & \ker(\cP\to \cF_y/H)+\mathcal I_{x}^{-1}L .
\end{array}
\]
Each formula gives a saturated rank-$m$ subbundle $\cP'\subset\cF'$ with locally free quotient $\cF'/\cP'$.

At every geometric point, exactly one of these four types occurs. Hence the substacks ${}_?\Hk_{\cN}^1$ stratify $\Hk_{\cN}^1$. Their geometry is read from the defining fiber conditions as follows.
\begin{itemize}
\item In type $0$, $H$ is a hyperplane in $\cF_y$ containing $\cP_y$ and such that $H^\perp\not\subset \cP_x$. Imposing the first condition gives the projective bundle of hyperplanes in
$\cF_y$ containing $\cP_y$, namely
$\bbP((\cF_y/\cP_y)^\vee)$, and the second condition is open.

\item In type $+$, $H$ is a hyperplane in $\cF_y$ satisfying the two open conditions $H\not\supset \cP_y$ and $H^\perp\not\subset \cP_x$.

\item To understand the type $-$ and $\pm$ spaces, it is more convenient to choose $L$ first. It must lie in the projective bundle of lines in $\cP_x$, namely $\bbP(\cP_x)$. Then the type $-$ parameter space is cut out by the closed condition
$L\subset \cP_y^\perp$, while the type $\pm$ parameter space is its open complement.

\end{itemize}
This completes the proof.
\end{proof}

\begin{figure}[htbp]
\centering
\begin{tikzpicture}[
        x=1cm,y=1cm,
        font=\small,
        fiber/.style={draw=black!45, fill=black!7, rounded corners=1pt},
        lower/.style={draw=red!70!black, line width=1.3pt, line cap=round},
        upper/.style={draw=red!70!black, line width=1.3pt, line cap=round},
        sub/.style={draw=blue!70!black, line width=1.7pt, line cap=round},
        base/.style={draw=black!25, line width=.4pt},
        panel/.style={draw=black!22, rounded corners=2pt, line width=.45pt}]
    \foreach \X/\Y/\T in {
        0/0/{0},
        4.7/0/{-},
        0/-4.2/{+},
        4.7/-4.2/{\pm}} {
        \begin{scope}[shift={(\X,\Y)}]
            \draw[panel] (-.25,-.68) rectangle (3.7,2.85);
            \draw[fiber] (0,0) rectangle (.9,2.15);
            \draw[fiber] (2.55,0) rectangle (3.45,2.15);
            \draw[base] (-.15,-.28) -- (3.6,-.28);
            \node at (1.72,2.62) {$\T$};
            \node[font=\scriptsize] at (.45,-.55) {$y$};
            \node[font=\scriptsize] at (3,-.55) {$x$};
        \end{scope}
    }

    \begin{scope}[shift={(0,0)}]
        \draw[lower] (.13,1.08) -- (.77,1.08);
        \draw[sub] (.24,1.08) -- (.66,1.08);
        \draw[sub] (2.66,.38) -- (3.43,1.80);
        \draw[upper] (2.86,1.08) -- (3.14,1.08);
    \end{scope}

    \begin{scope}[shift={(4.7,0)}]
        \draw[lower] (.13,1.08) -- (.77,1.08);
        \draw[sub] (.24,1.08) -- (.66,1.08);
        \draw[sub] (2.66,.38) -- (3.43,1.80);
        \draw[upper] (2.86,.75) -- (3.23,1.43);
    \end{scope}

    \begin{scope}[shift={(0,-4.2)}]
        \draw[lower] (.13,1.08) -- (.77,1.08);
        \draw[sub] (.30,.72) -- (.62,1.42);
        \draw[sub] (2.66,.38) -- (3.43,1.80);
        \draw[upper] (2.86,1.08) -- (3.14,1.08);
    \end{scope}

    \begin{scope}[shift={(4.7,-4.2)}]
        \draw[lower] (.13,1.08) -- (.77,1.08);
        \draw[sub] (.30,.72) -- (.62,1.42);
        \draw[sub] (2.66,.38) -- (3.43,1.80);
        \draw[upper] (2.86,.75) -- (3.23,1.43);
    \end{scope}
\end{tikzpicture}
\caption{Cartoon of the four possible modification types. The grey rectangles depict the fibers of $\cF_i$ at $x$ and $y$. The blue line depicts the fibers of $\cP$ at $x$ and $y$. The red line over $y$ depicts $H$ and the red line over $x$ depicts $H^\perp$.}
\label{fig:move-types}
\end{figure}
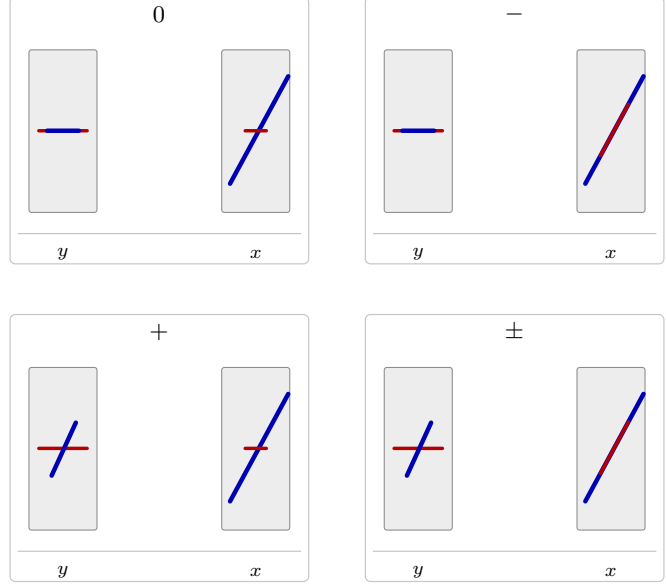

\begin{defn}
Let $\tau \co \{ 1, \ldots, r\} \rightarrow \{0, +, -, \pm\}$ be a labeling of the modifications by the four possible types. Equivalently, we may think of $\tau$ as a tuple $(I_0, I_+, I_-, I_{\pm})$ forming a partition of $\{1, \ldots, r\}$, with $I_?$ designating the indices of type $?$. 

We define $_\tau \Hk_{\cN}^r$ to be the locally closed stratum of $\Hk_{\cN}^r$ where the modifications in $I_?$ lie in the stratum of type $? \in \{0, +, -, \pm\}$. 
\end{defn}

\begin{cor}\label{cor:local-envelope}
Let $\tau=(I_0,I_+,I_-,I_\pm)$ be a partition of
$\{1,\ldots,r\}$. For every geometric point $w\in{}_\tau\Hk_{\cN}^r$
there exist
\begin{itemize}
\item a scheme $U$, 
\item a smooth morphism $q:U\longrightarrow{}_\tau\Hk_{\cN}^r$ 
whose image contains $w$, 
\item a locally closed immersion $U\hookrightarrow\overline U$,
and 
\item a smooth morphism $\overline U\longrightarrow\cN\times(X')^r $ 
extending $(\pr_r,\pr_{\mathrm{legs}})\circ q$, of relative dimension $D_\tau+e$, where
\begin{equation}\label{eq:local-envelope-dim}
        D_\tau:=\sum_{i=1}^r d_{\tau(i)}
        =(n-m-1)|I_0|+(n-1)|I_+|
        +(m-1)|I_-|+(m-1)|I_\pm|,
\end{equation}
and $e$ is the relative dimension of $q$ on the chosen component of $U$.
\end{itemize} 
\end{cor}

\begin{proof}
The result is trivial if $r =0$, so we henceforth assume that $r \geq 1$. 

For $1\leq a\leq r$, let $C_{[a,r]}$ denote the stratum of
chains
\[
(\cP_{a-1},\cF_{a-1})\dashrightarrow\cdots
\dashrightarrow(\cP_r,\cF_r)
\]
whose $i$-th modification has type $\tau(i)$. We prove, by
descending induction on $a$, the following statement: after
pullback to a smooth scheme chart $U\to C_{[a,r]}$ of relative
dimension $e$, every point of $U$ has a neighborhood locally closed in
a scheme smooth over $\cN\times (X')^{r-a+1}$ (via the last projection $\pr_r$ times the legs $x_a,\ldots,x_r$) of relative dimension
$\sum_{i=a}^r d_{\tau(i)}+e$. This is the assertion of the corollary when
$a=1$.

For each type $?\in\{0,+,-,\pm\}$, write
\[
        B_?\longrightarrow\cN\times X'
\]
for the ambient projective bundle appearing in
Proposition~\ref{prop:one-step-sat}. Its relative dimension is $d_?$ where
\[
        d_0=n-m-1,\qquad d_+=n-1,\qquad
        d_-=d_\pm=m-1.
\]
The base case $a=r$ is Proposition~\ref{prop:one-step-sat}: the
one-step stratum becomes, after pullback to a smooth scheme chart of
$\cN\times X'$, locally closed in the corresponding ambient projective
bundle $B_{\tau(r)}$. The chart dimension contributes the additional
$e$.

Suppose the assertion has been proved for $C_{[a+1,r]}$. Fix a point in $C_{[a,r]}$, and choose a point of a smooth scheme chart
$U\to C_{[a+1,r]}$ above its image, together with a locally closed immersion $U\hookrightarrow \overline U$,
where $\overline U$ is smooth over $\cN\times(X')^{r-a}$ of relative
dimension $\sum_{i=a+1}^r d_{\tau(i)}+e$. Put
\[
        \widetilde U:=C_{[a,r]}\times_{C_{[a+1,r]}}U.
\]
The morphism $\widetilde U\to C_{[a,r]}$ is smooth of relative dimension
$e$. The map
$\pr_a:C_{[a+1,r]}\to\cN$ pulls the ambient projective bundle
$B_{\tau(a)}$ back to a projective bundle
\[
        Q\longrightarrow U\times X'
\]
of relative dimension $d_{\tau(a)}$. By base change from
Proposition~\ref{prop:one-step-sat}, $\widetilde U$ is locally closed in
$Q$.

Choose a Zariski open neighborhood $V\subset U\times X'$ of the chosen
point on which the projective bundle $Q$ is trivial. Since $U\hookrightarrow\overline U$ is
locally closed, after shrinking around its image we may write
$V=\overline V\cap(U\times X')$ for an open
$\overline V\subset\overline U\times X'$. The trivialization identifies
$Q|_V$ with $\bbP^{d_{\tau(a)}}_V$, so
$\widetilde U|_V$ is locally closed in
$\bbP^{d_{\tau(a)}}_{\overline V}$. This projective-space bundle is
smooth over $\cN\times(X')^{r-a+1}$ of relative dimension
\[
        d_{\tau(a)}+\sum_{i=a+1}^r d_{\tau(i)}+e.
\]
Thus $\widetilde U|_V\to C_{[a,r]}$ and its embedding in $\bbP^{d_{\tau(a)}}_{\overline V}$ provide the
required chart to prove the induction step.
\end{proof}

\subsection{Framed shtukas}
We define the stack of framed shtukas $\Sht_{\cN}^r $ as the fibered product
\begin{equation}\label{eq:framed-shtuka-square}
\begin{tikzcd}
        \Sht_{\cN}^r \ar[r] \ar[d] & \Hk_{\cN}^r \ar[d,"{(\pr_0,\pr_r)}"] \\
        \cN \ar[r,"{(\Id, \Frob)}"] & \cN \times \cN
\end{tikzcd}
\end{equation}
For $\tau \co \{1, \ldots, r \} \rightarrow \{0,+, -, \pm \}$ we define ${}_\tau \Sht_{\cN}^r$ as the fibered product
\begin{equation}\label{eq:framed-shtuka-stratum-square}
\begin{tikzcd}
        _\tau \Sht_{\cN}^r \ar[r] \ar[d]  & _\tau \Hk_{\cN}^r \ar[d,"{(\pr_0,\pr_r)}"] \\
        \cN \ar[r,"{(\Id, \Frob)}"] & \cN \times \cN
\end{tikzcd}
\end{equation}
As $\tau$ varies over all such maps the $_\tau \Hk_{\cN}^r$ form a locally closed stratification of $ \Hk_{\cN}^r$, hence the $_\tau \Sht_{\cN}^r$ form a locally closed stratification of $\Sht_{\cN}^r$.

\begin{lemma}\label{lem:balance}
Let $\tau = (I_0, I_+, I_-, I_{\pm})$. If $_\tau \Sht^r_{\cN}$ is non-empty, then $|I_+|=|I_-|$.
\end{lemma}

\begin{proof}
Suppose $_\tau \Sht^r_{\cN}$ has a geometric point $(x_{\bu}, \cP_{\bu}, \cF_{\bu})$. Then $\cP_r \cong (\Id_{X'}\times\Frob_{\bar k})^*\cP_0$, so their total degrees on $X'$ agree. On the other hand, we have
\[
\deg \cP_0 - \deg \cP_r = |I_-| - |I_+|.
\]
Here $\Frob_{\bar k}$ is the $q$-Frobenius of the geometric test point in \eqref{eq:framed-shtuka-square}. Since $\cP_r \cong \Frob^* \cP_0$, we have $\deg \cP_r = \deg \cP_0$ hence if $_\tau \Sht^r_{\cN}$ is non-empty then we must have $|I_-| = |I_+|$.
\end{proof}

The following lemma generalizes \cite[Lemma 9.3]{FYZ} slightly: compared to \emph{loc. cit.} the smooth envelope is allowed to contain $U$ only as a locally closed subscheme, and we include an auxiliary $T$-factor.

\begin{lemma}\label{lem:VL-local}
Fix an algebraic closure $\ol{k}$ of $k$. For a $\ol{k}$-scheme $Z$, write
$Z^{(1)}:=Z\times_{\Spec\ol{k},\Frob_q}\Spec\ol{k}$ for its $q$-Frobenius twist and
$\Frob:Z\to Z^{(1)}$ for the relative $q$-Frobenius. Let $W,Z,T$ be finite-type schemes over $\ol{k}$, and let
\[
        (h_0,h_1,h_T):
        W\longrightarrow Z\times Z^{(1)}\times T
\]
be a morphism. Suppose that every point $w\in W$ has an open
neighborhood $U\subset W$ for which there exist a finite-type scheme
$\overline U$, a locally closed immersion $j :U\hookrightarrow\overline U$, and a smooth morphism
\[
        \overline h=(\overline h_1,\overline h_T):
        \overline U\longrightarrow Z^{(1)}\times T
\]
of relative dimension $d$, such that $\overline h\circ j=(h_1,h_T)|_U$. Define $V$ by the Cartesian square
\begin{equation}\label{eq:VL-square}
\begin{tikzcd}[ampersand replacement=\&]
V \arrow[r] \arrow[d] \&
W \arrow[d,"{(h_0,h_1)}"]\\
Z \arrow[r,"{(\Id,\Frob)}"'] \&
Z\times Z^{(1)} .
\end{tikzcd}
\end{equation}
Then every fiber of $V\to T$ has dimension at most $d$.
\end{lemma}

\begin{proof}
Fix a geometric point $v\in V_t$, with image $w\in W$ and $t \in T$. It is enough to prove $\dim_vV_t\le d$. Choose $U\hookrightarrow\overline U$ at $w$ as in the hypothesis and replace $W$ by $U$. We may then work with a single smooth morphism
\[
        \overline W\longrightarrow Z^{(1)}\times T
\]
with $W$ locally closed in $\overline W$. For the purpose of proving the Lemma, we may make the following reductions:
\begin{itemize}
\item Base change to a geometric point $\Spec\Omega\to T$ and replace $\ol k$ by $\Omega$, reducing to the case where $T=\Spec\Omega$.  
\item Replace $Z$ by an affine neighborhood $Z'$ of $h_0(w)$, and replace $\overline W$ by $\overline h_1^{-1}((Z')^{(1)})$ and $W$ by $h_0^{-1}(Z')\cap\overline h_1^{-1}((Z')^{(1)})$, reducing to the case where $Z$ is affine. 
\end{itemize} 
Since $Z$ is affine, we may write
\[
        Z=\Spec R,\qquad \text{where} \qquad 
        R=\Omega[x_1,\ldots,x_l]/I .
\]
Then
\[
        Z^{(1)}=\Spec R^{(1)},\qquad\text{where} \qquad 
        R^{(1)}\cong\Omega[\xi_1,\ldots,\xi_l]/I^{(1)},
\]
where $\xi_i=1\otimes x_i$ and $\Frob^*(\xi_i)=x_i^q$. After shrinking around $w$, we may assume that $\overline W$ is affine and $W\subset\overline W$ is closed. The standard local form of the smooth morphism $\overline h_1:\overline W\to Z^{(1)}$ of relative dimension $d$ implies that there is a presentation
\[
        \overline W=
        \Spec
        \frac{\Omega[\xi_1,\ldots,\xi_l,
        y_1,\ldots,y_{s+d}]_\Delta}
        {(I^{(1)}(\xi),r_1,\ldots,r_s)}
\]
for some element $\Delta$, with
\[
        J:=
        \det\left(\frac{\partial r_a}{\partial y_b}\right)_{1\le a,b\le s}
\]
invertible on $\overline W$.

Let $\bar f_i=h_0^*x_i\in\Gamma(W,\cO_W)$.  Because $W$ is closed in
$\overline W$, we may choose lifts
\[
        f_i\in
        \Omega[\xi_1,\ldots,\xi_l,y_1,\ldots,y_{s+d}]_\Delta
\]
of the functions $\bar f_i$.  Set
\[
        g_i:=\xi_i-f_i^q,\qquad 1\le i\le l .
\]
By \eqref{eq:VL-square}, the equations $\xi_i=\bar f_i^q$ cut out $V$ on $W$. Hence, on this affine neighborhood, $V$ is a closed subscheme of
\[
        U_{\mathrm{amb}}:=
        \Spec
        \frac{\Omega[\xi_1,\ldots,\xi_l,
        y_1,\ldots,y_{s+d}]_\Delta}
        {(g_1,\ldots,g_l,r_1,\ldots,r_s)} .
\]

Apart from the leading term $\xi_i$, each $g_i$ depends on the variables through $q$th powers. In characteristic $p$, we have
 $\frac{\partial g_i}{\partial \xi_j}=\delta_{ij}$ and $\frac{\partial g_i}{\partial y_b}=0$, so the Jacobian matrix for the equations
$g_1,\ldots,g_l,r_1,\ldots,r_s$, using the column order
\[
        \xi_1,\ldots,\xi_l
        \,;\,
        y_1,\ldots,y_s
        \,;\,
        y_{s+1},\ldots,y_{s+d},
\]
has the block form
\[
\begin{pmatrix}
\Id_l & 0 & 0\\
* & \left(\frac{\partial r_a}{\partial y_b}\right)_{1\le a,b\le s} & *
\end{pmatrix}.
\]
The matrix has rank $l+s$, so $U_{\mathrm{amb}}$ is smooth of dimension $(l+s+d)-(l+s)=d$. Since $V$ is locally a closed subscheme of $U_{\mathrm{amb}}$, this gives $\dim_vV_t\le d$, proving the lemma.
\end{proof}

\begin{cor}[Stratum bound]\label{cor:stratum-bound} Let $\tau = (I_0, I_+, I_-, I_\pm )$. Then we have
\begin{equation}\label{eq:stratum-bound}
\dim \left(_\tau \Sht_{\cN}^r \right) \leq (n-m)|I_0|+n|I_+|+m|I_-|+m|I_\pm|.
\end{equation}
\end{cor}

\begin{proof}
Let $D_\tau$ be the integer defined in
\eqref{eq:local-envelope-dim}, and let $T:=(X')^r$. We first show that every geometric fiber of ${}_\tau\Sht_{\cN}^r\longrightarrow T$
has dimension at most $D_\tau$.

We may base change to $\ol{k}$. The assertion is local on $T$.
Given a geometric point of $T$, choose a closed point
$v\in |X|$ whose inverse image in $X'$ is disjoint from all of its coordinates, and replace $T$ by the open neighborhood
\[
        T_v:=(X'\setminus\nu^{-1}(v))^r.
\]
As $v$ varies, these opens cover $T$.

It is also enough to work on finite-type truncations. Choose a Frobenius-stable
truncation datum $\mu$, consisting of a Harder--Narasimhan bound for the
$\cF_i$'s together with a lower bound for the locally constant integers
$\deg\cP_i$, and restrict to the locus on which all the framed bundles
\[
        (\cP_i,\cF_i),\qquad 0\leq i\leq r,
\]
are bounded by $\mu$. As $\mu$ varies, these loci exhaust
${}_\tau\Sht_{\cN}^r$.

Impose a sufficiently deep level structure at $v$, as in
\cite[proof of Proposition~9.1(2)]{FYZ}. Denote the resulting
finite-type scheme covering the corresponding truncation of
$\cN$ by $Z$, and denote the corresponding level Hecke stratum by
$W$. Since the legs lie in $T_v$, the level structure transports
through every modification, and we obtain morphisms
\[
        h_0:W\longrightarrow Z,\qquad
        h_r:W\longrightarrow Z^{(1)},\qquad
        h_T:W\longrightarrow T_v.
\]
Let $V$ be defined by the Cartesian square
\[
\begin{tikzcd}[ampersand replacement=\&]
V \arrow[r] \arrow[d] \&
W \arrow[d,"{(h_0,h_r)}"]\\
Z \arrow[r,"{(\Id,\Frob)}"'] \&
Z\times Z^{(1)} .
\end{tikzcd}
\]

Corollary~\ref{cor:local-envelope} implies that every point of $W$ has an open
neighborhood locally closed in a finite-type scheme smooth over $Z^{(1)}\times T_v$ of relative dimension $D_\tau$. Lemma~\ref{lem:VL-local} therefore shows that
every fiber of $V\longrightarrow T_v$ 
has dimension at most $D_\tau$.

Since forgetting level structure is a finite \'etale map, it does not affect fiber dimensions. It follows that
every geometric fiber of the map
${}_\tau\Sht_{\cN}^r\longrightarrow T$
has dimension at most $D_\tau$.

Since $\dim T=r$, we conclude that
\[
\dim({}_\tau\Sht_{\cN}^r)
 \leq D_\tau+r = (n-m)|I_0|+n|I_+|
       +m|I_-|+m|I_\pm|,
\]
as claimed.
\end{proof}

\subsubsection{Special cycles}

We prove Theorem \ref{thm:special-cycle-dimension-bound}. Since the coefficient map lands in the finite discrete set $\cA_\cE(k)$, the injective locus $\cZ_\cE^{r,\circ}$ is the finite disjoint union of the open-and-closed substacks $\cZ_\cE^r(a)^\circ$. It suffices to prove that if $m \leq n/3$, then $(\cZ_{\cE}^r(a)^\circ)_{\red}$ is exhausted by finite type open substacks with dimension $\leq r(n-m)$. We may thus impose a Harder--Narasimhan truncation and pass to a finite type open substack. Then we may apply Lemma \ref{lem:stratify-saturation} to the restriction of the universal family $\{t_{\bu} \co \cE \inj \cF_{\bu}\}$ (on a finite cover by a scheme) to obtain finitely many reduced locally closed strata where the saturation $\cP_{\bu} \subset \cF_{\bu}$ exists.

On each such reduced stratum, the saturations are compatible with the chain of modifications, thanks to the compatibility condition between each pair $t_{i-1}$ and $t_i$. Thus the universal data on the stratum define a morphism to $\Hk_{\cN}^r$, and the shtuka condition promotes it to a morphism to $\Sht_{\cN}^r$. We now refine by pulling back the locally closed stratification $\{{}_\tau\Hk_{\cN}^r\}_\tau$; on each refined stratum the type $\tau(i)\in\{0,+,-,\pm\}$ is constant for every leg.
It suffices to show that each such refined stratum has dimension $\leq r(n-m)$.

Pick any such stratum $\cS$, and let $\tau  = (I_0, I_+, I_-, I_{\pm})$ be its type. Then we obtain a map $\cS \rightarrow {_{\tau}\Sht_{\cN}^r}$. 

\begin{lemma}\label{lem:quasi-finite}
The map 
\begin{equation}\label{eq:factor-through-framed-shtukas}
\cS \rightarrow {_\tau \Sht_{\cN}^r}
\end{equation}
is quasi-finite.
\end{lemma}

\begin{proof}
Consider the composition
\[
\cS \rightarrow {_\tau \Sht_{\cN}^r} \rightarrow \Sht_{U(n)}^r.
\]
The composition $\cS\to\Sht_{U(n)}^r$ is obtained from the coefficient-$a$
injective-locus map
$\cZ_{\cE}^r(a)^\circ\to\Sht_{U(n)}^r$ by restricting to a locally closed stratum. Since
that map is quasi-finite by \cite[Proposition 7.5]{FYZ}, the composition
$\cS\to\Sht_{U(n)}^r$ is quasi-finite.  Because this composition factors
through the map \eqref{eq:factor-through-framed-shtukas}, every fiber of
\eqref{eq:factor-through-framed-shtukas} is contained in a fiber of
$\cS\to\Sht_{U(n)}^r$.  Hence \eqref{eq:factor-through-framed-shtukas} is
quasi-finite.
\end{proof}

\begin{proof}[Completion of the proof of Theorem \ref{thm:special-cycle-dimension-bound}]
Thanks to Lemma~\ref{lem:quasi-finite}, and after pulling back along smooth
charts of the target so that quasi-finite maps do not increase dimension, it
suffices to show that if
${}_\tau \Sht_{\cN}^r$ is non-empty, then
$\dim({}_\tau \Sht_{\cN}^r) \leq r(n-m)$. Let
$\tau = (I_0, I_+, I_-, I_\pm)$ and
\[
        u:=|I_+|=|I_-|,\qquad b:=|I_\pm|
\]
where the equality $|I_+|=|I_-|$ is Lemma~\ref{lem:balance}.

Corollary~\ref{cor:stratum-bound} gives
\begin{equation}\label{eq:dimension-after-balance}
        \dim ({}_\tau \Sht_{\cN}^r) \le (n-m)|I_0|+(n+m)u+mb .
\end{equation}
Since $r=|I_0|+2u+b$, we have
\begin{align}\label{eq:numerical-inequality}
&(n-m)r-\bigl((n-m)|I_0|+(n+m)u+mb\bigr) \nonumber \\
&\qquad=(n-m)(|I_0|+2u+b)-(n-m)|I_0|-(n+m)u-mb \nonumber \\
&\qquad=(n-3m)u+(n-2m)b .
\end{align}
The integers $u$ and $b$ are nonnegative.  The hypothesis $3m\le n$
implies $n-3m\ge0$, and also $n-2m=(n-3m)+m\ge0$.  Hence \eqref{eq:numerical-inequality} is nonnegative, giving
\[
        (n-m)|I_0|+(n+m)u+mb\le (n-m)r .
\]
Together with \eqref{eq:dimension-after-balance}, this proves
$\dim({}_\tau \Sht_{\cN}^r)\le (n-m)r$.
\end{proof}

\section{Proof of the Trace Conjecture in low corank}\label{sec: trace conjecture}

In \cite[\S 4--6]{FK}, the \emph{motivic sheaf-cycle correspondence} is developed, refining the $\ell$-adic version developed in \cite[\S 3--5]{FYZ3}. We use the notation and terminology of \cite{FK}, including:
\begin{itemize}
\item The notion of \emph{cohomological correspondences} between motivic sheaves.
\item The notion of \emph{pushable/pullable} squares and \emph{pushable/pullable} maps of correspondences.
\item The definition of \emph{trace} of a self-cohomological correspondence between \emph{geometric} motivic sheaves.
\item The shtuka-twisted trace.
\end{itemize}
The purpose of this section is to prove the low-corank form of the Trace Conjecture, Theorem \ref{thm:trace-conjecture-n/3}, which realizes the derived fundamental classes of special cycles as the shtuka-twisted traces of certain cohomological correspondences.
Notably, we expect that the argument will generalize well to other settings, such as orthosymplectic groups.

\subsection{Formulation of the Trace Conjecture}\label{ssec: trace conjecture formulation} Let $\cM := \cM_{H_1, H_2}$ be a derived Hitchin stack as in \cite[Definition 5.23]{FYZ2}. Here $BH_1$ is a smooth $m$-framed gerbe over $X$ and $BH_2$ is a smooth $n$-framed gerbe of unitary type over $X$. Let $\Hk_{\cM}^r$ be the derived Hecke stack for $\cM$ as defined in \cite[\S 5.6]{FYZ2}. Then we have a correspondence
\[
\begin{tikzcd}
\cM & \Hk_{\cM}^r \ar[l, "h_0"'] \ar[r, "h_r"]  &  \cM
\end{tikzcd}
\]

\begin{lemma}\label{lem: hecke quasismooth} The map $h_r \co \Hk_{\cM}^r \rightarrow \cM$ is quasi-smooth.
\end{lemma}

\begin{proof}
By \cite[Proposition 5.35]{FYZ2}, $\cM$ is the derived vector bundle over $\Bun_{H_1} \times \Bun_{H_2}$ associated to a certain perfect complex $\cK$, and $\Hk_{\cM}^r$ is the derived vector bundle over $\Bun_{H_1} \times \Hk_{H_2}^r$ associated to a certain perfect complex $\cK^\flat$. Let $\wt{\cM}  := \Tot_{\Bun_{H_1}\times\Hk_{H_2}^r} ( h_r^* \cK)$ be the base change of $\cM$ to $\Bun_{H_1}\times\Hk_{H_2}^r$. Since the map $\Hk_{H_2}^r \rightarrow \Bun_{H_2}$ is smooth by \cite[Lemma 2.6]{FYZ2}, the map $\wt{\cM} \rightarrow \cM$ is smooth. Hence it suffices to show that the map $\Hk_{\cM}^r \rightarrow \wt{\cM}$ is quasi-smooth. This follows from the observation that the cone of $\cK^\flat \rightarrow h_r^* \cK$ is represented by a locally free sheaf, since it arises as the relative cohomology of $X'$ with coefficients in a torsion coherent sheaf, as explained in the proof of \cite[Lemma 9.1.3(1)]{FYZ3}. Equivalently, $\Hk_{\cM}^r\to\wt{\cM}$ is the derived base change of the zero section of the vector bundle associated to this cone, so it is quasi-smooth.
\end{proof}

Consequently, the map of correspondences
\[
\begin{tikzcd}
\cM \ar[d] & \Hk_{\cM}^r \ar[l, "h_0"'] \ar[r, "h_r"]  \ar[d]  &\cM\ar[d]  \\
\pt & \pt \ar[l] \ar[r] & \pt
\end{tikzcd}
\]
is right pullable, so the pullback $\Corr_{\pt}(\Qsh{\pt}, \Qsh{\pt}) \rightarrow \Corr_{\Hk_{\cM}^r }(\Qsh{\cM}, \Qsh{\cM}\tw{-d(h_r)})$ exists. We define $\cc_{\cM} \in \Corr_{\Hk_{\cM}^r }(\Qsh{\cM}, \Qsh{\cM}\tw{-d(h_r)})$ to be the pullback of the identity cohomological correspondence $\cc_{\pt} \in \Corr_{\pt}(\Qsh{\pt}, \Qsh{\pt}) $. More explicitly, $\cc_{\cM}$ is the composition
\[
h_0^* \Qsh{\cM} \cong \Qsh{\Hk_{\cM}^r} \cong h_r^* \Qsh{\cM} \xrightarrow{\gys_{h_r}} h_r^! \Qsh{\cM}\tw{-d(h_r)}.
\]

\begin{conj}[The Trace Conjecture for Hitchin spaces, {\cite[\S 9.4]{FH}}]\label{thm: trace conjecture} For any $r \geq 0$, and for any $m \in [1, n]$, we have
\begin{equation}\label{eq: trace conjecture}
\Tr^{\Sht}(\cc_{\cM}) = [\Sht_{\cM}^r] \in \CH_{d(h_r)}(\Sht_{\cM}^r).
\end{equation}
\end{conj}

\begin{thm}[Trace Conjecture in low corank]\label{thm:trace-conjecture-n/3} For any $r \geq 0$ and for $m \leq n/3$, \eqref{eq: trace conjecture} holds true.
\end{thm}

\subsection{Localized version}\label{ssec: localized trace conjecture} We formulate an equivalent version of Conjecture \ref{thm: trace conjecture} which is localized over a point of $\Bun_{H_1}$. The stack $\cM$ comes equipped with maps to $\Bun_{H_1}$, such that
\begin{equation}\label{eq: Hk to Bun_H1}
\text{$\Hk_{\cM}^r \xrightarrow{h_i} \cM \rightarrow \Bun_{H_1}$ is independent of $i \in \{0, \ldots, r\}$.}
\end{equation}
Let $\cE \in \Bun_{H_1}(\F_q)$, let $\cM_{\cE}$ be the derived fiber of $\cM$ over $\cE$, and let $\Hk_{\cM_{\cE}}^r$ be the derived fiber of $\Hk_{\cM}^r$ over $\cE$.

By Lemma \ref{lem: hecke quasismooth}, the base changed map $h_r \co \Hk_{\cM_{\cE}}^r \rightarrow \cM_{\cE}$ is quasi-smooth, so we can define
\[
\cc_{\cM_{\cE}} \in \Corr_{\Hk_{\cM_{\cE}}^r}(\Qsh{\cM_{\cE}}, \Qsh{\cM_{\cE}}\tw{-d(h_r)})
\]
to be the pullback of the identity correspondence in $\Corr_{\pt}(\Qsh{\pt}, \Qsh{\pt})$.

\begin{thm}\label{thm: trace conjecture E} For any $r \geq 0$, and for $m \leq n/3$, we have
\begin{equation}\label{eq: trace conjecture E}
\Tr^{\Sht}(\cc_{\cM_{\cE}}) = [\cZ_{\cE}^r] \in \CH_{d(h_r)}(\cZ_{\cE}^r).
\end{equation}
\end{thm}

\begin{lemma}\label{lem:trace-local-global} For all $1 \leq m \leq n$, the identity \eqref{eq: trace conjecture} is equivalent to the collection of identities \eqref{eq: trace conjecture E} for all $\cE \in \Bun_{H_1}(\F_q)$.
\end{lemma}

\begin{proof}
Note that $\Sht_{\cM}^r$ is fibered over the fixed-point stack of $\Bun_{H_1}$. For $\cE \in \Bun_{H_1}(\F_q)$, put $\ol{\cE}:=[\{\cE\}/\Aut(\cE)]$ and let $\Gamma_{\cE}:=\Aut(\cE)(\F_q)$. The corresponding open-and-closed summand of $\Sht_{\cM}^r$ is
\[
\cZ_{\ol{\cE}}^r:=[\cZ_{\cE}^r/\Gamma_{\cE}],
\]
where the quotient is by the finite group of Frobenius-fixed automorphisms of $\cE$. This gives the decomposition
\begin{equation}\label{eq: trace reformulation 1}
[\Sht_{\cM}^r] = \sum_{\cE \in \Bun_{H_1}(\F_q)} [\cZ_{\ol{\cE}}^r] \in \CH_{d(h_r)}(\Sht_{\cM}^r).
\end{equation}
The RHS is an infinite sum, but over each connected component of $\Sht_{\cM}^r$ only finitely many summands are supported. Equivalently, only finitely many summands are supported on any finite-type truncation.

Consider the locally closed embedding $\ol{\cE} \inj \Bun_{H_1}$. We write $\cM_{\ol{\cE}}:=[\cM_{\cE}/\Aut(\cE)]$ for the stack-theoretic quotient by the full automorphism group scheme of $\cE$, and let $\cc_{\cM_{\ol{\cE}}}$ be the induced pullback of $\cc_{\cM}$. By \eqref{eq: Hk to Bun_H1}, $\Hk_{\cM}^r$ \emph{stabilizes} (cf. \cite[Definition 7.2.1]{FK}) $\cM_{\ol{\cE}}$, so its Frobenius twist is \emph{contracting} near $\cM_{\ol{\cE}}$ (cf. \cite[Example 7.2.2]{FK}). Hence \cite[Theorem 7.5.1]{FK} implies that
\begin{equation}\label{eq: trace reformulation 2}
\Tr^{\Sht}(\cc_{\cM})|_{\cZ_{\ol{\cE}}^r} = \Tr^{\Sht}(\cc_{\cM_{\ol{\cE}}}) \in \CH_{d(h_r)} (\cZ_{\ol{\cE}}^r).
\end{equation}

Let $q_{\cE}\co \cZ_{\cE}^r\to \cZ_{\ol{\cE}}^r$ be the natural finite \'{e}tale $\Gamma_{\cE}$-torsor. The smooth $\Aut(\cE)$-torsor $\cM_{\cE}\to\cM_{\ol{\cE}}$ induces a Cartesian map of correspondence data by \eqref{eq: Hk to Bun_H1}, which induces $q_{\cE}$ by formation of shtuka spaces. Compatibility of trace with smooth pullbacks, \cite[Proposition 6.2.2]{FK}, gives
\begin{equation}\label{eq: trace reformulation 3}
q_{\cE}^*\Tr^{\Sht}(\cc_{\cM_{\ol{\cE}}}) = \Tr^{\Sht}(\cc_{\cM_{\cE}})\in \CH_{d(h_r)} (\cZ_{\cE}^r).
\end{equation}
Moreover $q_{\cE}^*[\cZ_{\ol{\cE}}^r]=[\cZ_{\cE}^r]$, and with $\ol{\Q}$-coefficients the transfer identity $q_{\cE,!}q_{\cE}^*=|\Gamma_{\cE}|\operatorname{id}$ makes $q_{\cE}^*$ injective. Hence
\begin{equation}\label{eq: trace reformulation 4}
\Tr^{\Sht}(\cc_{\cM_{\ol{\cE}}}) = [\cZ_{\ol{\cE}}^r] \iff \Tr^{\Sht}(\cc_{\cM_{\cE}}) = [\cZ_{\cE}^r].
\end{equation}

If \eqref{eq: trace conjecture} holds, then we obtain \eqref{eq: trace conjecture E} by restricting to $\cZ_{\ol{\cE}}^r$ and then pulling back along $q_{\cE}$, using \eqref{eq: trace reformulation 2} and \eqref{eq: trace reformulation 3} to compute the restrictions. Conversely, if \eqref{eq: trace conjecture E} holds for every $\cE \in \Bun_{H_1}(\F_q)$, then \eqref{eq: trace reformulation 4} gives the desired equality on every open-and-closed summand $\cZ_{\ol{\cE}}^r$; summing these identities by \eqref{eq: trace reformulation 1} gives \eqref{eq: trace conjecture}.

\end{proof}

Henceforth we will focus on proving Theorem \ref{thm: trace conjecture E}. (In any case, this is the version that we will actually use later.)

\begin{remark}
The split-cover refinement needed for supermodularity is
stated separately as Theorem \ref{thm:split-trace-conjecture-n/3} in
Part \ref{part: supermodularity}.
\end{remark}

\subsection{The injective locus}
Recall that $\cM_{\cE}$ parametrizes an $H_2$-bundle $\cF$ along with a map $t \co \cE \rightarrow \cF$. Let $\cM_{\cE}^\circ \subset \cM_{\cE}$ be the locus where the map $t$ is injective (as a map of coherent sheaves) after restriction to all geometric points of the test scheme. This is the fiber over $\{\cE\}$ of \cite[Definition 3.4]{FYZ2}. This is an open subset preserved by the correspondence $\Hk_{\cM_{\cE}}^r$, in the sense that the restriction gives a map of correspondences
\[
\begin{tikzcd}
\cM_{\cE}^\circ \ar[d, hook]  & \Hk_{\cM_{\cE}^\circ}^r  \ar[l, "h_0"'] \ar[d, hook] \ar[r, "h_r"] & \cM_{\cE}^\circ \ar[d, hook]  \\
\cM_{\cE} & \Hk_{\cM_{\cE}}^r \ar[l, "h_0"'] \ar[r, "h_r"] & \cM_{\cE}
\end{tikzcd}
\]
with both squares Cartesian, hence pushable and pullable. In particular, we have the pullback cohomological correspondence $\cc_{\cM_{\cE}^\circ} \in \Corr_{\Hk_{\cM_{\cE}^\circ}^r }(\Qsh{\cM_{\cE}^\circ}, \Qsh{\cM_{\cE}^\circ} \tw{-d(h_r)})$ of $\cc_{\cM_{\cE}}$.

\begin{prop}\label{prop: trace of injective part}
For $m \leq n/3$, we have
\begin{equation}\label{eq: trace of injective part}
\Tr^{\Sht}(\cc_{\cM_{\cE}^\circ}) = [\cZ_{\cE}^{r, \circ}] \in \CH_{r(n-m)}(\cZ_{\cE}^{r, \circ}).
\end{equation}
\end{prop}

\begin{proof}
We can write $\cZ_{\cE}^{r, \circ}$ as the increasing union of Harder--Narasimhan truncations $\cZ_{\cE}^{r, \circ,\leq \nu}$. Here, for any of the fibers $\cM_{\star}$ below, the bounded Hecke correspondence means
\[
\Hk_{\cM_{\star}}^{r,\leq\nu}:=\bigcap_{i=0}^r h_i^{-1}(\cM_{\star}^{\leq\nu}),
\]
and $\cZ_{\star}^{r,\leq\nu}$ is its shtuka fixed-point stack. Thus $\cc_{\cM_{\star}^{\leq\nu}}$ denotes the open restriction of $\cc_{\cM_{\star}}$ to this bounded correspondence; a single Hecke modification need not preserve an HN bound, but these bounded open correspondences exhaust $\Hk_{\cM_{\star}}^r$ as $\nu$ varies. The argument below uses the identification
\begin{equation}\label{eq: injective HN Chow limit}
\CH_{r(n-m)}(\cZ_{\cE}^{r, \circ})  = \limit_{\nu} \CH_{r(n-m)}(\cZ_{\cE}^{r, \circ, \leq \nu}),
\end{equation}
so it suffices to show \eqref{eq: trace of injective part} after pullback to every Harder--Narasimhan truncation.

Fix the HN truncation $\nu$. We first choose a closed point $x_0\in|X|$ as follows. By Theorem \ref{thm:special-cycle-dimension-bound}, $\dim \cZ_{\cE}^{r,\circ,\le\nu}\le r(n-m)$, and the finite-type truncation has finitely many irreducible components of dimension $r(n-m)$. For each such component and each leg projection to $X'$, either the projection is dominant or its image is a single point of $X'$. Therefore, there is a finite subset $B\subset |X'|$ of points that can occur as the image of a non-dominant leg projection from an $r(n-m)$-dimensional component. Choose $x_0$ such that $\nu^{-1}(x_0)\cap B=\varnothing$; choosing $x_0$ on $X$ makes every divisor $D=d x_0$ and $\nu^*D$ defined over $\F_q$ and $\sigma$-invariant. Given this $x_0$, we may choose $d\gg0$ so that for $\wt{\cE}:=\cE(-\nu^*D)$, $\cM_{\wt{\cE}}^{\leq \nu}$ is smooth, cf. \cite[\S 10.1.1]{FYZ3}, so then \cite[Corollary 6.4.1]{FK} applies to the bounded correspondence above to yield
\[
\Tr^{\Sht}(\cc_{\cM_{\wt{\cE}}^{\leq \nu}}) = [\cZ_{\wt{\cE}}^{r, \leq \nu}] \in \CH_*(\cZ_{\wt{\cE}}^{r,\leq \nu}).
\]
Then compatibility of trace with open restrictions, the open-embedding case of \cite[Proposition 6.2.2]{FK}, implies that
\[
\Tr^{\Sht}(\cc_{\cM_{\wt{\cE}}^{\leq \nu, \circ}}) = [\cZ_{\wt{\cE}}^{r, \circ, \leq \nu}] \in \CH_*(\cZ_{\wt{\cE}}^{r, \circ,  \leq \nu}).
\]

Let $X'_0:=X'\setminus\nu^{-1}(x_0)$. We observe that $\cM_{\cE}^\circ \rightarrow \cM_{\wt{\cE}}^\circ$ is a closed embedding stabilized by $\Hk_{\cM_{\wt{\cE}}^\circ}^r|_{(X'_0)^r}$. Indeed, over $(X'_0)^r$ the modifications are isomorphisms near $\nu^*D$, so extension of the maps from $\wt{\cE}$ to $\cE$ across $\nu^*D$ is a closed condition at every point of the chain. Hence $\cc_{\cM_{\wt{\cE}}^{\leq \nu, \circ}}^{(1)}|_{(X'_0)^r}$ is contracting near $\cM_{\cE}^{\circ,  \leq \nu}$ by \cite[Example 7.2.2]{FK}. Furthermore, the derived vector bundle description of the Hecke stacks identifies $\Hk_{\cM_{\cE}^{\circ}}^r|_{(X'_0)^r}$ with the derived pullback of $\Hk_{\cM_{\wt{\cE}}^\circ}^r|_{(X'_0)^r}$ along $\cM_{\cE}^\circ\hookrightarrow\cM_{\wt{\cE}}^\circ$, and Gysin base change for this derived Cartesian square identifies $\cc_{\cM_{\cE}^{\leq \nu, \circ}}^{(1)}|_{(X'_0)^r}$ with the restriction of $\cc_{\cM_{\wt{\cE}}^{\leq \nu, \circ}}^{(1)}|_{(X'_0)^r}$ in the sense of \cite[Construction 7.2.3]{FK}. Hence \cite[Theorem 7.5.1]{FK} applies to yield
\begin{equation}\label{eq: restrict trace of injective part}
\Tr^{\Sht}(\cc_{\cM_{\cE}^{\circ, \leq \nu}}|_{(X'_0)^r}) = [\cZ_{\cE}^{r, \circ, \leq \nu}|_{(X'_0)^r}] \in \CH_{r(n-m)}(\cZ_{\cE}^{r, \circ, \leq\nu}|_{(X'_0)^r}).
\end{equation}
The choice of $x_0$ above ensures that no $r(n-m)$-dimensional component is contained in the boundary
\[
\partial \cZ := \cZ_{\cE}^{r,\circ,\leq\nu}\setminus \bigl(\cZ_{\cE}^{r,\circ,\leq\nu}|_{(X'_0)^r}\bigr).
\]
Thus, we have $\dim \partial \cZ<r(n-m)$ and hence $\CH_{r(n-m)}(\partial \cZ)=0$. By the excision sequence, this implies that the restriction map
\[
\CH_{r(n-m)} (\cZ_{\cE}^{r, \circ, \leq \nu}) \rightarrow \CH_{r(n-m)}( \cZ_{\cE}^{r, \circ, \leq \nu}|_{(X'_0)^r})
\]
is an isomorphism. Therefore, the desired equality \eqref{eq: trace of injective part} can be checked after restriction to $\cZ_{\cE}^{r, \circ, \leq \nu}|_{(X'_0)^r}$. After such restriction, it becomes \eqref{eq: restrict trace of injective part}, so we are done.
\end{proof}

\subsection{Twisting by Chow cohomology}

Let $C$ be a derived Artin stack. We use the Chow cohomology groups $\CH^i(C)$ defined in \S\ref{ssec:chow-groups-notation}.
Note that the $i$th Chern class of a vector bundle is valued in $\CH^i(C)$.

Suppose $L \in \Dmot{C}$ is dualizable, and $c \in \Hom_C( L,  L\tw{i})$. Then we may form the \emph{trace} of $c$ as an element of Chow cohomology, $\Tr(c) \in \CH^i(C)$.

\begin{example}\label{ex: trace chow}
Note that if $L$ is $\otimes$-invertible, then we have
\[
c \in \Hom_C(L, L \tw{i}) \cong  \Hom_C(L \otimes L^\vee, L \otimes L^\vee \tw{i}) \cong \Hom_C(\Qsh{C}, \Qsh{C}\tw{i}) = \CH^i(C).
\]
In this case, $\Tr(c)$ agrees with the image of $c$ in $\CH^i(C)$. %In particular, taking $L = \Qsh{C}$, we observe that we can twist cohomological correspondences by Chow cohomology classes, and the formation of Trace is linear over this twisting.
\end{example}

A cohomological correspondence $\cc \in \Corr_C(\cK_0, \cK_1\tw{-d})$ can be twisted by $c \in \Hom_C(L, L \tw{i})$ to obtain $\cc \otimes c \in \Corr_C(\cK_0 \otimes L, (\cK_1 \otimes L)\tw{-d+i})$.
Recall also that Chow groups form a module over Chow cohomology.

\begin{lemma}\label{lem: twist trace by Chern class}
Let $A \leftarrow C \rightarrow A$ be a correspondence of derived Artin stacks, and let $\rho_C\co\Fix(C)\to C$ be the natural map. Let $\cK \in \Dmotg{A}$ and $\cc \in \Corr_C(\cK, \cK\tw{j})$ for any $j \in \Z$. Let $c\in\CH^i(C)=\Hom_C(\Qsh{C},\Qsh{C}\tw{i})$, viewed as an endomorphism of the unit object on $C$. Then we have
\[
\Tr_C(\cc \otimes c) = \Tr_C(\cc)\cap \rho_C^*c \in \CH_*(\Fix(C)).
\]
\end{lemma}

\begin{proof}
Clear from the definition. 
\end{proof}

\begin{example}\label{ex: trace twist} Taking $L = \Qsh{C}$ in Example \ref{ex: trace chow}, and applying Lemma \ref{lem: twist trace by Chern class}, we learn that cohomological correspondences can be tensored with Chow cohomology, and that formation of trace is linear over such twisting.
\end{example}

\subsection{Degeneracy stratification}

For each saturated subbundle $\cK\subset\cE$, there is a closed substack defined by the condition that $t\co\cE\to\cF$ kills $\cK$, equivalently where $t$ factors through $\cE/\cK$. Inside it, let $\cM_{\cE}[\cK]$ be the open substack where the induced map $\cE/\cK\to\cF$ is injective as a map of coherent sheaves after restriction to every geometric point of the test scheme, as for $\cM_{\cE}^{\circ}$ above. This closed substack is stable under $\Hk_{\cM}^r$.

Thus, for each $\cK$ we have a map of correspondences
\begin{equation}\label{eq: degeneracy stratification 1}
\begin{tikzcd}
\cM_{\cE}[\cK] \ar[d] & \Hk_{\cM_{\cE}[\cK] }^r\ar[r] \ar[l]  \ar[d] & \cM_{\cE}[\cK] \ar[d] \\
\cM_{\cE} & \ar[l]  \Hk_{\cM_{\cE}}^r \ar[r] & \cM_{\cE}
\end{tikzcd}
\end{equation}
where both squares are Cartesian. Notice that we have obvious maps $\cM_{\cE/\cK}^\circ \rightarrow \cM_{\cE}[\cK]$ and $\Hk_{\cM_{\cE/\cK}^\circ}^r  \rightarrow \Hk_{\cM_{\cE}[\cK] }^r$, which are each isomorphisms on classical truncations. We will use nil-invariance for motivic Chow groups and cohomological correspondences to identify the corresponding groups for these derived stacks having the same classical truncation. The commutative square
\begin{equation}\label{eq: degeneracy excess intersection 1}
\begin{tikzcd}
\Hk_{\cM_{\cE/\cK}^\circ}^r \ar[r] \ar[d] & \cM_{\cE/\cK}^\circ \ar[d] \\
\Hk_{\cM_{\cE}}^r \ar[r] & \cM_{\cE}
\end{tikzcd}
\end{equation}
is not Cartesian (unless $\cK = 0$), but it is Cartesian on classical truncations: on classical points, $t_i(\cK)$ is generically zero and hence is zero as a subsheaf of the vector bundle $\cF_i$. Furthermore, both horizontal maps are quasi-smooth, by Lemma \ref{lem: hecke quasismooth}.

\begin{lemma}\label{lem: excess square}
The commutative square \eqref{eq: degeneracy excess intersection 1} is an \emph{excess intersection square} in the sense of \cite[Proposition 3.15]{KhanI}, with excess bundle $\cV_{\cK}^r$ of \cite[\S 5.8.1]{FYZ2}.
\end{lemma}

\begin{proof}
We have already observed that the horizontal maps in \eqref{eq: degeneracy excess intersection 1} are quasi-smooth and the square is Cartesian on classical truncations. Let $f \co \Hk_{\cM_{\cE/\cK}^\circ}^r \rightarrow \Hk_{\cM_{\cE}}^r$ be the left vertical map in \eqref{eq: degeneracy excess intersection 1}. It remains to identify the (derived) fiber of the map
\begin{equation}\label{eq: excess bundle 1}
\bT_{\Hk_{\cM_{\cE/\cK}^\circ}^r / \cM_{\cE/\cK}^\circ } \rightarrow f^* \bT_{\Hk_{\cM_{\cE}}^r/\cM_{\cE}}
\end{equation}
with $\cV_{\cK}^r\sm{[-2]}$. Consider the commutative diagram
\[
\begin{tikzcd}
\Hk_{\cM_{\cE/\cK}^\circ}^r \ar[r] \ar[d] & \Hk_{\cM_{\cE}}^r \ar[r] \ar[d] & \Hk_{H_2}^r \ar[d] \\
\cM_{\cE/\cK}^\circ \ar[r] & \cM_{\cE} \ar[r] & \Bun_{H_2}
\end{tikzcd}
\]
This induces a commutative diagram of perfect complexes on $\Hk_{\cM_{\cE/\cK}^\circ}^r$,
\begin{equation}\label{eq: tangent complex 1}
\begin{tikzcd}
\bT_{\Hk_{\cM_{\cE/\cK}^\circ}^r / \Hk_{H_2}^r} \ar[d]  \ar[r] & f^* \bT_{\Hk_{\cM_{\cE}}^r/ \Hk_{H_2}^r } \ar[d] \\
\bT_{\cM_{\cE/\cK}^\circ / \Bun_{H_2}} \ar[r] &  f^* \bT_{\cM_{\cE} / \Bun_{H_2}}
\end{tikzcd}
\end{equation}
The derived fibers along the columns of \eqref{eq: tangent complex 1} differ from the two terms in \eqref{eq: excess bundle 1} by the same smooth summand pulled back from $\bT_{\Hk_{H_2}^r/\Bun_{H_2}}$. The comparison is the identity on this common summand, so the derived fiber of \eqref{eq: excess bundle 1} is obtained by cancelling it and then interchanging the order of forming derived fibers in \eqref{eq: tangent complex 1}: first taking the derived fibers along the rows, and then the derived fiber along the column.

Referring to the notation of \cite[\S 5.8]{FYZ2}, \cite[Proposition 5.35]{FYZ2} implies that $\Hk_{\cM_{\cE}}^r \rightarrow \Hk_{H_2}^r$ is the derived vector bundle associated to $\ul{\RHom_{X'}(\cE, \cF^\flat_{\star})}$, and $\cM_{\cE} \rightarrow \Bun_{H_2}$ is the derived vector bundle associated to $\ul{\RHom_{X'}(\cE, \cF)}$, and similarly for $\cE$ replaced by $\cE/\cK$. Hence the derived fibers along the rows of \eqref{eq: tangent complex 1} are
\[
\begin{tikzcd}
\ul{\RHom_{X'}(\cK, \cF^\flat_\star)}\sm{[-1]} \ar[d] \\
\ul{\RHom_{X'}(\cK, \cF)} \sm{[-1]}
\end{tikzcd}
\]
and then forming the derived fiber along the column gives $\cV_{\cK}^r\sm{[-2]}$, as desired.

\end{proof}

\begin{cor}\label{cor: pullback cc excess}
We have
\[
\cc_{\cM_{\cE}[\cK]} = \cc_{\cM_{\cE/\cK}^\circ} \otimes c_{\rank(\cK)r }(\cV_{\cK}^r) \in \Corr_{\Hk_{\cM_{\cE}[\cK]}^r}(\Qsh{\cM_{\cE}[\cK]}, \Qsh{\cM_{\cE}[\cK]}\tw{-d(h_r)}).
\]
\end{cor}

\begin{proof}
We can factor the diagram \eqref{eq: degeneracy excess intersection 1} through the right half of \eqref{eq: degeneracy stratification 1}, obtaining a commutative square
\begin{equation}\label{eq: degeneracy excess intersection 2}
\begin{tikzcd}
\Hk_{\cM_{\cE/\cK}^\circ}^r \ar[r] \ar[d] & \cM_{\cE/\cK}^\circ \ar[d] \\
\Hk_{\cM_{\cE}[\cK]}^r \ar[r] & \cM_{\cE}[\cK]
\end{tikzcd}
\end{equation}
in which the vertical maps are isomorphisms on classical truncations. Since \eqref{eq: degeneracy excess intersection 1} factors through the derived Cartesian right square of \eqref{eq: degeneracy stratification 1}, Lemma \ref{lem: excess square} shows that \eqref{eq: degeneracy excess intersection 2} is also an excess intersection square, with excess bundle $\cV_{\cK}^r$. Then the excess intersection formula \cite[\S 6.1.12]{FYZ2} applies to yield
\[
[\Hk_{\cM_{\cE}[\cK]}^r/  \cM_{\cE}[\cK]] = c_{\rank(\cK)r} (\cV_{\cK}^r) [\Hk_{\cM_{\cE/\cK}^\circ}^r /  \cM_{\cE/\cK}^\circ]  \in \CH_{d(h_r)}(\Hk_{\cM_{\cE/\cK}^\circ}^r /  \cM_{\cE/\cK}^\circ).
\]
Under the isomorphism
\[
\Corr_{\Hk_{\cM_{\cE}[\cK]}^r}(\Qsh{\cM_{\cE}[\cK]}, \Qsh{\cM_{\cE}[\cK]}\tw{-d(h_r)}) \cong \CH_{d(h_r)}(\Hk_{\cM_{\cE}[\cK]}^r /  \cM_{\cE}[\cK])
\]
which is compatible with Gysin base change along the derived Cartesian squares above, the cohomological correspondence $\cc_{\cM_{\cE}[\cK]}$ maps to $[\Hk_{\cM_{\cE}[\cK]}^r/  \cM_{\cE}[\cK]]$. Under the nil-invariance identification with the top row, the correspondence $\cc_{\cM_{\cE/\cK}^\circ}\otimes c_{\rank(\cK)r}(\cV_{\cK}^r)$ maps to $c_{\rank(\cK)r}(\cV_{\cK}^r)\cap[\Hk_{\cM_{\cE/\cK}^\circ}^r /  \cM_{\cE/\cK}^\circ]$. The excess formula identifies these two classes, completing the proof.
\end{proof}

\begin{remark}\label{rem:degeneracy-stratum-identification}
The natural map $\cM_{\cE/\cK}^\circ \rightarrow \cM_{\cE}[\cK]$ is an isomorphism on classical truncations: an injective map $\cE/\cK\to\cF$ is the same thing as a map $\cE\to\cF$ with kernel exactly $\cK$. Thus on classical truncations we have $\Sht_{\cM_{\cE}[\cK]}^r \cong \cZ_{\cE/\cK}^{r,\circ}$. This is the identification used in Proposition \ref{prop: trace on degeneracy stratum}.
\end{remark}

\begin{prop}\label{prop: trace on degeneracy stratum}
Assume that $\rank(\cE/\cK) \leq n/3$. Let $\rho_{\cK}$ be the natural map from $\cZ_{\cE/\cK}^{r,\circ}$ to $\Hk_{\cM_{\cE/\cK}^\circ}^r$. Then we have
\[
\Tr^{\Sht}(\cc_{\cM_{\cE}[\cK] }) = \rho_{\cK}^*c_{\rank (\cK) r} (\cV_{\cK}^r)\cap[\cZ_{\cE/\cK}^{r, \circ}] \in \CH_*(\cZ_{\cE/\cK}^{r, \circ}).
\]
\end{prop}

\begin{proof}
By Corollary \ref{cor: pullback cc excess} and then Lemma \ref{lem: twist trace by Chern class}, we have
\[
\Tr^{\Sht}(\cc_{\cM_{\cE}[\cK]} ) = \Tr^{\Sht}(\cc_{\cM_{\cE/\cK}^\circ} \otimes c_{\rank(\cK)r }(\cV_{\cK}^r) ) = \Tr^{\Sht}(\cc_{\cM_{\cE/\cK}^\circ})\cap\rho_{\cK}^*c_{\rank (\cK) r} (\cV_{\cK}^r)  \in \CH_*(\cZ_{\cE/\cK}^{r, \circ}).
\]
If $\cE/\cK=0$, then $\cM_{\cE/\cK}^{\circ}=\Bun_{H_2}$ and
$\cZ_{\cE/\cK}^{r,\circ}=\Sht_{H_2}^r$.  The corresponding
cohomological correspondence is the canonical Hecke correspondence, so
\cite[Corollary 6.4.1]{FK} gives
$\Tr^{\Sht}(\cc_{\cM_{\cE/\cK}^{\circ}})
=[\cZ_{\cE/\cK}^{r,\circ}]$.  Hence assume that
$\rank(\cE/\cK)>0$. Then we may apply Proposition \ref{prop: trace of injective part}, which identifies
\[
\Tr^{\Sht}(\cc_{\cM_{\cE/\cK}^\circ})  = [\cZ_{\cE/\cK}^{r, \circ}] \in \CH_*(\cZ_{\cE/\cK}^{r, \circ}).
\]
\end{proof}

\subsection{Completion of the proof}

Let $\cM_{\geq d} \subset \cM_{\cE}$ be the closed substack where $\rank(\ker (t \co \cE \rightarrow \cF)) \geq d$ and let $\cM_{\leq d} \subset \cM_{\cE}$ be the open substack where $\rank(\ker (t \co \cE \rightarrow \cF)) \leq d$. These are the usual determinantal rank loci for the map $t$ on geometric fibers. They intersect in the locally closed substack $\cM_d \inj \cM_{\cE}$.

Furthermore, $j \co \cM_d \inj \cM_{\geq d}$ is an open embedding, while $i \co  \cM_{\geq d+1} \inj \cM_{\geq d}$ is a closed embedding, and both are stable under $\Hk_{\cM_{\cE}}^r$: along a Hecke chain the kernels $\ker(t_i)\subset\cE$ have the same generic fiber, hence the same saturation and the same rank. Thus the pullbacks $i^*$ and $j^*$ exist on cohomological correspondences.

Denote by $\Sht_{\cM_d}^r$ and $\Sht_{\cM_{\geq d+1}}^r$ the corresponding moduli stacks of shtukas. By \cite[Proposition 7.5.4]{FK}, the maps $\Sht_{\cM_d}^r \rightarrow \Sht_{\cM_{\cE}}^r = \cZ_{\cE}^r$ are open-closed for all $d$.

For any stable locally closed substack $\cN \subset \cM_{\cE}$ (in this subsection, $\cN$ will be one of $\cM_d$, $\cM_{\geq d}$, $\cM_{\geq d+1}$, or $\cM_{\cE}[\cK]$), write $\Hk_{\cN}^r$ for the derived preimage of $\cN$ under the right leg $h_r$ of $\Hk_{\cM_{\cE}}^r$; by stability, the left leg factors through $\cN$. We write $\cc_{\cN}$ for the pullback of $\cc_{\cM_{\cE}}$ to $\Hk_{\cN}^r$.

For $d<m$, decompose $\cM_{\geq d}=\cM_d\sqcup\cM_{\geq d+1}$ and compute in $\CH_*(\Sht_{\cM_{\geq d}}^r)$ as follows.
\begin{align}\label{eq: trace of stratification}
\Tr^{\Sht}(\cc_{\cM_{\geq d}}) &\stackrel{(1)}= \Tr^{\Sht}(j_! j^*  \cc_{\cM_{\geq d}})  + \Tr^{\Sht}(i_* i^* \cc_{\cM_{\geq d}}) \nonumber \\
& \stackrel{(2)}= \Tr^{\Sht}(j_!  \cc_{\cM_{d}}) +  \Tr^{\Sht}(i_* \cc_{\cM_{\geq d+1}}) \nonumber \\
&\stackrel{(3)}=  \Sht(j)_! (\Tr^{\Sht}(j_! \cc_{\cM_{d}})|_{\Sht_{\cM_{d}}^r}) + \Sht(i)_!(\Tr^{\Sht}(j_!  \cc_{\cM_{d}})|_{\Sht_{\cM_{\geq d+1}}^r} ) \nonumber \\
& \hspace{1cm} + \Sht(j)_!( \Tr^{\Sht}( i_* \cc_{\cM_{\geq d+1}})|_{\Sht_{\cM_{d}}^r})  + \Sht(i)_!( \Tr^{\Sht}( \cc_{\cM_{\geq d+1}})|_{\Sht_{\cM_{\geq d+1}}^r}  ).
\end{align}
The three numbered equalities are justified as follows.
\begin{enumerate}
\item Equality (1) is trace additivity; the proof of \cite[Lemma 5.2.7]{Jin24} applies verbatim. Both legs of the Frobenius-twisted correspondence preserve the two strata, so $j_!,j^*,i_*,i^*$ act on the restricted correspondences.
\item Equality (2) uses the derived Cartesian restriction squares
\[
\begin{tikzcd}
\Hk_{\cM_d}^r \ar[r] \ar[d]  & \cM_d \ar[d, "j"] \\
\Hk_{\cM_{\geq d}}^r \ar[r] & \cM_{\geq d}
\end{tikzcd}
\qquad
\begin{tikzcd}
\Hk_{\cM_{\geq d+1}}^r \ar[r] \ar[d]  & \cM_{\geq d+1} \ar[d, "i"] \\
\Hk_{\cM_{\geq d}}^r \ar[r] & \cM_{\geq d}
\end{tikzcd}
\]
which are Cartesian by the derived-preimage convention for $\Hk_{\cN}^r$ and identify $j^*\cc_{\cM_{\geq d}}=\cc_{\cM_d}$ and $i^*\cc_{\cM_{\geq d}}=\cc_{\cM_{\geq d+1}}$.
\item Equality (3) uses compatibility of $\Tr^{\Sht}$ with proper pushforward \cite[Proposition 6.2.1]{FK}; both $\Sht(j)$ and $\Sht(i)$ are open-and-closed embeddings. It also uses the decomposition
\[
\Sht_{\cM_{\geq d}}^r = \Sht_{\cM_d}^r \coprod \Sht_{\cM_{\geq d+1}}^r .
\]
\end{enumerate}

We have 
\begin{equation}\label{eq:sht-Md-decomposition}
\Sht_{\cM_d}^r = \coprod_{\substack{\cK \subset \subset \cE,\ \F_q\textup{-rational}\\ \rank(\cK) = d}} \Sht_{\cM_{\cE}[\cK]}^r.
\end{equation}

The four terms in the last line of \eqref{eq: trace of stratification} simplify as follows.
\begin{itemize}
\item The third summand vanishes because $\Tr^{\Sht}(i_*\cc_{\cM_{\geq d+1}})=\Sht(i)_!\Tr^{\Sht}(\cc_{\cM_{\geq d+1}})$ is supported on the disjoint open-and-closed summand $\Sht_{\cM_{\geq d+1}}^r$.

\item For the first summand, compatibility of trace with pullback along open embeddings gives
\begin{align*}
\Tr^{\Sht}(j_! \cc_{\cM_{d}})|_{\Sht_{\cM_{d}}^r}  = \Tr^{\Sht}(j^* j_! \cc_{\cM_{d}} ) = \Tr^{\Sht}(\cc_{\cM_{d}}).
\end{align*}
The decomposition \eqref{eq:sht-Md-decomposition} then gives
\begin{equation}\label{eq:trace-decomposition}
\Tr^{\Sht}(\cc_{\cM_{d}}) = \sum_{\substack{\cK \subset \subset \cE,\ \F_q\textup{-rational}\\ \rank(\cK) = d}}  (i_{\cK})_! i_{\cK}^* \Tr^{\Sht}(\cc_{\cM_{d}})
\end{equation}
where $i_{\cK}$ is the (open and closed) inclusion of $\Sht_{\cM_{\cE}[\cK]}^r$ into $\Sht_{\cM_d}^r $. 

The correspondence $\Hk_{\cM_d}^r$ stabilizes $\cM_{\cE}[\cK]$. In this Cartesian setting, the restriction of \cite[Construction 7.2.3]{FK} agrees with the pullback defining $\cc_{\cM_{\cE}[\cK]}$. Hence \cite[Example 7.2.2]{FK} and \cite[Theorem 7.5.1]{FK} give
\[
i_{\cK}^* \Tr^{\Sht}(\cc_{\cM_{d}}) = \Tr^{\Sht}(\cc_{\cM_{\cE}[\cK]}).
\]
Proposition \ref{prop: trace on degeneracy stratum} identifies this class as
\[
 \Tr^{\Sht}(\cc_{\cM_{\cE}[\cK]}) =:
 \alpha_{\cK}
 =\rho_{\cK}^*c_{\rank (\cK) r} (\cV_{\cK}^r)
 \cap[\cZ_{\cE/\cK}^{r, \circ}].
 \]
Substituting these observations into \eqref{eq:trace-decomposition} gives
\begin{align*}
\Tr^{\Sht}(\cc_{\cM_{d}}) &  =  \sum_{\substack{\cK \subset \subset \cE,\ \F_q\textup{-rational} \\ \rank(\cK) = d}} (i_{\cK})_!\alpha_{\cK}.
\end{align*}

\item Since $\Hk_{\cM_{\geq d}}^r$ stabilizes $\cM_{\geq d+1}$, $(\Hk_{\cM_{\geq d}}^r)^{(1)}$ is contracting near $\cM_{\geq d+1}$ by \cite[Example 7.2.2]{FK}, and then \cite[Theorem 7.5.1]{FK} shows that
\[
\Tr^{\Sht}(j_!  \cc_{\cM_{d}})|_{\Sht_{\cM_{\geq d+1}}^r} = 0.
\]
\end{itemize}
Write $i_{\cK,\ge d}:=\Sht(j)\circ i_{\cK}\co
\Sht_{\cM_{\cE}[\cK]}^r\to\Sht_{\cM_{\ge d}}^r$.
Equation \eqref{eq: trace of stratification} therefore becomes
\begin{align}\label{eq: trace of stratum}
\Tr^{\Sht}(\cc_{\cM_{\geq d}})   = \left( \sum_{\substack{\cK \subset \subset \cE,\ \F_q\textup{-rational} \\ \rank(\cK) = d}} (i_{\cK,\ge d})_!\alpha_{\cK} \right) + (\Sht(i)_! \Tr^{\Sht}( \cc_{\cM_{\geq d+1}})|_{\Sht_{\cM_{\geq d+1}}^r})
\end{align}

When $d=m$, we have $\cM_{\geq m}=\cM_m$ and $\cM_{\geq m+1}=\varnothing$, so \eqref{eq: trace of stratum} holds without its second summand. The identities for $0\le d\le m$ telescope after pushing each class $\alpha_{\cK}$ to $\CH_*(\cZ_{\cE}^r)$ along the open-closed map
\[
\iota_{\cK}\co
\cZ_{\cE/\cK}^{r,\circ}\cong \Sht_{\cM_{\cE}[\cK]}^r \longrightarrow \cZ_{\cE}^r,
\]
we obtain
\[
\Tr^{\Sht}(\cc_{\cM_{\cE}}) = \sum_{\substack{\cK \subset \subset \cE\\ \cK\ \F_q\textup{-rational}}} (\iota_{\cK})_!\alpha_{\cK},
\]
which agrees with $[\cZ_{\cE}^r]$ by \cite[Theorem 6.6]{FYZ2}. This concludes the proof of Theorem \ref{thm: trace conjecture E}. \qed

\part{General tools}

\section{Motivic derived Fourier transform}\label{sec: motivic derived Fourier transform}

In this section we develop a refined version of the homogeneous motivic Fourier transform introduced in \cite{FK}. Our version can be applied to inhomogeneous sheaves, and agrees with the homogeneous version of \emph{loc. cit.} on homogeneous motivic sheaves. There are two reasons we need this generalization:
\begin{itemize}
        \item We will need to apply the Fourier transform to inhomogeneous motivic sheaves.
        \item We will use the fact that the refined Fourier transform is given by a kernel which behaves like a ``local system'' (is trivial after a finite \'etale cover), which is not true of the homogeneous Fourier transform.
\end{itemize}

\subsection{Motivic Artin--Schreier sheaves}
Let $p=\operatorname{char}\F_q$. We recall the additive-character convention fixed in the Notation section: $\psi_0\co\F_p\to\ol{\Q}^{\times}$ is the character used to split the $p$-Artin--Schreier cover, and $\psi=\psi_0\circ\Tr_{\F_q/\F_p}$ is the resulting character of $\F_q$. Let
\begin{equation}\label{eq:motivic-AS-cover}
f\co \A^1_{\F_q}\longrightarrow \A^1_{\F_q},
\qquad
u\longmapsto u^p-u,
\end{equation}
be the Artin--Schreier morphism.  It is a finite \'{e}tale $\F_p$-torsor.  With respect to this action, set
\begin{equation}\label{eq:motivic-AS-idempotent}
e_{\psi_0}:=\frac{1}{p}\sum_{a\in \F_p}\psi_0(a)^{-1}[a]
\in \ol{\Q}[\F_p].
\end{equation}
This is the primitive idempotent projecting to the $\psi_0$-isotypic summand.  The motivic Artin--Schreier sheaf attached to $\psi_0$ is
\begin{equation}\label{eq:motivic-AS-sheaf}
\AS_{\psi_0}:=e_{\psi_0}\bigl(f_!\Qsh{\A^1_{\F_q}}\bigr)
\in \Dmot{\A^1_{\F_q}}.
\end{equation}

\begin{lemma}
Let $m\co\A^1_{\F_q}\times\A^1_{\F_q}\to\A^1_{\F_q}$ denote the addition map. Then there is a canonical isomorphism 
\begin{equation}\label{eq:motivic-AS-additivity}
m^*\AS_{\psi_0}\xrightarrow{\sim}
\pr_1^*\AS_{\psi_0}\otimes\pr_2^*\AS_{\psi_0}.
\end{equation}
Together with the canonical trivialization at the zero section, this makes $\AS_{\psi_0}$ into a character sheaf on the additive group $\G_a$.
\end{lemma}

\begin{proof}
Let us temporarily denote the Artin--Schreier cover as $f \co \wt \A^1_{\F_q} \rightarrow \A^1_{\F_q}$. Since $f$ is finite \'{e}tale, $f_!=f_*$, and $f_!\Qsh{\wt\A^1_{\F_q}}$ carries the regular action of the deck group $\F_p$.\footnote{We fix once and for all the \emph{pullback} convention for this action: $a\in\F_p$ acts via $\sigma_a^*$, where $\sigma_a(u)=u+a$ is the deck transformation. With this convention the image of the $e_{\psi_0}$-summand under any $\ell$-adic realization is the Artin--Schreier local system with trace function $x\mapsto\psi_0(x)$.} The addition map $m\co\A^1_{\F_q}\times\A^1_{\F_q}\to\A^1_{\F_q}$ lifts to a map $\wt m \co\wt\A^1_{\F_q}\times\wt\A^1_{\F_q}\to\wt\A^1_{\F_q}$, because
\[
(u+v)^p-(u+v)=(u^p-u)+(v^p-v).
\]
This covering map is equivariant for the addition homomorphism $\F_p\times\F_p\to\F_p$ of deck groups. The fiber of this homomorphism is the anti-diagonal $\{(a,-a)\}$, on which the character $\psi_0\boxtimes\psi_0$ is trivial. Therefore the $(\psi_0\boxtimes\psi_0)$-isotypic summand of $\pr_1^*f_!\Qsh{\wt\A^1}\otimes\pr_2^*f_!\Qsh{\wt\A^1}$ descends canonically to the $\psi_0$-isotypic summand of $m^*f_!\Qsh{\wt\A^1}$.

The addition map $(a,b)\mapsto a+b$ of $\F_p$-groups induces a map $\ol{\Q}[\F_p]\otimes_{\ol{\Q}}\ol{\Q}[\F_p]\to\ol{\Q}[\F_p]$ carrying
\[
e_{\psi_0}\boxtimes e_{\psi_0}\longmapsto e_{\psi_0}.
\]
This gives the isomorphism \eqref{eq:motivic-AS-additivity}.  The associativity and unit compatibilities follow from the associativity and unit of addition on the Artin--Schreier torsor, together with the canonical trivialization of the torsor over $0$.
\end{proof}

For an $\F_q$-stack $S$, let $\AS_{\psi,S}$ denote the pullback of $\AS_{\psi_0}$ to $\A^1_S$.  The construction of \eqref{eq:motivic-AS-additivity} is compatible with arbitrary base change in $S$.  

\subsection{Motivic derived Fourier transform}
The notion of \emph{derived vector bundles} is recalled in \cite[\S 6.1]{FYZ3}. Let $E\to S$ be a derived vector bundle over a derived Artin stack $S$ over $\F_q$, and let $\wh E\to S$ be its dual.  Write
\begin{equation}\label{eq:motivic-derived-evaluation}
\langle\, ,\,\rangle_E\co \wh E\times_S E\longrightarrow \A^1_S
\end{equation}
for the evaluation morphism, and let $\pr_{\wh E}$ and $\pr_E$ denote the two projections from $\wh E\times_S E$.  Let $d(E/S)$ be the virtual relative dimension of $E\to S$.  It is locally constant on $S$; whenever $S$ is disconnected, all shifts and Tate twists involving $d(E/S)$ are understood componentwise.  The motivic Fourier transform associated with $\psi$ is the $\ol{\Q}$-linear functor 
$\Dmot{E}\to
\Dmot{\wh E}$ given by
\begin{equation}\label{eq:motivic-derived-FT}
\FT_E^\psi(\cK)
:=
\pr_{\wh E,!}\bigl(\pr_E^*\cK\otimes
\langle\, ,\,\rangle_E^*\AS_{\psi,S}\bigr)\sm{[d(E/S)]}.
\end{equation}
This is the motivic analogue of the Fourier--Deligne transform. The shift in \eqref{eq:motivic-derived-FT} is the Deligne--Laumon normalization. We use $\FT$ for this derived motivic Fourier transform throughout.

\subsection{Formal properties of the motivic Fourier transform}
We now record the motivic derived analogues of the formal derived Fourier analysis used below.  The proofs of all statements in this subsection are deferred to the forthcoming work of Tong Zhou \cite{ZhouMotivicFourier}. They are formulated for the motivic categories $\Dmot{(-)}$ and for the Artin--Schreier kernel $\AS_{\psi,S}$ fixed above.  The motivic six-operation formalism for $\Dmot{(-)}$ on derived Artin stacks is the one recalled in \cite[\S 3.1]{FK}, ultimately based on Khan's extension of motivic homotopy theory and its six operations to derived Artin stacks \cite[Appendix A]{KhanI}. In the rest of this section we write $\FT_E=\FT_E^\psi$, unless the additive character must be displayed.

\begin{thm}[Formal properties]\label{thm:motivic-FT-formalism}
Let $E\to S$ be a derived vector bundle over a derived Artin stack over $\F_q$, let $\wh E\to S$ be its dual, and put $d=d(E/S)$.  The functor $\FT_E$ of \eqref{eq:motivic-derived-FT} satisfies the following canonical compatibilities.
\begin{enumerate}
\item\label{item:motivic-FT-base-change} For any morphism $h\co \wt S\to S$, write $\wt E=E\times_S\wt S$, and let $h^E\co\wt E\to E$ and $h^{\wh E}\co\wh{\wt E}\to\wh E$ be the induced maps.  There are canonical natural isomorphisms
\begin{equation}\label{eq:motivic-FT-base-change-star-shriek}
\FT_{\wt E}\circ(h^E)^*\xrightarrow{\sim}(h^{\wh E})^*\circ\FT_E,
\qquad
\FT_{\wt E}\circ(h^E)^!\xrightarrow{\sim}(h^{\wh E})^!\circ\FT_E,
\end{equation}
and
\begin{equation}\label{eq:motivic-FT-base-change-pushes}
\FT_E\circ h^E_!\xrightarrow{\sim}h^{\wh E}_!\circ\FT_{\wt E},
\qquad
\FT_E\circ h^E_*\xrightarrow{\sim}h^{\wh E}_*\circ\FT_{\wt E}.
\end{equation}

\item\label{item:motivic-FT-involutivity} There is a canonical involutivity isomorphism
\begin{equation}\label{eq:motivic-FT-involutivity}
\FT_{\wh E}^{\psi}\circ\FT_E^{\psi}
\xrightarrow{\sim}
[-1]_E^*\sm{(-d)},
\end{equation}
where $[-1]_E\co E\to E$ denotes fiberwise multiplication by $-1$.

\item\label{item:motivic-FT-linear-functoriality} Let $f\co E'\to E$ be a linear map of derived vector bundles over $S$, and let $\wh f\co\wh E\to\wh E'$ be its dual.  Put $d'=d(E'/S)$.  There are canonical natural isomorphisms
\begin{equation}\label{eq:motivic-FT-linear-functoriality}
\begin{gathered}
\wh f^*\circ\FT_{E'}\xrightarrow{\sim}\FT_E\circ f_!\sm{[d'-d]},
\qquad
\wh f^!\circ\FT_{E'}\xrightarrow{\sim}\FT_E\circ f_*\sm{[d-d'](d-d')},\\
\FT_{E'}\circ f^*\xrightarrow{\sim}\wh f_!\circ\FT_E\sm{[d-d'](d-d')},
\qquad
\FT_{E'}\circ f^!\xrightarrow{\sim}\wh f_*\circ\FT_E\sm{[d'-d]}.
\end{gathered}
\end{equation}
These isomorphisms intertwine the adjunctions $(f_!,f^!)$ and $(\wh f^*,\wh f_*)$, and the adjunctions $(f^*,f_*)$ and $(\wh f_!,\wh f^!)$, with the shifts and Tate twists displayed in \eqref{eq:motivic-FT-linear-functoriality}.

\item\label{item:motivic-FT-delta-constant} If $z_E\co S\to E$ and $z_{\wh E}\co S\to\wh E$ are the zero sections, set $\delta_E:=z_{E!}\Qsh{S}$ and $\delta_{\wh E}:=z_{\wh E!}\Qsh{S}$.  Then
\begin{equation}\label{eq:motivic-FT-delta-constant}
\FT_E(\delta_E)\xrightarrow{\sim}\Qsh{\wh E}\sm{[d]},
\qquad
\FT_E(\Qsh{E})\xrightarrow{\sim}\delta_{\wh E}\sm{[-d](-d)}.
\end{equation}

\item\label{item:motivic-FT-convolution-tensor} For $\cK_0,\cK_1\in\Dmot{E}$, with convolution taken along the additive group structure on $E/S$, there are canonical natural isomorphisms
\begin{equation}\label{eq:motivic-FT-convolution-tensor}
\FT_E(\cK_0\star\cK_1)
\xrightarrow{\sim}
\FT_E(\cK_0)\otimes\FT_E(\cK_1)\sm{[-d]},
\end{equation}
and
\begin{equation}\label{eq:motivic-FT-tensor-convolution}
\FT_E(\cK_0\otimes\cK_1)
\xrightarrow{\sim}
\FT_E(\cK_0)\star\FT_E(\cK_1)\sm{[d](d)},
\end{equation}
where the convolution on the right side of \eqref{eq:motivic-FT-tensor-convolution} is taken on $\wh E/S$.

\item\label{item:motivic-FT-Verdier-duality} Let $\DD_E$ and $\DD_{\wh E}$ denote Verdier duality on $E$ and $\wh E$.  There is a canonical natural isomorphism
\begin{equation}\label{eq:motivic-FT-Verdier-duality}
\DD_{\wh E}\bigl(\FT_E^\psi(\cK)\bigr)
\xrightarrow{\sim}
\FT_E^{\psi^{-1}}\bigl(\DD_E(\cK)\bigr)\sm{(d)}.
\end{equation}

\item\label{item:motivic-FT-Plancherel} For $\cK_1,\cK_2\in\Dmot{E}$, there is a canonical Plancherel isomorphism
\begin{equation}\label{eq:motivic-FT-Plancherel}
\pi_{\wh E!}\bigl(\FT_E(\cK_1)\otimes\FT_E(\cK_2)\bigr)
\xrightarrow{\sim}
\pi_{E!}\bigl(\cK_1\otimes[-1]_E^*\cK_2\bigr)\sm{(-d)}.
\end{equation}
\end{enumerate}
\end{thm}

\begin{prop}[Compatibility with proper base change]\label{prop:motivic-FT-proper-base-change}
Consider a Cartesian square of globally presented derived vector bundles over a common base, together with its Fourier-dual Cartesian square,
\begin{equation}\label{eq:motivic-FT-proper-base-change-square}
\begin{tikzcd}[ampersand replacement=\&]
\& B \ar[dl, "g'"'] \ar[dr, "f'"] \\
A  \ar[dr, "f"'] \& \& D  \ar[dl, "g"] \\
\& C
\end{tikzcd}
\qquad
\begin{tikzcd}[ampersand replacement=\&]
\& \wh C \ar[dl, "\wh f"'] \ar[dr, "\wh g"] \\
\wh A  \ar[dr, "\wh g'"'] \& \& \wh D  \ar[dl, "\wh f'"] \\
\& \wh B .
\end{tikzcd}
\end{equation}
Put $d_f=d(f)$ and $d_g=d(g)$.  Then the diagram of functors $\Dmot{A}\to\Dmot{\wh D}$
\begin{equation}\label{eq:motivic-FT-proper-base-change}
\begin{tikzcd}[ampersand replacement=\&, column sep=large, row sep=large]
\wh g_!\wh f^*\FT_A
\ar[r, "\sim"] \ar[d, "\sim"'] \&
\FT_D\, g^*f_!\sm{[d_f+d_g](d_g)}
\ar[d, "\sim"] \\
(\wh f')^*\wh g'_!\FT_A
\ar[r, "\sim"'] \&
\FT_D\, f'_!(g')^*\sm{[d_f+d_g](d_g)}
\end{tikzcd}
\end{equation}
commutes. Here the right vertical arrow is the Beck--Chevalley equivalence $g^*f_!\simeq f'_!(g')^*$ for the first square in \eqref{eq:motivic-FT-proper-base-change-square}, after applying $\FT_D$ and the displayed shifts and twists.  The left vertical arrow is the Beck--Chevalley equivalence $\wh g_!\wh f^*\simeq(\wh f')^*\wh g'_!$ for the Fourier-dual square. The horizontal arrows are the Fourier functoriality isomorphisms of \eqref{eq:motivic-FT-linear-functoriality}, and the vertical arrows are the proper base-change isomorphisms for \eqref{eq:motivic-FT-proper-base-change-square} and its Fourier-dual square.
\end{prop}

\begin{prop}[Fourier transform of the Gysin map]\label{prop:motivic-FT-Gysin}
Let $f\co E'\to E$ be a globally presented quasi-smooth linear map of derived vector bundles over $S$, and let $\wh f\co\wh E\to\wh E'$ be its dual. The quasi-smoothness of $f$ implies that $\wh f$ is separated and representable in derived schemes, by \cite[Lemma 6.1.5]{FYZ3}, so that the forget-supports transformation $\can(\wh f)\co\wh f_!\to\wh f_*$ is defined. Then the diagram of functors $\Dmot{E}\to\Dmot{\wh E'}$
\begin{equation}\label{eq:motivic-FT-Gysin}
\begin{tikzcd}[ampersand replacement=\&, column sep=large, row sep=large]
\wh f_!\FT_E \ar[r, "\can(\wh f)"] \ar[d, "\sim"'] \&
\wh f_*\FT_E \ar[d, "\sim"] \\
\FT_{E'}f^*\sm{[d(f)](d(f))} \ar[r, "{[f]}"'] \&
\FT_{E'}f^!\sm{[-d(f)]}
\end{tikzcd}
\end{equation}
commutes, where $[f]$ denotes the natural transformation induced by the relative fundamental class of $f$, and $\can(\wh f)\co\wh f_!\to\wh f_*$ is the forget-supports map.
\end{prop}

\subsection{Self-duality of Gaussians}
Recall our standing assumption that $p\ne2$. Throughout this subsection we write $\FT_E$ for $\FT_E^\psi$.
\begin{constr}[Induced Gaussian kernels]\label{constr:motivic-Gaussian-self-duality}
Let $E\to S$ be a vector bundle of virtual relative dimension $d(E/S)$, and let
\[
h_E\co E\xrightarrow{\sim}\wh E
\]
be a symmetric self-duality; that is, under the canonical biduality identification $E\cong\wh{\wh E}$, one has $h_E^\vee=h_E$.  Define
\begin{equation}\label{eq:motivic-induced-quadratic}
\beta_{h_E}\co
E\xrightarrow{(h_E,\Id)}
\wh E\times_S E
\xrightarrow{\langle\, ,\,\rangle_E}
\A^1_S .
\end{equation}
The symmetry of $h_E$ implies the following identity of morphisms $E\times_S E\to\A^1_S$:
\begin{equation}\label{eq:motivic-induced-quadratic-polarization}
\beta_{h_E}(e_1+e_2)-\beta_{h_E}(e_1)-\beta_{h_E}(e_2)
=
2\langle h_E(e_1),e_2\rangle_E.
\end{equation}
Thus $\beta_{h_E}$ is the quadratic function associated with the symmetric bilinear form defined by $h_E$, with the normalization responsible for the multiplication-by-$2$ map $[2]_{\wh E}$ in \eqref{eq:motivic-Gaussian-self-duality-on-E}.  The self-duality also determines the dual quadratic function
\begin{equation}\label{eq:motivic-induced-dual-quadratic}
\beta_{h_E}^{\vee}\co
\wh E\xrightarrow{(\Id,h_E^{-1})}
\wh E\times_S E
\xrightarrow{\langle\, ,\,\rangle_E}
\A^1_S .
\end{equation}
One has $h_E^*\beta_{h_E}^{\vee}=\beta_{h_E}$. Define the associated Gaussian object by
\begin{equation}\label{eq:motivic-induced-Gaussian}
\sG_{h_E}:=\beta_{h_E}^*\AS_{\psi,S}
\in\Dmot{E}.
\end{equation}
Via $h_E$, the positive-dual Gaussian
$(\beta_{h_E}^{\vee})^*\AS_{\psi,S}$ on $\wh E$ identifies canonically with $\sG_{h_E}$ on $E$.  On the other hand,
$h_E^*(-\beta_{h_E}^{\vee})=-\beta_{h_E}$, so upon using $h_E$ to identify $\wh{E}$ with $E$, the Fourier-dual Gaussian is
$(-\beta_{h_E})^*\AS_{\psi,S}$, which need not be isomorphic to $\sG_{h_E}$.
\end{constr}

\begin{lemma}[Fourier transform of Gaussian sheaves]\label{lem:motivic-Gaussian-self-duality}
With the notation of Construction \ref{constr:motivic-Gaussian-self-duality}, there is a canonical isomorphism
\begin{equation}\label{eq:motivic-Gaussian-self-duality-on-E}
[2]_{\wh E}^*\FT_E(\sG_{h_E})
\xrightarrow{\sim}
(-\beta_{h_E}^{\vee})^*\AS_{\psi,S}\otimes
\pi_{\wh E}^*\bigl(\pi_{E!}\sG_{h_E}\sm{[d(E/S)]}\bigr),
\end{equation}
where $[2]_{\wh E}\co\wh E\to\wh E$ is fiberwise multiplication by $2$, and $\pi_E\co E\to S$ and $\pi_{\wh E}\co\wh E\to S$ are the structure morphisms.
\end{lemma}

\begin{proof}
Consider the Cartesian square
\begin{equation}\label{diag:motivic-Gaussian-two-pullback-square}
\begin{tikzcd}[ampersand replacement=\&]
\wh E\times_S E \ar[r, "{[2]_{\wh E}\times\Id_E}"] \ar[d, "\pr_{\wh E}"'] \&
\wh E\times_S E \ar[d, "\pr_{\wh E}"]\\
\wh E \ar[r, "{[2]_{\wh E}}"'] \&
\wh E
\end{tikzcd}
\end{equation}
Base change for the square \eqref{diag:motivic-Gaussian-two-pullback-square}, together with the definition \eqref{eq:motivic-derived-FT} and the equality
$\pr_E\circ([2]_{\wh E}\times\Id_E)=\pr_E$, gives a canonical identification
\begin{equation}\label{eq:motivic-Gaussian-two-pullback}
[2]_{\wh E}^*\FT_E(\sG_{h_E})
\cong
\pr_{\wh E,!}\Bigl(\pr_E^*\sG_{h_E}\otimes
\bigl(\bigl\langle[2]_{\wh E}\circ\pr_{\wh E},\pr_E\bigr\rangle_E\bigr)^*\AS_{\psi,S}\Bigr)\sm{[d(E/S)]}.
\end{equation}
Let
\begin{equation}\label{eq:motivic-Gaussian-translation}
\tau_{h_E}\co\wh E\times_S E\longrightarrow\wh E\times_S E,
\qquad
(\eta,e)\longmapsto\bigl(\eta,e+h_E^{-1}(\eta)\bigr)
\end{equation}
be translation in the $E$-coordinate by $h_E^{-1}(\eta)$.  It is an automorphism over $\wh E$.  The polarization identity \eqref{eq:motivic-induced-quadratic-polarization} implies the following identity of morphisms $\wh E\times_S E\to\A^1_S$:
\begin{equation}\label{eq:motivic-Gaussian-completing-square}
\beta_{h_E}\circ\pr_E+
\bigl\langle[2]_{\wh E}\circ\pr_{\wh E},\pr_E\bigr\rangle_E
=
\beta_{h_E}\circ\pr_E\circ\tau_{h_E}
-\beta_{h_E}^{\vee}\circ\pr_{\wh E}.
\end{equation}
Indeed, for $(\eta,e) \in \wh{E} \times_S E$ and $v=h_E^{-1}(\eta)$, the right hand side of \eqref{eq:motivic-Gaussian-completing-square} is
$\beta_{h_E}(e+v)-\beta_{h_E}(v)=\beta_{h_E}(e)+2\langle h_E(e),v\rangle_E$, which equals $\beta_{h_E}(e)+\langle 2\eta,e\rangle_E$ by symmetry of $h_E$.
Applying the Artin--Schreier additivity isomorphism \eqref{eq:motivic-AS-additivity} to \eqref{eq:motivic-Gaussian-completing-square} yields a canonical isomorphism
\begin{equation}\label{eq:motivic-Gaussian-kernel-splitting}
\pr_E^*\sG_{h_E}\otimes
\bigl(\bigl\langle[2]_{\wh E}\circ\pr_{\wh E},\pr_E\bigr\rangle_E\bigr)^*\AS_{\psi,S}
\cong
\tau_{h_E}^*\pr_E^*\sG_{h_E}\otimes
\pr_{\wh E}^*\bigl((-\beta_{h_E}^{\vee})^*\AS_{\psi,S}\bigr).
\end{equation}
Substituting \eqref{eq:motivic-Gaussian-kernel-splitting} into \eqref{eq:motivic-Gaussian-two-pullback} and applying the projection formula for $\pr_{\wh E}$ gives
\begin{equation}\label{eq:motivic-Gaussian-after-projection-formula}
[2]_{\wh E}^*\FT_E(\sG_{h_E})
\cong
(-\beta_{h_E}^{\vee})^*\AS_{\psi,S}\otimes
\pr_{\wh E,!}\bigl(\tau_{h_E}^*\pr_E^*\sG_{h_E}\bigr)\sm{[d(E/S)]}.
\end{equation}
Because $\tau_{h_E}$ is an automorphism over $\wh E$, the canonical isomorphism
$\tau_{h_E,!}\tau_{h_E}^*\cong\Id$ gives
\begin{equation}\label{eq:motivic-Gaussian-translation-invariance}
\pr_{\wh E,!}\bigl(\tau_{h_E}^*\pr_E^*\sG_{h_E}\bigr)
\cong
\pr_{\wh E,!}\pr_E^*\sG_{h_E}.
\end{equation}
Finally, base change for the Cartesian square
\[
\begin{tikzcd}[ampersand replacement=\&]
\wh E\times_S E \ar[r, "\pr_E"] \ar[d, "\pr_{\wh E}"'] \&
E \ar[d, "\pi_E"]\\
\wh E \ar[r, "\pi_{\wh E}"'] \&
S
\end{tikzcd}
\]
identifies
\begin{equation}\label{eq:motivic-Gaussian-product-base-change}
\pr_{\wh E,!}\pr_E^*\sG_{h_E}
\cong
\pi_{\wh E}^*\pi_{E!}\sG_{h_E}.
\end{equation}
Combining \eqref{eq:motivic-Gaussian-after-projection-formula}, \eqref{eq:motivic-Gaussian-translation-invariance}, and \eqref{eq:motivic-Gaussian-product-base-change} gives
\[
[2]_{\wh E}^*\FT_E(\sG_{h_E})
\cong
(-\beta_{h_E}^{\vee})^*\AS_{\psi,S}\otimes
\pi_{\wh E}^*\bigl(\pi_{E!}\sG_{h_E}\sm{[d(E/S)]}\bigr),
\]
which is the desired isomorphism \eqref{eq:motivic-Gaussian-self-duality-on-E}.
\end{proof}

\begin{lemma}[Invertibility of the relative Gauss cohomology]\label{lem:motivic-relative-Gauss-invertible}
Maintain the notation of Construction \ref{constr:motivic-Gaussian-self-duality}, and assume in addition that $E\to S$ is a vector bundle in the classical sense (i.e., the total space of a locally free sheaf). Set
\begin{equation}\label{eq:motivic-relative-Gauss-object}
\mathfrak{G}(E,h_E):=
\pi_{E!}\sG_{h_E}\sm{[d(E/S)]}
\in\Dmot{S}.
\end{equation}
Then $\mathfrak{G}(E,h_E)$ is tensor-invertible.  More precisely, there is a canonical isomorphism
\begin{equation}\label{eq:motivic-relative-Gauss-inverse}
\mathfrak{G}(E,h_E)\otimes
\mathfrak{G}(E,-h_E)(d(E/S))
\xrightarrow{\sim}
\Qsh{S}.
\end{equation}
\end{lemma}

\begin{proof}
Apply Lemma \ref{lem:motivic-Gaussian-self-duality} to $h_E$ and to $-h_E$.  Since $p\ne2$, the map $[2]_{\wh E}$ is an automorphism over $S$, and Lemma \ref{lem:motivic-Gaussian-self-duality} gives
\begin{align*}
[2]_{\wh E}^*\FT_E(\sG_{h_E})&\cong (-\beta_{h_E}^{\vee})^*\AS_{\psi,S}\otimes\pi_{\wh E}^*\mathfrak{G}(E,h_E),\\
[2]_{\wh E}^*\FT_E(\sG_{-h_E})&\cong (-\beta_{-h_E}^{\vee})^*\AS_{\psi,S}\otimes\pi_{\wh E}^*\mathfrak{G}(E,-h_E).
\end{align*}
Here $\beta_{-h_E}^{\vee}=-\beta_{h_E}^{\vee}$, so
\[
(-\beta_{h_E}^{\vee})^*\AS_{\psi,S}\otimes(-\beta_{-h_E}^{\vee})^*\AS_{\psi,S}\cong\Qsh{\wh E}
\]
canonically. (Under the identification $h_E\co E\cong\wh E$ these Fourier-dual Gaussians are $\sG_{-h_E}$ and $\sG_{h_E}$, respectively, though we will not need this.)
Use the Plancherel isomorphism in \eqref{eq:motivic-FT-Plancherel} with $\cK_1=\sG_{h_E}$ and $\cK_2=\sG_{-h_E}$.  Because $\beta_{-h_E} \circ [-1]_E = \beta_{-h_E}$, one has $[-1]_E^*\sG_{-h_E}=\sG_{-h_E}$, and the tensor product $\sG_{h_E}\otimes\sG_{-h_E}$ is canonically $\Qsh{E}$.  Thus the right side of Plancherel is
\[
\pi_{E!}\Qsh{E}\sm{(-d(E/S))}.
\]
The left side, using the two displayed Fourier-transform identifications (after inserting $[2]_{\wh E}^*$, which does not change $\pi_{\wh E!}$ since $[2]_{\wh E}$ is an automorphism over $S$), is
\[
\mathfrak{G}(E,h_E)\otimes\mathfrak{G}(E,-h_E)\otimes \pi_{\wh E!}\Qsh{\wh E}.
\]
Since $E$ is a vector bundle in the classical sense, of rank $d(E/S)$, we have 
\[
\pi_{E!}\Qsh{E}\cong\pi_{\wh E!}\Qsh{\wh E}\cong \Qsh{S}\sm{[-2d(E/S)](-d(E/S))}.
\]
Cancelling this common factor leaves the canonical isomorphism
\[
\mathfrak{G}(E,h_E)\otimes\mathfrak{G}(E,-h_E)(d(E/S))
\cong
\Qsh{S},
\]
which is \eqref{eq:motivic-relative-Gauss-inverse}.
\end{proof}

\begin{lemma}
\label{lem:motivic-relative-Gauss-finite-etale}
Assume the hypotheses of Lemma \ref{lem:motivic-relative-Gauss-invertible} and that $d=d(E/S)$ is constant. Then there are tensor-invertible objects $\mathfrak g_d\in\Dmot{\Spec\F_q}$ and $\varepsilon(E,h_E)\in\Dmot{S}$ such that
\begin{itemize}
    \item There is a finite \'{e}tale surjection $u\co S'\to S$ such that
$u^*\varepsilon(E,h_E)\cong\Qsh{S'}$;
    \item There is an isomorphism
    \[
\mathfrak G(E,h_E)
\cong
\left.\mathfrak g_d\right|_S\otimes\varepsilon(E,h_E).
\]
\end{itemize}
\end{lemma}

\begin{proof}
The case $d=0$ is immediate. Suppose $d>0$, and set
\[
q_d(x_1,\ldots,x_d):=x_1^2+\cdots+x_d^2,
\qquad
\mathfrak g_d:=
\pi_{d!}q_d^*\AS_{\psi_0}\sm{[d]},
\qquad
\pi_d\co\A^d_{\F_q}\to\Spec\F_q.
\]
Lemma \ref{lem:motivic-relative-Gauss-invertible} shows that $\mathfrak g_d$ is tensor-invertible. Define
\[
\varepsilon(E,h_E):=
\mathfrak G(E,h_E)\otimes
\left(\left.\mathfrak g_d\right|_S\right)^{-1}.
\]
Let
\[
\pi_P\co P:=
\Isom_S\bigl((\cO_S^d,q_d),(E,\beta_{h_E})\bigr)\longrightarrow S.
\]
Because $p\ne2$, the quadratic form $\beta_{h_E}$ is \'{e}tale-locally isometric to $q_d$; hence $P$ is an $\mathrm O_d=\Aut(q_d)$-torsor. The universal isometry over $P$ induces
\begin{equation}\label{eq:orthogonal-trivialization}
\pi_P^*\mathfrak G(E,h_E)
\xrightarrow{\sim}
\left.\mathfrak g_d\right|_P,
\end{equation}
and therefore
$\pi_P^*\varepsilon(E,h_E)\cong\Qsh{P}$.

We have two pullbacks of \eqref{eq:orthogonal-trivialization} to $P \times_S P$, and under the isomorphism $P\times_S P\cong P\times\mathrm O_d$, the discrepancy between them is the action of $\mathrm O_d$ on $\mathfrak g_d$. Since $\mathfrak g_d$ is tensor-invertible, this action gives a class in $\Aut_{\Dmot{\mathrm O_d}}(\Qsh{\mathrm O_d})\cong \rH^0(\mathrm O_d,\Qsh{\mathrm O_d})^\times$. Its restriction to $\mathrm{SO}_d$ is constant because $\mathrm{SO}_d$ is connected, and its value at the identity is $1$. Hence $\varepsilon(E,h_E)$ is trivialized by the finite \'{e}tale cover $P/\mathrm{SO}_d\to S$.
\end{proof}

\section{Cohomological correspondences}\label{sec: cohomological correspondences}

The notion of \emph{cohomological correspondence} is reviewed in \cite[\S 4.1.1]{FYZ3}. Here we develop some supplementary technical material about cohomological correspondences that will be needed later.  

\subsection{Tensoring cohomological correspondences}\label{ssec: tensoring cohomological correspondences}
There is an operation of tensoring a cohomological correspondence with a map, defined formally as follows. Here $A_0 \xleftarrow{a_0} C^\flat \xrightarrow{a_1} A_1$ is a correspondence of derived Artin stacks, and $\cK_i, \cL_i \in \Dmot{A_i}$. Consider a cohomological correspondence
\[
a_0^*\cK_0 \xrightarrow{\cc} a_1^!\cK_1 \in \Corr_{C^\flat}(\cK_0, \cK_1).
\]
Suppose we also have a map
\[
u:a_0^*\cL_0\xrightarrow{\sim}a_1^*\cL_1 .
\]

\begin{defn}\label{def:tensor-coh-corr}
 The tensor product $\cc\otimes u \in \Corr_{C^\flat}(\cK_0 \otimes \cL_0, \cK_1 \otimes \cL_1)$ is the composite
\[
a_0^*(\cK_0\otimes \cL_0)
=a_0^*\cK_0\otimes a_0^*\cL_0
\xrightarrow{\cc\otimes u}
a_1^!\cK_1\otimes a_1^*\cL_1
\xrightarrow{\diamond}
a_1^!(\cK_1\otimes \cL_1).
\]
Here the arrow $\diamond$ is the
base change natural transformation $f^!(-)\otimes f^*(-)\to f^!(-\otimes -)$, applied with $f=a_1$. For a morphism $f\co A\to B$, this is the Beck--Chevalley morphism called the ``pull-pull'' base change transformation $\diamond$ in \cite[\S 3.5]{FYZ3}, obtained as the mate of the proper base change isomorphism for the Cartesian square
\begin{equation}\label{eq: tensoring-base-change-square}
\begin{tikzcd}
A \ar[r, "{(\Id_A,f)}"] \ar[d, "f"'] & A \times B \ar[d, "{f\times\Id_B}"] \\
B \ar[r, "\Delta_B"']  & B \times B
\end{tikzcd}
\end{equation}
\end{defn}

We are particularly interested in the case $\cK_0 = \cK_1$ and $\cL_0 = \cL_1$. 

\subsection{Compositions}\label{ssec: compositions of cohomological correspondences}
We will use compositions of ordinary and cohomological correspondences. Some of this material appears in \cite[\S 8]{FHM}. 

\subsubsection{Composing correspondences}
Two correspondences are composable if the codomain of one is the domain of the other: 
\begin{equation}\label{eq: two correspondences to compose}
\begin{tikzcd}
        & C_{01} \ar[dl, "a_0"'] \ar[dr, "a_1"] &   &  C_{12} \ar[dl, "c_1"'] \ar[dr, "c_2"]  \\ 
A_0 & &  A_1 & &  A_2 .
\end{tikzcd}
\end{equation}
In this case, let $C_{02}$ be the fibered product
\begin{equation}\label{eq: composition fiber product square}
\begin{tikzcd}
        & & C_{02}  \ar[dl, "\pr_{01}"'] \ar[dr, "\pr_{12}"] \\ 
        & C_{01} \ar[dl, "a_0"'] \ar[dr, "a_1"] &   &  C_{12} \ar[dl, "c_1"'] \ar[dr, "c_2"]  \\ 
A_0 & &  A_1 & &  A_2 .
\end{tikzcd}
\end{equation}
The composite correspondence is
\begin{equation}\label{eq: composite correspondence}
\begin{tikzcd}[ampersand replacement=\&]
A_0 \&
C_{02} \ar[l, "p_0"'] \ar[r, "p_2"] \&
A_2 ,
\end{tikzcd}
\qquad
p_0=a_0\pr_{01},\quad p_2=c_2\pr_{12}.
\end{equation}

We will apply this to a 1-legged Hecke correspondence of the form
\[
\begin{tikzcd}[ampersand replacement=\&, column sep=small]
Y \&
\Hk_Y^1 \ar[l, "h_0"'] \ar[r, "h_1"] \&
Y .
\end{tikzcd}
\]
For $r \geq 1$, the $r$-fold composition of such a correspondence is
\begin{equation}\label{eq: iterated Hecke correspondence}
\Hk_Y^r=
\underbrace{\Hk_Y^1\times_{h_1,Y,h_0}\Hk_Y^1
\times_{h_1,Y,h_0}\cdots
\times_{h_1,Y,h_0}\Hk_Y^1}_{r\text{ factors}},
\qquad
\begin{tikzcd}[ampersand replacement=\&]
Y \&
\Hk_Y^r \ar[l, "h_0"'] \ar[r, "h_r"] \&
Y .
\end{tikzcd}
\end{equation}
For $r=0$, we use the convention $\Hk_Y^0=Y$.

\subsubsection{Composing cohomological correspondences} Next we explain how to compose cohomological correspondences.

\begin{defn}\label{def: composite cohomological correspondence}
Let
\[
\cc_{01}\co a_0^*\cK_0\longrightarrow a_1^!\cK_1,
\qquad
\cc_{12}\co c_1^*\cK_1\longrightarrow c_2^!\cK_2
\]
be cohomological correspondences on \eqref{eq: two correspondences to compose}.  Their composite
$\cc_{12}\circ\cc_{01}\in\Corr_{C_{02}}(\cK_0,\cK_2)$ is defined to be the morphism on $C_{02}$
\begin{equation}\label{eq: composition cohomological correspondence}
\begin{aligned}
p_0^*\cK_0
&=\pr_{01}^*a_0^*\cK_0
\xrightarrow{\pr_{01}^*\cc_{01}}
\pr_{01}^*a_1^!\cK_1 \\
&\xrightarrow{\diamond}
\pr_{12}^!c_1^*\cK_1
\xrightarrow{\pr_{12}^!\cc_{12}}
\pr_{12}^!c_2^!\cK_2
=p_2^!\cK_2,
\end{aligned}
\end{equation}
where $\diamond$ is the pull-pull Beck--Chevalley transformation of \cite[\S 3.5]{FYZ3} for the pullback square in \eqref{eq: composition fiber product square}.
\end{defn}

Composition of cohomological correspondences is compatible with tensoring, in the following sense.

\begin{lemma}\label{lem: composition tensor compatibility}
With the notation of Definition \ref{def: composite cohomological correspondence}, suppose we are given
$u_{01}\co a_0^*\cL_0\xrightarrow{\sim}a_1^*\cL_1$ and
$u_{12}\co c_1^*\cL_1\xrightarrow{\sim}c_2^*\cL_2$. Let
$u_{02}\co p_0^*\cL_0\xrightarrow{\sim}p_2^*\cL_2$ be the composite
\begin{equation}\label{eq: composed tensoring isomorphism}
p_0^*\cL_0
=\pr_{01}^*a_0^*\cL_0
\xrightarrow{\pr_{01}^*u_{01}}
\pr_{01}^*a_1^*\cL_1
=\pr_{12}^*c_1^*\cL_1
\xrightarrow{\pr_{12}^*u_{12}}
\pr_{12}^*c_2^*\cL_2
=p_2^*\cL_2.
\end{equation}
Then we have 
\begin{equation}\label{eq: composition tensor compatibility}
(\cc_{12}\otimes u_{12})\circ(\cc_{01}\otimes u_{01})
=
(\cc_{12}\circ\cc_{01})\otimes u_{02} \in \Corr_{C_{02}}(\cK_0 \otimes \cL_0, \cK_2 \otimes \cL_2).
\end{equation}
\end{lemma}

\begin{proof}
After pulling Definition \ref{def:tensor-coh-corr} back to $C_{02}$, the tensoring square for $\cc_{01}$ is
\[
\begin{tikzcd}[ampersand replacement=\&, column sep = huge]
\pr_{01}^*a_0^*(\cK_0\otimes\cL_0)
\ar[r, "\pr_{01}^*\cc_{01}\otimes\pr_{01}^*u_{01}"] \ar[d, equals] \&
\pr_{01}^*\bigl(a_1^!\cK_1\otimes a_1^*\cL_1\bigr)
\ar[d, "\diamond"] \\
\pr_{01}^*a_0^*\cK_0\otimes\pr_{01}^*a_0^*\cL_0
\ar[r, "\pr_{01}^*(\cc_{01}\otimes u_{01})"'] \&
\pr_{01}^*a_1^!(\cK_1\otimes\cL_1)
\end{tikzcd}
\]
Compatibility of tensoring with the pull-pull map in \eqref{eq: composition cohomological correspondence} amounts to the commutativity of
\[
\begin{tikzcd}[ampersand replacement=\&, column sep = huge]
\pr_{01}^*a_1^!\cK_1\otimes\pr_{01}^*a_1^*\cL_1
\ar[r, "\pr_{01}^*\diamond_{\otimes}"] \ar[d, "\diamond_{\mathrm{pp}}\otimes\Id"'] \&
\pr_{01}^*a_1^!(\cK_1\otimes\cL_1)
\ar[d, "\diamond_{\mathrm{pp}}"] \\
\pr_{12}^!c_1^*\cK_1\otimes\pr_{12}^*c_1^*\cL_1
\ar[r, "\diamond_{\otimes}"'] \&
\pr_{12}^!c_1^*(\cK_1\otimes\cL_1)
\end{tikzcd}
\]
Here $\diamond_{\mathrm{pp}}$ is the pull-pull map for \eqref{eq: composition fiber product square}, and $\diamond_{\otimes}$ is the tensoring Beck--Chevalley map $f^!(-)\otimes f^*(-)\to f^!(-\otimes-)$. The morphism $\diamond_{\mathrm{pp}}\otimes\Id$ uses the identity $\pr_{01}^*a_1^*\cL_1=\pr_{12}^*c_1^*\cL_1$. Applying $\pr_{12}^!$ to the tensoring square for $\cc_{12}$ gives
\[
\begin{tikzcd}[ampersand replacement=\&, column sep = huge]
\pr_{12}^!c_1^*\cK_1\otimes\pr_{12}^*c_1^*\cL_1
\ar[r, "\pr_{12}^!\cc_{12}\otimes\pr_{12}^*u_{12}"] \ar[d, "\diamond_{\otimes}"'] \&
\pr_{12}^!c_2^!\cK_2\otimes\pr_{12}^*c_2^*\cL_2
\ar[d, "\diamond_{\otimes}"] \\
\pr_{12}^!c_1^*(\cK_1\otimes\cL_1)
\ar[r, "\pr_{12}^!(\cc_{12}\otimes u_{12})"'] \&
\pr_{12}^!c_2^!(\cK_2\otimes\cL_2)
\end{tikzcd}
\]
The morphism $\diamond_{\otimes}$ is the composite of the tensoring map for $\pr_{12}$ with $\pr_{12}^!$ applied to the tensoring map for $c_2$. The tensoring squares for $\cc_{01}$ and $\cc_{12}$ commute by definition and naturality. In the square involving $\diamond_{\mathrm{pp}}$, the two paths correspond to the two decompositions of the Cartesian rectangle obtained by concatenating the graph square of $a_1$ with \eqref{eq: composition fiber product square}. Compatibility of base-change transformations with concatenation shows that they agree, as in \cite[\S\S 3.4--3.5]{FYZ3} and (for the motivic setting) \cite[\S 4.2]{FK}. Consequently, the same composite is expressed either as $(\cc_{12}\otimes u_{12})\circ(\cc_{01}\otimes u_{01})$ or, after forming $\cc_{12}\circ\cc_{01}$ and $u_{02}$ as in \eqref{eq: composed tensoring isomorphism}, as $(\cc_{12}\circ\cc_{01})\otimes u_{02}$. This proves \eqref{eq: composition tensor compatibility}.
\end{proof}

\subsection{Cohomological co-correspondences}
Let 
\begin{equation}\label{eq:corr-cocorr}
\begin{tikzcd}
& C^\flat \ar[dl,"a_0"'] \ar[dr, "a_1"] \\
A_0  \ar[dr, "a_1'"'] & & A_1  \ar[dl, "a_0'"] \\
& C^{\sh}
\end{tikzcd}
\end{equation}
be a commutative diagram of derived Artin stacks which is Cartesian on the underlying reduced stacks. (This weaker condition suffices for the proper base-change step $a_{1!}a_0^*\cong(a'_0)^*a'_{1!}$ in the construction of \eqref{eq:corr-cocorr-conversion}, by topological invariance of $\Dmot{(-)}$, cf.\ \cite[Remark 4.2.4]{FK}: the comparison map from $C^\flat$ to the derived fiber product $A_0\times_{C^\sh}A_1$ induces an equivalence of motivic categories, under which the exchange isomorphism for the honest Cartesian square transports to this diamond.)
For $\cK_0\in\Dmot{A_0}$ and $\cK_1\in\Dmot{A_1}$, this Cartesian diamond gives the standard correspondence/co-correspondence conversion (cf. \cite[4.8.1]{FYZ3})
\begin{equation}\label{eq:corr-cocorr-conversion}
\gamma_C\co
\Corr_{C^\flat}(\cK_0,\cK_1)
\xrightarrow{\sim}
\CoCorr_{C^\sh}(\cK_0,\cK_1).
\end{equation}
It is the isomorphism obtained from adjunctions and proper base change:
\[
\begin{aligned}
\Hom_{C^\flat}(a_0^*\cK_0,a_1^!\cK_1)
&\cong
\Hom_{A_1}(a_{1!}a_0^*\cK_0,\cK_1) \\
&\cong
\Hom_{A_1}((a'_0)^*a'_{1!}\cK_0,\cK_1) \\
&\cong
\Hom_{C^\sh}(a'_{1!}\cK_0,a'_{0*}\cK_1).
\end{aligned}
\]
We write $\gamma_C^{-1}$ for the inverse operation, converting a cohomological co-correspondence on $C^\sh$ into the corresponding cohomological correspondence on $C^\flat$.

\subsection{Fourier transform of cohomological correspondences}\label{ssec: motivic derived FT of cc}
In this subsection all Fourier transforms are the motivic derived Fourier transforms of \S \ref{sec: motivic derived Fourier transform}; we write $\FT_E=\FT_E^\psi$.  The statements below are the motivic analogues of the cohomological-correspondence Fourier analysis in \cite[\S 7]{FYZ3}.  We do not repeat the proofs: the arguments of \cite[\S 7]{FYZ3} carry over verbatim with motivic sheaves, using the formalism of \S \ref{sec: motivic derived Fourier transform}.

We first treat correspondences of derived vector bundles over a fixed base.  Assume that
\begin{equation}\label{eq:corr-cocorr-Fourier}
\begin{tikzcd}
& C^\flat \ar[dl,"a_0"'] \ar[dr, "a_1"] \\
A_0  \ar[dr, "a_1'"'] & & A_1  \ar[dl, "a_0'"] \\
& C^{\sh}
\end{tikzcd}
\end{equation}
is a Cartesian diagram of derived vector bundles over a common base $S$.  Its dual diagram is
\begin{equation}\label{eq:corr-cocorr-Fourier-dual}
\begin{tikzcd}[ampersand replacement=\&]
\& \wh{C^{\sh}} \ar[dl,"\wh{a'_1}"'] \ar[dr, "\wh{a'_0}"] \\
\wh A_0  \ar[dr, "\wh a_0"'] \& \& \wh A_1  \ar[dl, "\wh a_1"] \\
\& \wh{C^{\flat}} .
\end{tikzcd}
\end{equation}
For $\cK_i\in\Dmot{A_i}$, put $\wh{\cK}_i:=\FT_{A_i}(\cK_i)$.  The Fourier transform of a cohomological correspondence first gives a co-correspondence on $\wh{C^\flat}$:
\begin{equation}\label{eq:motivic-derived-FT-corr-prime-same-base}
\FT'_{C^\flat} \co
\Corr_{C^\flat}(\cK_0,\cK_1)
\xrightarrow{\sim}
\CoCorr_{\wh{C^\flat}}
\bigl(\wh{\cK}_0,\wh{\cK}_1\sm{[d(a_0)+d(a_1)](d(a_0))}\bigr).
\end{equation}
Applying the correspondence/co-correspondence conversion $\gamma_{\wh C}^{-1}$ for the dual Cartesian diamond \eqref{eq:corr-cocorr-Fourier-dual} gives the same-base Fourier transform of cohomological correspondences
\begin{equation}\label{eq:motivic-derived-FT-corr-same-base}
\FT_{C^\flat}:=\gamma_{\wh C}^{-1}\circ\FT'_{C^\flat}\co
\Corr_{C^\flat}(\cK_0,\cK_1)
\xrightarrow{\sim}
\Corr_{\wh{C^\sh}}
\bigl(\wh{\cK}_0,\wh{\cK}_1\sm{[d(a_0)+d(a_1)](d(a_0))}\bigr).
\end{equation}

We also need the varying-base form.  Suppose that
\[
\begin{tikzcd}[ampersand replacement=\&]
A_0 \ar[d] \& C^\flat \ar[l, "p_0"'] \ar[r, "p_1"] \ar[d] \& A_1 \ar[d] \\
S_0 \& S_C \ar[l, "h_0"'] \ar[r, "h_1"] \& S_1
\end{tikzcd}
\]
is a map of correspondences in which $A_i\to S_i$ and $C^\flat\to S_C$ are derived vector bundles, and $p_0,p_1$ are linear.  Let $\wt A_i:=A_i\times_{S_i,h_i}S_C$, and let $\wt p_i\co C^\flat\to\wt A_i$ be the induced maps.  Let $C^\sh$ be the pushout of $\wt A_0\xleftarrow{\wt p_0}C^\flat\xrightarrow{\wt p_1}\wt A_1$ in derived vector bundles over $S_C$; then the diamond
\begin{equation}\label{eq:motivic-derived-varying-base-diamond}
\begin{tikzcd}[ampersand replacement=\&]
\& C^\flat \ar[dl, "\wt p_0"'] \ar[dr, "\wt p_1"] \\
\wt A_0 \ar[dr] \& \& \wt A_1 \ar[dl] \\
\& C^\sh
\end{tikzcd}
\end{equation}
is Cartesian.  Let $h_i^{\wh A}\co\wh{\wt A_i}\to\wh A_i$ be the base-change maps on the dual derived vector bundles.  The varying-base motivic derived Fourier transform is the canonical isomorphism
\begin{equation}\label{eq:motivic-derived-FT-corr-varying-base}
\FT_{C^\flat}\co
\Corr_{C^\flat}(\cK_0,\cK_1)
\xrightarrow{\sim}
\Corr_{\wh{C^\sh}}
\bigl((h_0^{\wh A})^*\FT_{A_0}(\cK_0),
(h_1^{\wh A})^!\FT_{A_1}(\cK_1)\sm{[d(\wt p_0)+d(\wt p_1)](d(\wt p_0))}\bigr).
\end{equation}
Following \cite[\S 7]{FYZ3}, this isomorphism is \emph{defined} as the composite of the following three steps. First, one uses the identification
\[
\Corr_{C^\flat}(\cK_0,\cK_1)
=
\Corr_{C^\flat}\bigl((h_0^{A})^{*}\cK_0,(h_1^{A})^{!}\cK_1\bigr)
\]
of correspondence groups over $S_C$, where $h_i^{A}\co\wt A_i\to A_i$ denote the projections. Second, one applies the transform \eqref{eq:motivic-derived-FT-corr-same-base} for the diamond \eqref{eq:motivic-derived-varying-base-diamond}. Third, one uses the base-change compatibilities $\FT_{\wt A_i}\circ(h_i^{A})^{*}\cong(h_i^{\wh A})^{*}\circ\FT_{A_i}$ and $\FT_{\wt A_i}\circ(h_i^{A})^{!}\cong(h_i^{\wh A})^{!}\circ\FT_{A_i}$, up to the indicated shifts and twists, of Theorem \ref{thm:motivic-FT-formalism}(1). When the base-change maps $h_i^{\wh A}$ are clear from context, we suppress them from the notation.

\begin{prop}[Push-pull compatibility]\label{prop:motivic-derived-FT-cc-functoriality}
Consider a morphism of varying-base linear correspondences
\[
\begin{tikzcd}[ampersand replacement=\&]
A_0 \ar[d, "f_0"'] \&
C^\flat \ar[l, "p_0"'] \ar[r, "p_1"] \ar[d, "f^\flat"] \&
A_1 \ar[d, "f_1"] \\
B_0 \&
D^\flat \ar[l, "q_0"] \ar[r, "q_1"'] \&
B_1
\end{tikzcd}
\]
over the same base correspondence $S_0\xleftarrow{h_0}S_C\xrightarrow{h_1}S_1$.  Let $\wt p_i$ and $\wt q_i$ denote the correspondence maps after base change to $S_C$, and let $C^\sh,D^\sh$ be the pushouts used in \eqref{eq:motivic-derived-varying-base-diamond}. The morphism of spans induces a map $f^\sh\co C^\sh\to D^\sh$. We denote its dual by $\wh f^\sh\co\wh{D^\sh}\to\wh{C^\sh}$, and similarly write $\wh f_i\co\wh B_i\to\wh A_i$ for the dual of $f_i$. Assume the induced diagram and its Fourier-dual diagram are globally presented.

If $f^\flat$ is left pushable, then for $\cK_i\in\Dmot{A_i}$ the diagram
\begin{equation}\label{eq:motivic-derived-FT-cc-push-pull}
\xymatrix{
\Corr_{C^\flat}(\cK_0,\cK_1) \ar[rr]^-{\FT_{C^\flat}}\ar[d]_{f^\flat_!} &&
\Corr_{\wh{C^\sh}}\bigl(\FT_{A_0}\cK_0,\FT_{A_1}\cK_1\sm{[d(\wt p_0)+d(\wt p_1)](d(\wt p_0))}\bigr)
\ar[d]^{(\wh f^\sh)^*} \\
\Corr_{D^\flat}(f_{0!}\cK_0,f_{1!}\cK_1)\ar[rr]^-{\TT_{\sm{[d(f_0)]}}\FT_{D^\flat}} &&
\Corr_{\wh{D^\sh}}\bigl((\wh f_0)^*\FT_{A_0}\cK_0,
(\wh f_1)^*\FT_{A_1}\cK_1\sm{[d(\wt q_0)+d(\wt q_1)+d(f_0)-d(f_1)](d(\wt q_0))}\bigr)
}
\end{equation}
commutes.  If $f^\flat$ is right pullable, then, setting $\delta_{f^\flat} := d(\wt p_1)-d(\wt q_1) = d(f^\flat)-d(f_1)$ for the defect of the right square (recall that the pullback of cohomological correspondences along a right pullable map carries the twist $\tw{-\delta_{f^\flat}}$, cf.\ \cite[\S 4.4]{FYZ3}), for $\cK_i\in\Dmot{B_i}$ the diagram
\begin{equation}\label{eq:motivic-derived-FT-cc-pull-push}
\begin{tikzcd}[ampersand replacement=\&, column sep=huge]
\Corr_{D^\flat}(\cK_0,\cK_1)
\ar[rr, "\FT_{D^\flat}"] \ar[d, "(f^\flat)^*"'] \&\&
\Corr_{\wh{D^\sh}}\bigl(\FT_{B_0}\cK_0,\FT_{B_1}\cK_1\sm{[d(\wt q_0)+d(\wt q_1)](d(\wt q_0))}\bigr)
\ar[d, "(\wh f^\sh)_!"] \\
\Corr_{C^\flat}(f_0^*\cK_0,f_1^*\cK_1\tw{-\delta_{f^\flat}})
\ar[rr, "\TT_{\sm{[d(f_0)](d(f_0))}}\FT_{C^\flat}"'] \&\&
\Corr_{\wh{C^\sh}}\bigl((\wh f_0)_!\FT_{B_0}\cK_0,
(\wh f_1)_!\FT_{B_1}\cK_1\sm{[d(\wt q_0)+d(\wt q_1)](d(\wt q_0))}\bigr)
\end{tikzcd}
\end{equation}
commutes.
\end{prop}

Next we investigate compatibility with compositions. Assume that the two correspondences in \eqref{eq: two correspondences to compose} are the upper halves of globally presented Cartesian diamonds of derived vector bundles
\begin{equation}\label{eq: composable Cartesian diamonds for Fourier composition}
\begin{tikzcd}[ampersand replacement=\&]
\& C_{01}^{\flat} \ar[dl, "a_0"'] \ar[dr, "a_1"] \\
A_0 \ar[dr, "a'_1"'] \& \& A_1 \ar[dl, "a'_0"]\\
\& C_{01}^{\sh}
\end{tikzcd}
\qquad
\begin{tikzcd}[ampersand replacement=\&]
\& C_{12}^{\flat} \ar[dl, "c_1"'] \ar[dr, "c_2"] \\
A_1 \ar[dr, "c'_2"'] \& \& A_2 \ar[dl, "c'_1"]\\
\& C_{12}^{\sh}.
\end{tikzcd}
\end{equation}
Here $C_{01}^{\flat}=C_{01}$ and $C_{12}^{\flat}=C_{12}$.  Let $C_{02}^{\flat}=C_{01}^{\flat}\times_{A_1}C_{12}^{\flat}$, and let $C_{02}^{\sh}$ be the lower term of the Cartesian diamond for the composite correspondence; equivalently, in global perfect-complex presentations, $C_{02}^{\sh}$ is the pushout of $C_{01}^{\sh}\leftarrow A_1\rightarrow C_{12}^{\sh}$.  Thus the Fourier-dual correspondence is
\begin{equation}\label{eq: dual composite correspondence}
\begin{tikzcd}[ampersand replacement=\&]
\wh A_0 \&
\wh{C_{02}^{\sh}}\cong \wh{C_{01}^{\sh}}\times_{\wh A_1}\wh{C_{12}^{\sh}}
\ar[l] \ar[r] \&
\wh A_2 .
\end{tikzcd}
\end{equation}
Write $\wh{\cK}_i=\FT_{A_i}(\cK_i)$, and set 
\[
\begin{gathered}
\sigma_{01}=d(a_0)+d(a_1),\quad \tau_{01}=d(a_0),\\
\sigma_{12}=d(c_1)+d(c_2),\quad \tau_{12}=d(c_1),\\
\sigma_{02}=d(p_0)+d(p_2),\quad \tau_{02}=d(p_0).
\end{gathered}
\]
Base change and additivity of virtual relative dimension give $\sigma_{02}=\sigma_{01}+\sigma_{12}$ and $\tau_{02}=\tau_{01}+\tau_{12}$.
For $ij= \in \{01,12,02\}$, let $\FT_{C_{ij}^{\flat}}$ denote the shifted-and-twisted Fourier transform of cohomological correspondences of \eqref{eq:motivic-derived-FT-corr-same-base}, so that
\[
\FT_{C_{ij}^{\flat}}(\cc_{ij})\in
\Corr_{\wh{C_{ij}^{\sh}}}
\bigl(\wh{\cK}_i,\wh{\cK}_j\sm{[\sigma_{ij}](\tau_{ij})}\bigr).
\]

\begin{lemma}\label{lem: composition Fourier compatibility}
In the globally presented setup \eqref{eq: composable Cartesian diamonds for Fourier composition}, Fourier transform is compatible with composition of cohomological correspondences.  More precisely,
\begin{equation}\label{eq: Fourier composition compatibility}
\FT_{C_{02}^{\flat}}(\cc_{12}\circ\cc_{01})
=
\bigl(\FT_{C_{12}^{\flat}}(\cc_{12})\sm{[\sigma_{01}](\tau_{01})}\bigr)
\circ
\FT_{C_{01}^{\flat}}(\cc_{01})
\in 
\Corr_{\wh{C_{02}^{\sh}}}
\bigl(\wh{\cK}_0,\wh{\cK}_2\sm{[\sigma_{02}](\tau_{02})}\bigr).
\end{equation}
\end{lemma}

\begin{proof}
By Definition \ref{def: composite cohomological correspondence}, the composition
$\cc_{12}\circ\cc_{01}$ is obtained by pulling $\cc_{01}$ and $\cc_{12}$
to $C_{02}^{\flat}=C_{01}^{\flat}\times_{A_1}C_{12}^{\flat}$, inserting the
pull-pull Beck--Chevalley transformation for the Cartesian square
\eqref{eq: composition fiber product square}, and then composing.  The two
diamonds in \eqref{eq: composable Cartesian diamonds for Fourier composition}
identify the Fourier-dual lower term of the composite with the fiber product
$\wh{C_{01}^{\sh}}\times_{\wh A_1}\wh{C_{12}^{\sh}}$, as displayed in
\eqref{eq: dual composite correspondence}.

The only non-formal input is the compatibility of this Beck--Chevalley
transformation with Fourier transform.  Applying Proposition
\ref{prop:motivic-FT-proper-base-change} to the Cartesian square
\eqref{eq: composition fiber product square}, using the two Cartesian
diamonds in \eqref{eq: composable Cartesian diamonds for Fourier composition},
shows that the Fourier transform of the pull-pull map used in the definition of
$\cc_{12}\circ\cc_{01}$ is the pull-pull Beck--Chevalley map for the dual
Cartesian square in \eqref{eq: dual composite correspondence}.  The
correspondence/co-correspondence conversions $\gamma_{\wh C}^{-1}$ are
defined by adjunction and proper base change, so they are compatible with this
identified Beck--Chevalley map.

After this identification, expanding the definition of
$\FT_{C_{ij}^{\flat}}$ in \eqref{eq:motivic-derived-FT-corr-same-base}
gives \eqref{eq: Fourier composition compatibility}.  The source and target twists agree because
base change in \eqref{eq: composable Cartesian diamonds for Fourier composition}
gives $\sigma_{02}=\sigma_{01}+\sigma_{12}$ and
$\tau_{02}=\tau_{01}+\tau_{12}$; this is exactly the additional
$\sm{[\sigma_{01}](\tau_{01})}$ applied to
$\FT_{C_{12}^{\flat}}(\cc_{12})$ in \eqref{eq: Fourier composition compatibility}.
\end{proof}

\subsection{Tensor-convolution duality}

Given a commutative group stack $A/S$ with addition operation $+ \co A \times_S A \rightarrow A$, we can define the convolution of $\cK, \cL \in \Dmot{A}$ as
\[
\cK \star\cL:=+_!(\cK \boxtimes_S \cL ) \in \Dmot{A}.
\]

We next define the Fourier-dual operation to tensoring. Assume now that the diagram \eqref{eq:corr-cocorr} is of commutative group stacks and homomorphisms over a common base $S$. Consider a
cohomological co-correspondence
\begin{equation}\label{eq:cocorr-convolution-inputs}
a'_{1!}\cK_0 \xrightarrow{\cc} a'_{0*}\cK_1 \in \CoCorr_{C^\sh}(\cK_0, \cK_1).
\end{equation}
Suppose we are given an isomorphism
\[
v:a'_{1!}\cL_0\xrightarrow{\sim}a'_{0!}\cL_1 .
\]
\begin{defn}
We define the convolution $\cc \star v \in \CoCorr_{C^\sh}(\cK_0 \star \cL_0, \cK_1 \star \cL_1)$ to be the composite
\begin{equation}\label{eq:conv-cocorr}
a'_{1!}(\cK_0 \star \cL_0)
\cong(a'_{1!}\cK_0)\star(a'_{1!}\cL_0)
\xrightarrow{\cc\star v}
(a'_{0*}\cK_1)\star(a'_{0!}\cL_1)
\rightarrow
a'_{0*}(\cK_1 \star \cL_1).
\end{equation}
Here the first isomorphism is the canonical convolution compatibility for the homomorphism $a'_1$. The last map is the Beck--Chevalley morphism, called the ``push-push'' base change natural transformation $\diamond$ in \cite[\S 3.3]{FYZ3}, obtained as the mate of the proper base change isomorphism for the Cartesian square
\[
\begin{tikzcd}[ampersand replacement=\&]
A \times_S A \ar[r, "+"] \ar[d, "f \times \Id"'] \& A \ar[d, "f"] \\
B \times_S A \ar[r, "+"] \& B
\end{tikzcd}
\]
with $f = a_0'$, where the bottom $+$ is the action map $(b,a)\mapsto b+f(a)$. This form of the square places the $*$-pushforward on the first factor, matching the middle term of \eqref{eq:conv-cocorr}; it is the Fourier-dual of the tensoring square \eqref{eq: tensoring-base-change-square}.
\end{defn}

\begin{lemma}\label{lem: tensor-convolution dual} Below, for $\cK \in \Dmot{A}$ we abbreviate $\wh \cK := \FT_A(\cK)$. 

Assume that the diagrams \eqref{eq:corr-cocorr-Fourier} and \eqref{eq: tensoring-base-change-square} with $f=a_1$
are globally presented in the sense of \cite[\S 6.3.1]{FYZ3} (so that motivic Fourier transform is compatible with proper base change, by Proposition \ref{prop:motivic-FT-proper-base-change}). Put $r_i:=d(A_i/S)$ and $\delta_i:=d(a_i)$ for $i=0,1$.

Let
\[
a_0^* \cK_0 \xrightarrow{\cc} a_1^! \cK_1 \in \Corr_{C^\flat}(\cK_0, \cK_1),
\qquad
u: a_0^*  \cL_0 \xrightarrow{\sim}a_1^*  \cL_1.
\]
Denote their Fourier transforms\footnote{See \eqref{eq:motivic-derived-FT-corr-prime-same-base} and \eqref{eq:motivic-derived-FT-corr-same-base} for the motivic derived Fourier transform of cohomological correspondences.} by
\[
\begin{aligned}
\wh{\cc} &:=\FT'_{C^\flat}(\cc) \in 
\CoCorr_{\wh{C^\flat}}
\bigl(\wh{\cK}_0,\wh{\cK}_1\sm{[\delta_0+\delta_1](\delta_0)}\bigr), \\
\wh u &:
\wh{a}_{0!}\wh{\cL}_0\xrightarrow{\sim}
\wh{a}_{1!} \wh{\cL}_1\sm{[\delta_0-\delta_1](\delta_0-\delta_1)},
\end{aligned}
\]
where $\wh{a}_i:\wh A_i\to\wh C^\flat$ denotes the map dual to
$a_i$. Under the Fourier-dual
tensor/convolution identification
\[
\FT_{A_i}(\cK_i\otimes\cL_i)
\cong \wh{\cK}_i\star\wh{\cL}_i\sm{[r_i](r_i)}
\]
from \eqref{eq:motivic-FT-tensor-convolution}, we have
\[
\TT_{\sm{[-r_0](-r_0)}}\FT'_{C^\flat}(\cc\otimes u)
=\wh{\cc}\star\wh u
\in
\CoCorr_{\wh{C^\flat}}
\bigl(\wh{\cK}_0\star\wh{\cL}_0,\wh{\cK}_1\star\wh{\cL}_1\sm{[2\delta_0](2\delta_0-\delta_1)}\bigr).
\]
\end{lemma}

\begin{proof}
By definition, $\cc \otimes u$ is given by the composition 
\begin{equation}\label{eq:cc-u-def}
\cc\otimes u \co a_0^*(\cK_0\otimes\cL_0) \rightarrow
a_1^!\cK_1\otimes a_1^*\cL_1
\xrightarrow{\diamond}
a_1^!(\cK_1\otimes\cL_1),
\end{equation}
where $\diamond$ is the Beck--Chevalley transformation mated to proper base change for the Cartesian diagram \eqref{eq: tensoring-base-change-square}. After applying
$\FT'_{C^\flat}$, using the preceding tensor/convolution identification, and then applying the shift-twist operator $\TT_{\sm{[-r_0](-r_0)}}$, the first arrow in \eqref{eq:cc-u-def} becomes
\[
(\wh{a}_{0!}\wh{\cK}_0)\star(\wh{a}_{0!}\wh{\cL}_0)
\rightarrow
(\wh{a}_{1*}\wh{\cK}_1\sm{[\delta_0+\delta_1](\delta_0)})
\star
(\wh{a}_{1!}\wh{\cL}_1\sm{[\delta_0-\delta_1](\delta_0-\delta_1)}).
\]
This is the first arrow in the definition \eqref{eq:conv-cocorr} of $\wh{\cc} \star \wh u$, applied to the dual diamond \eqref{eq:corr-cocorr-Fourier-dual} under the dictionary $a'_1=\wh{a}_0$, $a'_0=\wh{a}_1$, $C^\sh=\wh{C^\flat}$.

The Fourier transform of the second arrow in \eqref{eq:cc-u-def} is the Beck--Chevalley transformation mated to proper base change for the dual Cartesian diagram. Since we have assumed that the relevant diagrams are globally presented, Proposition \ref{prop:motivic-FT-proper-base-change} implies that the motivic Fourier transform carries the proper base change isomorphism for a Cartesian square to the
proper base change isomorphism for the Fourier-dual square. This implies that the motivic Fourier transform of the second arrow in \eqref{eq:cc-u-def} is the second arrow in the definition \eqref{eq:conv-cocorr} of $\wh{\cc} \star \wh u$, completing the proof.
\end{proof}

\section{A trace formula for the sheaf-cycle correspondence}\label{sec: trace formula}

In this section we prove a technical result establishing compatibility of pushforwards with formation of traces for certain \emph{nonproper} maps, the proper case being in \cite{FK}.

\subsection{Formulation}\label{ssec: trace formula formulation} Let
\[
\begin{tikzcd}S &  S^\flat \ar[l, "h_0"'] \ar[r, "h_1"] &  S
\end{tikzcd}
\]
be a correspondence of derived Artin stacks locally of finite type over $\F_q$. Let $\cE$ on $S$ and $\cE^\flat$ on $S^\flat$ be \emph{coconnective} perfect complexes, that is, perfect complexes of tor-amplitude in $[0,\infty)$ in our cohomological indexing. Recall from \cite[\S 6.1]{FYZ3} that $\Tot(\cE)$ is the functor $T\mapsto\Map_{\QCoh(T)}(\cO_T,u^*\cE)$, and set $E:=\Tot_S(\cE)$ and $E^\flat:=\Tot_{S^\flat}(\cE^\flat)$. By \cite[Lemma 6.1.5]{FYZ3}, these are relative affine derived schemes over their bases. They are classical affine schemes when the complexes are locally free in degree $0$; in general, their classical truncations are the associated abelian cones. Their zero sections are therefore closed embeddings.

For each $i\in\{0,1\}$, suppose we are given a map $\cE^\flat\to h_i^*\cE$ whose derived fiber has tor-amplitude in $[1,\infty)$. By \cite[Lemma 6.1.5]{FYZ3}, the induced map of total spaces $E^\flat\to h_i^*E$ is then a closed embedding of derived vector bundles.
This implies that the map of correspondences
\begin{equation}\label{eq: bundle correspondence map}
\begin{tikzcd}
E \ar[d, "f_0"] & E^\flat \ar[l, "a_0"'] \ar[r, "a_1"] \ar[d, "f^\flat"] & E  \ar[d, "f_1"] \\
S & S^\flat \ar[l, "h_0"'] \ar[r, "h_1"]  & S
\end{tikzcd}
\end{equation}
is pushable, since its left comparison map is the closed embedding $E^\flat\inj h_0^*E$. Consequently
\[
f_! \co \Corr_{E^\flat}(\cK_0, \cK_1) \rightarrow \Corr_{S^\flat}(f_{0!} \cK_0, f_{1!} \cK_1)
\]
is defined.

The derived pushforwards below are along schematic maps locally of finite type and therefore preserve geometricity of motivic sheaves (cf.\ \cite[\S 3.3]{FK}). The required compatibility is the following.

\begin{prop}\label{prop: trace of Hitchin pushforward} Assume the setup above, and let $\cc  \in \Corr_{E^\flat}(\cK, \cK\tw{-i})$ with $\cK \in \Dmotg{E}$. Then the induced map $\Sht(f) \co \Sht(E^\flat) \rightarrow \Sht(S^\flat)$ is proper, and we have
\[
\Tr^{\Sht}(f_{! } (\cc))  = \Sht(f)_! \Tr^{\Sht}(\cc) \in \CH_i(\Sht(S^\flat)).
\]
Traces are formed using the natural Weil structures (cf.\ \cite[\S 6.4.3]{FK}). 
\end{prop}

\begin{remark}
This does \emph{not} fall under the general compatibility results of \cite[\S 6.2]{FK}, the point being that the vertical maps in \eqref{eq: bundle correspondence map} are not proper. Proposition \ref{prop: trace of Hitchin pushforward} can be thought of as a (relative) version of the Grothendieck--Lefschetz trace formula for certain types of non-proper spaces. 
\end{remark}

\subsection{Proof of Proposition \ref{prop: trace of Hitchin pushforward}}\label{ssec: trace formula proof}

The strategy is to compactify the map, and show that the boundary does not contribute to the trace.

\sss{Projective completions}
The relevant projective completion is the following.

\begin{lemma}\label{lem: projective completion}
Let $T$ be a derived Artin stack locally of finite type over $\F_q$, let $\cF$ be a coconnective perfect complex on $T$, and let $F := \Tot_T(\cF)$. Define
\[
\bbP(F\oplus\cO) \,:=\, \bigl(\Tot_T(\cF\oplus\cO)\smallsetminus z_{F\oplus\cO}\bigr)/\G_m,
\]
where $z_{F\oplus\cO}$ denotes the zero section (a closed embedding, by coconnectivity) and $\G_m$ acts by fiberwise scaling. We likewise write $z_F$ for the zero section of $F$. Then:
\begin{enumerate}
\item The classical truncation of $\bbP(F\oplus\cO)$ is the projectivized abelian cone $\Proj_{T_{\cl}}\Sym(\cQ\oplus\cO)$, where $\cQ := H^0(\cF^*)$ is coherent. In particular $\bbP(F\oplus\cO) \rightarrow T$ is proper.
\item There is a stratification
\[
\bbP(F\oplus\cO) = F \sqcup \bbP(F),
\]
where $F \inj \bbP(F\oplus\cO)$ is the open locus on which the $\cO$-coordinate is invertible, and the boundary identifies with $\bbP(F) := (\Tot_T(\cF)\smallsetminus z_F)/\G_m$.
\end{enumerate}
\end{lemma}

\begin{proof}
(1) The fact that the
classical truncation of $\bbP(F\oplus\cO)$ is the projectivized abelian cone $\Proj_{T_{\cl}}\Sym(\cQ\oplus\cO)$ is clear from the definition. Properness is detected on classical truncations, so $\bbP(F\oplus\cO)\to T$ is proper.

(2) The projection $\cF\oplus\cO\to\cO$ gives a $\G_m$-equivariant fiberwise-linear map $\Tot_T(\cF\oplus\cO)\to\bbA^1_T$, hence a section of $\cO(1)$ on the quotient. Its derived fiber over $z_{\bbA^1_T}\co T\to\bbA^1_T$ is $\Tot_T(\cF)$, so its derived vanishing locus is $(\Tot_T(\cF)\smallsetminus z_F)/\G_m=\bbP(F)$. On the complement, the $\cO$-coordinate is invertible; rescaling it to $1$ identifies the complement with $\cF$. 
\end{proof}

Set $\ol E:=\bbP(E\oplus\cO)$ and $\ol{E^\flat}:=\bbP(E^\flat\oplus\cO)$. By Lemma \ref{lem: projective completion}, these are proper over their bases, with open interiors $\jmath_E\co E\inj\ol E$ and $\jmath^\flat\co E^\flat\inj\ol{E^\flat}$ and boundaries $\bbP(E)$ and $\bbP(E^\flat)$.

The maps $(\cE^\flat\to h_i^*\cE)\oplus\Id_{\cO}$ induce closed embeddings $\ol{E^\flat}\to h_i^*\ol E$. Since each has the form $(\text{linear})\oplus\Id_{\cO}$, it preserves the interior and boundary strata. Let $\ol a_i\co\ol{E^\flat}\to\ol E$ be the composite with the projection. The pullback of the boundary $\bbP(E)\subset\ol E$ pulls back under $\ol a_i$ to $\bbP(E^\flat)$, so the boundary is stabilized scheme-theoretically, as required in \cite[Definition 7.2.1]{FK}.

On the complementary opens, we have $\ol a_i^{\,-1}(E)=E^\flat$ as derived open substacks. Hence \eqref{eq: bundle correspondence map} extends to the compactified map
\begin{equation}\label{eq: compactified bundle correspondence map}
\begin{tikzcd}
\ol E \ar[d, "\ol f_0"] & \ol{E^\flat} \ar[l, "\ol a_0"'] \ar[r, "\ol a_1"] \ar[d, "\ol f^\flat"] & \ol E  \ar[d, "\ol f_1"] \\
S & S^\flat \ar[l, "h_0"'] \ar[r, "h_1"]  & S
\end{tikzcd}
\end{equation}
All vertical maps in \eqref{eq: compactified bundle correspondence map} are proper, and its comparison maps $\ol{E^\flat}\to h_i^*\ol E$ are closed embeddings; the map is therefore proper and pushable. The boundary strata form a correspondence $\bbP(E)\xleftarrow{\ol a_0}\bbP(E^\flat)\xrightarrow{\ol a_1}\bbP(E)$, which is the restriction of \eqref{eq: compactified bundle correspondence map} in the sense of \cite[Construction 7.2.3]{FK}. The open inclusions define a map of correspondences $\jmath$ from the top row of \eqref{eq: bundle correspondence map} to that of \eqref{eq: compactified bundle correspondence map}. Its squares are derived Cartesian, being preimages of open substacks, so $\jmath$ is pushable. Since \eqref{eq: bundle correspondence map} is the composite of $\jmath$ with \eqref{eq: compactified bundle correspondence map}, compatibility with composition of pushable maps (cf.\ \cite[\S 3.2, especially Proposition 3.2.3]{FYZ3}) gives
\begin{equation}\label{eq: trace formula factorization}
f_!\,\cc = \ol f_{!}\,(\jmath_!\,\cc).
\end{equation}

\sss{Decomposition of the shtuka stacks}
Both legs of \eqref{eq: compactified bundle correspondence map} preserve the stratification $\ol E=\bbP(E)\sqcup E$. Since the boundary is stabilized scheme-theoretically, the Frobenius twist of the compactified self-correspondence is contracting near $\bbP(E)$ by \cite[Example 7.2.2]{FK}. It follows from \cite[Proposition 7.5.4]{FK} that $\Sht(\bbP(E^\flat))\to\Sht(\ol{E^\flat})$ is an open-closed embedding on reduced substacks. Its complement is $\Sht(E^\flat)$: the underlying point of any shtuka lies in exactly one of the two strata, each stable under both legs. Thus we have 
\begin{equation}\label{eq: compactified Sht decomposition}
\Sht(\ol{E^\flat}) = \Sht(E^\flat) \sqcup \Sht(\bbP(E^\flat))
\end{equation}
on reduced substacks. Let $\jmath$ and $\imath$ denote the maps of correspondences
\[
\begin{tikzcd}
E \ar[d, "\jmath_0"'] & E^\flat \ar[l, "a_0"'] \ar[r, "a_1"] \ar[d, "\jmath^\flat"] & E \ar[d, "\jmath_1"] \\
\ol E & \ol{E^\flat} \ar[l, "\ol a_0"] \ar[r, "\ol a_1"'] & \ol E
\end{tikzcd}
\qquad
\begin{tikzcd}
\bbP(E) \ar[d, "\imath_0"'] & \bbP(E^\flat) \ar[l] \ar[r] \ar[d, "\imath^\flat"] & \bbP(E) \ar[d, "\imath_1"] \\
\ol E & \ol{E^\flat} \ar[l, "\ol a_0"] \ar[r, "\ol a_1"'] & \ol E
\end{tikzcd}
\]
given by the interior and boundary inclusions; in particular $\jmath_0 = \jmath_1 = \jmath_E$, and $\imath_0 = \imath_1$ is the inclusion $\bbP(E)\inj\ol E$. Write $\Sht(\jmath)$ and $\Sht(\imath)$ for the two open-closed embeddings. Furthermore, the morphism
\[
\Sht(\ol f) \co \Sht(\ol{E^\flat}) \rightarrow \Sht(S^\flat)
\]
is proper by \cite[Proposition 6.2.1]{FK} applied to the Frobenius twist of \eqref{eq: compactified bundle correspondence map}. Hence $\Sht(f) = \Sht(\ol f)\circ\Sht(\jmath)$ is proper, being the composition of an open-closed embedding with a proper map; in particular $\Sht(f)_!$ is defined on Chow groups, as asserted in Proposition \ref{prop: trace of Hitchin pushforward}.

\sss{Conclusion}
By \eqref{eq: trace formula factorization} and the compatibility of the trace with proper pushforward, \cite[Proposition 6.2.1]{FK}, applied to the proper map of correspondences \eqref{eq: compactified bundle correspondence map}, we have
\begin{equation}\label{eq: trace formula step 1}
\Tr^{\Sht}(f_!\,\cc) \,=\, \Tr^{\Sht}\bigl(\ol f_!(\jmath_!\,\cc)\bigr) \,=\, \Sht(\ol f)_!\,\Tr^{\Sht}(\jmath_!\,\cc).
\end{equation}
It remains to show that
\begin{equation}\label{eq: trace formula step 2}
\Tr^{\Sht}(\jmath_!\,\cc) \,=\, \Sht(\jmath)_!\,\Tr^{\Sht}(\cc) \,\in\, \CH_i(\Sht(\ol{E^\flat})).
\end{equation}
By the open-closed decomposition \eqref{eq: compactified Sht decomposition}, it suffices to compare the restrictions to its two components. On $\Sht(E^\flat)$, the open case of \cite[Proposition 6.2.2]{FK} identifies the restriction of the left side with $\Tr^{\Sht}(\jmath^*\jmath_!\,\cc)=\Tr^{\Sht}(\cc)$, which is also the restriction of the right side. On $\Sht(\bbP(E^\flat))$, the Frobenius twist is contracting near the boundary, and the restriction of $\jmath_!\,\cc$ to the boundary correspondence in the sense of \cite[Construction 7.2.3]{FK} is a cohomological correspondence for $\imath_0^*\,\jmath_{E!}\,\cK=0$. Hence \cite[Corollary 7.5.7]{FK} implies that the restriction of $\Tr^{\Sht}(\jmath_!\,\cc)$ vanishes, while the right side vanishes there for reasons of support. This proves \eqref{eq: trace formula step 2}.

Combining \eqref{eq: trace formula step 1} and \eqref{eq: trace formula step 2} with the equality $\Sht(\ol f)_!\,\Sht(\jmath)_! = \Sht(f)_!$ completes the proof. \qed

\part{The Modularity Conjecture}

\section{The genetic pattern of modularity}\label{sec:genetics-of-modularity}

In this section we establish certain sheaf-theoretic results which can be viewed as the ``genetics of modularity'' from Figure \ref{fig:cartoon}. The results here are essentially agnostic about the ground field.

\subsection{The bestiary of derived vector bundles}\label{ssec:bestiary}

Let $\cG \in \Bun_{GU^{-}(2m)}(k)$. Let $\cE_1, \cE_2 \subset \cG$ be transverse Lagrangians in the sense of \cite[\S 2]{FYZ3}. As discussed in \cite[\S 2, \S 9.1]{FYZ3}, the transversality assumption implies that we have short exact sequences
\begin{equation}\label{eq: bestiary torsion sequences}
\begin{gathered}
0 \rightarrow \sigma^* \cE_2 \rightarrow \cE_1^* \rightarrow Q_1 \rightarrow 0,\\
0 \rightarrow \sigma^* \cE_1 \rightarrow \cE_2^* \rightarrow Q_2 \rightarrow 0
\end{gathered}
\end{equation}
where $Q_1$ and $Q_2 \cong \sigma^* Q_1^* := \sigma^* \cExt^1_{X'}(Q_1, \cO_{X'})$ are torsion coherent sheaves on $X'$.\footnote{Our $Q_i$ agrees with the $\wt Q_i$ of \cite[\S 9.1]{FYZ3}, i.e.\ with $\sigma^*$ applied to the $Q_i$ of \cite[\S 2]{FYZ3}.} 

 Fix a Harder--Narasimhan truncation $\mu$ for $\Bun_{U(n)}$ and write $S = \Bun_{U(n)}^{\leq \mu}$. We let $\Hk_S^r$ be the corresponding Harder--Narasimhan truncation of the Hecke stack \cite[\S 9.1.2]{FYZ3}. This fits into a correspondence diagram 
 \[
 S \xleftarrow{h_0} \Hk_S^r \xrightarrow{h_r} S.
 \] 
 
From this setup we constructed a constellation of moduli spaces in \cite[\S 9.1,\S 9.2]{FYZ3}, and we will use the same notation here. 
\begin{enumerate}
\item For $i \in \{0,r\}$: we have derived vector bundles $U_i,V_i, W_i$ over $S$; and their respective dual derived vector bundles $W_i^\perp = \wh{U}_i, V'_i = \wh{V}_i, U_i^\perp = \wh{W}_i$. 

\item For $i \in \{0,r\}$: pulling back $U_i, V_i, W_i$ to $\Hk_S^r$ via $h_i$ gives derived vector bundles $\wt U_i, \wt V_i, \wt W_i$. Similarly, we have $\wt W_i^\perp, \wt V_i', \wt U_i^\perp$. 

\item The Hecke stacks 
\[
\Hk_U^\flat, \Hk_U^\sharp, \Hk_V^\flat, \Hk_V^\sharp, \Hk_W^\flat, \Hk_W^\sharp,
\]
which are derived vector bundles over $\Hk_S^r$; and their respective dual derived vector bundles 
\[
\Hk_{W^\perp}^{\sharp}, \Hk_{W^\perp}^{\flat}, \Hk_{\wh{V}}^\sharp, \Hk_{\wh{V}}^\flat, \Hk_{U^\perp}^\sharp, \Hk_{U^\perp}^\flat.
\]
\end{enumerate}

We briefly recall the definitions of these objects.  

\subsubsection{$U,V,W$ and their duals} Let $\cF_{\univ}$ denote the universal Hermitian bundle on $X'\times S$. The six base vector bundles are the total spaces of the following perfect complexes on $S$:
\[
\begin{array}{lll}
\cU=\ul{\RHom(\cF_{\univ}^*,\cE_1^*)},&
\cV=\ul{\RHom(\cF_{\univ}^*,Q_1)},&
\cW=\ul{\RHom(\cF_{\univ}^*,\sigma^*\cE_2\sm{[1]})},\\[2mm]
\cU^\perp=\ul{\RHom(\cF_{\univ}^*,\cE_2^*)},&
\cV'=\ul{\RHom(\cF_{\univ}^*,Q_2)},&
\cW^\perp=\ul{\RHom(\cF_{\univ}^*,\sigma^*\cE_1\sm{[1]})}.
\end{array}
\]
Then $U=\Tot_S(\cU)$, $V=\Tot_S(\cV)$, $W=\Tot_S(\cW)$, and similarly for the perpendicular collection. For $i=0,r$ we write $U_i,V_i,W_i$ for the corresponding copies over the $i$th copy $S_i$ of $S$, and we write $\wt U_i,\wt V_i,\wt W_i$ for their pullbacks along $h_i\co\Hk_S^r\to S_i$. The exact sequences \eqref{eq: bestiary torsion sequences} induce exact triangles in $\Perf(S)$,
\begin{equation}\label{eq: bestiary base triangles}
\cU\to\cV\to\cW,
\qquad
\cU^\perp\to\cV'\to\cW^\perp .
\end{equation}
Passing to total spaces, this induces the following pair of derived Cartesian squares over $S$:
\begin{equation}\label{eq: UVW zero fibers}
\begin{tikzcd}
U \ar[r, "f"] \ar[d, "\pi_U"'] & V \ar[d, "g"] \\
S \ar[r, "z_W"'] & W
\end{tikzcd}
\qquad \qquad 
\begin{tikzcd}
U^\perp \ar[r, "\wh g"] \ar[d, "\wh{z}_W"'] & \wh V \ar[d, "\wh{f}"] \\
S \ar[r, "\wh{\pi}_U"'] & W^\perp
\end{tikzcd}
\end{equation}
Here $z_W$ and $\wh{\pi}_U=z_{W^\perp}$ denote the zero sections of their respective derived vector bundles, while $\pi_U$ and $\wh{z}_W  = \pi_{U^\perp}$ is the structure morphism $U^\perp\to S$. From \eqref{eq: UVW zero fibers} we have corresponding derived Cartesian squares $(U_i,V_i,S_i,W_i)$ and $(U_i^\perp,\wh V_i,S_i,W_i^\perp)$ for $i=0,r$.
Here $\wh V_i$ denotes the Serre-dual vector bundle to $V_i$; relative Serre duality identifies $\Tot_S(\cV')$ with $\wh V$, identifies $U^\perp$ with $\wh W$, and identifies $W^\perp$ with $\wh U$. 

\subsubsection{Hecke correspondences}
We next recall the Hecke versions. Over $\Hk_S^r$ there are two universal complexes on $X'\times\Hk_S^r$.  For an $R$-point $(\cF_\star)\in\Hk_S^r(R)$, the complex $\cF_\bu^\flat$ on $X'_R$ is the two-term complex in degrees $0$ and $1$
\[
\left(\cF_{1/2}^\flat\oplus\cdots\oplus\cF_{r-1/2}^\flat\right)
\longrightarrow
\left(\cF_1\oplus\cdots\oplus\cF_{r-1}\right),
\]
where the differential maps $(s_{1/2},\ldots,s_{r-1/2})$ to $(s_{1/2}-s_{3/2},\ldots,s_{r-3/2}-s_{r-1/2})$. The sheaf $\cF_\bu^\sharp$ fits into the exact triangle
\begin{equation}\label{eq: bestiary flat sharp triangle}
\cF_\bu^\flat\to \cF_0\oplus\cF_r\to \cF_\bu^\sharp .
\end{equation}
The map $\cF_\bu^\flat \rightarrow \cF_0\oplus\cF_r$ has components induced by the inclusions $\cF_{1/2}^\flat \subset \cF_0$ and $\cF_{r-1/2}^\flat \subset \cF_r$, with the sign conventions of \cite[\S 9.2]{FYZ3}. However, when $r=0$, these are interpreted according to the following convention: $\Hk_S^0 = S$ with the identity correspondence, $\cF_\bu^\flat := \cF_0$ (for $r=0$ the displayed two-term complex would be the zero complex), and \eqref{eq: bestiary flat sharp triangle} is defined to be the split triangle $\cF_0 \xrightarrow{(1,1)} \cF_0\oplus\cF_0 \rightarrow \cF_0$, so that $\cF_\bu^\sharp = \cF_0$.

Replacing $\cF_{\univ}^*$ in the definitions of $\cU,\cV,\cW$ by $(\cF_\bu^\flat)^*$ gives perfect complexes $\cU^\flat_{\Hk},\cV^\flat_{\Hk},\cW^\flat_{\Hk}$ on $\Hk_S^r$, and hence derived vector bundles $\Hk_U^\flat,\Hk_V^\flat,\Hk_W^\flat$. Replacing $(\cF_\bu^\flat)^*$ by $(\cF_\bu^\sharp)^*$ gives $\Hk_U^\sharp,\Hk_V^\sharp,\Hk_W^\sharp$. Applying the same construction to the second exact sequence in \eqref{eq: bestiary torsion sequences} gives $\Hk_{U^\perp}^{\flat/\sharp}$, $\Hk_{\wh V}^{\flat/\sharp}$, and $\Hk_{W^\perp}^{\flat/\sharp}$. These fit into derived Cartesian squares
\begin{equation}\label{eq: Hk UVW zero fibers}
\begin{tikzcd}
\Hk_U^\flat \ar[r, "f"] \ar[d, "\pi"'] & \Hk_V^\flat \ar[d, "g"] \\
\Hk_S^r \ar[r, "z_{\Hk_W^\flat}"'] & \Hk_W^\flat
\end{tikzcd}
\qquad
\begin{tikzcd}
\Hk_{U^\perp}^\flat \ar[r, "\wh g"] \ar[d, "\wh z"'] & \Hk_{\wh V}^\flat \ar[d, "\wh f"] \\
\Hk_S^r \ar[r, "\wh \pi"'] & \Hk_{W^\perp}^\flat
\end{tikzcd}
\end{equation}
The analogous $\sharp$-squares are obtained from the same construction using $(\cF_\bu^\sharp)^*$. (In the second square of \eqref{eq: Hk UVW zero fibers}, the maps with $\wh{()}$ are the Fourier duals of the corresponding $\sharp$-maps; for instance $\wh g = \widehat{g^{\sharp}}$ under the identification $\wh{\Hk_W^{\sharp}} = \Hk_{U^\perp}^{\flat}$.)

For $i\in\{0,r\}$, write
$h_i^U,h_i^V,h_i^W,h_i^{U^\perp},h_i^{\wh V},h_i^{W^\perp}$
for the corresponding base-change projections to the bundles. We use the maps
\[
\begin{gathered}
a_i:=h_i^U\circ\wt a_i,\qquad
b_i:=h_i^V\circ\wt b_i,\qquad
c_i:=h_i^W\circ\wt c_i,\\
a_i^\perp:=h_i^{U^\perp}\circ\wt a_i^\perp,\qquad
\mbeta_i:=h_i^{\wh V}\circ\wt\mbeta_i,\qquad
c_i^\perp:=h_i^{W^\perp}\circ\wt c_i^\perp.
\end{gathered}
\]

The exact triangle \eqref{eq: bestiary flat sharp triangle} induces three derived Cartesian squares of derived vector bundles over $\Hk_S^r$:
\begin{equation}\label{eq: three squares 1}
\begin{tikzcd}
& \Hk_U^{\flat} \ar[dl, "\wt{a}_0"'] \ar[dr, "\wt{a}_r"] \\ 
\wt{U}_0 \ar[dr, "\wt{a}_r'"'] & & \wt{U}_r \ar[dl, "\wt{a}_0'"]  \\
& \Hk_U^{\sharp} 
\end{tikzcd} \hspace{1cm}
\begin{tikzcd}
& \Hk_V^{\flat} \ar[dl, "\wt{b}_0"'] \ar[dr, "\wt{b}_r"] \\ 
\wt{V}_0 \ar[dr, "\wt{b}_r'"'] & & \wt{V}_r \ar[dl, "\wt{b}_0'"]  \\
& \Hk_V^{\sharp} 
\end{tikzcd}\hspace{1cm}
\begin{tikzcd}
& \Hk_W^{\flat} \ar[dl, "\wt{c}_0"'] \ar[dr, "\wt{c}_r"] \\ 
\widetilde{W}_0 \ar[dr, "\widetilde{c}_r'"'] & & \widetilde{W}_r \ar[dl, "\widetilde{c}_0'"]  \\
& \Hk_W^{\sharp} 
\end{tikzcd}
\end{equation}
Switching the roles of $\cE_{1}$ (resp. $Q_1$) and $\cE_{2}$ (resp. $Q_2$), the exact triangle \eqref{eq: bestiary flat sharp triangle} induces three (derived) Cartesian squares of derived vector bundles over $\Hk_S^r$:
\begin{equation}\label{eq: three squares 2}
\begin{tikzcd}
& \Hk_{U^{\perp}}^{\flat} \ar[dl, "\wt{a}_0^{\perp}"'] \ar[dr, "\wt{a}_r^{\perp}"] \\ 
\wt{U}_0^{\perp} \ar[dr, "(\wt{a}_r')^{\perp}"'] & & \wt{U}_r^{\perp} \ar[dl, "(\wt{a}_0')^{\perp}"]  \\
& \Hk_{U^{\perp}}^{\sharp} 
\end{tikzcd} \hspace{1cm}
\begin{tikzcd}
& \Hk_{\wh V}^{\flat} \ar[dl, "\wt\mbeta_{0}"'] \ar[dr, "\wt\mbeta_{r}"] \\ 
\wh{\wt V}_0 \ar[dr, "\wt\mbeta'_{r}"'] & & \wh{\wt V}_r \ar[dl, "\wt\mbeta'_{0}"]  \\
& \Hk_{\wh V}^{\sharp} 
\end{tikzcd}\hspace{1cm}
\begin{tikzcd}
& \Hk_{W^{\perp}}^{\flat} \ar[dl, "\wt{c}_0^{\perp}"'] \ar[dr, "\wt{c}_r^{\perp}"] \\ 
\widetilde{W}_0^{\perp} \ar[dr, "(\widetilde{c}_r')^{\perp}"'] & & \widetilde{W}_r^{\perp} \ar[dl, "(\widetilde{c}_0')^{\perp}"]  \\
& \Hk_{W^{\perp}}^{\sharp} 
\end{tikzcd}
\end{equation}

The construction is summarized by the following pair of diagrams:
\begin{equation}\label{eq: big diagram for E_1}
\adjustbox{scale=0.8, center}{\begin{tikzcd}
& & \Hk_U^{\flat} \ar[dl, "\wt{a}_0"', color=blue] \ar[dr, "\wt{a}_r"] \ar[ddd, bend left, "f"', color=orange] \\
& \ar[dl, "h_0^U"'] \wt{U}_0 \ar[ddd, "\wt{f}_0"', color=red] \ar[dr, "\wt{a}_r'"'] && \wt{U}_r \ar[dl, "\wt{a}_0'"] \ar[ddd, "\wt{f}_r"] \ar[dr, "h_r^U"] \\
U_0 \ar[ddd, "f_0"]  & & \Hk_U^{\sharp}  \ar[ddd, bend right, "f^{\sharp}"]   &  &  U_r  \ar[ddd, "f_r"] \\
& & \Hk_V^{\flat} \ar[dl, "\wt{b}_0"', color=green] \ar[dr, "\wt{b}_r"] \ar[ddd, bend left, "g"']   \\
& \wt{V}_0 \ar[ddd, "\wt{g}_0"']  \ar[dr, "\wt{b}_r'"']  \ar[dl, "h_0^V"'] && \wt{V}_r \ar[dl, "\wt{b}_0'"]  \ar[ddd, "\wt{g}_r"]  \ar[dr, "h_r^V"] \\
V_0 \ar[ddd, "g_0"] & & \Hk_V^{\sharp} \ar[ddd, bend right, "g^{\sharp}"]  & & V_r  \ar[ddd, "g_r"] \\
& & \Hk_W^{\flat}  \ar[dl, "\wt{c}_0"'] \ar[dr, "\wt{c}_r"] \\
& \wt{W}_0 \ar[dr, "\wt{c}_r'"'] \ar[dl, "h_0^W"']  & &  \wt{W}_r \ar[dl, "\wt{c}_0'"]  \ar[dr, "h_r^W"]  \\
W_0 & & \Hk_W^{\sharp}  & & W_r 
\end{tikzcd}\hspace{1cm}
\begin{tikzcd}
& & \Hk_{U^{\perp}}^{\flat} \ar[dl, "\wt{a}_0^{\perp}"'] \ar[dr, "\wt{a}_r^{\perp}"] \ar[ddd, bend left, "f^{\perp}"'] \\
& \ar[dl, "h_0^{U^{\bot}}"'] \wt{U}_0^{\perp} \ar[ddd, "\wt{f}_0^{\perp}"'] \ar[dr, "(\wt{a}_r')^{\perp}"'] && \wt{U}_r^{\perp} \ar[dl, "(\wt{a}_0')^{\perp}"] \ar[ddd, "\wt f_r^{\perp}"] \ar[dr, "h_r^{U^{\bot}}"]  \\
U_0^{\perp} \ar[ddd, "f_0^{\perp}"] & & \Hk_{U^{\perp}}^{\sharp}  \ar[ddd, bend right, "(f^{\sharp})^{\perp}"]   & & U_r^{\perp} \ar[ddd, "f_r^{\perp}"] \\
& & \Hk_{\wh{V}}^{\flat} \ar[dl, "\wt{\mbeta}_0"'] \ar[dr, "\wt{\mbeta}_r"] \ar[ddd, bend left, "g^{\bot}"']   \\
& \ar[dl, "h_0^{\wh V}"']  \wh{\wt{V}}_0 \ar[ddd, "\wt{g}_0^{\perp}"', color=red]  \ar[dr, "\wt\mbeta'_{r}"', color=green]  && \wh{\wt{V}}_r \ar[dl, "\wt\mbeta'_{0}"]  \ar[ddd, "\wt{g}_r^{\perp}"]  \ar[dr, "h_r^{\wh V}"]  \\
\wh{V}_0 \ar[ddd, "g_0^{\perp}"]  & & \Hk_{\wh{V}}^{\sharp} \ar[ddd, bend right, "(g^{\sharp})^{\perp}", color=orange] & & \wh{V}_r \ar[ddd, "g_r^{\perp}"]  \\	
& & \Hk_{W^{\perp}}^{\flat}  \ar[dl, "\wt{c}_0^{\perp}"'] \ar[dr, "\wt{c}_r^{\perp}"] \\
& \ar[dl, "h_0^{W^{\bot}}"']  \wt{W}_0^{\perp} \ar[dr, "(\wt{c}_r')^{\perp}"', color=blue]  & &  \wt{W}_r^{\perp} \ar[dl, "(\wt{c}_0')^{\perp}"]  \ar[dr, "h_r^{W^{\bot}}"]  \\
W_0^{\perp} & & \Hk_{W^{\perp}}^{\sharp} & & W_r^{\perp} 
\end{tikzcd}
}
\end{equation}
In each diagram: 
\begin{itemize}
\item The maps in the columns are base changes (along the maps $h_i$, where appropriate) of the morphisms in the derived zero-fiber squares \eqref{eq: UVW zero fibers} and \eqref{eq: Hk UVW zero fibers}.
\item The three diamonds in the middle are derived Cartesian. 
\item The four parallelograms on the left and right sides are derived  Cartesian. 
\end{itemize}
The diagram on the right is dual to the diagram on the left. The duality exchanges $U$ with $W^{\perp}$, $V$ with $\wh{V}$, and $W$ with $U^{\perp}$, and exchanges $\flat$ and $\sharp$ superscripts. Examples of dual morphisms are colored with the same color. By \cite[\S 9.3]{FYZ3}, each of these diagrams is globally presented, so that the global-presentation formalism applies to them.

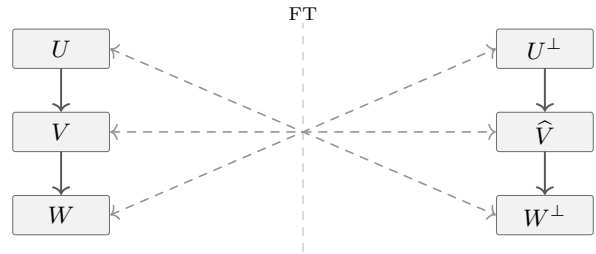
\begin{figure}[htbp]
\centering
\begin{tikzpicture}[
        x=1cm,y=1cm,
        font=\small,
        line cap=round,
        line join=round,
        bundle/.style={draw=black!55, fill=black!5, rounded corners=1pt,
            minimum width=1.28cm, minimum height=.52cm, inner sep=2pt},
        arr/.style={->, draw=black!65, line width=.75pt},
        dual/.style={<->, draw=black!42, dashed, line width=.6pt}]
    \node[bundle] (U) at (0,2.2) {$U$};
    \node[bundle] (V) at (0,1.1) {$V$};
    \node[bundle] (W) at (0,0) {$W$};

    \node[bundle] (Up) at (6.4,2.2) {$U^\perp$};
    \node[bundle] (Vh) at (6.4,1.1) {$\wh V$};
    \node[bundle] (Wp) at (6.4,0) {$W^\perp$};

    \draw[arr] (U) -- (V);
    \draw[arr] (V) -- (W);
    \draw[arr] (Up) -- (Vh);
    \draw[arr] (Vh) -- (Wp);

    \draw[black!28, dashed] (3.2,-.48) -- (3.2,2.68);
    \node[font=\scriptsize, fill=white, inner sep=1pt] at (3.2,2.67) {$\FT$};

    \draw[dual] (U.east) -- (Wp.west);
    \draw[dual] (V.east) -- (Vh.west);
    \draw[dual] (W.east) -- (Up.west);
\end{tikzpicture}
\caption{Cartoon of the dual pair of diagrams in \eqref{eq: big diagram for E_1}. Fourier duality reverses the towers by exchanging $U$ with $W^\perp$, $V$ with $\wh V$, and $W$ with $U^\perp$.}
\label{fig:uvw-duality-cartoon}
\end{figure}

\subsection{Departure from \cite{FYZ3}} \label{ssec:departure}
The bestiary setup in \S \ref{ssec:bestiary} is identical to the one from \cite[\S 9.1--9.2]{FYZ3}. At this point, our proof will diverge in a crucial way from \cite{FYZ3}. Let us indicate the shape of this departure, leaving some precise definitions for later. 

The Trace Conjecture implies that
\[
\Tr^{\Sht}(\cc_U) = [\cZ_{\cE_1}^r],
\qquad
\Tr^{\Sht}(\cc_{U^\perp}) = [\cZ_{\cE_2}^r].
\]
Therefore, where the quadratic function $\beta$ is defined in \S\ref{ssec: gaussians for E1}, the trace
\[
\Tr^{\Sht}\bigl(\pi_! (\cc_U \otimes \beta^* \AS_{\psi, S})\bigr)
\]
is essentially (i.e., up to signs and powers of $q$) the higher theta series $\wt Z_m^{n, r}(\cG, \cE_1)$, while the trace
\[
\Tr^{\Sht}\bigl(\pi^\perp_! (\cc_{U^\perp} \otimes (-\beta)^* \AS_{\psi, S})\bigr)
\]
is essentially $\wt Z_m^{n, r}(\cG, \cE_2)$. 

The strategy of \cite{FYZ3} is to relate $\cc_U \otimes \beta^* \AS_{\psi, S}$ with $\cc_{U^\perp} \otimes (-\beta)^* \AS_{\psi, S}$ by pushing them both forward to the correspondence 
\begin{equation}\label{eq:V-correspondence}
V_0 \leftarrow \Hk_V^\flat \rightarrow V_r
\end{equation}
and comparing them there. This comparison turns out to ultimately be a manifestation of the Plancherel formula. 

However, this strategy has a key defect: in order to effect an actual comparison in $\CH_*(\Sht_S^r)$, we would have to push forward from \eqref{eq:V-correspondence} to 
\begin{equation}\label{eq:S-correspondence}
S \leftarrow \Hk_S^r \rightarrow S .
\end{equation}
It turns out that the map from \eqref{eq:V-correspondence} to \eqref{eq:S-correspondence} is not pushable, so this is not possible. This is due to a certain ``wildness'' of the geometry of $\Hk_V^\flat$ when the legs meet the support of the torsion sheaf $Q_1$ defining $V$. On the complement of this locus, the map is in fact pushable, which is why this approach can work there. However, as the transverse pair $(\cE_1,\cE_2)$ varies, the support of the torsion sheaf $Q_1$ varies uncontrollably, so one is left only with automorphy over the generic fiber. 

The key new insight here is that the relation between $\cc_U \otimes \beta^* \AS_{\psi, S}$ and $\cc_{U^\perp} \otimes (-\beta)^* \AS_{\psi, S}$ is already baked into their genetics: they are both \emph{pulled back} from a pair of essentially self-dual cohomological correspondences on \eqref{eq:V-correspondence}. From this starting point, the desired comparison can be effected by Fourier duality, without trying to push forward to \eqref{eq:V-correspondence}.

\subsection{Geometric properties}\label{sssec: geometric properties for E1}  We establish some needed geometric properties of the maps in \eqref{eq: big diagram for E_1}.

We use the terminology of \emph{pushable} and \emph{pullable} squares \cite[\S 3.1 and \S 4.2]{FYZ3}. A commutative square induces a comparison map $\alpha$ to the fibered product
\begin{equation}\label{eq: pushable pullable square comparison}
\begin{tikzcd}[ampersand replacement=\&, column sep=large, row sep=large]
A \ar[rr] \ar[dd] \ar[dr, dashed] \& \& B \ar[dd] \\
\& C\times_D B \ar[ur] \ar[dl] \& \\
C \ar[rr] \& \& D
\end{tikzcd}
\end{equation}
and we say that it is pushable if the (dashed) comparison map is proper, and pullable if it is quasi-smooth. We say that a map of correspondences 
\begin{equation}\label{eq: pushable pullable correspondence map}
\begin{tikzcd}[ampersand replacement=\&, column sep=large, row sep=large]
A_0 \ar[d, "f_0"'] \& C \ar[l, "c_0"'] \ar[r, "c_1"] \ar[d, "f"] \& A_1 \ar[d, "f_1"] \\
B_0 \& D \ar[l, "d_0"] \ar[r, "d_1"'] \& B_1
\end{tikzcd}
\end{equation}
is \emph{pushable} if the left square is pushable, and \emph{pullable} if the right square is pullable. Following \cite[\S 3.1]{FYZ3}, we call the virtual relative dimension of the comparison map of a pullable square its \emph{defect}. Pullback of cohomological correspondences along a pullable map of correspondences carries a twist by the defect, which we track at each use.

\begin{lemma}\label{lem:pushable-and-pullable}
(1) The maps of correspondences
\[
\begin{tikzcd}
U_0 \ar[d, "f_0"] &
\Hk_U^\flat \ar[r, "a_r"] \ar[d, "f"] \ar[l, "a_0"'] &
U_r \ar[d, "f_r"] \\
V_0 &
\Hk_V^\flat \ar[l, "b_0"'] \ar[r, "b_r"] &
V_r
\end{tikzcd}
\qquad
\begin{tikzcd}
U_0^\perp \ar[d, "f_0^\perp"] &
\Hk_{U^\perp}^\flat \ar[r, "a_r^\perp"] \ar[d, "f^\perp"] \ar[l, "a_0^\perp"'] &
U_r^\perp \ar[d, "f_r^\perp"] \\
\wh V_0 &
\Hk_{\wh V}^\flat \ar[l, "\mbeta_0"'] \ar[r, "\mbeta_r"] &
\wh V_r
\end{tikzcd}
\]
are pullable.

(2) The maps of correspondences
\[
\begin{tikzcd}
U_0 \ar[d, "\pi_0"] &
\Hk_U^\flat \ar[r, "a_r"] \ar[d, "\pi"] \ar[l, "a_0"'] &
U_r \ar[d, "\pi_r"] \\
S_0 &
\Hk_S^r \ar[l, "h_0"'] \ar[r, "h_r"] &
S_r
\end{tikzcd}
\qquad
\begin{tikzcd}
U_0^\perp \ar[d, "\pi_0^\perp"] &
\Hk_{U^\perp}^\flat \ar[r, "a_r^\perp"] \ar[d, "\pi^\perp"] \ar[l, "a_0^\perp"'] &
U_r^\perp \ar[d, "\pi_r^\perp"] \\
S_0 &
\Hk_S^r \ar[l, "h_0"'] \ar[r, "h_r"] &
S_r
\end{tikzcd}
\]
are pushable.

(3) The maps of correspondences
\[
\begin{tikzcd}[ampersand replacement=\&]
V_0 \ar[d, "g_0"] \&
\Hk_V^\flat \ar[r, "b_r"] \ar[d, "g"] \ar[l, "b_0"'] \&
V_r \ar[d, "g_r"] \\
W_0 \&
\Hk_W^\flat \ar[l, "c_0"'] \ar[r, "c_r"] \&
W_r
\end{tikzcd}
\qquad
\begin{tikzcd}[ampersand replacement=\&]
\wh V_0 \ar[d, "g_0^\perp"] \&
\Hk_{\wh V}^\flat \ar[r, "\mbeta_r"] \ar[d, "g^{\bot}"] \ar[l, "\mbeta_0"'] \&
\wh V_r \ar[d, "g_r^\perp"] \\
W_0^\perp \&
\Hk_{W^\perp}^\flat \ar[l, "c_0^\perp"'] \ar[r, "c_r^\perp"] \&
W_r^\perp
\end{tikzcd}
\]
are pushable.

(4) The maps of correspondences
\[
\begin{tikzcd}[ampersand replacement=\&]
S_0 \ar[d, "z_{W_0}"] \&
\Hk_S^r \ar[r, "h_r"] \ar[d, "z_{\Hk_W^\flat}"] \ar[l, "h_0"'] \&
S_r \ar[d, "z_{W_r}"] \\
W_0 \&
\Hk_W^\flat \ar[l, "c_0"'] \ar[r, "c_r"] \&
W_r
\end{tikzcd}
\qquad
\begin{tikzcd}[ampersand replacement=\&]
S_0 \ar[d, "z_{W_0^\perp}"] \&
\Hk_S^r \ar[r, "h_r"] \ar[d, "z_{\Hk_{W^\perp}^\flat}"] \ar[l, "h_0"'] \&
S_r \ar[d, "z_{W_r^\perp}"] \\
W_0^\perp \&
\Hk_{W^\perp}^\flat \ar[l, "c_0^\perp"'] \ar[r, "c_r^\perp"] \&
W_r^\perp
\end{tikzcd}
\]
are pullable.

\end{lemma}

\begin{proof}

(4) We focus first on the left map of correspondences. By definition, we need to show pullability of the right square, in other words the map from $\Hk_S^r$ to the fibered product of $\Hk_W^\flat$ and $S_r$ over $W_r$ is quasi-smooth. The canonical base-change isomorphism
\[
\Hk_W^\flat\times_{W_r}S_r
\cong
\Hk_W^\flat\times_{\wt W_r}\Hk_S^r
\]
identifies its comparison map with the comparison map of the diagram
\begin{equation}\label{eq:SW-right-pullable-check}
\begin{tikzcd}
\Hk_S^r \ar[r, "\Id"] \ar[d, "z_{\Hk_W^\flat}"'] & \Hk_S^r \ar[d, "z_{\wt W_r}"] \\
\Hk_W^\flat \ar[r, "\wt c_r"'] & \wt W_r 
\end{tikzcd}
\end{equation}
It therefore suffices to show that \eqref{eq:SW-right-pullable-check} is right pullable. The morphism $\wt c_r:\Hk_W^\flat\to\wt W_r$ is smooth (a vector bundle projection) by \cite[Lemma 9.1.3(3)]{FYZ3}, while $\Id_{\Hk_S^r}$ is smooth. Hence \eqref{eq:SW-right-pullable-check} is pullable by \cite[Example 3.1.3]{FYZ3}. The same argument applies to the other map of correspondences, finishing the proof of (4).

We next deduce (1) from (4). By symmetry it suffices to handle the left map of correspondences. We similarly reduce to pullability of the diagram 
\begin{equation}\label{eq:UV-right-pullable-check}
\begin{tikzcd}
\Hk_U^\flat \ar[r, "\wt a_r"] \ar[d, "f"'] & \wt U_r \ar[d, "\wt f_r"] \\
\Hk_V^\flat \ar[r, "\wt b_r"'] & \wt V_r .
\end{tikzcd}
\end{equation}
We have derived Cartesian squares
\begin{equation}\label{eq:UV-SW-basechange-squares}
\begin{tikzcd}[ampersand replacement=\&, column sep=large, row sep=large]
\Hk_U^\flat \ar[r, "\pi"] \ar[d, "f"'] \& \Hk_S^r \ar[d, "z_{\Hk_W^\flat}"] \\
\Hk_V^\flat \ar[r, "g"'] \& \Hk_W^\flat
\end{tikzcd}
\qquad
\begin{tikzcd}[ampersand replacement=\&, column sep=large, row sep=large]
\wt U_r \ar[r] \ar[d, "\wt f_r"'] \& \Hk_S^r \ar[d, "z_{\wt W_r}"] \\
\wt V_r \ar[r, "\wt g_r"'] \& \wt W_r .
\end{tikzcd}
\end{equation}
The first square in \eqref{eq:UV-SW-basechange-squares} identifies $\Hk_U^\flat$ with $\Hk_V^\flat\times_{\Hk_W^\flat}\Hk_S^r$, and the second identifies $\wt U_r$ with $\wt V_r\times_{\wt W_r}\Hk_S^r$. Hence the comparison map of \eqref{eq:UV-right-pullable-check} is the derived base change of the comparison map of \eqref{eq:SW-right-pullable-check}, and quasi-smoothness is stable under derived base change.

(3) follows from (1) by duality: the first map of correspondences in (1) is dual to the second map in (3), and the second map in (1) is dual to the first map in (3). Indeed, dualization exchanges the $\flat$- and $\sharp$-vertices. It also exchanges pullability and pushability by \cite[Lemma 7.2.1]{FYZ3} in its varying-base form (see the end of \cite[\S 7.2]{FYZ3}).

(2) is deduced from (3) by base change, similarly to how (1) was deduced from (4). Indeed, for the left square in (2), the comparison map $\wt a_0\co\Hk_U^\flat\to\wt U_0$ is the top horizontal arrow in the Cartesian square
\begin{equation}\label{eq:US-VW-left-comparison-basechange}
\begin{tikzcd}[ampersand replacement=\&, column sep=large, row sep=large]
\Hk_U^\flat \ar[r, "\wt a_0"] \ar[d, "f"'] \& \wt U_0 \ar[d] \\
\Hk_V^\flat \ar[r] \& \Hk_W^\flat\times_{W_0}V_0 .
\end{tikzcd}
\end{equation}
The lower horizontal arrow in \eqref{eq:US-VW-left-comparison-basechange} is the left-square comparison map for (3), so $\wt a_0$ is its derived base change along $z_{\Hk_W^\flat}$. We conclude using that properness is stable under derived base change.
\end{proof}

In the proof of Lemma \ref{lem:pushable-and-pullable}, we showed that the map of correspondences
\[
\begin{tikzcd}
\wt{U}_0 \ar[d] & \Hk_U^\flat \ar[l, "\wt a_0"'] \ar[r, "\wt a_r"] \ar[d] & \wt{U}_r \ar[d] \\
 \wt{V}_0 &  \Hk_V^\flat \ar[l, "\wt b_0"'] \ar[r, "\wt b_r"] &  \wt{V}_r
 \end{tikzcd}
\]
is pullable. 

\begin{lemma}\label{lem: V-to-S pullable}
The map of correspondences 
\[
\begin{tikzcd}
 \wt{V}_0 \ar[d] &  \Hk_V^\flat \ar[l, "\wt b_0"'] \ar[r, "\wt b_r"] \ar[d]   &  \wt{V}_r\ar[d]  \\
\Hk_S^r  & \Hk_S^r \ar[l, "\Id"'] \ar[r, "\Id"]  & \Hk_S^r
 \end{tikzcd}
\]
is pullable. Consequently, the map of correspondences
\[
\begin{tikzcd}
\wt{U}_0 \ar[d] & \Hk_U^\flat \ar[l, "\wt a_0"'] \ar[r, "\wt a_r"] \ar[d] & \wt{U}_r \ar[d] \\
\Hk_S^r &  \Hk_S^r \ar[l, "\Id"'] \ar[r, "\Id"] & \Hk_S^r
 \end{tikzcd}
\] 
is pullable. The same statements hold after replacing $\wt{U}_i$ (resp. $\wt{V}_i$) with $U_i$ (resp. $V_i$) and the bottom left (resp. bottom right) $\Hk_S^r$ with $S$.
\end{lemma}

\begin{proof} 
The last sentence follows immediately from the definition of pullability, as in the proof of Lemma \ref{lem:pushable-and-pullable}. 

The second assertion follows from stability of pullable maps under composition (two-out-of-three for pullable squares, \cite[\S 3.1]{FYZ3}). We now focus on proving the first statement. This amounts to showing that the square 
\[
\begin{tikzcd}
\Hk_V^\flat \ar[r, "\wt b_r"] \ar[d] & \wt{V}_r \ar[d] \\
\Hk_S^r \ar[r, "\Id"'] &  \Hk_S^r
\end{tikzcd}
\]
is pullable, which is equivalent to $\wt b_r$ being quasi-smooth. This holds by \cite[Lemma 9.1.3(2)]{FYZ3}. 

\end{proof}

\begin{lemma}\label{lem: V-row endpoint maps vdim zero}
For each $i$, the maps 
\[
\wt b_i\co \Hk_V^\flat\longrightarrow \wt V_i
\]
of \eqref{eq: big diagram for E_1} are quasi-smooth of virtual relative dimension $0$.
\end{lemma}

\begin{proof}
The quasi-smoothness is \cite[Lemma 9.1.3(2)]{FYZ3}. For the relative dimension, recall that over an $R$-point $(\cF_\star)$ of $\Hk_S^r$, the map $\wt b_i$ is induced by the map $\cF_i^*\longrightarrow \cF_{\bu}^{\flat *}$. Write $\cF_{\univ,\bu}^{\flat}$ for the universal complex whose
pullback at $(\cF_\star)$ is $\cF_\bu^\flat$, and define the universal
perfect complex
\[
T_{i,\univ}:=\operatorname{cofib}
\bigl(h_i^*\cF_{\univ}^*\longrightarrow
\cF_{\univ,\bu}^{\flat *}\bigr)
\quad\text{in }\Perf(X'\times\Hk_S^r),
\]
and let $T_i$ be its pullback to $X'_R$.  At a classical $R$-point of
$\Hk_S^r$, $T_i$ is a torsion sheaf supported at the legs.
The relative tangent complex of $\wt b_i$ is the pullback from $\Hk_S^r$ of
$\ul{\RHom(T_{i,\univ},Q_1)}$. The virtual relative dimension may be checked at geometric points.
There Riemann--Roch gives
$\chi_{\mathrm{Eul}}(X',\RHom(T_i,Q_1))=0$, because both $T_i$ and $Q_1$ are torsion.
\end{proof}

\begin{lemma}\label{lem: one-leg V-row endpoint maps lci}
Assume $r=1$.  For $i=0,r$, the maps
\[
\wt b_i\co \Hk_V^\flat\longrightarrow \wt V_i
\]
are LCI morphisms of virtual relative dimension $0$.
\end{lemma}

\begin{proof}
When $r=1$, the complex $\cF_\bu^\flat$ in \eqref{eq: bestiary flat sharp triangle} is the vector bundle $\cF_{1/2}^\flat$ placed in degree $0$.  Thus $(\cF_\bu^\flat)^*=(\cF_{1/2}^\flat)^*$ is also a vector bundle.  Since $Q_1$ is torsion, the relative Hom complexes
\[
\ul{\RHom\bigl((\cF_\bu^\flat)^*,Q_1\bigr)}
\quad\text{and}\quad
\ul{\RHom(\cF_i^*,Q_1)}
\]
on $\Hk_S^1$ are represented by locally free coherent sheaves in degree $0$.  Hence $\Hk_V^\flat$ and $\wt V_i$ are ordinary vector bundles over $\Hk_S^1$, and $\wt b_i$ is a morphism of vector bundles over $\Hk_S^1$. A map of vector bundles over the same base is always LCI, since it can be factored as the composition of a regular embedding into its graph, followed by a smooth projection. Both bundles here are locally free of the same rank $n \cdot \length(Q_1)$, so the virtual relative dimension is $0$.
\end{proof}

\subsection{Gaussians}\label{ssec: gaussians for E1}

Abbreviate $\wt S:= \Hk_S^r$, $V^\flat:=\Hk_V^\flat$ and $V^\sh:=\Hk_V^\sharp$.

For $i=0,r$, let
$\wt\pi_i\co\wt V_i\to \wt S$ be the projection to the base of the derived vector bundle, and let $h_i^V\co\wt V_i\to V_i\cong V$ be the base change of $h_i\co \wt S \to S$.  

We have a derived Cartesian diagram of vector bundles over $\wt S$.
\begin{equation}\label{eq: Gaussian self-dual square}
\begin{tikzcd}[ampersand replacement=\&]
\& V^\flat \ar[dl, "\wt b_0"'] \ar[dr, "\wt b_r"] \\
\wt V_0 \ar[dr, "\wt b_r'"'] \& \& \wt V_r \ar[dl, "\wt b_0'"] \\
\& V^\sh .
\end{tikzcd}
\end{equation}

Following \cite[\S 2]{FYZ3}, the torsion sheaf $Q_1$ carries a canonical Hermitian structure. In the notation of \emph{loc.\ cit.}\ one has $Q_1=\sigma^*Q$, and the pullback of the Hermitian structure $h_{12}\co Q\xrightarrow{\sim}\sigma^*Q^*$ is
\[
\sigma^*h_{12}\co Q_1=\sigma^*Q\xrightarrow{\sim}Q^*\cong\sigma^*Q_1^*,
\qquad Q^*:=\cExt^1_{X'}(Q,\cO_{X'}).
\]
We continue to denote this transported form on $Q_1$ by $h_{12}$. (There are two natural choices $h_{12}$ and $h_{21}$, related by $h_{12} = -h_{21}$ \cite[\S 2]{FYZ3}; we fix the normalization $h_{12}$ throughout, following \cite{FYZ3}.) Tensoring $h_{12}$ with the Hermitian structure of $\cF_{\univ}$ and applying relative Serre duality along $X'\times S\rightarrow S$ produces a symmetric self-duality
\[
h_V\co V\xrightarrow{\ \sim\ }\wh V,
\]
whose associated quadratic function on $\F_q$-points is the quadratic form $\frq_{12}$ of \cite[\S 2]{FYZ3}.

We single out the features that we will use.
\begin{itemize}
\item The symmetric self-duality $h_V\co V\xrightarrow{\sim}\wh V$ pulls back to symmetric self-dualities
\[
h_{V,i}\co\wt V_i\xrightarrow{\sim}\wh{\wt V_i}
\]
over $\wt S$, for $i=0,r$.
\item There is an identification $V^\flat\xrightarrow{\sim}\wh{V^\sh}$. Thus, the diagram \eqref{eq: Gaussian self-dual square} is self-dual. 
\item With respect to these identifications, $\wt b_r'$ is the dual morphism to $\wt b_0$, and $\wt b_0'$ is the dual morphism to $\wt b_r$. (These three assertions are established in \cite[\S 9.1--9.2]{FYZ3}, transported along the self-duality $h_V$ defined above.)
\item The diagram \eqref{eq: Gaussian self-dual square}, the tensoring square \eqref{eq: tensoring-base-change-square} with $f=\wt b_r$, and their Fourier-dual diagrams are globally presented in the sense of \cite[Definition 6.3.1]{FYZ3}.
\item The morphism $\wt b_r$ is quasi-smooth of virtual relative dimension $0$ by Lemma \ref{lem: V-row endpoint maps vdim zero}, and $\wt b'_0$ is separated, so that the relative derived fundamental class $[\wt b_r]\co\wt b_r^*\to\wt b_r^!$ and the forget-supports map $\can(\wt b'_0)\co\wt b'_{0!}\to\wt b'_{0*}$ are defined.
\end{itemize}
Write $d(V/S)$ for the virtual relative dimension of $V\to S$, equivalently of either $\wt\pi_i\co\wt V_i\to \wt S$.

\subsubsection{Gaussian sheaves}
Let $\AS_{\psi_0}\in\Dmot{\A^1_{\F_q}}$ be the motivic Artin--Schreier sheaf of \eqref{eq:motivic-AS-sheaf}, and let $\AS_{\psi,S}$ and $\AS_{\psi,\wt S}$ denote its pullbacks to $\A^1_S$ and $\A^1_{\wt S}$.  Applying Construction \ref{constr:motivic-Gaussian-self-duality} to $h_V$ gives
\[
\beta:=\beta_{h_V}\co V\longrightarrow \A^1_S,
\qquad
\beta^\vee:=\beta_{h_V}^{\vee}\co \wh V\longrightarrow \A^1_S .
\]
Set
\[
\sG_\beta:=\beta^*\AS_{\psi,S}\in\Dmot{V},
\qquad
\sG_{-\beta}:=(-\beta)^*\AS_{\psi,S}\in\Dmot{V}.
\]
For $i=0,r$, let
\[
\beta_i:=(h_i^V)^*\beta\co\wt V_i\longrightarrow\A^1_{\wt S}.
\]
Equivalently, $\beta_i$ is the quadratic function induced by the pulled-back self-duality $h_{V,i}$.  Let $\beta_i^\vee$ denote the corresponding dual quadratic function supplied by Construction \ref{constr:motivic-Gaussian-self-duality}.

\begin{lemma} We have 
\begin{equation}\label{eq: Gaussian beta compatibility}
\beta_0\circ\wt b_0=\beta_r\circ\wt b_r
\quad\text{as morphisms }V^\flat\to\A^1_{\wt S}.
\end{equation}
\end{lemma}

\begin{proof} Indeed, writing $\iota_\flat \co V^\flat \cong \wh{V^\sh}$ for the identification above and $\iota_\sh \co V^\sh \cong \wh{V^\flat}$ for its dual, the duality assertions give 
\begin{align*}
        \wh{\wt b_0}\circ h_{V,0} = \iota_\sh \circ \wt b'_r \qquad \text{and} \qquad 
        \wh{\wt b_r}\circ h_{V,r} = \iota_\sh \circ \wt b'_0.
\end{align*}
Hence 
\begin{align*}
        \beta_0(\wt b_0 x) = \langle \iota_\sh\, \wt b'_r\, \wt b_0\, x,\ x\rangle, \qquad \text{and} \qquad 
        \beta_r(\wt b_r x) = \langle \iota_\sh\, \wt b'_0\, \wt b_r\, x,\ x \rangle,
\end{align*} 
which agree by the commutativity of the diagram \eqref{eq: Gaussian self-dual square}.
\end{proof} 

Define
\begin{equation}\label{eq: Gaussian indexed sheaves definition}
\sG_{\beta,i}:=\beta_i^*\AS_{\psi,\wt S}\in\Dmot{\wt V_i},
\qquad
\sG_{-\beta,i}:=(-\beta_i)^*\AS_{\psi,\wt S}\in\Dmot{\wt V_i}.
\end{equation}
The equality $h_{V,i}^*\beta_i^\vee=\beta_i$ identifies the positive-dual Gaussian with $\sG_{\beta,i}$, while $h_{V,i}^*(-\beta_i^\vee)=-\beta_i$, so the transported Fourier-dual Gaussian is $\sG_{-\beta,i}$.

\subsubsection{Cohomological correspondences}
Pulling $\AS_{\psi,\wt S}$ back along \eqref{eq: Gaussian beta compatibility} gives
\begin{equation}\label{eq: Gaussian pullback iso}
u_{\beta}\co
\wt b_0^*\sG_{\beta,0}\xrightarrow{\sim}
\wt b_r^*\sG_{\beta,r} .
\end{equation}
Applying the same construction to
$(-\beta_0)\circ\wt b_0=(-\beta_r)\circ\wt b_r$ gives
\begin{equation}\label{eq: Gaussian negative pullback iso}
u_{-\beta}\co
\wt b_0^*\sG_{-\beta,0}\xrightarrow{\sim}
\wt b_r^*\sG_{-\beta,r} .
\end{equation}
By Lemma \ref{lem: V-to-S pullable}, we may define a cohomological correspondence $\cc_{\wt V} \in \Corr_{V^\flat}(\Qsh{\wt V_0},\Qsh{\wt V_r})$ as the pullback of the identity cohomological correspondence on $\wt S$. Concretely, $\cc_{\wt V}$ is given explicitly by the composition 
\begin{equation}\label{eq: Gaussian ccV definition}
\cc_{\wt V} \co
\wt b_0^*\Qsh{\wt V_0}\cong \wt b_r^*\Qsh{\wt V_r}
\xrightarrow{[\wt b_r]}
\wt b_r^!\Qsh{\wt V_r} .
\end{equation}

Tensoring \eqref{eq: Gaussian ccV definition} with \eqref{eq: Gaussian pullback iso} gives
\begin{equation}\label{eq: Gaussian tensor definition}
\cc_{\wt V} \otimes u_{\beta}\in
\Corr_{V^\flat}\bigl(\sG_{\beta,0},\sG_{\beta,r}\bigr),
\end{equation}
and tensoring with \eqref{eq: Gaussian negative pullback iso} gives
\[
\cc_{\wt V} \otimes u_{-\beta}\in
\Corr_{V^\flat}\bigl(\sG_{-\beta,0},\sG_{-\beta,r}\bigr).
\]
Both tensor products are in the sense of Definition \ref{def:tensor-coh-corr}.

\subsubsection{Fourier transform of Gaussians}
Let $\pi\co V\to S$ be the structure morphism.  Set
\begin{equation}\label{eq: Gaussian endpoint Gauss factor}
\mathfrak{G}_Q:=\pi_!\sG_\beta\sm{[d(V/S)]}\in\Dmot{S},
\qquad
\mathfrak{G}_{Q,i}:=h_i^*\mathfrak{G}_Q\in\Dmot{\wt S}
\quad (0\leq i\leq r).
\end{equation}
By the definition of $V$ in \S\ref{ssec:bestiary}, it is the derived vector bundle associated to the perfect complex $\cV = \ul{\RHom(\cF_{\univ}^*, Q_1)}$. Since $\cF_{\univ}^*$ is 
locally free and $Q_1$ is a fixed torsion sheaf on $X'$, this is a locally free sheaf in degree $0$, so $V$ is a vector bundle in the classical sense. Hence Lemma \ref{lem:motivic-relative-Gauss-invertible} implies that $\mathfrak{G}_Q$, and with it each $\mathfrak{G}_{Q,i}$, is tensor-invertible.  

Base change for $\wt V_i=V\times_{S,h_i}\wt S$, followed by Lemma \ref{lem:motivic-Gaussian-self-duality}, gives a canonical isomorphism
\begin{equation}\label{eq: Gaussian shifted Fourier kernel}
\relax[2]_{\wt V_i}^*h_{V,i}^*\FT_{\wt V_i}(\sG_{\beta,i})
\xrightarrow{\sim}
\sG_{-\beta,i}\otimes\wt\pi_i^*\mathfrak{G}_{Q,i}.
\end{equation}
The scalar automorphism $[2]$ acts on $\wt V_0,\wt V_r,V^\flat$, and $V^\sh$.  We write $[2]_V^*$ for the induced scalar pullback on the corresponding sheaves, correspondences, and co-correspondences.

We define a normalized Fourier transform of cohomological correspondences
\begin{equation}\label{eq:FT-corr}
\begin{aligned}
\relax[2]_V^*\TT_{\sm{[-d(V/S)](-d(V/S))}}\FT_{V^\flat}\co\;&
\Corr_{V^\flat}(\sG_{\beta,0},\sG_{\beta,r}) \\
&\longrightarrow
\Corr_{V^\flat}\Bigl(
(\sG_{-\beta,0}\otimes\wt\pi_0^*\mathfrak{G}_{Q,0})
\sm{[-d(V/S)](-d(V/S))}, \\
&\hspace{7em}
(\sG_{-\beta,r}\otimes\wt\pi_r^*\mathfrak{G}_{Q,r})
\sm{[-d(V/S)](-d(V/S))}
\Bigr),
\end{aligned}
\end{equation}
the composite of the Fourier transform \eqref{eq:motivic-derived-FT-corr-same-base} for \eqref{eq: Gaussian self-dual square} (whose legs have virtual relative dimension $0$ by Lemma \ref{lem: V-row endpoint maps vdim zero}, so that no shift or twist appears), the identification $V^\flat\cong\wh{V^\sh}$, \eqref{eq: Gaussian shifted Fourier kernel}, and the displayed common shift and Tate twist.

\begin{lemma}\label{lem: one-leg Gaussian purity}
Assume $r=1$.  For each sign $\epsilon=\pm1$, with $u_{+\beta}:=u_\beta$, the morphism
\[
[\wt b_r]_{\epsilon\beta}\co
\wt b_r^*\sG_{\epsilon\beta,r}\longrightarrow
\wt b_r^!\sG_{\epsilon\beta,r}
\]
obtained from the relative derived fundamental class $[\wt b_r]$ is an isomorphism.  Consequently, after applying the isomorphisms
$u_{\epsilon\beta}\co\wt b_0^*\sG_{\epsilon\beta,0}\xrightarrow{\sim}\wt b_r^*\sG_{\epsilon\beta,r}$, there are canonical identifications
\begin{equation}\label{eq: Gaussian one-leg scalar identification}
\Corr_{V^\flat}(\sG_{\epsilon\beta,0},\sG_{\epsilon\beta,r})
= 
\Hom_{V^\flat}(\wt b_0^*\sG_{\epsilon\beta,0},\wt b_r^!\sG_{\epsilon\beta,r}) \cong 
\End_{V^\flat}(\wt b_0^*\sG_{\epsilon\beta,0}) \cong 
\rH^0(V^\flat,\Qsh{V^\flat}).
\end{equation}
\end{lemma}

\begin{proof}
By Lemma \ref{lem: V-row endpoint maps vdim zero}, the morphism $\wt b_r$ is quasi-smooth of virtual relative dimension $0$. Thus the relative derived fundamental class has the form $[\wt b_r]\co\wt b_r^*\to\wt b_r^!$ without any additional shift or Tate twist.

Let $p_{\epsilon,r}\co\widetilde V_{\epsilon,r}\to\wt V_r$ be the finite \'{e}tale Artin--Schreier cover obtained by pulling back \eqref{eq:motivic-AS-cover} along the composite of $\epsilon\beta_r\co\wt V_r\to\A^1_{\wt S}$ with the projection $\A^1_{\wt S}\rightarrow\A^1_{\F_q}$, with $\epsilon=\pm1$.  Then $\sG_{\epsilon\beta,r}
=e_{\psi_0}\,p_{\epsilon,r!}\Qsh{\widetilde V_{\epsilon,r}}$ where $e_{\psi_0}$ is the idempotent of $\ol{\Q}[\F_p]$ defined in \eqref{eq:motivic-AS-idempotent}.
Form the Cartesian square
\[
\begin{tikzcd}[ampersand replacement=\&]
\widetilde V_{\epsilon}^{\flat} \ar[r, "\widetilde b_{r,\epsilon}"] \ar[d, "q_{\epsilon}"'] \&
\widetilde V_{\epsilon,r} \ar[d, "p_{\epsilon,r}"] \\
V^\flat \ar[r, "\wt b_r"'] \& \wt V_r .
\end{tikzcd}
\]
Since $p_{\epsilon,r}$ is finite \'{e}tale, proper base change gives isomorphisms
\[
\wt b_r^*\sG_{\epsilon\beta,r}
\cong
e_{\psi_0}\,q_{\epsilon!}\Qsh{\widetilde V_{\epsilon}^{\flat}},
\qquad
\wt b_r^!\sG_{\epsilon\beta,r}
\cong
e_{\psi_0}\,q_{\epsilon!}\widetilde b_{r,\epsilon}^!\Qsh{\widetilde V_{\epsilon,r}} .
\]
By Lemma \ref{lem: one-leg V-row endpoint maps lci}, the map
$\widetilde b_{r,\epsilon}$ is a classical LCI morphism of virtual relative
dimension $0$. Absolute purity for rational Beilinson
motives \cite[Theorem 14.4.1]{CD19} implies that $\widetilde b_{r,\epsilon}^!\Qsh{\widetilde V_{\epsilon,r}} \cong \Qsh{\widetilde V_{\epsilon}^{\flat}}$, and then that $\wt b_r^! \sG_{\epsilon\beta,r} \cong \wt b_r^* \sG_{\epsilon\beta,r} \cong \wt b_0^* \sG_{\epsilon\beta,0}$. Since $\wt b_0^*\sG_{\epsilon\beta,0}$ is $\otimes$-invertible, the result follows.
\end{proof}

\begin{lemma}\label{lem: Gaussian generic V-row isomorphisms}
Let
\[
\begin{aligned}
X^\circ&:=X-\nu(|Q_1|\cup|Q_2|),\\
\Hk_S^{1,\circ}&:=\Hk_S^1\times_X X^\circ,\\
V^{\flat,\circ}&:=V^\flat\times_{\Hk_S^1}\Hk_S^{1,\circ},\\
V^{\sh,\circ}&:=V^\sh\times_{\Hk_S^1}\Hk_S^{1,\circ}.
\end{aligned}
\]
For $i=0,1$, let $\wt V_i:=V\times_{S,h_i}\Hk_S^1$, and let
$\wt b_i\co V^\flat\to\wt V_i$ be the map over $\Hk_S^1$. Then the restrictions $\wt b_i^\circ\co V^{\flat,\circ}\xrightarrow{\sim}\wt V_i^\circ$ are isomorphisms.
\end{lemma}

\begin{proof}
We prove the assertion for $\wt b_i^\circ$; the proof for $\wt b_i'{}^\circ$ is identical, applied to the map $(\cF_\bu^\sharp)^* \rightarrow \cF_i^*$ obtained by dualizing \eqref{eq: bestiary flat sharp triangle}, whose cokernel is again torsion supported on $\Gamma_{x'}\cup\Gamma_{\sigma x'}$.  Over an $R$-point $(\cF_\star)$ of $\Hk_S^{1,\circ}$, the map $\wt b_i$ is induced by the morphism
\[
\frp_i^*\co \cF_i^*\longrightarrow \cF_{\bu}^{\flat *}
\]
from the one-leg local model.  The cokernel of $\frp_i^*$ is a torsion sheaf supported set-theoretically on $\Gamma_{x'}\cup\Gamma_{\sigma x'}$, where $x'$ denotes the leg of $(\cF_\star)$: for $i=0$ the support lies on $\Gamma_{x'}$, and for $i=1$ on $\Gamma_{\sigma x'}$. Since the fiber product defining $\Hk_S^{1,\circ}$ is taken over $X$, the image of the leg in $X$ lies in $X^\circ$, so both $x'$ and $\sigma(x')$ are disjoint from $|Q_1|\cup|Q_2|$ (note that this set is $\sigma$-stable, since $Q_2\cong\sigma^*Q_1^*$). Therefore $\frp_i^*$ is an isomorphism in an open neighborhood of
$|Q_1|\times\Spec R\subset X'_R$, and hence induces an isomorphism
\[
\RHom_{X'_R}(\cF_{\bu}^{\flat *},Q_1\ot R)
\xrightarrow{\sim}
\RHom_{X'_R}(\cF_i^*,Q_1\ot R).
\]
This identification is functorial in $R$, so it identifies the corresponding total spaces over $\Hk_S^{1,\circ}$.
\end{proof}

\begin{defn}
On $V^{\flat,\circ}$, Lemma \ref{lem: Gaussian generic V-row isomorphisms} and \eqref{eq: Gaussian beta compatibility} give an isometry
\[
\wt b_1^\circ\circ(\wt b_0^\circ)^{-1}\co
(\wt V_0^\circ,\beta_0|_{\wt V_0^\circ})
\xrightarrow{\sim}
(\wt V_1^\circ,\beta_1|_{\wt V_1^\circ})
\]
We denote by
\begin{equation}\label{eq:xi-circ}
\xi^\circ \co
\left.\wt b_0^*\wt\pi_0^*\mathfrak{G}_{Q,0}\right|_{V^{\flat,\circ}}
\xrightarrow{\sim}
\left.\wt b_1^*\wt\pi_1^*\mathfrak{G}_{Q,1}\right|_{V^{\flat,\circ}}
\end{equation}
the induced transport isomorphism of relative Gauss factors in $\Dmot{V^{\flat,\circ}}$.
\end{defn}

\begin{prop}\label{prop: Fourier Gaussian cc}
\begin{enumerate}
\item Assume $r=1$. There is a unique isomorphism $\xi_1\co \wt b_0^*\wt\pi_0^*\mathfrak{G}_{Q,0}\xrightarrow{\sim}\wt b_1^*\wt\pi_1^*\mathfrak{G}_{Q,1}$ in $\Dmot{V^\flat}$ whose restriction to $V^{\flat,\circ}$ is $\xi^\circ$. Moreover, with respect to the normalized map \eqref{eq:FT-corr}, we have 
\begin{align}\label{eq: Gaussian cc Fourier identity}
& \relax[2]_V^*\!\left(
\TT_{\sm{[-d(V/S)](-d(V/S))}}
\FT_{V^\flat}(\cc_{\wt V}\otimes u_{\beta})
\right)
 =
(\cc_{\wt V}\otimes u_{-\beta})\otimes\xi_1\sm{[-d(V/S)](-d(V/S))} \\
& \in \Corr_{V^\flat}\bigl(
(\sG_{-\beta,0}\otimes\wt\pi_0^*\mathfrak{G}_{Q,0})\sm{[-d(V/S)](-d(V/S))},\,
(\sG_{-\beta,1}\otimes\wt\pi_1^*\mathfrak{G}_{Q,1})\sm{[-d(V/S)](-d(V/S))}
\bigr) \nonumber,
\end{align}
the tensor operation being that of Definition \ref{def:tensor-coh-corr}.
\item
For general $r$, the identity \eqref{eq: Gaussian cc Fourier identity} holds with $\xi_1$ replaced by the isomorphism $\xi_r\co \wt b_0^*\wt\pi_0^*\mathfrak{G}_{Q,0}\xrightarrow{\sim}\wt b_r^*\wt\pi_r^*\mathfrak{G}_{Q,r}$ defined as the composite of the pullbacks of $\xi_1$ along the factorization of $(\wt V_0\leftarrow V^\flat\rightarrow\wt V_r)$ into one-leg correspondences (and $\xi_0:=\Id$), as depicted in \eqref{eq: Gaussian cc r-leg factorization diagram}.
\begin{equation}\label{eq: Gaussian cc r-leg factorization diagram}
\begin{tikzcd}[ampersand replacement=\&, column sep=small, row sep=large]
\& V_{1/2}^{\flat}
\ar[dl, "\wt b_0^{(1)}"']
\ar[dr, "\wt b_1^{(1)}"]
\& \& V_{3/2}^{\flat}
\ar[dl, "\wt b_1^{(2)}"']
\ar[dr, "\wt b_2^{(2)}"]
\& \& \cdots
\& \& V_{r-1/2}^{\flat}
\ar[dl, "\wt b_{r-1}^{(r)}"']
\ar[dr, "\wt b_r^{(r)}"]
\& \\
\wt V_0 \& \& \wt V_1 \& \& \wt V_2
\& \& \wt V_{r-1} \& \& \wt V_r .
\end{tikzcd}
\end{equation}
\end{enumerate}
\end{prop}

\begin{proof}
If $r=0$, both sides are the identity correspondence by the zero-leg convention, and $\xi_0=\Id$, so the assertion is immediate. Hence assume $r>0$.

First we will reduce to the case $r=1$. For each $1\leq i\leq r$, let $\rho_i\co\wt S\to\Hk_S^1$ be the map recording the $i$-th one-step modification, and pull the one-leg identity back along $\rho_i$. This gives the correspondence
\[
\wt V_{i-1} \leftarrow V_{i-1/2}^\flat \rightarrow \wt V_i
\]
It is globally presented over the common base $\wt S$, being the base change of the globally presented one-leg diamond. By \eqref{eq: Gaussian cc r-leg factorization diagram}, these correspondences compose to $V^\flat$. Base change and functoriality of relative fundamental classes \cite[Theorems 3.12--3.13]{KhanI}, together with Definition \ref{def: composite cohomological correspondence}, identify $\cc_{\wt V}$ with the corresponding composite of the one-leg cohomological correspondences defined for $r=1$. Lemma \ref{lem: composition tensor compatibility} shows that the cohomological correspondence $\cc_{\wt V} \otimes u_\beta$ also factors as the composite of the $r=1$ case, while Lemma \ref{lem: composition Fourier compatibility} shows that Fourier transform is compatible with compositions of cohomological correspondences. Thus we reduce to the case $r=1$, which we now assume.

By Lemma \ref{lem:motivic-relative-Gauss-invertible}, each $\mathfrak{G}_{Q,i}$ is tensor-invertible; set
\[
\cL:=
\bigl(\wt b_0^*\wt\pi_0^*\mathfrak{G}_{Q,0}\bigr)^{-1}
\otimes
\wt b_1^*\wt\pi_1^*\mathfrak{G}_{Q,1}
\in\Dmot{V^\flat}.
\]
The object $\cL$ is tensor-invertible. Tensor the identifications of Lemma \ref{lem: one-leg Gaussian purity} for $\epsilon=-1$ with $\wt b_i^*\wt\pi_i^*\mathfrak{G}_{Q,i}$ and apply the common shift and Tate twist in \eqref{eq:FT-corr}. This identifies the target group of \eqref{eq:FT-corr} with $\rH^0(V^\flat,\cL)$, sending $s$ to $(\cc_{\wt V}\otimes u_{-\beta})\otimes s\sm{[-d(V/S)](-d(V/S))}$; here $\cc_{\wt V}\otimes u_{-\beta}$ corresponds to $1$ under \eqref{eq: Gaussian one-leg scalar identification}. There is therefore a unique $s\in\rH^0(V^\flat,\cL)$ such that
\[
\relax[2]_V^*\bigl(\TT_{\sm{[-d(V/S)](-d(V/S))}}\FT_{V^\flat}(\cc_{\wt V}\otimes u_\beta)\bigr)=(\cc_{\wt V}\otimes u_{-\beta})\otimes s\sm{[-d(V/S)](-d(V/S))}.
\]
We must show that $s$ is an isomorphism restricting on $V^{\flat,\circ}$ to transport along the isometries of Lemma \ref{lem: Gaussian generic V-row isomorphisms}. Let $\pi^\flat\co V^\flat\to\wt S=\Hk_S^1$ be the structure morphism; then $\pi^\flat=\wt\pi_i\circ\wt b_i$ for $i=0,1$.
Lemma
\ref{lem:motivic-relative-Gauss-finite-etale}, applied to $(V,h_V)$, gives
\[
\mathfrak G_Q\cong
\left.\mathfrak g_{d(V/S)}\right|_S\otimes\varepsilon_Q,
\]
where $\varepsilon_Q$ is tensor-invertible and becomes trivial after a finite \'{e}tale base change. Cancelling the constant Gauss factors in the definition of $\cL$ gives
\[
\cL\cong
(\pi^\flat)^*\bigl(h_0^*\varepsilon_Q^{-1}
\otimes h_1^*\varepsilon_Q\bigr).
\]
The RHS is likewise trivial after a finite \'{e}tale base change. Moreover, $V^{\flat,\circ}$ is dense in every connected component, since each component of $\Hk_S^1$ dominates $X$ through the smooth leg map and hence meets the dense open $\Hk_S^{1,\circ}$. It follows that restriction to $V^{\flat,\circ}$ is injective on $\rH^0(-,\cL)$.

On $V^{\flat,\circ}$, the maps $\wt b_i^\circ$ identify the correspondence with the graph of an isometry of quadratic bundles. Naturality of \eqref{eq: Gaussian shifted Fourier kernel} then identifies the restriction of the left side with $(\cc_{\wt V}\otimes u_{-\beta})|_{V^{\flat,\circ}}$ tensored with the tautological transport isomorphism $\xi^\circ$. Thus $s|_{V^{\flat,\circ}}=\xi^\circ$. This is an isomorphism, and the finite \'{e}tale local triviality used above implies that $s$ is an isomorphism. Taking $\xi_1:=s$ completes the proof.
\end{proof}

\subsubsection{The fundamental identity}
We now pass from the $\wt S$-relative correspondence
$\wt V_0\leftarrow V^\flat\rightarrow \wt V_r$ to the
varying-base correspondence
\[
V_0\longleftarrow V^\flat\longrightarrow V_r.
\]
Recall that
\[
b_i:=h_i^V\circ \wt b_i\co V^\flat\longrightarrow V_i
\qquad (i=0,r).
\]
The maps $h_i\co\wt S\to S$ are smooth, schematic, and separated by
\cite[Lemma 6.9]{FYZ}; hence their base changes
$h_i^V\co\wt V_i\to V_i$ are smooth of virtual relative dimension
$d(h_i)$.  By Lemma \ref{lem: V-row endpoint maps vdim zero}, the maps
$\wt b_i$ are quasi-smooth of virtual relative dimension $0$.  Thus
$b_i=h_i^V\circ\wt b_i$ is quasi-smooth of virtual relative dimension
$d(h_i)$.

For $\cK_0\in\Dmot{V_0}$ and $\cK_r\in\Dmot{V_r}$, we use the following cohomological correspondence group with varying base:
\begin{equation}\label{eq: Gaussian varying-base Corr identification}
\begin{aligned}
\Corr_{V^\flat}(\cK_0,\cK_r)
&:=\Hom_{V^\flat}(b_0^*\cK_0,b_r^!\cK_r) \\
&=\Hom_{V^\flat}\bigl(\wt b_0^*(h_0^V)^*\cK_0,
\wt b_r^!(h_r^V)^!\cK_r\bigr).
\end{aligned}
\end{equation}

The equality
$\beta\circ b_0=\beta\circ b_r$, viewed as an equality of morphisms
$V^\flat\to\A^1_{\F_q}$, induces an isomorphism 
\begin{equation}\label{eq: Gaussian varying-base pullback iso}
u_\beta\co
b_0^*\sG_\beta\xrightarrow{\sim} b_r^*\sG_\beta,
\end{equation}
and similarly $(-\beta)\circ b_0=(-\beta)\circ b_r$ induces
$u_{-\beta}$.

We define $\cc_V \in \Corr_{V^\flat}(\Qsh{V_0}, \Qsh{V_r}\tw{-d(h_r)})$ to be the pullback of the identity correspondence on $\pt$. Concretely, it is the composition 
\begin{equation}\label{eq: Gaussian varying-base cc definition}
\cc_V\co
b_0^*\Qsh{V_0}\cong b_r^*\Qsh{V_r}
\xrightarrow{[b_r]}
b_r^!\Qsh{V_r}\tw{-d(h_r)}
\end{equation}
where $[b_r]$ is the relative fundamental class attached to the quasi-smooth map $b_r$. We compare it with $\cc_{\wt V}$ through
\eqref{eq: Gaussian varying-base Corr identification}.  Since $h_r^V$ is
the base change of the smooth map $h_r$, relative purity for $h_r^V$
gives
\begin{equation}\label{eq: Gaussian hrV purity cancellation}
[h_r^V]\co
(h_r^V)^*\Qsh{V_r}
\xrightarrow{\sim}
(h_r^V)^!\Qsh{V_r}\tw{-d(h_r)}.
\end{equation}
Compositional compatibility of relative derived fundamental classes (cf.\ \cite[\S 3.2]{KhanI}) identifies $[b_r]$ with the composite of $\wt b_r^*([h_r^V])$ followed by $[\wt b_r]$; note that for the smooth map $h_r^V$ the fundamental class coincides with the purity isomorphism \eqref{eq: Gaussian hrV purity cancellation}.  Tensoring
\eqref{eq: Gaussian varying-base cc definition} with
\eqref{eq: Gaussian varying-base pullback iso} gives a cohomological correspondence 
\begin{equation}\label{eq: Gaussian varying-base tensor definition}
\cc_V\otimes u_\beta
\in
\Corr_{V^\flat}\bigl(\sG_\beta,\sG_\beta\tw{-d(h_r)}\bigr).
\end{equation}
Similarly, tensoring with $u_{-\beta}$ gives
\begin{equation}\label{eq: Gaussian varying-base negative tensor definition}
\cc_V\otimes u_{-\beta}
\in
\Corr_{V^\flat}\bigl(\sG_{-\beta},\sG_{-\beta}\tw{-d(h_r)}\bigr).
\end{equation}

Let $\FT_{V^\flat}$ denote, in this varying-base formulation, the
Fourier transform of \eqref{eq:motivic-derived-FT-corr-varying-base}.  After applying the Gaussian isomorphism of
Lemma \ref{lem:motivic-Gaussian-self-duality}, the normalized varying-base
Fourier transform is
\begin{equation}\label{eq: Gaussian varying-base FT map}
\begin{aligned}
\relax[2]_V^*\TT_{\sm{[-d(V/S)](-d(V/S))}}\FT_{V^\flat}\co\quad
&\Corr_{V^\flat}\bigl(\sG_\beta,\sG_\beta\tw{-d(h_r)}\bigr)\\
&\longrightarrow
\Corr_{V^\flat}\Bigl(
\begin{aligned}[t]
&(\sG_{-\beta}\otimes\pi^*\mathfrak{G}_Q)\sm{[-d(V/S)](-d(V/S))},\\
&(\sG_{-\beta}\otimes\pi^*\mathfrak{G}_Q)\sm{[-d(V/S)](-d(V/S))}\tw{-d(h_r)}
\end{aligned}\Bigr).
\end{aligned}
\end{equation}
Under \eqref{eq: Gaussian varying-base Corr identification}, the target
group in \eqref{eq: Gaussian varying-base FT map} becomes the target group
of \eqref{eq:FT-corr}; in the notation of
\eqref{eq: Gaussian indexed sheaves definition}, this is the pair of
identifications
\begin{align*}
(h_0^V)^*\bigl((\sG_{-\beta}\otimes\pi^*\mathfrak{G}_Q)\sm{[-d(V/S)](-d(V/S))}\bigr)
&=
\bigl(\sG_{-\beta,0}\otimes\wt\pi_0^*\mathfrak{G}_{Q,0}\bigr)\sm{[-d(V/S)](-d(V/S))},\\
(h_r^V)^!\bigl((\sG_{-\beta}\otimes\pi^*\mathfrak{G}_Q)\sm{[-d(V/S)](-d(V/S))}\tw{-d(h_r)}\bigr)
&\cong
\bigl(\sG_{-\beta,r}\otimes\wt\pi_r^*\mathfrak{G}_{Q,r}\bigr)\sm{[-d(V/S)](-d(V/S))}.
\end{align*}
The second identification is obtained by tensoring
\eqref{eq: Gaussian hrV purity cancellation} with
$(h_r^V)^*(\sG_{-\beta}\otimes\pi^*\mathfrak{G}_Q)\sm{[-d(V/S)](-d(V/S))}$; the twist $\tw{-d(h_r)}$ cancels the
purity shift and Tate twist in $(h_r^V)^!$.

\begin{prop}[Varying-base Gaussian correspondence]\label{prop: Fourier Gaussian cc varying base}
With respect to the normalized varying-base map
\eqref{eq: Gaussian varying-base FT map}, one has
\begin{align}\label{eq: Gaussian varying-base cc Fourier identity}
&\relax[2]_V^*\!\left(
\TT_{\sm{[-d(V/S)](-d(V/S))}}
\FT_{V^\flat}
(\cc_V\otimes u_\beta) \right)
= (\cc_V\otimes u_{-\beta})\otimes\xi_r\sm{[-d(V/S)](-d(V/S))}\\
&\in \Corr_{V^\flat}\Bigl(
(\sG_{-\beta}\otimes\pi^*\mathfrak{G}_Q)\sm{[-d(V/S)](-d(V/S))}, 
(\sG_{-\beta}\otimes\pi^*\mathfrak{G}_Q)\sm{[-d(V/S)](-d(V/S))}\tw{-d(h_r)}
\Bigr),
\end{align}
where $\xi_r$ is the isomorphism of Proposition \ref{prop: Fourier Gaussian cc}, viewed through \eqref{eq: Gaussian varying-base Corr identification} (note that $b_i^*\pi^*\mathfrak{G}_Q = \wt b_i^*\wt\pi_i^*\mathfrak{G}_{Q,i}$).
\end{prop}

\begin{proof}
By definition of the varying-base Fourier transform \eqref{eq:motivic-derived-FT-corr-varying-base}, it can be related to the fixed-base Fourier transform for different sheaves. In this case, the translation goes through 
\eqref{eq: Gaussian varying-base Corr identification}. After applying that identification, $\cc_V\otimes u_\beta$ becomes
$\cc_{\wt V}\otimes u_\beta$: the class $[b_r]$ becomes
$[\wt b_r]$ by the composition and purity discussion preceding
\eqref{eq: Gaussian varying-base tensor definition}, and
$(h_r^V)^!(\sG_\beta\tw{-d(h_r)})\cong\sG_{\beta,r}$.  The same
argument identifies $\cc_V\otimes u_{-\beta}$ with
$\cc_{\wt V}\otimes u_{-\beta}$, and the Gauss factors correspond via $\mathfrak{G}_{Q,i}=h_i^*\mathfrak{G}_Q$, matching the shifted target groups.

By the definition of the varying-base transform \eqref{eq:motivic-derived-FT-corr-varying-base}, the varying-base Fourier transform of cohomological correspondences \eqref{eq: Gaussian varying-base FT map} corresponds, under the identification \eqref{eq: Gaussian varying-base Corr identification}, to the fixed Fourier transform \eqref{eq:FT-corr}. Therefore, after translated between fixed-base and varying-based groups of cohomological correspondences using \eqref{eq: Gaussian varying-base Corr identification}, the image of \eqref{eq: Gaussian varying-base cc Fourier identity} becomes precisely the identity of Proposition \ref{prop: Fourier Gaussian cc}, with the same $\xi_r$. 
\end{proof}

\section{Proof of the Modularity Conjecture}\label{sec:proof-modularity-conj}
In this section we complete the proof of the Modularity Conjecture of \cite{FYZ2}. We first recall its formulation. A point of $\Bun_{GU^{-}(2m)}$ is a triple $(\cG,\mathfrak M,h)$, where $\cG$ is a rank $2m$ vector bundle on $X'$, $\mathfrak M$ is a line bundle on $X$, and
\[
h\co \cG\xrightarrow{\sim}\sigma^*\cG^\vee\ot\nu^*\mathfrak M
=\sigma^*\cG^*\ot\nu^*(\omega_X\ot\mathfrak M)
\]
is skew-Hermitian.  Put $\frL:=\omega_X\ot\mathfrak M$.  Let $\Bun_{\wt P_m}$ be the stack of quadruples $(\cG,\mathfrak M,h,\cE)$, with $(\cG,\mathfrak M,h)\in\Bun_{GU^{-}(2m)}$ and $\cE\subset\cG$ a Lagrangian sub-bundle. Forgetting $\cE$ defines a surjective map
\begin{equation}\label{eq: unitary parabolic forgetful}
\Bun_{\wt P_m}(\F_q)\surj
\Bun_{GU^{-}(2m)}(\F_q).
\end{equation}

With the normalization of \cite[\S 4.6]{FYZ2}, the higher theta series is the function
\begin{equation}\label{eq: unitary theta function notation}
\wt Z_m^{n,r}\co
\Bun_{\wt P_m}(\F_q)\longrightarrow
\CH_{r(n-m)}(\Sht_{GU(n)}^r),
\end{equation}
whose value at $(\cG,\mathfrak M,h,\cE)$ is
\begin{equation}\label{eq: unitary theta value recall}
\wt Z_m^{n,r}(\cG,\cE)
=
\chi(\det\cE)\,
q^{n(\deg\cE-\deg\frL-\deg\omega_X)/2}
\sum_{a\in\cA_{\cE,\frL}(\F_q)}
\psi(\langle e_{\cG,\cE},a\rangle)\,
[\cZ_{\cE,\frL}^{r}(a)]
\end{equation}
with the notation explained as follows:
\begin{itemize}
\item $\chi: \Pic_{X'}(k)\to \ol{\Q}^{\times}$ is a character satisfying $\chi\circ\nu^*=\y^{n}$ on $\Pic_X(k)$, where $\y: \Pic_{X}(k)\to \{\pm1\}$ is the character corresponding to the double cover $X'/X$. We also write $\eta_{F'/F}=\y$.
\item $\psi:\F_{q}\to \ol{\Q}^{\times}$ is a nontrivial character.
\item $\cA_{\cE,\frL}(\F_q)$ is the set of Hermitian maps $a: \cE\to \s^{*}\cE^{\vee}\ot\nu^{*}\frL$. 
\item $e_{\cG,\cE}\in\Ext^1(\sigma^*\cE^*\ot\nu^*\frL,\cE)$ is the extension class for the short exact sequence
\begin{equation}\label{eq: unitary theta extension class}
0\longrightarrow\cE\longrightarrow\cG
\longrightarrow\sigma^*\cE^*\ot\nu^*\frL
\longrightarrow0,
\end{equation}
and $\langle e_{\cG,\cE},a\rangle$ is the Serre duality pairing.  
\end{itemize}

\begin{conj}[The Modularity Conjecture]\label{conj: modularity} The function $\wt Z_m^{n,r}$ descends along \eqref{eq: unitary parabolic forgetful}; in other words, $\wt Z_m^{n,r}(\cG,\cE)$ is independent of the choice of Lagrangian $\cE \subset \cG$.
\end{conj}

To keep the notation manageable, we write out the proof below only in the case of the trivial-similitude fiber $\frL=\cO_X$. The same ideas apply straightforwardly to the general case. On this fiber $\mathfrak M=\omega_X^{-1}$, the skew-Hermitian form is an isomorphism $\cG\cong\sigma^*\cG^*$.  We write $\Bun_{GU^{-}(2m)}^{\mathrm{triv}}$ and $\Bun_{P_m}$ for the corresponding fibers of $\Bun_{GU^{-}(2m)}$ and $\Bun_{\wt P_m}$, and we write $\Sht^r_{U(n)}=\Sht^r_{U(n),\cO_X}$. We likewise abbreviate $\cA_{\cE} := \cA_{\cE,\cO_X}$ and $\cZ^r_{\cE}(a) := \cZ^r_{\cE,\cO_X}(a)$.

\subsection{Modularity in low corank}\label{ssec: unitary modularity low corank}

We first prove the Modularity Conjecture in the range $m \leq n/3$, where we have established the Trace Conjecture in Theorem \ref{thm:trace-conjecture-n/3}.

\begin{thm}[Modularity in low corank]\label{thm:unitary-low-corank-modularity-section}
Assume $m\le n/3$. Then for every $\cG\in\Bun_{GU^{-}(2m)}^{\mathrm{triv}}(\F_q)$ and every pair of Lagrangian sub-bundles $\cE_1,\cE_2\subset\cG$,
\begin{equation}\label{eq: unitary low corank modularity}
\wt Z_m^{n,r}(\cG,\cE_1)
=
\wt Z_m^{n,r}(\cG,\cE_2)
\in \CH_{r(n-m)}(\Sht_{U(n)}^r).
\end{equation}
\end{thm}

\noindent\emph{Proof of Theorem
\ref{thm:unitary-low-corank-modularity-section}.}

\subsubsection{Reductions}\label{sssec: unitary low corank reductions} Let $S=\Bun_{U(n)}^{\le\mu}$ be a Harder--Narasimhan truncation and let $\Sht_S^r\subset\Sht_{U(n)}^r$ be the induced open substack. We have
\begin{equation}\label{eq: unitary HN Chow limit}
\CH_{r(n-m)}(\Sht_{U(n)}^r)
=
\varprojlim_{\mu}
\CH_{r(n-m)}(\Sht_S^r).
\end{equation}
Thus it is enough to prove \eqref{eq: unitary low corank modularity} after restriction to every $\CH_{r(n-m)}(\Sht_S^r)$.  The transverse reduction of \cite[\S 2.2]{FYZ3} says that, in a fixed fiber of $p_m$, the equivalence relation generated by transverse pairs of Lagrangians is the full relation.  Thus we may assume that $\cE_1,\cE_2\subset\cG$ are transverse.  We will prove that
\begin{equation}\label{eq: unitary transverse target}
\wt Z_m^{n,r}(\cG,\cE_1)\big|_S
=
\wt Z_m^{n,r}(\cG,\cE_2)\big|_S .
\end{equation}

Now the formulations and results of \S \ref{sec:genetics-of-modularity} apply, and we maintain the notation from there.

\subsubsection{Preparations}
To summarize the setup, we have short exact sequences
\begin{equation}\label{eq: unitary transverse Q}
0\longrightarrow\sigma^*\cE_2
\longrightarrow\cE_1^*
\longrightarrow Q_1
\longrightarrow0,
\qquad
0\longrightarrow\sigma^*\cE_1
\longrightarrow\cE_2^*
\longrightarrow Q_2
\longrightarrow0,
\end{equation}
where $Q_2\cong\sigma^*Q_1^*$ and $Q_1^*=\cExt^1_{X'}(Q_1,\cO_{X'})$.  Let $D_Q$ be the effective divisor on $X$, counted with lengths, whose pullback to $X'$ is the scheme-theoretic support divisor of $Q_1$, equivalently of $Q_2$.  Thus $\det(Q_1)\cong\nu^*\cO_X(D_Q)$.  In what follows $U,V,W,U^\perp,\wh V,W^\perp$, their Hecke correspondences, and the maps $f,\pi,f^\perp,\pi^\perp$ are as in \eqref{eq: bestiary base triangles}, \eqref{eq: UVW zero fibers}, \eqref{eq: Hk UVW zero fibers}, and \eqref{eq: big diagram for E_1}.  Abbreviate $d:=d(V/S)$.

The skew-Hermitian form on $\cG$ induces two opposite Hermitian forms on the common torsion quotient, as in \cite[\S 2.3]{FYZ3}:
\begin{equation}\label{eq: unitary Q hermitian forms}
h_{12},h_{21}\co Q_1\xrightarrow{\sim}\sigma^*Q_1^*,
\qquad h_{21}=-h_{12}.
\end{equation}
(In particular $\operatorname{div}(Q_1)$ is $\sigma$-invariant, which justifies the existence of the divisor $D_Q$ on $X$ introduced above.) Let $\beta=\beta_{12}\co V\to\A^1_S$ be the quadratic function attached to $h_{12}$, and let $-\beta=\beta_{21}$ be the quadratic function attached to $h_{21}$.  We write $\sG_\beta=\beta^*\AS_{\psi,S}$ and $\sG_{-\beta}=(-\beta)^*\AS_{\psi,S}$.

\begin{lemma}\label{lem: unitary relative Gauss}
The relative Gauss cohomology of \eqref{eq: Gaussian endpoint Gauss factor}, equivalently
\begin{equation}\label{eq: unitary relative Gauss trace}
\mathfrak{G}_Q:=
\pi_{V!}\sG_\beta\sm{[d]}
\in \Dmot{S}
\end{equation}
has constant Frobenius trace function on $S(\F_q)$, which for every $s\in S(\F_q)$ is equal to
\begin{equation}\label{eq: unitary relative Gauss scalar}
\Tr(\Frob,\mathfrak{G}_Q)(s)
=
(-1)^d q^{d/2}\eta_{F'/F}(D_Q)^n .
\end{equation}
More generally, for every finite extension $\F_{q^k}$ and every $s\in S(\F_{q^k})$, one has
\[
\Tr(\Frob_q^k,\mathfrak{G}_Q)(s)
=(-1)^d q^{kd/2}\eta_{F'_k/F_k}(D_{Q,k})^n
=\bigl((-1)^d q^{d/2}\eta_{F'/F}(D_Q)^n\bigr)^k
\]
where the subscript $k$ indicates base change to $\F_{q^k}$. 
\end{lemma}

\begin{proof}
The statement is pointwise, so fix $s\in S(\F_{q^k})$; to lighten notation take $k=1$. The formation of $\pi_{V!}\sG_\beta$ commutes with the base change along $s\rightarrow S$, and since $Q_1$ is torsion, the fiber $V_s = \Hom(\cF_s^*,Q_1)$ is a finite-dimensional $\F_q$-quadratic space with quadratic form $\frq_{12}$. By the Grothendieck--Lefschetz trace formula,
\begin{equation}\label{eq: unitary relative Gauss fiber sum}
\Tr(\Frob,\pi_{V!}\sG_\beta)(s)=\sum_{v\in V_s(\F_q)}\psi(\beta_s(v)),
\end{equation}
the Gauss sum of $(V_s,\frq_{12})$. The computation of \cite[Lemma 2.3.8]{FYZ3} gives
\[
\sum_{v\in V_s(\F_q)}\psi(\beta_s(v))
=q^{d/2}\eta_{F'/F}(D_Q)^n,
\]
independently of the Hermitian bundle $\cF_s$. For general $k$ with the additive character $\psi_k:=\psi\circ\Tr_{\F_{q^k}/\F_q}$, the same placewise calculation over $X_k$ gives
\begin{equation}\label{eq:gauss-sum}
\sum_{v\in V_s(\F_{q^k})}\psi_k(\beta_s(v))
=q^{kd/2}\eta_{F'_k/F_k}(D_{Q,k})^n.
\end{equation}
Moreover,
$\eta_{F'_k/F_k}(D_{Q,k})=\eta_{F'/F}(D_Q)^k$. Finally, the shift $\sm{[d]}$ in
\eqref{eq: unitary relative Gauss trace} contributes the sign $(-1)^d$.
Here $d=n\deg_{X'}(\nu^*D_Q)=2n\deg D_Q$ is even, so this sign is $1$. Together with \eqref{eq:gauss-sum}, this proves the result.
\end{proof}

Let $\pi^\flat\co V^\flat\to\Hk_S^1$ be the projection. This is a vector bundle in the classical sense by the proof of Lemma \ref{lem: one-leg V-row endpoint maps lci}. Therefore $(\pi^\flat)^*$ is fully faithful (by homotopy invariance), and via the identifications
\[
(\pi^\flat)^*h_i^*\mathfrak{G}_Q=b_i^*\pi^*\mathfrak{G}_Q
=\wt b_i^*\wt\pi_i^*\mathfrak{G}_{Q,i},
\]
the isomorphism $\xi_1$ of Proposition \ref{prop: Fourier Gaussian cc} descends uniquely to an isomorphism
\[
\bar\xi_1\co h_0^*\mathfrak{G}_Q\xrightarrow{\sim}h_1^*\mathfrak{G}_Q.
\]
For $0\leq i\leq r$, let
$h_i\co\Hk_S^r\to S_i$ record the $i$-th intermediate bundle, where
each $S_i$ is a copy of $S$. For $1\leq i\leq r$, let
$\rho_i\co\Hk_S^r\to\Hk_S^1$ record the $i$-th modification. Then
$\rho_i^*\bar\xi_1$ is an isomorphism from
$h_{i-1}^*\mathfrak{G}_Q$ to $h_i^*\mathfrak{G}_Q$. For general $r$,
let
\begin{equation}\label{eq:bar-xi}
\bar\xi_r:=
\rho_r^*\bar\xi_1\circ\cdots\circ\rho_1^*\bar\xi_1
\co h_0^*\mathfrak{G}_Q\xrightarrow{\sim}h_r^*\mathfrak{G}_Q
\end{equation}
and set $\bar\xi_0=\Id$. The one-leg factorization is displayed
schematically as
\begin{equation}\label{eq: bar xi r one-leg factorization diagram}
\begin{tikzcd}[ampersand replacement=\&, column sep=small, row sep=large]
\& \Hk_S^1 \ar[dl, "h_0"'] \ar[dr, "h_1"]
\& \& \Hk_S^1 \ar[dl, "h_0"'] \ar[dr, "h_1"]
\& \& \cdots
\& \& \Hk_S^1 \ar[dl, "h_0"'] \ar[dr, "h_1"]
\& \\
S_0 \& \& S_1 \& \& S_2
\& \& S_{r-1} \& \& S_r ,
\end{tikzcd}
\end{equation}
where the $i$-th copy of $\Hk_S^1$ is the correspondence from
$S_{i-1}$ to $S_i$, pulled back to $\Hk_S^r$ along $\rho_i$.
Thus the pullback of $\bar\xi_r$ along $\Hk_V^\flat\to\Hk_S^r$ is the isomorphism $\xi_r$ of Proposition \ref{prop: Fourier Gaussian cc}.

\begin{cor}\label{cor: Gauss twist trace constant}
The trace of $(\mathfrak{G}_Q,\bar\xi_r)$ twisted by Frobenius is the constant class
\[
\Tr^{\Sht}(\mathfrak{G}_Q,\bar\xi_r)
=
(-1)^d q^{d/2}\eta_{F'/F}(D_Q)^n\cdot 1_{\Sht_S^r}
\]
in $\rH^0(\Sht_S^r,\Qsh{})$, where $1_{\Sht_S^r}$ denotes the unit
section.
\end{cor}

\begin{proof}
By Lemma \ref{lem:motivic-relative-Gauss-invertible}, applied as in the paragraph following \eqref{eq: Gaussian endpoint Gauss factor}, the object $\mathfrak{G}_Q$ is tensor-invertible. Hence the trace is given by an element of $\rH^0(\Sht_S^r,\Qsh{})$, i.e., a locally constant function. It is enough to compute this locally constant function on a nonempty open substack of each connected component.

Let $\Sht_S^{r,\circ}\subset\Sht_S^r$ be the locus where all legs avoid $\nu^{-1}(D_Q)$. If $r=0$, this is all of $\Sht_S^0$. If $r>0$, let $C$ be a connected component of $\Sht_S^r$. By \cite[\S 6]{FYZ}, the leg map $\Sht^r_{U(n)}\to (X')^r$ is smooth, hence open; therefore the image of $C$ is a nonempty open subset of a connected component of $(X')^r$. Since $(X'\setminus\nu^{-1}(D_Q))^r$ is dense in every connected component of $(X')^r$, the component $C$ meets $\Sht_S^{r,\circ}$.

Over $\Sht_S^{r,\circ}$, all bundle modifications in the shtuka restrict to isomorphisms on a neighborhood of $\nu^{-1}(D_Q)$. Thus the restriction of $\bar\xi_r$ is the $r$-fold composition of the tautological isomorphism \eqref{eq:xi-circ}. Fix a geometric point $y$ of $\Sht_S^{r,\circ}$. Let
$\tau_y$ be the isometry from
$\cF_0|_{\nu^{-1}(D_Q)}$ to
$\cF_r|_{\nu^{-1}(D_Q)}$ obtained by composing the Hecke modifications, and
let
$\varphi_y\co\cF_r\xrightarrow{\sim}\Frob_q^*\cF_0$ be the identification in the definition of shtuka. The composite
$(\varphi_y|_{\nu^{-1}(D_Q)})\circ\tau_y$ is the Frobenius-semilinear
isometry that supplies descent on the finite quadratic space
$V_y=\Hom(\cF_0^*,Q_1)$. By the definition of $\xi^\circ$ in
\eqref{eq:xi-circ} and of $\bar\xi_r$, the Frobenius-twisted endomorphism
of the Gauss cohomology at $y$ is precisely the transport induced by this
semilinear isometry. The twisted Grothendieck--Lefschetz formula identifies
its trace with the Gauss sum on the fixed subspace, which is the
descended finite Hermitian quadratic space over $\F_q$. The computation of
\cite[Lemma 2.3.8]{FYZ3}, applied to that descended space, gives
$q^{d/2}\eta_{F'/F}(D_Q)^n$, and the shift $\sm{[d]}$ contributes
$(-1)^d$.
\end{proof}

\subsubsection{Sheaf-cycle correspondence}
Recall the cohomological correspondence $\cc_V$ from
\eqref{eq: Gaussian varying-base cc definition}. On the Fourier-dual row, write
\[
\cc_{\wh V}\in
\Corr_{\Hk_{\wh V}^{\flat}}\bigl(\Qsh{\wh V_0},\Qsh{\wh V_r}\tw{-d(h_r)}\bigr)
\]
for the analogous pullback of the identity correspondence, defined using the quasi-smooth map $\mbeta_r:\Hk_{\wh V}^{\flat}\to\wh V_r$. Under the self-duality identifications used in Proposition \ref{prop: Fourier Gaussian cc varying base}, $\cc_V$ identifies with $\cc_{\wh V}$. The pullback cohomological
correspondences $\cc_U=f^*\cc_V$ and
$\cc_{U^\perp}=(f^\perp)^*\cc_{\wh V}$ exist by Lemma
\ref{lem:pushable-and-pullable}.  We also use the Gaussian pullback
isomorphisms $u_\beta$ and $u_{-\beta}$ defined in and immediately after
\eqref{eq: Gaussian varying-base pullback iso}: $u_\beta$ is
induced by the equality $\beta\circ b_0=\beta\circ b_r$, and $u_{-\beta}$
is induced by the equality $(-\beta)\circ b_0=(-\beta)\circ b_r$.

\begin{lemma}\label{lem: unitary trace Fourier comparison}
In $\CH_{r(n-m)}(\Sht_S^r)$, one has
\begin{equation}\label{eq: unitary trace comparison before theta}
\begin{aligned}
&\Sht(\pi)_!\Tr^{\Sht}\!\left(\cc_U\otimes f^*u_\beta\right)\\
&\quad =
q^{-d+d(U/S)+d/2}
\eta_{F'/F}(D_Q)^n\,
\Sht(\pi^\perp)_!\Tr^{\Sht}\!\left(\cc_{U^\perp}\otimes(f^\perp)^*u_{-\beta}\right).
\end{aligned}
\end{equation}
\end{lemma}

\begin{proof}
Under the self-duality identifications used in Proposition
\ref{prop: Fourier Gaussian cc varying base}, the right hand side of
\eqref{eq: Gaussian varying-base cc Fourier identity} is the cohomology correspondence
$\cc_{\wh V}\otimes u_{-\beta}$ on the row with
$\Hk_{\wh V}^{\flat}$ in \eqref{eq: big diagram for E_1}.
By
definition of $\cc_U$ and $\cc_{U^\perp}$, and by compatibility of
pullback with tensoring (the pullable analogue of Lemma \ref{lem: composition tensor compatibility}, proved by the same concatenation-of-base-change argument), we have
\begin{equation}\label{eq: unitary trace Fourier pullback of Gaussian cc}
f^*(\cc_V\otimes u_\beta)
=
\cc_U\otimes f^*u_\beta,
\qquad
(f^\perp)^*(\cc_{\wh V}\otimes u_{-\beta})
=
\cc_{U^\perp}\otimes (f^\perp)^*u_{-\beta}.
\end{equation}

We next compare the pushed correspondences on the base correspondence
$S\xleftarrow{h_0}\Hk_S^r\xrightarrow{h_r}S$. Apply the
Base Change Theorem for cohomological correspondences, \cite[Theorem
4.4.2]{FK}, to the square of maps of correspondences induced by the first
square in \eqref{eq: Hk UVW zero fibers}. The corresponding arrows for its
top, right, left, and bottom maps are, respectively,
$f$, $g$, $\pi$, and $z_{\Hk_W^\flat}$.
Parts (1)--(4) of Lemma \ref{lem:pushable-and-pullable} say respectively
that $f$ is pullable, $\pi$ is pushable, $g$ is pushable, and
$z_{\Hk_W^\flat}$ is pullable. The corresponding commutative squares are
the base changes of the derived Cartesian square
\eqref{eq: UVW zero fibers}, hence are themselves Cartesian. Thus the hypotheses of the Base Change Theorem
hold, and it gives the canonical identification
\begin{equation}\label{eq: unitary trace Fourier base-change}
\pi_!f^*(\cc_V\otimes u_\beta)
\cong
z_{\Hk_W^\flat}^{\,*}\,g_!(\cc_V\otimes u_\beta),
\end{equation}
where $z_{\Hk_W^\flat}\co\Hk_S^r\to\Hk_W^\flat$ denotes the zero
section in \eqref{eq: Hk UVW zero fibers}.

We now apply the motivic derived Fourier transform to
\eqref{eq: unitary trace Fourier base-change}.  On the $S$-row the Fourier
transform is the identity, since $S$ is the zero derived vector bundle over
itself. The required hypotheses are precisely the global-presentation
hypotheses asserted for both diagrams after
\eqref{eq: big diagram for E_1} in \S\ref{ssec:bestiary}. First, the left-pushable
part of Proposition \ref{prop:motivic-derived-FT-cc-functoriality}, applied to
the map with vertical arrow $g$, gives
\begin{equation}\label{eq: unitary trace Fourier g functoriality}
\TT_{\sm{[d(g_0)]}}\FT_{\Hk_W^\flat}\!\left(g_!(\cc_V\otimes u_\beta)\right)
\cong
(f^\perp)^*
\FT_{\Hk_V^\flat}(\cc_V\otimes u_\beta).
\end{equation}
Here the Fourier-dual map has vertical arrow
$f^\perp=\widehat{g^\sharp}$, denoted $\wh g$ in the second diagram of
\eqref{eq: Hk UVW zero fibers}. Second, the right-pullable part of
Proposition \ref{prop:motivic-derived-FT-cc-functoriality}, applied to the
zero-section map
$z_{\Hk_W^\flat}\co\Hk_S^r\to\Hk_W^\flat$, gives
\begin{equation}\label{eq: unitary trace Fourier zero section functoriality}
\FT_{\Hk_S^r}\!\left(z_{\Hk_W^\flat}^{\,*}g_!(\cc_V\otimes u_\beta)\right)
\cong
\pi^\perp_!
\TT_{\sm{[-d(z_{W,0})](-d(z_{W,0}))}}
\FT_{\Hk_W^\flat}\!\left(g_!(\cc_V\otimes u_\beta)\right).
\end{equation}
The Fourier-dual of the zero section $z_W\co S\to W$ is the
projection $U^\perp=\wh W\to S$, hence the pushforward in
\eqref{eq: unitary trace Fourier zero section functoriality} is
$\pi^\perp_!$.  Here $z_{W,0}\co S\to W_0$ denotes the zero section over
the first copy of $S$. Combining
\eqref{eq: unitary trace Fourier base-change},
\eqref{eq: unitary trace Fourier g functoriality}, and
\eqref{eq: unitary trace Fourier zero section functoriality}, and then using
Proposition \ref{prop: Fourier Gaussian cc varying base}, gives the
following identity of cohomological correspondences on $\Hk_S^r$:
\begin{equation}\label{eq: unitary normalized pushed correspondence}
\pi_!f^*(\cc_V\otimes u_\beta)
\cong
\left(\pi^\perp_!(f^\perp)^*(\cc_{\wh V}\otimes u_{-\beta})\right)
\otimes(\mathfrak{G}_Q,\bar\xi_r)
\sm{[d-2d(U/S)](d-d(U/S))}.
\end{equation}
Here $\bar\xi_r$ is as in \eqref{eq:bar-xi}, and $\otimes(\mathfrak{G}_Q,\bar\xi_r)$ is the tensor operation of Definition \ref{def:tensor-coh-corr} with this datum.
The shift and Tate twist in
\eqref{eq: unitary normalized pushed correspondence} are obtained by combining
the following contributions:
\[
\begin{array}{l|l}
\textup{operation} & \textup{shift and twist} \\
\hline
\textup{normalized FT \eqref{eq: Gaussian varying-base FT map}}
& \sm{[d](d)} \\
\textup{Fourier commutation with $g_!$ \eqref{eq: unitary trace Fourier g functoriality}}
& \TT_{\sm{[-d(g_0)]}} \\
\textup{zero-section functoriality \eqref{eq: unitary trace Fourier zero section functoriality}}
& \TT_{\sm{[-d(z_{W,0})](-d(z_{W,0}))}} \\
\textup{common normalization on the Gauss factor} &
\sm{[-d](-d)} \\
\hline
\textup{total after rank additivity}
& \sm{[d-2d(U/S)](d-d(U/S))}.
\end{array}
\]
The extra row records that the normalized Gaussian Fourier identity carries the factor $\mathfrak{G}_Q\sm{[-d](-d)}$, while \eqref{eq: unitary normalized pushed correspondence} is written as a twist by $\mathfrak{G}_Q$ together with the displayed common shift and Tate twist.
Finally, the scalar
multiplication $[2]_V^*$ appearing in
\eqref{eq: Gaussian varying-base FT map} commutes with all linear maps in
\eqref{eq: UVW zero fibers} and \eqref{eq: Hk UVW zero fibers}, and becomes the identity on $S$.

Taking Frobenius traces in
\eqref{eq: unitary normalized pushed correspondence}, and using Corollary \ref{cor: Gauss twist trace constant} for the Gauss twist, gives
\begin{equation}\label{eq: unitary normalized weighted trace comparison}
\begin{aligned}
\Tr^{\Sht}\!\left(\pi_! f^*(\cc_V\otimes u_\beta)\right)
&=
(-1)^d q^{-d+d(U/S)}
\Tr(\Frob,\mathfrak{G}_Q)\,
\Tr^{\Sht}\!\left(\pi^\perp_!(f^\perp)^*(\cc_{\wh V}\otimes u_{-\beta})\right).
\end{aligned}
\end{equation}

It remains to commute the pushforwards past the formation of the trace, using Proposition \ref{prop: trace of Hitchin pushforward}. We check that the maps of correspondences in question satisfy its hypotheses.  Indeed, $\cF_\bu^\flat$ has tor-amplitude $[0,1]$, so its dual has amplitude $[-1,0]$; applying relative $\RHom$ into the fixed vector bundles and then relative cohomology on the curve produces coconnective perfect complexes.  The same calculation applies to the Fourier-dual row.  The comparison maps $\wt a_i$ and $\wt a_i^\perp$ are closed embeddings of derived vector bundles by \cite[Lemma 9.1.3(1)]{FYZ3}; equivalently, the fibers of the defining maps of perfect complexes have tor-amplitude $[1,\infty)$ by \cite[Lemma 6.1.5]{FYZ3}.  Finally, $\AS_{\psi}$ is a direct summand of the pushforward of the unit along the finite \'{e}tale Artin--Schreier cover \eqref{eq:motivic-AS-cover}; hence it is geometric, and so are its Gaussian pullbacks and their tensor products with the unit coefficient objects.  Thus Proposition \ref{prop: trace of Hitchin pushforward} applies and gives
\begin{equation}\label{eq: unitary weighted pushforward compatibility}
\begin{aligned}
\Tr^{\Sht}\!\left(\pi_!f^*(\cc_V\otimes u_\beta)\right)
&=
\Sht(\pi)_!\Tr^{\Sht}\!\left(\cc_U\otimes f^*u_\beta\right),\\
\Tr^{\Sht}\!\left(\pi^\perp_!(f^\perp)^*(\cc_{\wh V}\otimes u_{-\beta})\right)
&=
\Sht(\pi^\perp)_!\Tr^{\Sht}\!\left(\cc_{U^\perp}\otimes(f^\perp)^*u_{-\beta}\right).
\end{aligned}
\end{equation}
Thus \eqref{eq: unitary normalized weighted trace comparison} and
\eqref{eq: unitary weighted pushforward compatibility} imply
\begin{equation}\label{eq: unitary trace Fourier comparison with Gauss scalar}
\begin{aligned}
&\Sht(\pi)_!\Tr^{\Sht}\!\left(\cc_U\otimes f^*u_\beta\right)\\
&\quad =
(-1)^d q^{-d+d(U/S)}
\Tr(\Frob,\mathfrak{G}_Q)\,
\Sht(\pi^\perp)_!\Tr^{\Sht}\!\left(\cc_{U^\perp}
\otimes(f^\perp)^*u_{-\beta}\right).
\end{aligned}
\end{equation}
Finally, substituting the relative Gauss-sum calculation
\eqref{eq: unitary relative Gauss scalar}, namely
\[
\Tr(\Frob,\mathfrak{G}_Q)
=
(-1)^dq^{d/2}\eta_{F'/F}(D_Q)^n,
\]
into \eqref{eq: unitary trace Fourier comparison with Gauss scalar}
gives \eqref{eq: unitary trace comparison before theta}, with scalar
\[
(-1)^d q^{-d+d(U/S)}\cdot(-1)^d q^{d/2}\eta_{F'/F}(D_Q)^n
=
q^{-d+d(U/S)+d/2}\eta_{F'/F}(D_Q)^n.
\]
\end{proof}

We now express higher theta functions as a trace. 

\begin{lemma}\label{lem: unitary trace theta calculation}
The two pushed trace classes satisfy
\begin{equation}\label{eq: unitary theta normalizations}
\begin{aligned}
\wt Z_m^{n,r}(\cG,\cE_1)\big|_S
&=
\chi(\det\cE_1)
q^{n(\deg\cE_1-\deg\omega_X)/2}
\Sht(\pi)_!\Tr^{\Sht}\!\left(\cc_U\otimes f^*u_\beta\right),\\
\wt Z_m^{n,r}(\cG,\cE_2)\big|_S
&=
\chi(\det\cE_2)
q^{n(\deg\cE_2-\deg\omega_X)/2}
\Sht(\pi^\perp)_!\Tr^{\Sht}\!\left(
\cc_{U^\perp}\otimes(f^\perp)^*u_{-\beta}
\right).
\end{aligned}
\end{equation}
\end{lemma}

\begin{proof}
Since Artin--Schreier sheaves are tensor-invertible, Example \ref{ex: trace chow} and Lemma \ref{lem: twist trace by Chern class} yield
\[
\begin{aligned}
\Tr^{\Sht}\!\left(\cc_U\otimes f^*u_\beta\right)
&=
\Tr^{\Sht}(\cc_U)\cap \Tr^{\Sht}(f^*u_\beta),\\
\Tr^{\Sht}\!\left(\cc_{U^\perp}\otimes(f^\perp)^*u_{-\beta}\right)
&=
\Tr^{\Sht}(\cc_{U^\perp})\cap
\Tr^{\Sht}((f^\perp)^*u_{-\beta}).
\end{aligned}
\]

We identify the two pushed trace classes with the two unnormalized
theta sums.  Let
$\cM_{\cE_i,S}\to\cA_{\cE_i}$ be the Hitchin stack over $S$ for the
Lagrangian $\cE_i$. The moduli descriptions of the two $U$-rows give
canonical identifications
\[
\Sht(\Hk_U^\flat)\cong\cZ_{\cE_1}^r|_S,
\qquad
\Sht(\Hk_{U^\perp}^\flat)\cong\cZ_{\cE_2}^r|_S.
\]
Under these identifications, $\cc_U$ and $\cc_{U^\perp}$ are,
respectively, the restrictions over $S$ of the canonical Hitchin
cohomological correspondences $\cc_{\cM_{\cE_1}}$ and
$\cc_{\cM_{\cE_2}}$. The fixed-point Hitchin base is (using
\cite[Lemma 6.14]{FYZ2})
\begin{equation}\label{eq: unitary Hitchin coefficient decomposition}
\cA_{\cE_i}\times_{(\Id,\Frob),\,\cA_{\cE_i}\times\cA_{\cE_i},\,\Delta}
\cA_{\cE_i}
\cong
\coprod_{a\in\cA_{\cE_i}(\F_q)}\Spec\F_q .
\end{equation}
The low-corank Trace Conjecture, Theorem \ref{thm: trace conjecture E}, applies because $m\le n/3$.  Restricted to the open-and-closed summand indexed by $a$ in \eqref{eq: unitary Hitchin coefficient decomposition}, it identifies the unweighted trace class with the derived fundamental class of $\cZ_{\cE_i}^r(a)|_S$. (Here we use that $\Tr^{\Sht}$ commutes with restriction to open substacks, by \cite[Proposition 6.2.2]{FK}.)  The extension-class calculation of \cite[\S 2.3 and Lemma 10.2.7]{FYZ3} identifies the trace of $f^*u_\beta$ on the $a$-summand for $\cE_1$ with $\psi(\langle e_{\cG,\cE_1},a\rangle)$, and identifies the trace of $(f^\perp)^*u_{-\beta}$ on the $a$-summand for $\cE_2$ with $\psi(\langle e_{\cG,\cE_2},a\rangle)$.  Thus
\begin{equation}\label{eq: unitary theta trace identification 1}
\Sht(\pi)_!\Tr^{\Sht}\!\left(\cc_U\otimes f^*u_\beta\right)
=
\sum_{a\in\cA_{\cE_1}(\F_q)}
\psi(\langle e_{\cG,\cE_1},a\rangle)
[\cZ_{\cE_1}^{r}(a)]\big|_S,
\end{equation}
and
\begin{equation}\label{eq: unitary theta trace identification 2}
\Sht(\pi^\perp)_!\Tr^{\Sht}\!\left(
\cc_{U^\perp}\otimes(f^\perp)^*u_{-\beta}
\right)
=
\sum_{a\in\cA_{\cE_2}(\F_q)}
\psi(\langle e_{\cG,\cE_2},a\rangle)
[\cZ_{\cE_2}^{r}(a)]\big|_S.
\end{equation}
Comparing \eqref{eq: unitary theta trace identification 1} and
\eqref{eq: unitary theta trace identification 2} with (the trivial similitude
specialization of) \eqref{eq: unitary theta value recall} gives
\eqref{eq: unitary theta normalizations}.
\end{proof}

Combining Lemma \ref{lem: unitary trace theta calculation} with
Lemma \ref{lem: unitary trace Fourier comparison} gives
\begin{equation}\label{eq: unitary theta comparison before accounting}
\begin{aligned}
\wt Z_m^{n,r}(\cG,\cE_1)\big|_S
&=
q^{-d+d(U/S)+d/2+\frac{n(\deg\cE_1-\deg\omega_X)}{2}
-\frac{n(\deg\cE_2-\deg\omega_X)}{2}}\\
&\quad\cdot
\eta_{F'/F}(D_Q)^n
\chi(\det\cE_1)\chi(\det\cE_2)^{-1}\,
\wt Z_m^{n,r}(\cG,\cE_2)\big|_S.
\end{aligned}
\end{equation}

\subsubsection{Completion of the proof of Theorem \ref{thm:unitary-low-corank-modularity-section}}
It remains to show that the scalar factor in \eqref{eq: unitary theta comparison before accounting}, namely
\[
q^{-d+d(U/S)+d/2+\frac{n(\deg\cE_1-\deg\omega_X)}{2}
-\frac{n(\deg\cE_2-\deg\omega_X)}{2}}
\eta_{F'/F}(D_Q)^n
\chi(\det\cE_1)\chi(\det\cE_2)^{-1}
\]
is equal to $1$.

Applying
$\ul{\RHom(\cF_{\univ}^*,-)}$ to the first exact sequence in
\eqref{eq: unitary transverse Q} gives an exact triangle
\[
\ul{\RHom(\cF_{\univ}^*,\sigma^*\cE_2)}
\longrightarrow
\ul{\RHom(\cF_{\univ}^*,\cE_1^*)}
\longrightarrow
\ul{\RHom(\cF_{\univ}^*,Q_1)}.
\]
The total spaces of the middle and right terms are $U$ and $V$,
respectively, and $d=d(V/S)$.  Hence additivity of virtual ranks gives,
on each connected component of $S$,
\begin{equation}\label{eq: unitary q exponent direct rank}
-d+d(U/S)+\frac d2
=
\chi_{\mathrm{Eul}}(X',\cF_{\univ}\otimes\sigma^*\cE_2)
+\frac12\chi_{\mathrm{Eul}}(X',\cF_{\univ}\otimes Q_1).
\end{equation}
For a geometric point $\cF$ of
$S$, the Hermitian condition
$\cF\cong\sigma^*\cF^\vee$, with
$\cF^\vee=\cHom(\cF,\omega_{X'})$, implies
\[
\deg\cF=\frac n2\deg\omega_{X'}=n\deg\omega_X.
\]
Riemann--Roch on $X'$ therefore gives
\[
\chi_{\mathrm{Eul}}(X',\cF\otimes\sigma^*\cE_2)=n\deg\cE_2.
\]
Since $Q_1$ is torsion,
\[
\chi_{\mathrm{Eul}}(X',\cF\otimes Q_1)=n\deg Q_1.
\]
The first exact sequence in \eqref{eq: unitary transverse Q} gives
$\deg Q_1=-\deg\cE_1-\deg\cE_2$.  Substituting these two
Euler characteristic computations into
\eqref{eq: unitary q exponent direct rank} gives
\begin{equation}\label{eq: unitary q exponent direct cancellation}
-d+d(U/S)+\frac d2
=
\frac{n(\deg\cE_2-\deg\cE_1)}2.
\end{equation}
Consequently the total exponent of $q$ in
\eqref{eq: unitary theta comparison before accounting} is
\[
\frac{n(\deg\cE_2-\deg\cE_1)}2
+\frac{n(\deg\cE_1-\deg\omega_X)}2
-\frac{n(\deg\cE_2-\deg\omega_X)}2
=0.
\]

It remains to identify the character factor.  Taking determinants in
the first exact sequence of \eqref{eq: unitary transverse Q} gives
\[
\det(\cE_1^*)
\cong
\det(\sigma^*\cE_2)\otimes\det(Q_1).
\]
By the definition of $D_Q$, the determinant of the torsion sheaf
$Q_1$ is $\nu^*\cO_X(D_Q)$.  Equivalently,
\begin{equation}\label{eq: unitary determinant direct accounting}
\det\cE_1\otimes\nu^*\cO_X(D_Q)
\cong
\sigma^*(\det\cE_2^*).
\end{equation}
The character $\chi$ restricts to
$\eta_{F'/F}^{\,n}$ on $\nu^*\Pic_X(\F_q)$.  Moreover
$\chi(\sigma^*L)=\chi(L)^{-1}$ for every
$L\in\Pic_{X'}(\F_q)$, because
$\sigma^*L\otimes L=\nu^*\Nm(L)$ and
$\eta_{F'/F}$ is trivial on norms.  Applying $\chi$ to
\eqref{eq: unitary determinant direct accounting} gives
\begin{equation}\label{eq: unitary character direct cancellation}
\chi(\det\cE_1)\eta_{F'/F}(D_Q)^n
=
\chi(\det\cE_2).
\end{equation}

Equation \eqref{eq: unitary q exponent direct cancellation} makes the $q$-power in \eqref{eq: unitary theta comparison before accounting} equal to $1$, and \eqref{eq: unitary character direct cancellation} does the same for the character and Gauss factor. Hence the scalar is $1$, proving Theorem \ref{thm:unitary-low-corank-modularity-section}. \qed

\subsection{The embedding trick}\label{ssec: unitary all corank reduction}
We now pass from Theorem \ref{thm:unitary-low-corank-modularity-section} to the full Modularity Conjecture using a version of the ``embedding trick'' for theta series. A similar idea has appeared before, for instance, in the work of Yuan--Zhang--Zhang \cite{YZZ}, Kudla \cite{Kud21}, and Howard--Madapusi \cite{HM22} on modularity of theta series on orthogonal Shimura varieties. This will involve considering Hermitian shtukas of varying ranks, so we introduce some temporary notation that makes the rank explicit. We write
\[
\wt Z_m^{N,r}(\cG,\cE)
\in
\CH_{r(N-m)}(\Sht^r_{U(N)})
\]
for the higher theta series of Hermitian rank $N$ and corank $m$.
For each $N$, the theta series is normalized using an admissible
character
\[
\chi_N\co \Pic_{X'}(k)\longrightarrow \ol{\Q}^{\times}
\]
whose restriction along $\nu^*\co\Pic_X(k)\to\Pic_{X'}(k)$ is
$\eta_{F'/F}^{\,N}$. Note that this condition
implies $\chi_N(\sigma^*L)=\chi_N(L)^{-1}$, since
$\sigma^*L\otimes L=\nu^*\Nm(L)$ and
$\eta_{F'/F}^{\,N}(\Nm(L))=1$.

\begin{lemma}\label{lem: unitary direct sum factorization}
The direct-sum morphism
\begin{equation}\label{eq: unitary add map}
\add_{N_1,N_2}\co
\Sht^0_{U(N_1)}\times\Sht^r_{U(N_2)}
\longrightarrow
\Sht^r_{U(N_1+N_2)}
\end{equation}
is LCI, so that the pullback $\add^*_{N_1, N_2}$ exists on Chow groups, via the Gysin pullback \eqref{eq: Gysin pullback}. Let $N_1,N_2\ge m$, and suppose
\begin{equation}\label{eq: unitary character product}
\chi_{N_1+N_2}=\chi_{N_1}\chi_{N_2}.
\end{equation}
Then for every $(\cG,\cE)\in\Bun_{P_m}(k)$ one has
\begin{equation}\label{eq: unitary direct sum factorization}
\add_{N_1,N_2}^*
\left(\wt Z_m^{N_1+N_2,r}(\cG,\cE)\right)
=
\wt Z_m^{N_1,0}(\cG,\cE)\bt
\wt Z_m^{N_2,r}(\cG,\cE).
\end{equation}
\end{lemma}

\begin{proof}
The source and target of $\add_{N_1,N_2}$ are smooth: $\Sht^r_{U(N)}$ is smooth by \cite[\S 6]{FYZ}, while $\Sht^0_{U(N_1)}$ is the discrete groupoid $\Bun_{U(N_1)}(\F_q)$. Hence $\add_{N_1,N_2}$ is LCI.  By \cite[Lemma 6.9(2)]{FYZ}, $\Sht^r_{U(N)}$ has pure dimension $rN$, including the $r$ leg parameters, whereas $\Sht^0_{U(N_1)}$ has dimension $0$.  Thus
\[
d(\add_{N_1,N_2})=rN_2-r(N_1+N_2)=-rN_1.
\]
Consequently its Gysin pullback carries
$\CH_{r(N_1+N_2-m)}$ to
$\CH_{r(N_1+N_2-m)-rN_1}=\CH_{r(N_2-m)}$, which is precisely the degree of the product class on the right side of \eqref{eq: unitary direct sum factorization}. For target rank $N$, let $\cZ_{\cE}^{N,r}(a)\subset\Sht^r_{U(N)}$ denote the special cycle of Hermitian parameter $a\in\cA_{\cE}(k)$. Functoriality in the unitary variable \cite[Proposition 7.5]{FYZ2} gives, for every $a$,
\begin{equation}\label{eq: unitary fixed-a direct-sum pullback}
\add_{N_1,N_2}^*
[\cZ_{\cE}^{N_1+N_2,r}(a)]
=
\sum_{a_1+a_2=a}
[\cZ_{\cE}^{N_1,0}(a_1)]\bt
[\cZ_{\cE}^{N_2,r}(a_2)]
\end{equation}
in $\CH_*(\Sht^0_{U(N_1)}\times\Sht^r_{U(N_2)})$. 
%In more detail, apply \cite[Proposition 7.5]{FYZ2} to $B(U(N_1)\times U(N_2))\to BU(N_1+N_2)$, restrict to the open-closed component $\Sht^0_{U(N_1)}\times\Sht^r_{U(N_2)}$ on which all legs modify the second factor, and pull back along $\pt\to B\Aut(\cE)$ as in \cite[Example 7.6]{FYZ2}. The source decomposes by the two Hermitian parameters, whose sum is the target parameter. At an $R$-point $(t_1,t_2)$, this is the identity\footnote{The isomorphism of derived stacks in \cite[Proposition 7.5]{FYZ2} gives \eqref{eq: unitary fixed-a direct-sum pullback} before taking fundamental classes.}
%\begin{equation}\label{eq: unitary direct-sum Hermitian parameter}
%a(t_1\oplus t_2)
%=
%\sigma^*(t_1\oplus t_2)^\vee\circ h_{\cF_1\oplus\cF_2}\circ(t_1\oplus t_2)
%=
%a(t_1)+a(t_2),
%\end{equation}
%because $h_{\cF_1\oplus\cF_2}=h_{\cF_1}\oplus h_{\cF_2}$ is block diagonal. 
Additivity of the Artin--Schreier character and \eqref{eq: unitary fixed-a direct-sum pullback} now give the unnormalized factorization
\[
\begin{aligned}
&\add_{N_1,N_2}^*
\left(
\sum_{a}
\psi(\langle e_{\cG,\cE},a\rangle)
[\cZ_{\cE}^{N_1+N_2,r}(a)]
\right)\\
&\quad =
\sum_{a_1,a_2}
\psi(\langle e_{\cG,\cE},a_1+a_2\rangle)
[\cZ_{\cE}^{N_1,0}(a_1)]\bt
[\cZ_{\cE}^{N_2,r}(a_2)]\\
&\quad =
\left(
\sum_{a_1}
\psi(\langle e_{\cG,\cE},a_1\rangle)
[\cZ_{\cE}^{N_1,0}(a_1)]
\right)
\bt
\left(
\sum_{a_2}
\psi(\langle e_{\cG,\cE},a_2\rangle)
[\cZ_{\cE}^{N_2,r}(a_2)]
\right).
\end{aligned}
\]
The character identity \eqref{eq: unitary character product} likewise gives
\[
\chi_{N_1+N_2}(\det\cE)\,
q^{(N_1+N_2)(\deg\cE-\deg\omega_X)/2}
=
\prod_{i=1}^2
\chi_{N_i}(\det\cE)\,
q^{N_i(\deg\cE-\deg\omega_X)/2}.
\]
Together, these two factorizations prove \eqref{eq: unitary direct sum factorization}.
\end{proof}

Throughout the rest of this subsection, we use the identification of $\CH_0(\Sht^0_{U(N)})$ with the space of $\ol\Q$-valued functions on the groupoid $\Bun_{U(N)}(\F_q)$, as in \cite[\S 4]{FYZ2} and \cite[\S 2.3]{FYZ3}; under this identification, $[\cZ^0_{\cE}(a)](\cF)=\#\{t\co\cE\rightarrow\cF \mid a(t)=a\}$, with no automorphism weighting.

\begin{lemma}\label{lem: unitary zero leg nonvanishing}
Fix $m\le n$ and $(\cG,\cE)\in\Bun_{P_m}(k)$.  There is a point $\cF_0\in\Sht^0_{U(2n)}(k)$ such that
\[
\wt Z_m^{2n,0}(\cG,\cE)(\cF_0)\ne0.
\]
\end{lemma}
\begin{proof}
Choose a rank $n$ vector bundle $\cH$ on $X'$ defined over $\F_q$, for instance $\cO_{X'}^{\oplus n}$.  For
\[
\cF_0=\cH\oplus\sigma^*\cH^\vee
\]
we have $\sigma^*\cF_0^\vee\simeq\sigma^*\cH^\vee\oplus\cH$, and the isomorphism $\cF_0\xrightarrow{\sim}\sigma^*\cF_0^\vee$ exchanging the two summands is Hermitian.  Thus $\cF_0$ defines a $k$-point of $\Bun_{U(2n)}$, hence a point $\cF_0\in\Sht^0_{U(2n)}(k)$. Then
\begin{equation}\label{eq: unitary zero leg gauss sum}
\wt Z_m^{2n,0}(\cG,\cE)(\cF_0)
=
\chi_{2n}(\det\cE)q^{n(\deg\cE-\deg\omega_X)}
\sum_{t\in\Hom_{X'}(\cE,\cF_0)}
\psi\bigl(\langle e_{\cG,\cE},a(t)\rangle\bigr),
\end{equation}
where $a(t)=\sigma^*t^\vee\circ h_{\cF_0}\circ t$ is the Hermitian parameter induced by $t$.  The prefactor in \eqref{eq: unitary zero leg gauss sum} is nonzero, and the function $t\mapsto\langle e_{\cG,\cE},a(t)\rangle$ is a homogeneous quadratic form on the finite $k$-vector space $\Hom_{X'}(\cE,\cF_0)$. Hence \eqref{eq: unitary zero leg gauss sum} is nonzero.
\end{proof}

\begin{thm}[The Modularity Conjecture]\label{thm: unitary full modularity}
For all $m\le n$ and all $r\ge0$, the map
\[
\wt Z_m^{n,r}\co
\Bun_{P_m}(k)
\longrightarrow
\CH_{r(n-m)}(\Sht^r_{U(n)})
\]
descends to $\Bun_{GU^-(2m)}^{\mathrm{triv}}(k)$.  Equivalently, if $\cG$ is a trivial-similitude skew-Hermitian bundle of rank $2m$ and $\cE_1,\cE_2\subset\cG$ are Lagrangian subbundles, then
\[
\wt Z_m^{n,r}(\cG,\cE_1)=
\wt Z_m^{n,r}(\cG,\cE_2)
\]
in $\CH_{r(n-m)}(\Sht^r_{U(n)})$.
\end{thm}

\begin{proof}
Fix the rank $n$ character $\chi_n$.  For the auxiliary ranks set $\chi_{2n}=1$ and $\chi_{3n}=\chi_n$.

Let $\cG$ be a trivial-similitude skew-Hermitian bundle of rank $2m$, and let $\cE_1,\cE_2\subset\cG$ be Lagrangian subbundles.  Since $m\le n=(3n)/3$, Theorem \ref{thm:unitary-low-corank-modularity-section} applies to give
\begin{equation}\label{eq: unitary rank 3n equality}
\wt Z_m^{3n,r}(\cG,\cE_1)
=
\wt Z_m^{3n,r}(\cG,\cE_2)
\in \CH_{r(3n-m)}(\Sht^r_{U(3n)}) .
\end{equation}
Pulling \eqref{eq: unitary rank 3n equality} back along
\[
\add_{2n,n}\co
\Sht^0_{U(2n)}\times\Sht^r_{U(n)}
\longrightarrow
\Sht^r_{U(3n)}
\]
gives, by Lemma \ref{lem: unitary direct sum factorization},
\begin{equation}\label{eq: unitary tensor equality}
\wt Z_m^{2n,0}(\cG,\cE_1)\bt
\wt Z_m^{n,r}(\cG,\cE_1)
=
\wt Z_m^{2n,0}(\cG,\cE_2)\bt
\wt Z_m^{n,r}(\cG,\cE_2).
\end{equation}

The $r=0$ case of \cite[Conjecture 4.15]{FYZ2} is the classical automorphy of the theta series for the dual pair $(GU^-(2m),GU(2n))$, and is proved in \cite[\S\S 2.2--2.3]{FYZ3}. Therefore we have $\wt Z_m^{2n,0}(\cG,\cE_1)=
\wt Z_m^{2n,0}(\cG,\cE_2)$. Putting this into \eqref{eq: unitary tensor equality} yields
\begin{equation}\label{eq: unitary common factor}
\wt Z_m^{2n,0}(\cG,\cE_1)\bt
\left(
\wt Z_m^{n,r}(\cG,\cE_1)-
\wt Z_m^{n,r}(\cG,\cE_2)
\right)=0.
\end{equation}

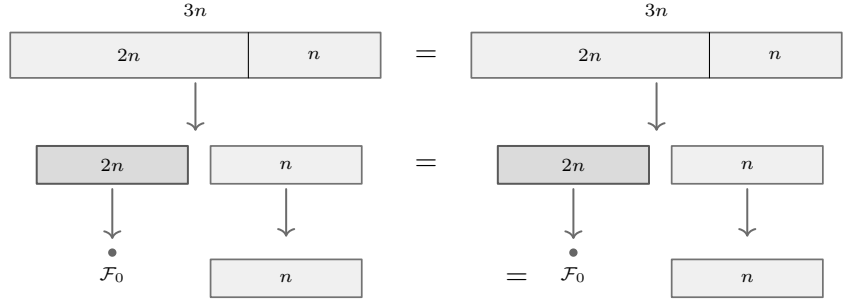
\begin{figure}[htbp]
\centering
\begin{tikzpicture}[
        x=1cm,y=1cm,
        font=\small,
        line cap=round,
        line join=round,
        block/.style={draw=black!55, fill=black!6, line width=.65pt},
        common/.style={draw=black!65, fill=black!14, line width=.75pt},
        arr/.style={->, draw=black!58, line width=.75pt},
        eq/.style={font=\large}]
    \draw[block] (0,2.45) rectangle (4.9,3.05);
    \draw[block] (6.1,2.45) rectangle (11,3.05);
    \draw (3.15,2.45) -- (3.15,3.05);
    \draw (9.25,2.45) -- (9.25,3.05);
    \node[font=\scriptsize] at (1.57,2.75) {$2n$};
    \node[font=\scriptsize] at (4.02,2.75) {$n$};
    \node[font=\scriptsize] at (7.67,2.75) {$2n$};
    \node[font=\scriptsize] at (10.12,2.75) {$n$};
    \node[eq] at (5.5,2.75) {$=$};
    \node[font=\scriptsize] at (2.45,3.35) {$3n$};
    \node[font=\scriptsize] at (8.55,3.35) {$3n$};

    \draw[arr] (2.45,2.38) -- (2.45,1.75);
    \draw[arr] (8.55,2.38) -- (8.55,1.75);

    \draw[common] (.35,1.05) rectangle (2.35,1.55);
    \draw[block] (2.65,1.05) rectangle (4.65,1.55);
    \draw[common] (6.45,1.05) rectangle (8.45,1.55);
    \draw[block] (8.75,1.05) rectangle (10.75,1.55);
    \node[font=\scriptsize] at (1.35,1.3) {$2n$};
    \node[font=\scriptsize] at (3.65,1.3) {$n$};
    \node[font=\scriptsize] at (7.45,1.3) {$2n$};
    \node[font=\scriptsize] at (9.75,1.3) {$n$};
    \node[eq] at (5.5,1.3) {$=$};

    \draw[arr] (1.35,.98) -- (1.35,.32);
    \draw[arr] (7.45,.98) -- (7.45,.32);
    \fill[black!60] (1.35,.14) circle (.055);
    \fill[black!60] (7.45,.14) circle (.055);
    \node[font=\scriptsize] at (1.35,-.15) {$\cF_0$};
    \node[font=\scriptsize] at (7.45,-.15) {$\cF_0$};

    \draw[arr] (3.65,.98) -- (3.65,.35);
    \draw[arr] (9.75,.98) -- (9.75,.35);
    \draw[block] (2.65,-.45) rectangle (4.65,.05);
    \draw[block] (8.75,-.45) rectangle (10.75,.05);
    \node[font=\scriptsize] at (3.65,-.2) {$n$};
    \node[font=\scriptsize] at (9.75,-.2) {$n$};
    \node[eq] at (6.7,-.2) {$=$};
\end{tikzpicture}
\caption{Cartoon for the embedding-and-cancellation step following \eqref{eq: unitary common factor}. Pullback from rank $3n$ splits into a rank $2n$ factor and a rank $n$ factor; evaluation at $\cF_0$ makes the common rank $2n$ factor nonzero, so the rank $n$ factors agree.}
\label{fig:embedding-cancellation}
\end{figure}

By Lemma \ref{lem: unitary zero leg nonvanishing}, there is a point
$\cF_0\in\Sht^0_{U(2n)}(k)$ with $\wt Z_m^{2n,0}(\cG,\cE_1)(\cF_0)\neq 0$.
Pulling \eqref{eq: unitary common factor} back along
$\{\cF_0\}\times\Sht^r_{U(n)}\to
\Sht^0_{U(2n)}\times\Sht^r_{U(n)}$ gives
\[
\wt Z_m^{2n,0}(\cG,\cE_1)(\cF_0)
\left(
\wt Z_m^{n,r}(\cG,\cE_1)-
\wt Z_m^{n,r}(\cG,\cE_2)
\right)=0
\quad\text{in}\quad
\CH_{r(n-m)}(\Sht^r_{U(n)}).
\]
Since the scalar $\wt Z_m^{2n,0}(\cG,\cE_1)(\cF_0)$ is nonzero, we conclude that
\[
\wt Z_m^{n,r}(\cG,\cE_1)=
\wt Z_m^{n,r}(\cG,\cE_2)
\quad\text{in}\quad
\CH_{r(n-m)}(\Sht^r_{U(n)}).
\]
This proves the theorem.
\end{proof}

\part{Supermodularity for general linear groups}\label{part: supermodularity}

\section{Supermodularity}
\subsection{Formulation of supermodularity}
We begin by formulating supermodularity. Fix integers $m \leq n$ and a decomposition $m=m_1+m_2$ with $m_1,m_2\geq0$. Let $P_{(m_1,m_2)}\subset \GL(m)$ be the standard parabolic subgroup with Levi quotient $\GL(m_1)\times \GL(m_2)$, with the convention $\GL(0)=\{e\}$. Thus a point of $\Bun_{P_{(m_1,m_2)}}(\F_q)$ is a rank $m$ vector bundle $\cG$ on $X$ together with a rank $m_1$ subbundle $\cE_1\subset \cG$ whose quotient has the form $\cE_2^*$ for a rank $m_2$ vector bundle $\cE_2$. Equivalently, $(\cE_1,\cG)$ determines a short exact sequence
\begin{equation}\label{eq: formulation parabolic extension}
0\longrightarrow \cE_1\longrightarrow \cG\longrightarrow \cE_2^*\longrightarrow 0,
\end{equation}
and we denote its extension class by $e_{\cG}\in \Ext^1(\cE_2^*,\cE_1)$.

Let $\mu=(\mu_1,\ldots,\mu_r)$ be a sequence with $\mu_i\in\{\pm1\}$ and $\sum_i\mu_i=0$. For a pair $(\cE_1,\cE_2)$ as in \eqref{eq: formulation parabolic extension}, the special cycle of \cite[\S 4]{FYZ2} gives a class
\[
[\cZ_{\cE_1,\cE_2}^{\mu}]\in \CH_{\frac{r}{2}(2n-m)}(\Sht_n^\mu).
\]
The cycle is naturally decomposed over the Hitchin base $\Hom(\cE_1,\cE_2^\vee)$, where $\cE_2^\vee=\cE_2^*\otimes\omega_X$:
\begin{equation}\label{eq: formulation special cycle decomposition}
[\cZ_{\cE_1,\cE_2}^{\mu}]
=
\sum_{a\in\Hom(\cE_1,\cE_2^\vee)}
[\cZ_{\cE_1,\cE_2}^{\mu}(a)].
\end{equation}

For a non-trivial additive character $\psi\co \F_q\to \ol{\Q}^{\times}$, the associated higher theta series is
\begin{equation}\label{eq: formulation theta series}
\wt{Z}_{m_1,m_2}^{\psi,\mu}\co
\Bun_{P_{(m_1,m_2)}}(\F_q)
\longrightarrow
\CH_{\frac{r}{2}(2n-m)}(\Sht_n^\mu),
\end{equation}
whose value at $(\cE_1,\cG)$ is
\begin{equation}\label{eq: formulation theta value}
\wt{Z}_{m_1,m_2}^{\psi,\mu}(\cE_1,\cG)
=
q^{n\deg\cE_2}
\sum_{a\in\Hom(\cE_1,\cE_2^\vee)}
\psi(\langle e_{\cG},a\rangle)
[\cZ_{\cE_1,\cE_2}^{\mu}(a)].
\end{equation}
Here $\langle e_{\cG},a\rangle\in\F_q$ is the Serre duality pairing between $e_{\cG}$ and $a$. When $m_2=0$, the set $\Hom(\cE_1,\cE_2^\vee)$ has a single element and \eqref{eq: formulation theta value} reduces to $[\cZ_{\cG,0}^{\mu}]$. The notation in Theorem \ref{thm: intro supermodularity} suppresses $\psi$; Lemma \ref{lem: ind of psi} below proves that the value in \eqref{eq: formulation theta value} is independent of the choice of $\psi$.

The Modularity Conjecture for general linear groups \cite[Conjecture 4.24]{FYZ2} asserts that \eqref{eq: formulation theta series} descends from $\Bun_{P_{(m_1,m_2)}}(\F_q)$ to $\Bun_m(\F_q)$, i.e., that \eqref{eq: formulation theta value} depends on the subbundle $\cE_1\subset\cG$ only through $\cG$. Supermodularity is the stronger assertion that the higher theta series is also essentially independent of the parabolic.

\begin{thm}[Supermodularity]\label{thm: intro supermodularity} For a cycle $[Z]\in \CH_*(\Sht_n^\mu)$ we write
\[
[Z]^{\shear}
:=
\sum_{d\in\Z}q^{-m_2d+m_2n(g-1)}[Z]_d,
\]
where $[Z]_d$ is the projection to the component of $\Sht_n^\mu$ on which the initial vector bundle has degree $d$.  The sum is understood componentwise with respect to this open-and-closed degree decomposition: equivalently, $[Z]^{\shear}$ is the class whose restriction to the degree-$d$ component is $q^{-m_2d+m_2n(g-1)}[Z]_d$.
Then we have
\begin{equation}\label{eq: formulation supermodularity}
\wt{Z}_{m_1,m_2}^{\mu}(\cE_1,\cG)
=
[\cZ_{\cG,0}^{\mu}]^{\shear}
=
\sum_{d\in\Z}q^{-m_2d+m_2n(g-1)}[\cZ_{\cG,0}^{\mu}]_d.
\end{equation}
\end{thm}

\subsection{Proof of Theorem \ref{thm: intro supermodularity} for $r=0$}\label{ssec: r=0} 

In this subsection we prove Theorem \ref{thm: intro supermodularity} in the special case $r=0$, as a toy model for the general case. In this case, all the geometry becomes trivial and the result is classical. However, it is delicate to formulate a proof that generalizes to $r>0$, and we have not found our particular argument elsewhere in the literature. 

For a rank $n$ vector bundle $\cF$ on $X$, the $r=0$ theta function is
\begin{equation}\label{eq: r=0 theta}
\wt Z^{0}_{m}(\cE_1, \cG)_{\cF} := q^{n(\deg \cE_2)} \sum_{\substack{t_1 \in \Hom(\cE_1, \cF) \\ t_2   \in  \Hom(\cE_2, \cF^{\vee})}} \psi (\langle e_{\cG}, a(t_1,t_2) \rangle).
\end{equation}
We write the theta function \eqref{eq: r=0 theta} as follows. Serre duality gives a pairing 
\[
\ev \co \Hom(\cE_1, \cF) \times \Hom(\cF, \cE_1 \otimes \omega_X\sm{[1]}) \rightarrow \F_q.
\]
Let $f \co \Hom(\cF , \cE_2^\vee) \rightarrow \Ext^1(\cF, \cE_1 \otimes \omega_X)$ be the boundary map in the long exact sequence
\begin{equation}\label{eq: r=0 LES}
\ldots  \rightarrow \Hom(\cF, \cE_2^\vee) \xrightarrow{f} \Ext^1(\cF, \cE_1 \otimes \omega_X) \xrightarrow{g} \Ext^1(\cF, \cG \otimes \omega_X) \rightarrow \ldots
\end{equation}
For $u\in\Hom(\cF,\cE_2^\vee)$, corresponding to $t_2\in\Hom(\cE_2,\cF^\vee)$, set
\[
a(t_1,t_2):=u\circ t_1\in\Hom(\cE_1,\cE_2^\vee).
\]
With the standard connecting-homomorphism convention, $f(u)$ is the pullback of the extension
\[
0\to\cE_1\otimes\omega_X\to\cG\otimes\omega_X\to\cE_2^\vee\to0
\]
along $u$. Hence
\[
\ev(t_1,f(u))
=\Tr_{\cE_1}\bigl(f(u)\circ t_1\bigr)
=\Tr_{\cE_2^*}\bigl((u\circ t_1)\circ e_{\cG}\bigr)
=\langle e_{\cG},a(t_1,t_2)\rangle .
\]
Noting that $\Hom( \cE_2, \cF^{\vee}) \cong \Hom(\cF , \cE_2^{\vee})$, the theta function \eqref{eq: r=0 theta} can be reformulated as 
\begin{equation}\label{eq: r=0 theta 2}
\wt Z^{0}_{m}(\cE_1, \cG)_{\cF} = q^{n(\deg \cE_2)}  \sum_{\substack{t_1 \in \Hom(\cE_1, \cF) \\ s_1 \in \Ext^1(\cF, \cE_1 \otimes \omega_X)}} \psi ( \ev(t_1, s_1)) (f_! \bbm{1}_{\Hom(\cF, \cE_2^\vee)})(s_1) ,
\end{equation}
where $\bbm{1}$ is the indicator function. Set
\[
\phi:=f_! \bbm{1}_{\Hom(\cF,\cE_2^\vee)}.
\]
Consider the commutative diagram below:
\begin{equation}\label{eq: outline commutative diagram}
\begin{tikzcd}
& \Hom(\cE_1, \cF) \times \Hom(\cF, \cE_2^\vee) \ar[dd, "\Id \times f"] \ar[dr, "\pr_2"']  \\
\Hom(\cG, \cF) \ar[dd, "\wh{g}"]  & & \Hom(\cF, \cE_2^\vee) \ar[dd, "f"] \\
& \Hom(\cE_1, \cF) \times \Ext^1(\cF, \cE_1 \otimes \omega_X) \ar[dr, "\pr_2"']  \ar[dl, "\pr_1"]  \ar[dd, "\pi"] \\
\Hom(\cE_1, \cF)   \ar[dr, "\pi_1"'] \ar[dd, "\wh{f}"]  & & \Ext^1(\cF, \cE_1 \otimes \omega_X) \ar[dl, "\pi_2"] \ar[dd, "g"]  \\
& \{ \cF\}  \\
\Ext^1(\cE_2^*, \cF) & & \Ext^1(\cF, \cG \otimes \omega_X)
\end{tikzcd}
\end{equation}
Here $\wh g$ and $\wh f$ are the Serre-dual maps of $g$ and
$f$, respectively, and the left column is the linear dual of the right
column under Serre duality.
In these terms, we have 
\begin{equation}\label{eq: r=0 theta 3}
\wt Z^{0}_{m}(\cE_1, \cG)_{\cF}= q^{n(\deg \cE_2)}  \pi_! (\ev^* \psi \cdot \pr_2^* \phi).
\end{equation}
With the finite Fourier transform $\FT$ normalized as in \cite[\S 2.3.7]{FYZ3}, we have\footnote{We write $\hom(\ldots) := \dim \Hom(\ldots)$ and $\ext(\ldots) := \dim \Ext(\ldots)$, etc.}
\[
\pr_{1!}(\ev^* \psi \cdot \pr_2^*\phi)
=
(-1)^{\hom(\cE_1, \cF)}\FT(\phi).
\]
We may thus rewrite \eqref{eq: r=0 theta 3} as 
\begin{equation}\label{eq: r=0 theta 4}
\wt Z^{0}_{m}(\cE_1, \cG)_{\cF}= q^{n(\deg \cE_2)}  (-1)^{\hom(\cE_1, \cF)} \pi_{1!}\FT(\phi).
\end{equation}

Now, by exactness of \eqref{eq: r=0 LES}, we have 
\begin{equation}\label{eq: descent for functions}
\phi = q^{\dim\ker(f)} g^* \delta_{\Ext^1(\cF, \cG \otimes \omega_X)},
\end{equation}
where $\delta$ is the delta function with value $1$ at $0 \in \Ext^1(\cF, \cG \otimes \omega_X)$. Therefore, using the functoriality of $\FT$ as in \cite[(2.3.58) and Example 2.3.7]{FYZ3}, we have 
\begin{align*}
\FT(\phi) &= q^{\dim\ker(f)} \FT(g^* \delta_{\Ext^1(\cF, \cG \otimes \omega_X)})  \\
&= q^{\dim\ker(f) + \ext^1(\cF, \cE_1 \otimes \omega_X) - \ext^1(\cF, \cG \otimes \omega_X)} (-1)^{\ext^1(\cF, \cE_1 \otimes \omega_X)}  \wh{g}_! \bbm{1}_{\Hom(\cG, \cF) }.
\end{align*}
Combining this with \eqref{eq: r=0 theta 4}, and using $\hom(\cE_1, \cF) = \ext^1(\cF, \cE_1 \otimes \omega_X)$ to simplify signs, we obtain
\begin{equation}\label{eq: r=0 theta 5}
\wt Z^{0}_{m}(\cE_1, \cG)_{\cF} = q^{n \deg \cE_2 + \dim\ker(f) + \ext^1(\cF, \cE_1 \otimes \omega_X) - \ext^1(\cF, \cG \otimes \omega_X)}  \pi_{1!} \wh{g}_! \bbm{1}_{\Hom(\cG, \cF) }.
\end{equation}
By the long exact sequence \eqref{eq: r=0 LES}, we have 
\[
\dim\ker(f) = \hom(\cF, \cG \otimes \omega_X) - \hom(\cF, \cE_1 \otimes \omega_X).
\]
Therefore, the exponent of $q$ in \eqref{eq: r=0 theta 5} can be rewritten as
\[
n\deg\cE_2+\chi_{\mathrm{Eul}}(X,\RHom(\cF, \cG \otimes \omega_X)) - \chi_{\mathrm{Eul}}(X,\RHom(\cF, \cE_1 \otimes \omega_X))
=n\deg\cE_2+\chi_{\mathrm{Eul}}(X,\RHom(\cF,\cE_2^\vee)).
\]
By Riemann--Roch,
\[
\chi_{\mathrm{Eul}}(X,\RHom(\cF,\cE_2^\vee))
=
-m_2\deg\cF-n\deg\cE_2+m_2n(g-1),
\]
so the exponent of $q$ in \eqref{eq: r=0 theta 5} may be rewritten as $-m_2\deg\cF+m_2n(g-1)$.  We find that
\begin{equation}\label{eq: r=0 theta 6}
\wt Z^{0}_{m}(\cE_1, \cG)_{\cF} = q^{-m_2\deg\cF+m_2n(g-1)}  \pi_{1!} \wh{g}_! \bbm{1}_{\Hom(\cG, \cF) }.
\end{equation}
Since $\pi_{1!}\wh{g}_!\bbm{1}_{\Hom(\cG,\cF)}=\#\Hom(\cG,\cF)=[\cZ_{\cG,0}^{0}]_{\cF}$, and the prefactor is the shear factor of Theorem \ref{thm: intro supermodularity} on the component $d=\deg\cF$, this confirms Theorem \ref{thm: intro supermodularity} in the case $r=0$. In particular, this visibly depends on the choice of sub-bundle $\cE_1 \subset \cG$ only through its rank.  \qed

\subsection{Outline for general $r$} \label{ssec: outline}

For $m\le n/3$, the proof of Theorem \ref{thm: intro supermodularity} follows the same pattern as the case $r=0$. Its four ingredients are these.
\begin{enumerate}
\item Derived vector bundles $U\to V\to W$ over $\Bun_n$ geometrize the sequence
\[
\Hom(\cF, \cE_2^\vee) \rightarrow \Ext^1(\cF, \cE_1 \otimes \omega_X) \rightarrow \Ext^1(\cF, \cG \otimes \omega_X)
\]
in the right column of \eqref{eq: outline commutative diagram}. 

The bundles $U'=\wh W$, $V'=\wh V$, and $W'=\wh U$ form a second sequence $U'\to V'\to W'$ geometrizing
\[
\Hom(\cG, \cF)  \rightarrow \Hom(\cE_1, \cF)   \rightarrow \Ext^1(\cE_2^*, \cF) 
\] 
in the left column of \eqref{eq: outline commutative diagram}. 

\item The six Hecke correspondences $\Hk_U^\mu,\Hk_V^\mu,\Hk_W^\mu$ and $\Hk_{U'}^\mu,\Hk_{V'}^\mu,\Hk_{W'}^\mu$ are derived vector bundles over $\Hk_n^\mu$.

\item The cycle $[\cZ_{\cE_1,\cE_2}^\mu]$, whose components over the Hitchin base are the Fourier coefficients of $\wt Z_{m_1,m_2}^\mu(\cE_1,\cG)$, is the trace of a cohomological correspondence $\cc_U$ on $\Hk_U^\mu$; the cycle $[\cZ_{\cG,0}^\mu]$ is the trace of a correspondence $\cc_{U'}$ on $\Hk_{U'}^\mu$. The first calculation uses Theorem \ref{thm:split-trace-conjecture-n/3} under the separate bounds $m_1,m_2\le n/3$ and does not require $m_1+m_2\le n/3$. The second applies it to $(m,0)$ and uses $m=m_1+m_2\le n/3$. This lifts supermodularity to the correspondence level. The dotted arrows in the diagram are Fourier dualities.
\[
\begin{tikzcd}
{[\cZ_{\cE_1, \cE_2}^\mu]} & \cc_U \ar[d] \ar[l, dashed, "\Tr^{\Sht}"'] & U \ar[d, "f"]  \ar[ddrrr, dotted] & & & U' = \wh{W} \ar[d, "f' = \wh{g}"] \ar[ddlll, dotted]  & \cc_{U'} \ar[d] \ar[r, dashed, "\Tr^{\Sht}"]  & {[\cZ_{\cG, 0}^\mu]} \\  
& f_! \cc_U & V \ar[d, "g"]  \ar[rrr, dotted] & & & \ar[lll, dotted]   V'  = \wh{V} \ar[d, "g'   =\wh{f}"]  & f_!' \cc_{U'} \\
 & & W \ar[rrruu, dotted] & & &  W' = \wh{U} \ar[uulll, dotted] 
\end{tikzcd}
\]

\item The derived Fourier transform identifies $f_!\cc_U$ with $f'_!\cc_{U'}$ up to the explicit shift and twist. Pushforward to $\Hk_S^\mu$, followed by the trace formula of \S \ref{sec: trace formula}, gives Theorem \ref{thm: intro supermodularity} for $m\le n/3$; here $S\subset\Bun_n$ is the Harder--Narasimhan truncation of \S \ref{sssec: UVW}.
\end{enumerate}
For arbitrary $m$, apply the low-corank identity at target rank $3n$, pull it back along $\Sht_{2n}^0\times\Sht_n^\mu\to\Sht_{3n}^\mu$, and cancel the nonzero zero-leg factor.

\section{The genetic pattern of supermodularity}
In this section we prove the sheaf-theoretic results that can be regarded as the ``genetics of supermodularity'' in the sense of Figure \ref{fig:cartoon}. Again, this section is essentially agnostic of the ground field. 

\subsection{The bestiary of derived vector bundles}

We suppose a short exact sequence of vector bundles on $X$, 
\begin{equation}\label{eq: EGE SES}
0 \rightarrow \cE_1 \rightarrow \cG \rightarrow \cE_2^* \rightarrow 0
\end{equation}
where $\rank \cG = m$, $\rank \cE_1 = m_1$, and $\rank \cE_2 = m_2$. The constructions and arguments below apply uniformly for all $m_1, m_2 \ge 0$; for instance, when $m_1 = 0$ one has $V = S$, and the Fourier transform along $V$ is the identity. (The case $m_2 = 0$ is in any case tautological, cf.\ the proof of Theorem \ref{thm: gl full supermodularity}.)

Below we will construct the collection of spaces promised in \S \ref{ssec: outline}. For reasons of familiarity we will use the same notation for the setup as used in \S \ref{sec:genetics-of-modularity} (and also in the prior works \cite{FYZ3}, and \cite{FK}), but the actual meaning of the spaces is very different from taking those earlier definitions and specializing them to the situation at hand.

\subsubsection{Definition of $U, V$ and $W$}\label{sssec: UVW}

Below we write $S$ for a Harder--Narasimhan truncation $\Bun_n^{\leq \nu}$. (The reason for this truncation is to guarantee \emph{global presentability} in the sense of \cite[\S 6.3.1]{FYZ3}.) We define several derived vector bundles over $S$. Let $\cF_{\univ}$ be the universal vector bundle over $X \times S$. Below $R$ denotes any animated $\F_{q}$-algebra.
\begin{itemize}

\item Define $\cU:=\ul{\RHom(\cF_{\univ}, \cE_2^\vee)}$ to be the perfect complex on $S$ whose pullback to an $R$-point $\cF \in S(R)$ is naturally in $R$ isomorphic to  $\RHom_{X_R}(\cF, \cE_2^\vee \ot R)$ regarded as a derived $R$-module. Let $U := \Tot_S(\cU)$ be the associated derived vector bundle over $S$.

\item Define $\cV:=\ul{\RHom(\cF_{\univ}, \cE_1 \otimes \omega_X\sm{[1]})}$ to be the perfect complex on $S$ whose pullback to an $R$-point $\cF \in S(R)$ is naturally in $R$ isomorphic to $\RHom_{X_R}(\cF,  \cE_1 \ot \omega_{X} \ot R\sm{[1]})$ regarded as a derived $R$-module. Let $V := \Tot_S(\cV)$ be the associated derived vector bundle over $S$.

\item Define $\cW:=\ul{\RHom(\cF_{\univ}, \cG \ot \omega_X\sm{[1]})}$ to be the perfect complex on $S$ whose pullback to an $R$-point $\cF \in S(R)$ is naturally in $R$ isomorphic to $\RHom_{X_R}(\cF, \cG \ot \omega_X \ot  R\sm{[1]})$ regarded as a derived $R$-module. Let $W := \Tot_S(\cW)$ be the associated derived vector bundle over $S$. 
\end{itemize} 

The short exact sequence \eqref{eq: EGE SES} induces an exact triangle $\cU \to \cV \to \cW$ in $\Perf(S)$, which geometrically is equivalent to the derived Cartesian square of total spaces over $S$: 
\begin{equation}\label{eq: derived cartesian UVWS}
\begin{tikzcd}
U \ar[r, "f"] \ar[d, "\pi"']  & V \ar[d, "g"] \\
S \ar[r, "z_W"] & W
\end{tikzcd}
\end{equation}
where $z_W$ is the zero section.

\subsubsection{Hecke stacks} Let $\mu = (\mu_1, \mu_2, \ldots, \mu_r)$ be a sequence such that $\mu_i \in \{\pm 1\}$ for each $i$ and $\sum_i \mu_i = 0$. Then one has the Hecke stack $\Hk_{n}^{\mu}$, whose $R$-points are diagrams
\begin{equation}\label{eq: flat hecke diagram}
\begin{tikzcd}
\cF_0 \ar[r, dashed, "x_1"']   & \cF_{1} \ar[r, dashed, "x_2"'] & \ldots \ar[r, dashed] & \cF_{r-1} \ar[r, dashed, "x_r"'] & \cF_r 
\end{tikzcd} 
\end{equation}
where each $x_i \in X(R)$ and each $\cF_i \in \Bun_n(R)$, and the dashed arrow $\cF_{i-1} \dashrightarrow \cF_i$ is an elementary modification along the graph of $x_i$, which is upper of length $1$ if $\mu_i = +1$ and lower of length $1$ if $\mu_i = -1$, in the sense of \cite[Definition 5.1]{YZ}. We shall abbreviate such diagrams as $(\cF_{\star}) \in \Hk_{n}^\mu(R)$. 

Given $(\cF_\star)$, we define for each $i = 1, \ldots, r$,
\[
\cF_{i-1/2}^\flat := \cF_{i-1}\cap \cF_i
\quad \text{and} \quad 
\cF_{i-1/2}^\sharp := \cF_{i-1}+\cF_i
\]
inside the common meromorphic bundle determined by the modification. This fits into an obvious pullback square
\begin{equation}\label{eq: sharp/flat}
\begin{tikzcd}
& \cF_{i-1/2}^\flat \ar[dl, hook'] \ar[dr, hook] \\
\cF_{i-1} \ar[dr, hook] & & \cF_i \ar[dl, hook'] \\
& \cF_{i-1/2}^\sharp
\end{tikzcd}
\end{equation}

There are maps 
\begin{equation}\label{eq: h_i untruncated}
h_i \co \Hk_{n}^{\mu} \rightarrow \Bun_{n}, \quad i =0, \ldots, r
\end{equation}
projecting to the datum of $\cF_i$, as well as $\pr\co \Hk_{n}^{\mu} \rightarrow X^r$ projecting to the data of $(x_1, \ldots, x_r)$. 

We define the open substack $\Hk_S^\mu \subset \Hk_{n}^\mu$ as 
\[
\Hk_S^\mu := h_0^{-1}(S) \cap h_1^{-1} (S) \cap \ldots \cap h_r^{-1}(S).
\]
Therefore the maps \eqref{eq: h_i untruncated} restrict to give 
\begin{equation}\label{eq: h_i}
h_i \co \Hk_{S}^\mu \rightarrow S, \quad i = 0, \ldots, r.
\end{equation}

Given a diagram $(\cF_{\star}) \in \Hk_{S}^\mu(R)$, define $\cF_{\bu}^{\flat}$ to be the object of $\Perf(X_R)$ represented by the two-term complex in cohomological degrees $0$ and $1$
\[
(\cF_{1/2}^{\flat} \oplus \ldots \oplus \cF_{r-1/2}^{\flat}) \rightarrow (\cF_1 \oplus \ldots \oplus \cF_{r-1}),
\]
where the map sends $(s_{1/2},\cdots, s_{r-1/2})$ to $(s_{1/2}-s_{3/2}, \cdots, s_{r-3/2}-s_{r-1/2})$, with reference to the commutative diagram \eqref{eq: sharp/flat}. This object need not be quasi-isomorphic to a coherent sheaf.

Define $\cF_{\bu}^{\sh}$ to be the object represented by the two-term complex in cohomological degrees $-1$ and $0$
\[
 (\cF_1 \oplus \ldots \oplus \cF_{r-1})\rightarrow (\cF_{1/2}^{\sh} \oplus \ldots \oplus \cF_{r-1/2}^{\sh})
\]
where the map sends $(s_1, \ldots, s_{r-1})$ to $(s_1, s_1-s_2, \ldots,  s_{r-2}-s_{r-1}, -s_{r-1})$. As in the bestiary construction of \cite[\S 9.1]{FYZ3}, this object is quasi-isomorphic to a coherent \emph{sheaf} on $X_R$, which may not be locally free; we view it as a perfect complex concentrated in degree $0$.

We have a natural map of perfect complexes on $X_{R}$
\begin{equation}\label{eq: frp_i^flat}
\frp_{i}^\flat: \cF_{\bu}^{\flat}\to \cF_{i}, \quad i=0,1,\cdots, r
\end{equation}
that is the composition of the projection to $\cF^{\flat}_{i-1/2}$ and the natural inclusion $\cF^{\flat}_{i-1/2}\incl \cF_{i}$ when $i>0$, and  the composition of the projection to $\cF^{\flat}_{i+1/2}$ and the natural inclusion $\cF^{\flat}_{i+1/2}\incl \cF_{i}$ when $i<r$. (Both constructions give the same map up to explicit chain homotopy when $0<i<r$.) Similarly, we have a natural map 
\begin{equation}\label{eq: frp_i^sh}
\frp_i^{\sharp} \co  \cF_i \rightarrow \cF_{\bu}^{\sh}, \quad i=0, \ldots, r
\end{equation}

As $\cF_{\star}$ varies in $\Hk_{S}^{\mu}$, the construction $\cF_{\star}\mapsto \cF_{\bu}^{\flat}$ gives a universal perfect complex $\cF_{\univ, \bu}^{\flat}$ over $X\times \Hk_{S}^{\mu}$, and similarly $\cF_\star \mapsto \cF_{\bu}^\sharp$ gives a universal coherent sheaf $\cF_{\univ, \bu}^{\sh}$ over $X\times \Hk_{S}^{\mu}$, viewed as a perfect complex concentrated in degree $0$. These fit into an exact triangle in $\Perf(X\times\Hk_S^\mu)$,
\begin{equation}\label{eq: sharp flat triangle}
\cF_{\univ, \bu}^{\flat} \rightarrow h_0^* \cF_{\univ} \oplus h_r^* \cF_{\univ} \rightarrow  
\cF_{\univ, \bu}^{\sh} 
\end{equation}
in which the first map has components $(\frp_0^\flat,\frp_r^\flat)$ and the second is the difference $\frp_0^\sharp - \frp_r^\sharp$.

When $r=0$, the previous paragraph should be understood as follows: $\Hk_S^\mu=S$, the Hecke correspondence is the identity correspondence, and $\cF_\bu^\flat:=\cF_0=:\cF_\bu^\sharp$. Both correspondence maps are the identity, and the triangle \eqref{eq: sharp flat triangle} is defined to be the split triangle
\[
\cF_0\xrightarrow{(1,1)}\cF_0\oplus\cF_0
\xrightarrow{(1,-1)}\cF_0.
\]

\subsubsection{Hecke stacks for $U, V$ and $W$}\label{sss: Hk UVW}

We now define several perfect complexes on $\Hk_S^\mu$, using similar notation as in \S \ref{sssec: UVW}, so the explanations will be abbreviated. 
\begin{itemize} 
\item Define $\cU^{\mu}_{\Hk}:=\ul{\RHom(\cF_{\univ, \bu}^{\sh}, \cE_2^\vee)}$.
Let $\Hk_U^\mu := \Tot_{\Hk_S^\mu}(\cU^{\mu}_{\Hk})$ be the associated derived vector bundle over $\Hk_S^\mu$.
\item Define $\cV^{\mu}_{\Hk}:=\ul{\RHom(\cF^{\sh}_{\univ, \bu}, \cE_1 \otimes \omega_X\sm{[1]})}$.
Let $\Hk_V^\mu := \Tot_{\Hk_S^\mu}(\cV^{\mu}_{\Hk})$ be the associated derived vector bundle over $\Hk_S^\mu$. 
\item Define $\cW^{\mu}_{\Hk}:=\ul{\RHom(\cF^{\sh}_{\univ, \bu}, \cG \otimes \omega_X\sm{[1]})}$.
Let $\Hk_W^\mu := \Tot_{\Hk_S^\mu}(\cW^{\mu}_{\Hk})$ be the associated derived vector bundle over $\Hk_S^\mu$. 
\end{itemize}
From \eqref{eq: EGE SES}, we get an exact triangle $\cU^{\mu}_{\Hk}\to \cV^{\mu}_{\Hk}\to \cW^{\mu}_{\Hk}$ of perfect complexes on $\Hk_S^\mu$, which geometrically is equivalent to the derived Cartesian square of derived vector bundles over $\Hk_S^\mu$:
\begin{equation}\label{eq: Hk UVW Cart}
\begin{tikzcd}
\Hk_U^\mu \ar[r, "f"] \ar[d, "\pi"] & \Hk_V^\mu \ar[d, "g"] \\
\Hk_S^\mu \ar[r, "z"] & \Hk_W^\mu
\end{tikzcd}
\end{equation}
where $z$ is the zero section.

\subsubsection{Definition of $U', V', W'$}

We define further perfect complexes on $S$, using the same notational pattern as in \S \ref{sssec: UVW}. 
\begin{itemize}
\item Define $\cU':=\ul{\RHom(\cG, \cF_{\univ})}$.
Let $U':=\Tot_S(\cU')$ be its associated derived vector bundle over $S$.

\item Define $\cV':=\ul{\RHom(\cE_1, \cF_{\univ})}$.
Let $V':=\Tot_S(\cV')$ be its associated derived vector bundle over $S$.

\item Define $\cW' :=\ul{\RHom(\cE_2^*, \cF_{\univ}\sm{[1]})}$.
Let $W':=\Tot_S(\cW')$ be its associated derived vector bundle over $S$.
\end{itemize}

From \eqref{eq: EGE SES}, we have a derived Cartesian square
\begin{equation}\label{eq: derived cartesian U'V'W'S'}
\begin{tikzcd}
U' \ar[r, "f'"] \ar[d, "\pi'"']  & V' \ar[d, "g'"] \\
S \ar[r, "z_{W'}"] & W'
\end{tikzcd}
\end{equation}
where $z_{W'}$ is the zero section. 

\subsubsection{Hecke stacks for $U',V',W'$} 

We now define several perfect complexes on $\Hk_S^\mu$ which will be used to construct the Hecke stacks for $U',V',W'$. 
\begin{itemize} 
\item Define $(\cU')^{\mu}_{\Hk}:=\ul{\RHom(\cG, \cF_{\univ, \bu}^{\flat})}$.
Let $\Hk_{U'}^\mu := \Tot_{\Hk_S^\mu}((\cU')^{\mu}_{\Hk})$ be the associated derived vector bundle over $\Hk_S^\mu$.
\item Define $(\cV')^{\mu}_{\Hk}:=\ul{\RHom(\cE_1, \cF^{\flat}_{\univ, \bu})}$.
Let $\Hk_{V'}^\mu := \Tot_{\Hk_S^\mu}((\cV')^{\mu}_{\Hk})$ be the associated derived vector bundle over $\Hk_S^\mu$. 
\item Define $(\cW')^{\mu}_{\Hk}:=\ul{\RHom(\cE_2^*,  \cF_{\univ,\bu}^{\flat}\sm{[1]})}$.
Let $\Hk_{W'}^\mu := \Tot_{\Hk_S^\mu}((\cW')^{\mu}_{\Hk})$ be the associated derived vector bundle over $\Hk_S^\mu$. 
\end{itemize}
From \eqref{eq: EGE SES}, we get an exact triangle $(\cU')^{\mu}_{\Hk}\to (\cV')^{\mu}_{\Hk}\to (\cW')^{\mu}_{\Hk}$ of perfect complexes on $\Hk_S^\mu$, which geometrically is equivalent to the derived Cartesian square of derived vector bundles over $\Hk_S^\mu$:
\begin{equation}\label{eq: Hk U'V'W' Cart}
\begin{tikzcd}
\Hk_{U'}^\mu \ar[r, "f'"] \ar[d, "\pi'"] & \Hk_{V'}^\mu \ar[d, "g'"] \\
\Hk_S^\mu \ar[r, "z'"] & \Hk_{W'}^\mu
\end{tikzcd}
\end{equation}
where $z'$ is the zero section.

\subsection{Dualities} We record the dualities that relate $U,V,W$ and their Hecke stacks with $U',V',W'$ and their Hecke stacks. 

\subsubsection{Dual bundles} 
\begin{lemma}\label{lem: dual vector bundles} 
As derived vector bundles over $S$, Serre duality identifies $V'$ with the dual $\wh V$ of $V$, $U'$ with the dual $\wh{W}$ of $W$, and $W'$ with the dual $\wh{U}$ of $U$. Under this identification, the derived fiber square \eqref{eq: derived cartesian U'V'W'S'} is identified with the dual fiber square to \eqref{eq: derived cartesian UVWS}
\begin{equation*}
\begin{tikzcd}
\wh W \ar[r, "\wh{g}"] \ar[d]  & \wh V \ar[d, "\wh{f}"] \\
S \ar[r, "z_{\wh U}"] & \wh U
\end{tikzcd}
\end{equation*}
\end{lemma}

\begin{proof}
Immediate from the definitions and Serre duality (applied on $X_R$ for all animated $\F_q$-algebras $R$). \end{proof}

\subsubsection{Hecke stacks} Let $\wh{\Hk}_E^\mu$ be the dual derived bundle to $\Hk_E^\mu$.

For each $i \in \{0,r\}$, let $S_i$ be the corresponding copy of $S$. For $E \in \{U,V,W, U', V', W'\}$, write $E_i$ for the corresponding copy over $S_i$, and let $\wt{E}_i \rightarrow \Hk_S^\mu$ be the pullback of $E_i \rightarrow S_i$ along $h_i \co \Hk_S^\mu \rightarrow S_i$. The base change maps are denoted
\begin{equation}\label{eq: hiUVW}
h_{i}^{E}: \wt E_{i}\to E_i.
\end{equation} 
For $i \in \{0,r\}$, the map $\Hk_E^\mu \rightarrow \wt{E}_i $ dualizes to $\wh{\wt{E}}_i \rightarrow \wh{\Hk}_E^\mu$. 

\begin{lemma}\label{lem: dual Hecke stacks}
Under the identification of $U'_i \xrightarrow{f_i'} V'_i \xrightarrow{g_i'} W'_i$ with $\wh{W}_i \xrightarrow{\wh{g}_i} \wh{V}_i \xrightarrow{\wh{f}_i} \wh{U}_i$ from Lemma \ref{lem: dual vector bundles}, the six diagrams in \eqref{eq: GL three squares 1} and \eqref{eq: GL three squares 2} are derived Cartesian.
\begin{equation}\label{eq: GL three squares 1}
\adjustbox{max width=\textwidth, center}{%
\begin{tikzcd}
& \Hk_U^{\mu} \ar[dl, "\wt{a}_0"'] \ar[dr, "\wt{a}_r"] \\ 
\wt{U}_0 \ar[dr, "\wh{\wt c_0'}"'] & & \wt{U}_r \ar[dl, "\wh{\wt c_r'}"]  \\
& \wh{\Hk}_{W'}^{\mu} 
\end{tikzcd} \hspace{1cm}
\begin{tikzcd}
& \Hk_V^{\mu} \ar[dl, "\wt{b}_0"'] \ar[dr, "\wt{b}_r"] \\ 
\wt{V}_0 \ar[dr, "\wh{\wt b_0'}"'] & & \wt{V}_r \ar[dl, "\wh{\wt b_r'}"]  \\
& \wh{\Hk}_{V'}^{\mu} 
\end{tikzcd}\hspace{1cm}
\begin{tikzcd}
& \Hk_W^{\mu} \ar[dl, "\wt{c}_0"'] \ar[dr, "\wt{c}_r"] \\ 
\wt{W}_0 \ar[dr, "\wh{\wt a_0'}"'] & & \wt{W}_r \ar[dl, "\wh{\wt a_r'}"]  \\
& \wh{\Hk}_{U'}^{\mu} 
\end{tikzcd}
}
\end{equation}
\begin{equation}\label{eq: GL three squares 2}
\adjustbox{max width=\textwidth, center}{%
\begin{tikzcd}
& \Hk_{U'}^{\mu} \ar[dl, "\wt{a}_0'"'] \ar[dr, "\wt{a}_r'"] \\ 
\wt{U}_0' \ar[dr, "\wh{\wt c_0}"'] & & \wt{U}_r' \ar[dl, "\wh{\wt c_r}"]  \\
& \wh{\Hk}_{W}^{\mu} 
\end{tikzcd} \hspace{1cm}
\begin{tikzcd}
& \Hk_{V'}^{\mu} \ar[dl, "\wt{b}_{0}'"'] \ar[dr, "\wt{b}_{r}'"] \\ 
\wt V_0' \ar[dr, "\wh{\wt b_0}"'] & & \wt V_r' \ar[dl, "\wh{\wt b_r}"]  \\
& \wh{\Hk}_{V}^{\mu} 
\end{tikzcd}\hspace{1cm}
\begin{tikzcd}
& \Hk_{W'}^{\mu} \ar[dl, "\wt{c}_0'"'] \ar[dr, "\wt{c}_r'"] \\ 
\wt{W}_0' \ar[dr, "\wh{\wt a_0}"'] & & \wt{W}_r' \ar[dl, "\wh{\wt a_r}"]  \\
& \wh{\Hk}_{U}^{\mu} 
\end{tikzcd}
}
\end{equation}
\end{lemma}

\begin{proof}
This follows from Serre duality (applied on $X_R$ for all animated $\F_q$-algebras $R$) and the exact triangle \eqref{eq: sharp flat triangle}. 
\end{proof}

\subsubsection{Summary}\label{sssec: summary} For $i\in\{0,r\}$, the maps \eqref{eq: frp_i^sh} induce the unprimed maps and the maps \eqref{eq: frp_i^flat} induce the primed maps in
\begin{equation}\label{eq: tilde maps}
\begin{aligned}
\wt a_i &: \Hk_U^\mu\to\wt U_i, &
\wt a_i' &: \Hk_{U'}^\mu\to\wt U_i',\\
\wt b_i &: \Hk_V^\mu\to\wt V_i, &
\wt b_i' &: \Hk_{V'}^\mu\to\wt V_i',\\
\wt c_i &: \Hk_W^\mu\to\wt W_i, &
\wt c_i' &: \Hk_{W'}^\mu\to\wt W_i'.
\end{aligned}
\end{equation}

Composing the maps in \eqref{eq: tilde maps} with the corresponding maps $h_i^E$ of \eqref{eq: hiUVW}, we obtain
\begin{equation}\label{eq: no tilde maps}
\begin{aligned}
a_i:=h_i^U\circ\wt a_i &: \Hk_U^\mu\to U_i, &
a_i':=h_i^{U'}\circ\wt a_i' &: \Hk_{U'}^\mu\to U_i',\\
b_i:=h_i^V\circ\wt b_i &: \Hk_V^\mu\to V_i, &
b_i':=h_i^{V'}\circ\wt b_i' &: \Hk_{V'}^\mu\to V_i',\\
c_i:=h_i^W\circ\wt c_i &: \Hk_W^\mu\to W_i, &
c_i':=h_i^{W'}\circ\wt c_i' &: \Hk_{W'}^\mu\to W_i'.
\end{aligned}
\end{equation}

We have a pair of commutative diagrams
\begin{equation}\label{eq: big diagram}
\adjustbox{scale=0.8, center}{\begin{tikzcd}
& & \Hk_U^{\mu} \ar[dl, "\wt{a}_0"'] \ar[dr, "\wt{a}_r"] \ar[ddd, bend left, "f"'] \\
& \ar[dl, "h_0^U"'] \wt{U}_0 \ar[ddd, "\wt{f}_0"'] \ar[dr, "\wh{\wt c_0'}"'] && \wt{U}_r \ar[dl, "\wh{\wt c_r'}"] \ar[ddd, "\wt{f}_r"] \ar[dr, "h_r^U"] \\
U_0 \ar[ddd, "f_0"]  & & \wh{\Hk}_{W'}^{\mu}  \ar[ddd, bend right, "\wh{g'}"]   &  &  U_r  \ar[ddd, "f_r"] \\
& & \Hk_V^{\mu} \ar[dl, "\wt{b}_0"'] \ar[dr, "\wt{b}_r"] \ar[ddd, bend left, "g"']   \\
& \wt{V}_0 \ar[ddd, "\wt{g}_0"']  \ar[dr, "\wh{\wt b_0'}"']  \ar[dl, "h_0^V"'] && \wt{V}_r \ar[dl, "\wh{\wt b_r'}"]  \ar[ddd, "\wt{g}_r"]  \ar[dr, "h_r^V"] \\
V_0 \ar[ddd, "g_0"] & &  \wh{\Hk}_{V'}^{\mu}   \ar[ddd, bend right, "\wh{f'}"]  & & V_r  \ar[ddd, "g_r"] \\
& & \Hk_W^{\mu}  \ar[dl, "\wt{c}_0"'] \ar[dr, "\wt{c}_r"] \\
& \wt{W}_0 \ar[dr, "\wh{\wt a_0'}"'] \ar[dl, "h_0^W"']  & &  \wt{W}_r \ar[dl, "\wh{\wt a_r'}"]  \ar[dr, "h_r^W"]  \\
W_0 & &  \wh{\Hk}_{U'}^{\mu}    & & W_r 
\end{tikzcd}\hspace{1cm}
\begin{tikzcd}
& & \Hk_{U'}^{\mu} \ar[dl, "\wt{a}_0'"'] \ar[dr, "\wt{a}_r'"] \ar[ddd, bend left, "f'"'] \\
& \ar[dl, "h_0^{U'}"'] \wt{U}_0' \ar[ddd, "\wt{f}_0'"'] \ar[dr, "\wh{\wt c_0}"'] && \wt{U}_r' \ar[dl, "\wh{\wt c_r}"] \ar[ddd, "\wt{f}_r'"] \ar[dr, "h_r^{U'}"]  \\
U_0' \ar[ddd, "f_0'"] & & \wh{\Hk}_{W}^{\mu}  \ar[ddd, bend right, "\wh{g}"]   & & U_r' \ar[ddd, "f_r'"] \\
& & \Hk_{V'}^{\mu} \ar[dl, "\wt{b}_0'"'] \ar[dr, "\wt{b}_r'"] \ar[ddd, bend left, "g'"']   \\
& \ar[dl, "h_0^{V'}"']  \wt{V}_0' \ar[ddd, "\wt{g}_0'"']  \ar[dr, "\wh{\wt b_0}"']  && \wt{V}_r' \ar[dl, "\wh{\wt b_r}"]  \ar[ddd, "\wt{g}_r'"]  \ar[dr, "h_r^{V'}"]  \\
V_0' \ar[ddd, "g_0'"]  & & \wh{\Hk}_{V}^{\mu} \ar[ddd, bend right, "\wh{f}"] & & V_r' \ar[ddd, "g_r'"]  \\	
& & \Hk_{W'}^{\mu}  \ar[dl, "\wt{c}_0'"'] \ar[dr, "\wt{c}_r'"] \\
& \ar[dl, "h_0^{W'}"']  \wt{W}_0' \ar[dr, "\wh{\wt a_0}"']  & &  \wt{W}_r' \ar[dl, "\wh{\wt a_r}"]  \ar[dr, "h_r^{W'}"]  \\
W_0' & & \wh{\Hk}_U^\mu & & W_r'
\end{tikzcd}
}
\end{equation}
In each diagram: 
\begin{itemize}
\item The maps in the columns come from exact triangles of perfect complexes. 
\item The three diamonds in the middle are derived Cartesian. 
\item The four parallelograms on the left and right sides are derived  Cartesian. 
\end{itemize}
The primed diagram in \eqref{eq: big diagram} is dual to the unprimed diagram. The duality exchanges $U$ with $W'$, $V$ with $V'$, and $W$ with $U'$.
	
\begin{remark}[Global presentability]
By the same considerations as in \cite[\S 9.3]{FYZ3}, the diagram \eqref{eq: big diagram} is globally presented. 
\end{remark}

\subsection{Geometric properties}\label{sssec: geometric properties}  We establish some needed geometric properties of the maps in \eqref{eq: big diagram}. 

\begin{lemma}\label{lem: hecke corr quasismooth tilde} For $i\in\{0,r\}$, we have the following properties of the morphisms in \eqref{eq: tilde maps}.
\begin{enumerate}
\item Each of $\wt{a}_i$, $\wt{a}_i'$, and $\wt{b}_i'$ is a quasi-smooth closed embedding. 
\item Each of $\wt{b}_i, \wt{c}_i$, and $\wt{c}_i'$ is a smooth vector bundle. 
\end{enumerate}
\end{lemma}

\begin{proof}
We will first discuss the maps $\wt{a}_i$, $\wt{b}_i, \wt{c}_i$. For an $R$-point over $(\cF_\star)$, let
\[
\cT_i := \operatorname{cofib}(\cF_i \xrightarrow{\frp_{i}^\sh} \cF_{\bu}^{\sh}).
\]
This glues to a universal complex $\cT_{i,\univ}$ on $X\times\Hk_S^\mu$. Note that on classical truncations, $\cT_i$ is a torsion coherent sheaf on $X \times \Hk_S^\mu$, supported on the graph of the legs. The relative tangent complex of $\wt a_i$ to the $R$-point $(\cF_\star)$ is $\RHom_{X_R}(\cT_i,\cE_2^\vee\otimes R)$, i.e., the pullback of $\ul{\RHom(\cT_{i,\univ},\cE_2^\vee)}$ from $X \times \Hk_S^\mu$. As in \cite[Lemma 9.1.3]{FYZ3}, this complex has tor-amplitude $[0,1]$, so \cite[Lemma 6.1.5]{FYZ3} makes $\wt a_i$ quasi-smooth. Closedness may be checked on classical truncations, where since $\cT_i$ is torsion, this complex is represented by a vector bundle in degree $1$. Thus $\wt a_i$ is a closed embedding, by \cite[Lemma 6.1.5(3)]{FYZ3}.

The analysis of $\wt b_i$ and $\wt c_i$ is similar. The relative tangent complexes of $\wt b_i$ and $\wt c_i$ are the pullbacks of $\ul{\RHom(\cT_{i,\univ},\cE_1\otimes\omega_X\sm{[1]})}$ and $\ul{\RHom(\cT_{i,\univ},\cG\otimes\omega_X\sm{[1]})}$. Both are locally vector bundles in degree $0$; hence $\wt b_i$ and $\wt c_i$ are smooth vector bundles.

Next we turn to $\wt{a}_i', \wt{b}_i'$, and $\wt{c}_i'$. Let $\cK_i:=\operatorname{fib}(\cF_{\bu}^{\flat}\xrightarrow{\frp_i^\flat}\cF_i)$
be the derived fiber in $\Perf(X_R)$, and let $\cQ_{i,\univ}$ be the universal cofiber of $\cF_{\univ,\bu}^\flat\to h_i^*\cF_{\univ}$. Then $\cK_i$ is the pullback of $\cQ_{i,\univ}\sm{[-1]}$. At an $R$-point, the relative tangent complexes of $\wt a_i'$, $\wt b_i'$, and $\wt c_i'$ are computed by $\RHom_{X_R}(\cG,\cK_i)$, $\RHom_{X_R}(\cE_1,\cK_i)$, and $\RHom_{X_R}(\cE_2^*,\cK_i\sm{[1]})$, respectively. Globally, they are the pullbacks of $\ul{\RHom(\cG,\cQ_{i,\univ})}\sm{[-1]}$, $\ul{\RHom(\cE_1,\cQ_{i,\univ})}\sm{[-1]}$, and $\ul{\RHom(\cE_2^*,\cQ_{i,\univ})}$. The first two are locally vector bundles in degree $1$, and the last in degree $0$. Thus $\wt a_i'$ and $\wt b_i'$ are quasi-smooth closed embeddings, while $\wt c_i'$ is a smooth vector bundle.
\end{proof}

By the argument of \cite[Lemma 6.9]{FYZ}, each map $h_i \co \Hk_S^\mu \rightarrow S$ is smooth, schematic, and separated. Since, for every $E\in\{U,V,W,U',V',W'\}$, the map $h_i^E$ in \eqref{eq: hiUVW} is a base change of $h_i$, we get the following corollary. 

\begin{cor}\label{cor: hecke corr quasismooth} For $i\in\{0,r\}$, we have the following properties of the morphisms in \eqref{eq: no tilde maps}.
\begin{enumerate}
\item Each of $a_i, a_i', b_i'$ is quasi-smooth, schematic, and separated. 
\item Each of $b_i,c_i, c_i'$ is smooth, schematic, and separated. 
\end{enumerate}
\end{cor}

\begin{lemma}\label{lem: cube side pushable/pullable} The commutative diagram 
\begin{equation}\label{eq: base change cube}
\xymatrix{ U_{0}\ar[ddd]^{\pi_{0}}\ar[dr]^{f_{0}} && \ar[ll]_-{a_{0}}\ar[rr]^-{a_{r}}\Hk^{\mu}_{U}\ar[ddd]^{\pi}\ar[dr]^{f} && U_{r}\ar[ddd]^{\pi_{r}}\ar[dr]^{f_{r}} \\
& V_{0}\ar[ddd]^{g_{0}} && \ar[ll]_-{b_{0}}\ar[rr]^-{b_{r}} \Hk^{\mu}_{V}\ar[ddd]^{g} && V_{r}\ar[ddd]^{g_{r}}\\
\\
S_{0}\ar[dr]^{z_{0}} && \ar[ll]^-{h_{0}}\ar[rr]_-{h_{r}} \Hk^{\mu}_{S}\ar[dr]^{z} && S_{r}\ar[dr]^{z_{r}}\\
& W_{0} && \ar[ll]_-{c_{0}}\ar[rr]^-{c_{r}} \Hk^{\mu}_{W} && W_{r} 
}
\end{equation}
satisfies the conditions in \cite[\S 4.4.1]{FK}. (Here, $z_{i}$ and $z$ are the inclusions of zero sections, and $\pi_{i}$ and $\pi$ are the natural projections.) The analogous diagram, obtained by replacing every object and map in the list $U,V,W,f,g,\pi,z,a_i,b_i,c_i$ by its counterpart with $'$, satisfies the same conditions.
\end{lemma}

\begin{proof}
The schematic and separated properties needed for the horizontal maps in \eqref{eq: base change cube} follow from Corollary \ref{cor: hecke corr quasismooth}. 

The vertical squares $(U_{0},V_{0}, S_{0}, W_{0})$, $(\Hk^{\mu}_{U},\Hk^{\mu}_{V},\Hk^{\mu}_{S},\Hk^{\mu}_{W})$ and $(U_{r}, V_{r}, S_{r}, W_{r})$ are derived Cartesian by \eqref{eq: derived cartesian UVWS} and \eqref{eq: Hk UVW Cart}. In particular, the rightmost square is pullable with defect zero. It remains to check the pushability or pullability of the remaining squares.
\begin{enumerate}[(a)]
\item The square $(\Hk^{\mu}_{U}, U_{0}, \Hk^{\mu}_{V},V_{0})$ is pushable. For this it suffices to base change all relevant spaces to $\Hk^{\mu}_{S}$ and show instead that
\begin{equation*}
\xymatrix{ \wt U_{0}\ar[d]^{\wt f_{0}} & \Hk^{\mu}_{U} \ar[l]_{\wt a_{0}}\ar[d]^{f}\\
\wt V_{0} & \Hk^{\mu}_{V}\ar[l]_{\wt b_{0}}
}
\end{equation*}
is pushable. As $\wt a_{0}$ is a closed embedding and $\wt b_{0}$ is separated by Lemma \ref{lem: hecke corr quasismooth tilde}, the pushability follows from \cite[Example 3.1.2]{FYZ3}. 
\item The square $(\Hk^{\mu}_{S}, S_{0}, \Hk^{\mu}_{W},W_{0})$ is pushable. After base change to $\Hk^{\mu}_{S}$, it suffices to show that
\begin{equation*}
\xymatrix{ \Hk^{\mu}_{S}\ar[d]^{\wt z_{0}} & \Hk^{\mu}_{S} \ar@{=}[l]\ar[d]^{z}\\
\wt W_{0} & \Hk^{\mu}_{W}\ar[l]_{\wt c_{0}}
}
\end{equation*}
is pushable. As $\Id_{\Hk^{\mu}_{S}}$ is proper and $\wt c_{0}$ is separated by Lemma \ref{lem: hecke corr quasismooth tilde}, the pushability follows from \cite[Example 3.1.2]{FYZ3}. 
\item The square $(\Hk^{\mu}_{U}, U_{r}, \Hk^{\mu}_{S}, S_{r})$ is pullable. It suffices to check that
\begin{equation*}
\xymatrix{\Hk^{\mu}_{U} \ar[d]^{\pi}\ar[r]^{\wt a_{r}} & \wt U_{r}\ar[d]^{\wt \pi_{r}}\\
\Hk^{\mu}_{S} \ar@{=}[r]&\Hk^{\mu}_{S}}
\end{equation*}
is pullable. As $\wt a_{r}$ is quasi-smooth (by Lemma \ref{lem: hecke corr quasismooth tilde}) and $\Id$ is smooth, this follows from \cite[Example 3.1.3]{FYZ3}. 
\item The square $(\Hk^{\mu}_{V}, V_{r}, \Hk^{\mu}_{W}, W_{r})$ is pullable. It suffices to check that
\begin{equation*}
\xymatrix{\Hk^{\mu}_{V} \ar[d]^{g}\ar[r]^{\wt b_{r}} & \wt V_{r}\ar[d]^{\wt g_{r}}\\
\Hk^{\mu}_{W} \ar[r]^{\wt c_{r}} &\wt W_{r}
}
\end{equation*}
is pullable. As $\wt c_{r}$ and $\wt b_{r}$ are smooth by Lemma \ref{lem: hecke corr quasismooth tilde}, this follows from \cite[Example 3.1.3]{FYZ3}. 
\end{enumerate}
The analogous diagram with $'$ is treated identically, using \eqref{eq: derived cartesian U'V'W'S'} and \eqref{eq: Hk U'V'W' Cart} for the vertical derived Cartesian squares. The four side squares are checked by the same base-change reductions as in (a)--(d), with Lemma \ref{lem: hecke corr quasismooth tilde} applied to the $'$ maps; in the analogue of (d), one uses that $\wt b_r'$ is quasi-smooth and $\wt c_r'$ is smooth.
\end{proof}

\section{The proof of supermodularity}

Finally, in this section we complete the proof of Theorem \ref{thm: intro supermodularity} (supermodularity).

\subsection{Recollections on higher theta series}\label{ssec: higher theta series} The higher theta series
\[
\wt{Z}^{\psi, \mu}_{m_1,m_2} \co \Bun_{P_{(m_1, m_2)}}(\F_q) \rightarrow \CH_{\frac{r}{2}(2n-m)}(\Sht^\mu_n)
\]
is defined in \cite[\S 4.8]{FYZ2}\footnote{Our $E_1$ and $E_2$ here correspond to $E^{(1)}$ and $E^{(2)}$ of \cite[\S 4.8]{FYZ2}.}. Unlike in \cite{FYZ2}, we emphasize here the a priori dependence on the non-trivial additive character $\psi$, although we will shortly see that (in this general linear case) it is independent of the choice of $\psi$. The value of $\wt{Z}^{\psi, \mu}_{m_1,m_2}$ on a pair $(\cE_1, \cG) \in  \Bun_{P_{(m_1, m_2)}}(\F_q) $ is
\[
q^{n \deg \cE_2} \sum_{a \in \Hom(\cE_1, \cE_2^\vee)} \psi(\langle e_{\cG}, a \rangle) [\cZ^{\mu}_{\cE_1, \cE_2}(a)]
\]
where $e_{\cG} \in \Ext^1(\cE_2^*, \cE_1)$ is the extension class of $\cG$, paired with $a$ using Serre duality. 
When $m_2=0$, the Hom set in the displayed formula has the single element $0$, and the formula reduces to $[\cZ_{\cG,0}^{\mu}]$.

\begin{lemma}\label{lem: ind of psi}
For any two non-trivial additive characters $\psi, \psi' \co \F_q \rightarrow \ol{\Q}^\times$, we have 
\begin{equation}\label{eq: ind of psi}
\wt{Z}_{m_1, m_2}^{\psi, \mu} = \wt{Z}_{m_1, m_2}^{\psi', \mu}.
\end{equation}
\end{lemma}

\begin{proof} There are two scaling $\F_q^\times$-actions on $\cZ_{\cE_1, \cE_2}^\mu$: one scaling the map $t_1 \co \cE_1 \rightarrow \cF_i$, and the other scaling the map $t_2 \co \cF_i \rightarrow \cE_2^\vee$. The derived fundamental class $[\cZ_{\cE_1, \cE_2}^\mu] \in \CH(\cZ_{\cE_1, \cE_2}^\mu)$ is equivariant for each of these two actions; equivalently, it is pulled back from $\CH(\cZ_{\cE_1, \cE_2}^\mu/(\F_q^\times \times  \F_q^\times))$.

Fix the action scaling $t_1$, and set $\psi_s(x):=\psi(sx)$ for $s\in\F_q^\times$. If $p_a\co\cZ_{\cE_1,\cE_2}^{\mu}(a)\to\Sht_n^\mu$ is the projection, scaling $t_1$ induces an isomorphism $\cZ_{\cE_1,\cE_2}^{\mu}(a)
\xrightarrow{\sim}
\cZ_{\cE_1,\cE_2}^{\mu}(sa)$. Therefore, their pushforwards to $\CH_{\frac{r}{2}(2n-m)}(\Sht^\mu_n)$ satisfy $[\cZ_{\cE_1,\cE_2}^{\mu}(sa)]
=[\cZ_{\cE_1,\cE_2}^{\mu}(a)]$. Using the notation of \eqref{eq: formulation theta value} and reindexing by $b=sa$, we obtain
\[
\wt Z_{m_1,m_2}^{\psi_s,\mu}
=q^{n\deg\cE_2}\sum_b
\psi(\langle e_{\cG},b\rangle)
[\cZ_{\cE_1,\cE_2}^{\mu}(s^{-1}b)]
=\wt Z_{m_1,m_2}^{\psi,\mu}.
\]
Every two non-trivial additive characters of $\F_q$ are related by $\psi'=\psi_s$ for a unique $s\in\F_q^\times$, proving \eqref{eq: ind of psi}.
\end{proof}

The stack $\Sht_S^\mu$, defined as the fibered product 
\[
\begin{tikzcd}
\Sht_S^\mu \ar[r] \ar[d] & \Hk_S^\mu \ar[d] \\
S \ar[r, "\Id \times \Frob"] & S \times S
\end{tikzcd}
\]
is the open Harder--Narasimhan truncation of $\Sht_n^\mu$ on which every intermediate bundle $\cF_i$ lies in $S=\Bun_n^{\leq\nu}$.
As $\nu$ ranges over all HN polygons, the union of the $\Sht_S^\mu$ exhausts $\Sht_n^\mu$. Therefore, we have 
\begin{equation}\label{eq: gl HN Chow limit}
\CH_*(\Sht_n^\mu) = \limit_{\nu} \CH_*(\Sht_S^\mu).
\end{equation}
Thus, to prove Theorem \ref{thm: intro supermodularity}, it suffices to prove the restriction of the statement to $\CH(\Sht_S^\mu)$ for all $S$. The pullback of $\cZ^{\mu}_{\cE_1, \cE_2}$ to $\Sht_S^\mu$ is $\Sht_{V' \times_S U}^\mu$, the derived fibered product 
\[
\begin{tikzcd}
\Sht_{V' \times_S U}^\mu \ar[r] \ar[d] & \Hk_{V'}^\mu \times_{\Hk_S^\mu} \Hk_U^\mu \ar[d] \\
V' \times_S U \ar[r, "\Id \times \Frob"] & (V' \times_S U) \times (V' \times_S U)
\end{tikzcd}
\]

\subsection{Supermodularity in low corank} 

We first prove supermodularity in the low-corank range using the
following refinement of the Trace Conjecture in the split case.

Let $X'=X^{(1)}\coprod X^{(2)}$, with the involution exchanging the
two components, and identify a Hermitian bundle with its restriction to
$X^{(1)}$.  Under this convention, a leg on $X^{(2)}$ gives an upper
modification and has sign $\mu_i=+1$, whereas a leg on $X^{(1)}$
gives a lower modification and has sign $\mu_i=-1$.

\begin{thm}[Trace Conjecture in low corank, split case]\label{thm:split-trace-conjecture-n/3}
Let $\cE=(\cE_1,\cE_2)$ be a pair of vector bundles on $X$ of ranks $m_1,m_2$, such that $0\leq m_1\leq n/3$ and $0\leq m_2\leq n/3$. Let $\cM_{\cE_1,\cE_2}$ be the corresponding localized split Hitchin stack. For a sequence $\mu\in\{\pm1\}^r$ with
$\sum_i\mu_i=0$, put
\[
r_\pm:=\#\{i:\mu_i=\pm1\},
\qquad
d_\mu:=r_+(n-m_1)+r_-(n-m_2)
=\frac r2(2n-m_1-m_2).
\]
Let $\cc^\mu_{\cM_{\cE_1,\cE_2}}$ denote the restriction of the cohomological correspondence $\cc_{\cM_{\cE_1,\cE_2}}$, defined as in \S\ref{ssec: localized trace conjecture}, to the open-and-closed component of the leg space $(X')^r$ indexed by $\mu$.
Then
\begin{equation}\label{eq:split-refined-trace-conjecture}
\Tr^{\Sht}\bigl(\cc_{\cM_{\cE_1,\cE_2}}^\mu\bigr)
=
[\cZ_{\cE_1,\cE_2}^{\mu}]
\in
\CH_{d_\mu}\bigl(\cZ_{\cE_1,\cE_2}^{\mu}\bigr).
\end{equation}
\end{thm}

\begin{proof}
Fix $\mu$ and work on the corresponding open-and-closed component of
the leg space $(X')^r$.  The split Hitchin stack, its cohomological
correspondence, and its kernel decomposition are the split-cover
specializations of the Hermitian constructions used in Theorem
\ref{thm: trace conjecture E}. The proof of Theorem \ref{thm: trace conjecture E} respects this
component and goes through with no essential changes.
\end{proof}

In the application to Theorem \ref{thm: modularity in low corank}, $m=m_1+m_2\leq n/3$.  Therefore the
split theorem applies both to $(\cE_1,\cE_2)$ and to $(\cG,0)$, whose
rank pairs are $(m_1,m_2)$ and $(m,0)$, respectively.

\begin{thm}[Supermodularity in low corank]\label{thm: modularity in low corank} If $m \leq n/3$, we have 
\begin{equation}\label{eq: modularity in low corank}
\wt{Z}_{m_1,m_2}^{\psi,\mu}(\cE_1,\cG)
=
[\cZ_{\cG,0}^{\mu}]^{\shear}
=
\sum_{d\in\Z}q^{-m_2d+m_2n(g-1)}[\cZ_{\cG,0}^{\mu}]_d
\in \CH_{\frac{r}{2}(2n-m)}(\Sht_{n}^\mu).
\end{equation}
\end{thm}

\noindent\emph{Proof of Theorem \ref{thm: modularity in low corank}.}
We write $d(U/S)$, $d(V/S)$, $d(W/S)$, etc. for the virtual relative dimensions of the corresponding derived vector bundles over $S$. These integers are locally constant on connected components of $S$. In formulas involving the degree of the universal bundle, such as \eqref{eq: modularity RHS trace}, all shifts, Tate twists, and Chow degrees are read componentwise on the fixed-degree connected component under consideration.

\subsubsection{Theta sheaves}  Consider the commutative diagram 
\begin{equation}\label{eq: theta FT diagram}
\begin{tikzcd}
& V' \times_S U \ar[dd, "\Id \times f"] \ar[dr, "\pr_2"]  \\
U'\ar[dd, "f'"]  & & U\ar[dd, "f"] \\
& V' \times_S V \ar[dr, "\pr_2"']  \ar[dl, "\pr_1"]  \\
V'   \ar[dr, "\pi_{V'}"'] \ar[dd, "g'"]  & & V\ar[dl, "\pi_V"] \ar[dd, "g"]  \\
& S \\
W' & & W
\end{tikzcd}
\end{equation}
Let
\[
\cL_V:=\langle\, ,\,\rangle_V^*\AS_{\psi,S}\in \Dmot{V'\times_S V}
\]
be the Artin--Schreier kernel appearing in \eqref{eq:motivic-derived-FT}. We define 
\[
\Theta_{\cE_1, \cE_2} = \pi_{V'!} \pr_{1!} (\pr_2^* f_! \Qsh{U} \otimes \cL_V )  \in \Dmotg{S},
\]
and
\[
\Theta_{\cG, 0} = \pi_{V'!}  f'_! (\Qsh{U'}) \in \Dmotg{S}.
\]
Unwinding the definition of the motivic Fourier transform \eqref{eq:motivic-derived-FT}, we may rewrite 
\begin{equation}\label{eq: theta as FT}
\Theta_{\cE_1, \cE_2} =  \pi_{V'!}  \FT_V(f_! \Qsh{U})\sm{[-d(V/S)]}
\end{equation}
where $\FT_V$ is the motivic Fourier transform of \eqref{eq:motivic-derived-FT}.

\begin{constr}We will construct an isomorphism 
\begin{equation}\label{eq: theta sheaf isom}
\Theta_{\cE_1, \cE_2} \cong \Theta_{\cG,0}  \tw{-d(U/S)} \in \Dmotg{S}.
\end{equation}
Note that $f_! \Qsh{U} \cong g^* \delta_W$ by proper base change applied to the (derived) Cartesian square \eqref{eq: derived cartesian UVWS}. Then, using the Fourier Transform functorialities obtained from Theorem \ref{thm:motivic-FT-formalism} and Lemma \ref{lem: dual vector bundles}, we have
\begin{align*}
\Theta_{\cE_1, \cE_2}
&=\pi_{V'!}\FT_V(f_!\Qsh{U})\sm{[-d(V/S)]} \\
&\cong\pi_{V'!}\FT_V(g^*\delta_W)\sm{[-d(V/S)]} \\
&\cong\pi_{V'!}\wh g_!\FT_W(\delta_W)
  \tw{-d(V/S)}\sm{[d(W/S)](d(W/S))} \\
&\cong\pi_{V'!}f'_!\Qsh{U'}\tw{-d(V/S)+d(W/S)} \\
&=\Theta_{\cG,0}\tw{-d(U/S)}.
\end{align*}
\end{constr}

\subsubsection{Cohomological correspondences}

We define the following cohomological correspondences. 
\begin{enumerate}[(i)]
\item Since $h_r \co \Hk_S^\mu \rightarrow S$ is smooth, the map of correspondences  
\[
\begin{tikzcd}
S \ar[d] & \Hk_S^\mu \ar[l, "h_0"'] \ar[d] \ar[r, "h_r"] & S \ar[d] \\
\pt & \pt \ar[l] \ar[r] & \pt 
\end{tikzcd}
\]
is right pullable. We define $\cc_S \in \Corr_{\Hk_S^\mu}(\Qsh{S}, \Qsh{S}\tw{-d(h_r)})$ to be the pullback of the identity correspondence in $\Corr_{\pt}(\Qsh{\pt}, \Qsh{\pt})$. More explicitly, it is the composition 
\[
h_0^* \Qsh{S} \cong \Qsh{\Hk_S^\mu}\cong h_r^* \Qsh{S} \xrightarrow{\gys_{h_r}} h_r^!\Qsh{S}\tw{-d(h_r)}.
\]
\item As shown in part (c) of the proof of Lemma \ref{lem: cube side pushable/pullable}, the map of correspondences 
\[
\begin{tikzcd}
U \ar[d, "\pi_0"] & \Hk_U^\mu \ar[l, "a_0"'] \ar[d, "\pi"] \ar[r, "a_r"] & U \ar[d, "\pi_r"] \\
S  & \Hk_S^\mu \ar[l, "h_0"'] \ar[r, "h_r"] & S 
\end{tikzcd}
\]
is right pullable. We define $\cc_U = \pi^* \cc_S \in \Corr_{\Hk_U^\mu}(\Qsh{U}, \Qsh{U}\tw{-d(a_r)})$. Note that we could also have defined $\cc_U$ as the pullback from the identity correspondence in $\Corr_{\pt}(\Qsh{\pt}, \Qsh{\pt})$. Explicitly, $\cc_U$ is the composition 
\[
a_0^* \Qsh{U} \cong \Qsh{\Hk_U^\mu}\cong a_r^* \Qsh{U} \xrightarrow{\gys_{a_r}} a_r^!\Qsh{U} \tw{-d(a_r)}. 
\]
\item The analogous construction to (ii) defines $\cc_{U'} \in \Corr_{\Hk_{U'}^\mu}(\Qsh{U'}, \Qsh{U'}\tw{-d(a_r')})$.
\item As shown in part (b) of the proof of Lemma \ref{lem: cube side pushable/pullable}, the map of correspondences 
\[
\begin{tikzcd}
S \ar[d, "z_0"] & \Hk_S^\mu \ar[l, "h_0"'] \ar[r, "h_r"] \ar[d, "z"] & S  \ar[d,"z_r"] \\
W  & \Hk_W^\mu \ar[l, "c_0"']  \ar[r, "c_r"] & W
\end{tikzcd}
\]
is left pushable. We define $\cc_W = z_! \cc_S \in \Corr_{\Hk_W^\mu}(\delta_W, \delta_W \tw{-d(h_r)})$. 

\end{enumerate}

\begin{defn}[Twist and shift of cohomological correspondences]
For a cohomological correspondence $\cc \in \Corr_C(\cK_0, \cK_1)$ and $c \in \Z$, we abbreviate
\[
\TT_{\tw{c}} \cc := \TT_{\sm{[2c](c)}} \cc \in \Corr_C(\cK_0\tw{c}, \cK_1 \tw{c})
\]
where the notation is as in \cite[\S 6.5]{FK}. 
\end{defn}

\begin{prop}\label{prop: modularity of cc} 
(1) With respect to the isomorphism $f_! \Qsh{U} \cong g^* \delta_W$ coming from proper base change for the Cartesian square \eqref{eq: derived cartesian UVWS}, we have 
\begin{equation}\label{eq: modularity of cc (1)} 
f_! \cc_U =  g^* \cc_W \in \Corr_{\Hk_V^\mu}(f_! \Qsh{U}, f_! \Qsh{U} \tw{-d(a_r)} ).
\end{equation}
(2) We have 
\begin{equation}\label{eq: modularity of cc (2)} 
\begin{aligned}
\FT(f_! \cc_U)
&\cong \TT_{\sm{[2d(W/S)-d(V/S)](d(W/S)-d(V/S))}}
f'_!(\cc_{U'}) \\
&\in \Corr_{\Hk_{V'}^\mu}
\bigl(\FT_V(f_!\Qsh{U}),\FT_V(f_!\Qsh{U})\tw{-d(a_r')}\bigr).
\end{aligned}
\end{equation}
\end{prop}

\begin{proof} (1) Thanks to Lemma \ref{lem: cube side pushable/pullable}, we may apply \cite[Theorem 4.4.2]{FK}, in the middle isomorphism below: 
\[
f_! \cc_U = f_! \pi^* \cc_S \cong g^* z_! \cc_S = g^* \cc_W.
\]

(2) By the global-presentation statement preceding Lemma \ref{lem: hecke corr quasismooth tilde}, Proposition \ref{prop:motivic-derived-FT-cc-functoriality} applies to these diagrams. Lemma \ref{lem: cube side pushable/pullable}(b), together with Lemmas \ref{lem: dual vector bundles} and \ref{lem: dual Hecke stacks}, identifies the dual map with $\pi'=\wh z$ and gives
\begin{equation}\label{eq: FT of delta cc}
\FT (\cc_W)  = \FT ( z_! \cc_S) \cong \TT_{\sm{[d(W/S)]}} (\pi')^* \FT_0 (\cc_S)  =  \TT_{\sm{[d(W/S)]}} (\pi')^* \cc_S = \TT_{\sm{[d(W/S)]}} \cc_{U'},
\end{equation}
where the isomorphism is the left-pushable case of Proposition \ref{prop:motivic-derived-FT-cc-functoriality}. We used here that $\FT_0(\cc_S)=\cc_S$ because the zero derived vector bundle is self-dual.

Lemma \ref{lem: cube side pushable/pullable}(d) gives right pullability for $g$. The right-pullable case of Proposition \ref{prop:motivic-derived-FT-cc-functoriality}, together with (1), gives
\begin{align*}
\FT(f_! \cc_U) & \stackrel{\eqref{eq: modularity of cc (1)}}\cong \FT(g^* \cc_W) =  \TT_{\sm{[d(W/S)-d(V/S)](d(W/S)-d(V/S))}}\wh{g}_! \FT (\cc_W) \\
& \stackrel{\eqref{eq: FT of delta cc}}= \TT_{\sm{[2d(W/S)-d(V/S)](d(W/S)-d(V/S))}} f'_!   (\cc_{U'}),
\end{align*}
as desired.
\end{proof}

\subsubsection{Supermodularity} We will prove \eqref{eq: modularity in low corank} after restriction to $\CH_*(\Sht_S^\mu)$; as explained in \S \ref{ssec: higher theta series}, this suffices to prove Theorem \ref{thm: modularity in low corank}.
We write $(\ldots)|_S$ for base change of all projections along $S\inj\Bun_n$: for Hecke correspondences it means restriction to $\Hk_S^\mu=\bigcap_{i=0}^r h_i^{-1}(S)$, and for shtuka stacks and Chow classes it means restriction to $\Sht_S^\mu$.

By Lemma \ref{lem: hecke corr quasismooth tilde}, the map of correspondences
\[
\begin{tikzcd}
V'_0 \ar[d] & \Hk_{V'}^\mu \ar[d, "\pi_{V'}"]  \ar[l, "b'_0"'] \ar[r,"b'_r"]  & V'_r   \ar[d] \\
S_0 & \Hk_S^\mu \ar[l, "h_0"'] \ar[r, "h_r"] & S_r
\end{tikzcd}
\]
is pushable.  It also satisfies the remaining hypotheses of Proposition
\ref{prop: trace of Hitchin pushforward}.  Indeed, the complexes
$\cV'=\ul{\RHom(\cE_1,\cF_{\univ})}$ and
$(\cV')_{\Hk}^\mu=
\ul{\RHom(\cE_1,\cF_{\univ,\bu}^\flat)}$ are coconnective.
The correspondence maps have fiber
$\ul{\RHom(\cE_1,\cQ_{i,\univ})}[-1]$, which has tor-amplitude
$[1,\infty)$, and Lemma \ref{lem: hecke corr quasismooth tilde}
identifies $\wt b_i'$ as the resulting closed embedding of derived vector
bundles.  The relevant coefficient objects are obtained from unit objects
by pullback, tensor product, and pushforward along schematic maps locally of
finite type, so they are geometric by \cite[Remark 3.3.2]{FK}.
Proposition \ref{prop: trace of Hitchin pushforward} therefore implies that
pushforward along $\pi_{V'}$ is compatible with formation of $\Tr^{\Sht}$.

With respect to the identification $\Theta_{\cE_1, \cE_2} \cong \Theta_{\cG,0}\tw{-d(U/S)}$ from \eqref{eq: theta sheaf isom}, Proposition \ref{prop: modularity of cc}(2), together with the shift in \eqref{eq: theta as FT}, gives an identification
\begin{equation}\label{eq: modularity cc prop}
\TT_{\sm{[-d(V/S)]}}\pi_{V'!} \FT( f_! \cc_U) \cong \TT_{\tw{-d(U/S)}} \pi_{V'!} f'_!(\cc_{U'})  \in \Corr_{\Hk_S^\mu}(\Theta_{\cE_1, \cE_2}, \Theta_{\cE_1, \cE_2}\tw{-d(a_r')}).
\end{equation}
Indeed, Proposition \ref{prop: modularity of cc}(2) gives
\[
\pi_{V'!}\FT(f_!\cc_U)\cong
\TT_{\sm{[d(V/S)]}}\TT_{\tw{-d(U/S)}}\pi_{V'!}f'_!(\cc_{U'}).
\]
Recall from \eqref{eq: theta as FT} that $\Theta_{\cE_1,\cE_2}=\pi_{V'!}\FT_V(f_!\Qsh U)\sm{[-d(V/S)]}$. 

Now we invoke the \emph{motivic sheaf-cycle correspondence} of \cite[\S 6]{FK}. It is clear from \cite[Remark 3.3.2]{FK} that $\Theta_{\cE_1, \cE_2} \cong \Theta_{\cG,0}\tw{-d(U/S)}$ lies in $\Dmotg{S} \subset \Dmot{S}$, so we may form the (shtuka-twisted) \emph{trace} of \eqref{eq: modularity cc prop} in the sense of \cite[\S 5.4]{FK}. On the RHS, Theorem
\ref{thm:split-trace-conjecture-n/3} gives the final equality in
\begin{align}\label{eq: modularity RHS trace}
q^{n\deg\cE_2}\Tr^{\Sht}(\TT_{\tw{-d(U/S)}} \pi_{V'!} f'_! (\cc_{U'})) &\stackrel{\text{Prop \ref{prop: trace of Hitchin pushforward}}}{=}  q^{n\deg\cE_2}\Sht(\pi_{V'} f')_!  \Tr^{\Sht}(\TT_{\tw{-d(U/S)}} \cc_{U'}) \nonumber \\
&=  \sum_{d \in \Z} q^{- m_2d + m_2 n (g-1)}  [\cZ_{\cG, 0}^{\mu}]_d |_S
\end{align}
Here $[\cZ_{\cG,0}^{\mu}]_d$ is the projection of $[\cZ_{\cG,0}^{\mu}]$ to the component of $\Sht_n^\mu$ consisting of shtukas with degree $d \in \Z$. On that component, the trace identity of \cite[\S 6.5]{FK} gives
\[
\Tr^{\Sht}(\TT_{\tw{-d(U/S)}}\cc)=q^{d(U/S)}\Tr^{\Sht}(\cc),
\]
because $\tw{-d(U/S)}=[-2d(U/S)](-d(U/S))$. Using Theorem \ref{thm:split-trace-conjecture-n/3} to compute the shtuka-twisted trace of $\cc_{U'}$ and the calculation from \S \ref{ssec: r=0}, we obtain 
\[
d(U/S) = \chi_{\mathrm{Eul}}(X,\RHom(\cF, \cE_2^\vee))  = -n\deg\cE_2 - m_2 \deg \cF +  m_2n(g-1),
\]
so that
$q^{n\deg\cE_2}q^{d(U/S)}
=q^{-m_2d+m_2n(g-1)}$.
Finally, the virtual dimension calculation gives
$d(a_r')=\frac r2(2n-m)$, so every term lies in the Chow degree asserted in \eqref{eq: modularity in low corank}.

On the left side of \eqref{eq: modularity cc prop}, referring to the notation of \eqref{eq: theta FT diagram}, we will show that 
\begin{equation}\label{eq: modularity LHS trace}
q^{n\deg\cE_2}\Tr^{\Sht}(\TT_{\sm{[-d(V/S)]}}\pi_{V'!} \FT(f_! \cc_U) ) = \wt{Z}_{m_1, m_2}^{\psi, \mu}(\cE_1, \cG)|_S.
\end{equation}
\vspace{-6pt}
To begin, we have
\begin{align*}
\Tr^{\Sht}(\TT_{\sm{[-d(V/S)]}}\pi_{V'!} \FT(f_! \cc_U) ) &   = \Tr^{\Sht}(\pi_{V'!} \pr_{1!} (\pr_2^* f_!  \cc_U \otimes \cL_V))\\
& \stackrel{(1)}=  \Tr^{\Sht}(\pi_{V'!} \pr_{1!}  (\Id \times f)_! (\pr_2^* \cc_U \otimes  (\Id \times f)^*\cL_V)) \\
& \stackrel{(2)}= \Sht(\pi_{V'}  \pr_1  (\Id \times f) )_!  \Tr^{\Sht}(\pr_2^* \cc_U \otimes  (\Id \times f)^*\cL_V)
\end{align*}
where the steps are justified as follows. 
\begin{enumerate}
\item This follows from the Base Change Theorem for cohomological correspondences,
\cite[Theorem 4.4.2]{FK}, whose hypotheses are justified by Lemma \ref{lem: cube side pushable/pullable}. 
\item This follows from  Proposition \ref{prop: trace of Hitchin pushforward} applied to the
composite projection from $V'\times_S U$, whose hypotheses are satisfied by Lemma
\ref{lem: hecke corr quasismooth tilde}, and \eqref{eq:motivic-AS-sheaf}.

%\item is an instance of the compatibility of pullbacks of cohomological correspondences with compositions, \cite[\S 4.2.6]{FK}. 
\end{enumerate}

We claim that 
\begin{equation}\label{eq: modularity claim 1}
q^{n\deg\cE_2}\Sht(\pi_{V'}  \pr_1  (\Id \times f) )_!  \Tr^{\Sht}(\pr_2^* \cc_U \otimes  (\Id \times f)^*\cL_V)  =  \wt{Z}_{m_1, m_2}^{\psi, \mu}(\cE_1, \cG)|_S,
\end{equation}
which would complete the proof of \eqref{eq: modularity LHS trace}.
First, write $\cM:=\cM_{\cE_1,\cE_2}$ for the localized Hitchin stack of Theorem \ref{thm:split-trace-conjecture-n/3}, so that $V'\times_S U=\cM|_S$.  Both $\pr_2^*\cc_U$ and $\cc_{\cM}$ are obtained by pulling back the identity correspondence on the point, so $\pr_2^*\cc_U=\cc_{\cM}$ by transitivity of pullback of cohomological correspondences.
Since $m_1,m_2\leq m\leq n/3$,
\eqref{eq:split-refined-trace-conjecture} gives
\begin{equation}\label{eq: modularity claim 1.5}
\Tr^{\Sht}(\pr_2^* \cc_U) = \Tr^{\Sht}(\cc_{\cM}) =  [\cZ_{\cE_1, \cE_2}^\mu]|_S \in \CH_*(\cZ_{\cE_1, \cE_2}^\mu|_S).
\end{equation}
The map of correspondences, in which the vertical maps are the composite of the projection to the Hitchin base $\Hom(\cE_1,\cE_2^\vee)$ with the linear functional $\langle e_{\cG},-\rangle$, 
\begin{equation}\label{eq: ev map of correspondences}
\begin{tikzcd}
\cM \ar[d] & \Hk_{\cM}^\mu \ar[r] \ar[l] \ar[d]  &  \cM \ar[d] \\
\A^1 & \ar[l, equals] \A^1 \ar[r, equals] & \A^1  
\end{tikzcd}
\end{equation}
induces after applying the $\Sht$ construction a map $\cZ_{\cE_1, \cE_2}^\mu \xrightarrow{\ev} \F_q$, which in turn induces a decomposition 
\[
\cZ_{\cE_1, \cE_2}^{\mu} = \coprod_{\lambda \in \F_q}  \underbrace{(\cZ_{\cE_1, \cE_2}^{\mu} )_\lambda}_{ \ev^{-1}(\lambda)}. 
\]
By inspection of the definitions, the map $\ev$ coincides with the pairing of $e_{\cG} \in \Ext^1(\cE_2^*, \cE_1)$ with the map from $\cZ_{\cE_1, \cE_2}^\mu$ to the Hitchin base $\Hom(\cE_1, \cE_2^\vee)$. Therefore, we may rewrite the higher theta series as 
\begin{equation}\label{eq: modularity claim 3}
\wt{Z}_{m_1, m_2}^{\psi, \mu}(\cE_1, \cG) = q^{n\deg\cE_2}\sum_{a \in \Hom(\cE_1, \cE_2^\vee)} \psi(\langle e_{\cG}, a \rangle)[\cZ_{\cE_1, \cE_2}^{\mu}(a)] = q^{n\deg\cE_2}\sum_{\lambda \in \F_q} \psi(\lambda)  [\cZ_{\cE_1, \cE_2}^{\mu} ]_\lambda 
\end{equation}
where $[\cZ_{\cE_1, \cE_2}^{\mu} ]_\lambda $ is the projection of $[\cZ_{\cE_1, \cE_2}^{\mu} ]$ to $\CH_*((\cZ_{\cE_1, \cE_2}^{\mu})_\lambda)$\footnote{This projection is the same as the derived fundamental class of $(\cZ_{\cE_1, \cE_2}^{\mu} )_\lambda$.} followed by pushforward to $\Sht_n^\mu$. After restricting to $S$, this pushforward is what was denoted $\Sht(\pi_{V'}  \pr_1  (\Id \times f) )_!$ in \eqref{eq: modularity claim 1}, but we suppress it below for ease of notation. 

%The scaling $\G_m$-actions on $\Hk_{V'}^\mu$ and on $\Hk_U^\mu$, as derived vector bundles over $\Hk_S^\mu$, induce two scaling $\F_q^\times$-action on $\cZ_{\cE_1, \cE_2}^\mu$. 

For $\lambda \in \F_q \subset \A^1_{\F_q}$, write $\cM_\lambda$, etc. for the derived fiber over $\lambda$, and $\iota_\lambda \co \cM_\lambda \inj \cM$ for the tautological closed embedding. By the commutativity of \eqref{eq: ev map of correspondences}, the correspondence $\cM \leftarrow \Hk_{\cM}^\mu \rightarrow \cM$ \emph{stabilizes} $\cM_\lambda$ in the sense of \cite[Definition 7.2.1]{FK}, hence its Frobenius twist is \emph{contracting} near $\cM_{\lambda}$ \cite[Example 7.2.2]{FK}. Hence \cite[Construction 7.2.3]{FK} applies, to give a restriction of (cohomological) correspondences 
\[
\iota_{\lambda}^* \co \Corr_{\Hk_{\cM}^\mu} (\cK_0, \cK_1) \rightarrow \Corr_{\Hk_{\cM_\lambda}^\mu}(\iota_\lambda^* \cK_0, \iota_\lambda^* \cK_1).
\]
By \cite[Theorem 7.5.1]{FK}, this restriction is compatible with the formation of $\Tr^{\Sht}$, so that we have 
\begin{equation}\label{eq: modularity claim 4}
\Tr^{\Sht}(\iota_\lambda^* \cc_{\cM} \otimes \iota_\lambda^*(\Id \times f)^*\cL_V)  = \Sht(\iota_\lambda)^* \Tr^{\Sht}(\cc_{\cM} \otimes (\Id \times f)^*\cL_V)  = \psi(\lambda)  [\cZ_{\cE_1, \cE_2}^{\mu} ]_\lambda |_S
\end{equation}
Here, the first equality is the compatibility \cite[Theorem 7.5.1]{FK}. For the second equality, note that $\iota_\lambda^*(\Id\times f)^*\cL_V$ is the constant invertible object with Frobenius trace $\psi(\lambda)$ (the restriction of the Artin--Schreier kernel along the constant map $\lambda$), so the left side equals $\psi(\lambda)\cdot\Sht(\iota_\lambda)^*\Tr^{\Sht}(\cc_{\cM})$, which is $\psi(\lambda)[\cZ_{\cE_1,\cE_2}^{\mu}]_\lambda|_S$ by \eqref{eq: modularity claim 1.5} and \cite[Theorem 7.5.1]{FK} applied to $\cc_\cM$. Summing \eqref{eq: modularity claim 4} over all $\lambda \in \F_q$, we find that 
\[
\Tr^{\Sht}(\cc_{\cM} \otimes (\Id \times f)^*\cL_V)  = \sum_{\lambda \in \F_q} \psi(\lambda) [\cZ_{\cE_1, \cE_2}^{\mu} ]_\lambda|_S,
\]
which, in combination with \eqref{eq: modularity claim 3} and the normalizing factor $q^{n\deg\cE_2}$ from \eqref{eq: formulation theta value}, finally establishes \eqref{eq: modularity claim 1}, hence also \eqref{eq: modularity LHS trace}.

Putting things together, we have found
\begin{align*}
\wt{Z}_{m_1, m_2}^{\psi,\mu}(\cE_1, \cG)|_S & \stackrel{\eqref{eq: modularity LHS trace}}=  q^{n\deg\cE_2}\Tr^{\Sht}(\TT_{\sm{[-d(V/S)]}}\pi_{V'!} \FT (f_! \cc_U)) \\
&\stackrel{\eqref{eq: modularity cc prop}}=  q^{n\deg\cE_2}\Tr^{\Sht}(\TT_{\tw{-d(U/S)}} \pi_{V'!}  f'_! (\cc_{U'})) \\
& \stackrel{\eqref{eq: modularity RHS trace}} =  \sum_{d \in \Z} q^{- m_2d + m_2 n (g-1)}  [\cZ_{\cG, 0}^{\mu}]_d|_S
\end{align*}
as desired.  \qed 

\subsection{Embedding and cancellation in all coranks}\label{ssec: gl all corank reduction}

We now pass from the low-corank supermodularity identity to the full statement by the same embedding-and-cancellation mechanism used in \S \ref{ssec: unitary all corank reduction}.  Since the argument compares shtukas of different ranks, we temporarily write the target rank as a superscript:
\[
\wt Z_{m_1,m_2}^{N,\mu}(\cE_1,\cG)
\in \CH_*(\Sht_N^\mu),
\qquad
[\cZ_{\cG,0}^{N,\mu}]\in \CH_*(\Sht_N^\mu).
\]
For a cycle $Z\in \CH_*(\Sht_N^\mu)$, define
\begin{equation}\label{eq: gl target-rank shear}
Z^{\shear_N}
:=
\sum_{d\in\Z}q^{-m_2d+m_2N(g-1)}Z_d,
\end{equation}
where $Z_d$ is the component on which the initial rank $N$ bundle has degree $d$, and the sum is interpreted componentwise as in Theorem \ref{thm: intro supermodularity}.  Thus $\shear_n$ is the operator denoted $\shear$ in \eqref{eq: formulation supermodularity}.

\begin{lemma}[Direct-sum factorization]\label{lem: gl direct sum factorization}
Let $N_1,N_2\ge m$.  The direct-sum morphism
\begin{equation}\label{eq: gl direct-sum add map}
\add_{N_1,N_2}^{\mu}\co
\Sht_{N_1}^{0}\times\Sht_{N_2}^{\mu}
\longrightarrow
\Sht_{N_1+N_2}^{\mu}
\end{equation}
is LCI (both source and target being smooth Deligne--Mumford stacks, cf.\ \cite[\S 5]{YZ}), so that the Gysin pullback $(\add_{N_1,N_2}^{\mu})^*$ of \eqref{eq: Gysin pullback} exists on Chow groups. For every $(\cE_1,\cG)\in\Bun_{P_{(m_1,m_2)}}(\F_q)$, one has
\begin{equation}\label{eq: gl theta direct-sum factorization}
(\add_{N_1,N_2}^{\mu})^*
\wt Z_{m_1,m_2}^{N_1+N_2,\mu}(\cE_1,\cG)
=
\wt Z_{m_1,m_2}^{N_1,0}(\cE_1,\cG)
\bt
\wt Z_{m_1,m_2}^{N_2,\mu}(\cE_1,\cG).
\end{equation}
Moreover, the sheared special cycles attached to $(\cG,0)$ satisfy
\begin{equation}\label{eq: gl sheared-special-cycle-factorization}
(\add_{N_1,N_2}^{\mu})^*
\bigl[\cZ_{\cG,0}^{N_1+N_2,\mu}\bigr]^{\shear_{N_1+N_2}}
=
\bigl[\cZ_{\cG,0}^{N_1,0}\bigr]^{\shear_{N_1}}
\bt
\bigl[\cZ_{\cG,0}^{N_2,\mu}\bigr]^{\shear_{N_2}}.
\end{equation}
\end{lemma}

\begin{proof}
The dimension formula for the smooth stacks in \cite[\S 5]{YZ} gives
$\dim\Sht_N^\mu=rN$ and $\dim\Sht_{N_1}^0=0$.  Hence
\[
d(\add_{N_1,N_2}^\mu)=-rN_1.
\]
In particular, the Gysin pullback sends the target degree
$\frac r2(2(N_1+N_2)-m)$ to
\[
\frac r2(2(N_1+N_2)-m)-rN_1
=\frac r2(2N_2-m),
\]
the degree of either product class in
\eqref{eq: gl theta direct-sum factorization} and
\eqref{eq: gl sheared-special-cycle-factorization}.
Write $\cZ_{\cE_1,\cE_2}^{N,\mu}(a)$ when the target rank is $N$. Functoriality of derived special cycles in the target variable, \cite[Proposition 7.5]{FYZ2}, and the identity $a(t^{(1)}\oplus t^{(2)})=a(t^{(1)})+a(t^{(2)})$ give
\[
(\add_{N_1,N_2}^{\mu})^*
[\cZ_{\cE_1,\cE_2}^{N_1+N_2,\mu}(a)]
=\sum_{a_1+a_2=a}
[\cZ_{\cE_1,\cE_2}^{N_1,0}(a_1)]\bt
[\cZ_{\cE_1,\cE_2}^{N_2,\mu}(a_2)].
\]
Multiplying by $\psi(\langle e_{\cG},a\rangle)$, summing over $a$, and using $q^{(N_1+N_2)\deg\cE_2}=q^{N_1\deg\cE_2}q^{N_2\deg\cE_2}$ proves \eqref{eq: gl theta direct-sum factorization}.

The same derived pullback for $(\cG,0)$, decomposed by the degrees of the two summands, gives, for every $d$,
\[
(\add_{N_1,N_2}^{\mu})^*
[\cZ_{\cG,0}^{N_1+N_2,\mu}]_d
=
\sum_{d_1+d_2=d}
[\cZ_{\cG,0}^{N_1,0}]_{d_1}\bt
[\cZ_{\cG,0}^{N_2,\mu}]_{d_2}.
\]
Equation \eqref{eq: gl sheared-special-cycle-factorization} follows after summing over $d_1,d_2$, since the exponent in \eqref{eq: gl target-rank shear} is additive under $(N,d)=(N_1+N_2,d_1+d_2)$.
\end{proof}

\begin{lemma}\label{lem: gl zero leg nonvanishing}
Assume $m\le n$ and fix $(\cE_1,\cG)\in\Bun_{P_{(m_1,m_2)}}(\F_q)$.  Then there is a point $\cF_0\in\Sht_{2n}^{0}(k)$ such that
\[
\wt Z_{m_1,m_2}^{2n,0}(\cE_1,\cG)(\cF_0)\ne0.
\]
\end{lemma}

\begin{proof}
Note that $\Sht_{2n}^{0}(k) \simeq \Bun_{2n}(k)$ is always non-empty. For any $\cF_0\in\Bun_{2n}(k)$,
the calculation of \S\ref{ssec: r=0}, applied with target rank $2n$, gives
\[
\wt Z_{m_1,m_2}^{2n,0}(\cE_1,\cG)(\cF_0)
= q^{-m_2\deg\cF_0+2m_2n(g-1)}\,\#\Hom(\cG,\cF_0)\ne 0,
\]
as in \eqref{eq: r=0 theta 6}: the value is a positive integer times a power of $q$.
\end{proof}

\begin{thm}[Supermodularity]\label{thm: gl full supermodularity}
For all $m\le n$, all decompositions $m=m_1+m_2$, all $\mu$ with entries in $\{\pm1\}$ and $\sum_i\mu_i=0$, and all $(\cE_1,\cG)\in\Bun_{P_{(m_1,m_2)}}(\F_q)$, one has
\begin{equation}\label{eq: gl full supermodularity}
\wt Z_{m_1,m_2}^{\mu}(\cE_1,\cG)
=
\bigl[\cZ_{\cG,0}^{\mu}\bigr]^{\shear}
=
\sum_{d\in\Z}q^{-m_2d+m_2n(g-1)}
[\cZ_{\cG,0}^{\mu}]_d.
\end{equation}
\end{thm}

\begin{proof}
If $m_2=0$, the identity is tautological from \eqref{eq: formulation theta value} and \eqref{eq: formulation supermodularity}.  We therefore assume $m_2>0$.  Since $m\le n=(3n)/3$, the low-corank supermodularity identity proved in Theorem \ref{thm: modularity in low corank} gives
\begin{equation}\label{eq: gl rank-3n-supermodularity}
\wt Z_{m_1,m_2}^{3n,\mu}(\cE_1,\cG)
=
\bigl[\cZ_{\cG,0}^{3n,\mu}\bigr]^{\shear_{3n}}
\quad\text{in}\quad
\CH_*(\Sht_{3n}^{\mu}).
\end{equation}
Pulling \eqref{eq: gl rank-3n-supermodularity} back along
$\add_{2n,n}^{\mu}\co
\Sht_{2n}^{0}\times\Sht_n^\mu
\longrightarrow
\Sht_{3n}^{\mu}$ 
gives, using Lemma \ref{lem: gl direct sum factorization},
\begin{equation}\label{eq: gl tensor equality before cancellation}
\wt Z_{m_1,m_2}^{2n,0}(\cE_1,\cG)
\bt
\wt Z_{m_1,m_2}^{n,\mu}(\cE_1,\cG)
=
\bigl[\cZ_{\cG,0}^{2n,0}\bigr]^{\shear_{2n}}
\bt
\bigl[\cZ_{\cG,0}^{n,\mu}\bigr]^{\shear_n}.
\end{equation}
The $r=0$ calculation in \S \ref{ssec: r=0}, applied with target rank $2n$, gives
\[
\wt Z_{m_1,m_2}^{2n,0}(\cE_1,\cG)
=
\bigl[\cZ_{\cG,0}^{2n,0}\bigr]^{\shear_{2n}}.
\]
Substituting this identity into \eqref{eq: gl tensor equality before cancellation} yields
\begin{equation}\label{eq: gl common-factor equality}
\wt Z_{m_1,m_2}^{2n,0}(\cE_1,\cG)
\bt
\left(
\wt Z_{m_1,m_2}^{n,\mu}(\cE_1,\cG)
-
\bigl[\cZ_{\cG,0}^{n,\mu}\bigr]^{\shear_n}
\right)
=0
\end{equation}
in $\CH_*(\Sht_{2n}^{0}\times\Sht_n^\mu)$.

By Lemma \ref{lem: gl zero leg nonvanishing}, there is a point $\cF_0\in\Sht_{2n}^{0}(k)$ such that
\begin{equation}\label{eq: gl nonvanishing}
\wt Z_{m_1,m_2}^{2n,0}(\cE_1,\cG)(\cF_0)\ne0.
\end{equation}
Pulling back along $\{\cF_0\}\times\Sht_{n}^\mu\to\Sht_{2n}^{0}\times\Sht_{n}^\mu$ gives
\[
\wt Z_{m_1,m_2}^{2n,0}(\cE_1,\cG)(\cF_0)
\left(
\wt Z_{m_1,m_2}^{n,\mu}(\cE_1,\cG)
-
\bigl[\cZ_{\cG,0}^{n,\mu}\bigr]^{\shear_n}
\right)
=0.
\]
The scalar in \eqref{eq: gl nonvanishing} is nonzero, so
\[
\wt Z_{m_1,m_2}^{n,\mu}(\cE_1,\cG)
=
\bigl[\cZ_{\cG,0}^{n,\mu}\bigr]^{\shear_n},
\]
completing the proof.
\end{proof}

\bibliographystyle{amsalpha}
\bibliography{Bibliography}

@article {BHKRY1,
    AUTHOR = {Bruinier, Jan H. and Howard, Benjamin and Kudla, Stephen S.
              and Rapoport, Michael and Yang, Tonghai},
     TITLE = {Modularity of generating series of divisors on unitary
              {S}himura varieties},
   JOURNAL = {Ast\'{e}risque},
  FJOURNAL = {Ast\'{e}risque},
    NUMBER = {421, Diviseurs arithm\'{e}tiques sur les vari\'{e}t\'{e}s orthogonales
              et unitaires de Shimura},
      YEAR = {2020},
     PAGES = {7--125},
      ISSN = {0303-1179},
      ISBN = {978-2-85629-927-2},
   MRCLASS = {14G35 (11F27 11F55 11G18 14G40)},
       DOI = {10.24033/ast},
       URL = {https://doi.org/10.24033/ast},
}

@article {Bor99,
    AUTHOR = {Borcherds, Richard E.},
     TITLE = {The {G}ross-{K}ohnen-{Z}agier theorem in higher dimensions},
   JOURNAL = {Duke Math. J.},
  FJOURNAL = {Duke Mathematical Journal},
    VOLUME = {97},
      YEAR = {1999},
    NUMBER = {2},
     PAGES = {219--233},
      ISSN = {0012-7094},
   MRCLASS = {11F55 (11F30 11F50 11G18)},
MRREVIEWER = {Rainer Schulze-Pillot},
       DOI = {10.1215/S0012-7094-99-09710-7},
       URL = {https://doi.org/10.1215/S0012-7094-99-09710-7},
}

@article {BWR15,
    AUTHOR = {Bruinier, Jan Hendrik and Westerholt-Raum, Martin},
     TITLE = {Kudla's modularity conjecture and formal {F}ourier-{J}acobi
              series},
   JOURNAL = {Forum Math. Pi},
  FJOURNAL = {Forum of Mathematics. Pi},
    VOLUME = {3},
      YEAR = {2015},
     PAGES = {e7, 30},
   MRCLASS = {11F46 (14C25)},
MRREVIEWER = {Haigang Zhou},
       DOI = {10.1017/fmp.2015.6},
       URL = {https://doi.org/10.1017/fmp.2015.6},
}

@misc{BR25,
      title={Correction to ``{K}udla's {M}odularity {C}onjecture and {F}ormal {F}ourier-{J}acobi {S}eries''},
      author={Bruinier, Jan Hendrik and Raum, Martin},
      year={2025},
      eprint={2510.05031},
      archivePrefix={arXiv},
      primaryClass={math.NT}
}

@misc{BRZ24,
      title={Modularity of special {$0$}-cycles on toroidal compactifications of {S}himura varieties},
      author={Bruinier, Jan Hendrik and Rosu, Eugenia and Zemel, Shaul},
      year={2024},
      eprint={2404.06254},
      archivePrefix={arXiv},
      primaryClass={math.NT}
}

@book {CD19,
    AUTHOR = {Cisinski, Denis-Charles and D\'{e}glise, Fr\'{e}d\'{e}ric},
     TITLE = {Triangulated categories of mixed motives},
    SERIES = {Springer Monographs in Mathematics},
 PUBLISHER = {Springer, Cham},
      YEAR = {[2019] \copyright 2019},
     PAGES = {xlii+406},
      ISBN = {978-3-030-33241-9; 978-3-030-33242-6},
   MRCLASS = {14F42 (14C15 14C35 18G80 19D55)},
MRREVIEWER = {Igor A. Rapinchuk},
       DOI = {10.1007/978-3-030-33242-6},
       URL = {https://doi.org/10.1007/978-3-030-33242-6},
}

@misc{Chen1,
      title={Co-rank 1 Arithmetic Siegel--Weil I: Local non-Archimedean},
      author={Chen, Ryan},
      year={2024},
      eprint={2405.01426},
      archivePrefix={arXiv},
      primaryClass={math.NT}
}

@misc{Chen2,
      title={Co-rank 1 Arithmetic Siegel--Weil III: Geometric local-to-global},
      author={Chen, Ryan},
      year={2024},
      eprint={2405.01428},
      archivePrefix={arXiv},
      primaryClass={math.NT}
}

@misc{FK,
      title={Modularity of higher theta series {I}{I}: {C}how group of the generic fiber}, 
      author={Tony Feng and Adeel Khan},
      year={2024},
      eprint={2403.19711},
      archivePrefix={arXiv},
      primaryClass={math.NT}
}

@article {FYZ,
    AUTHOR = {Feng, Tony and Yun, Zhiwei and Zhang, Wei},
     TITLE = {Higher {S}iegel--{W}eil formula for unitary groups: the
              non-singular terms},
   JOURNAL = {Invent. Math.},
  FJOURNAL = {Inventiones Mathematicae},
    VOLUME = {235},
      YEAR = {2024},
    NUMBER = {2},
     PAGES = {569--668},
      ISSN = {0020-9910},
   MRCLASS = {11F72 (11F67 11F70 14D23)},
       DOI = {10.1007/s00222-023-01228-y},
       URL = {https://doi.org/10.1007/s00222-023-01228-y},
}

@article {FYZ2,
    AUTHOR = {Feng, Tony and Yun, Zhiwei and Zhang, Wei},
     TITLE = {Higher theta series for unitary groups over function fields},
   JOURNAL = {Ann. Sci. \'{E}c. Norm. Sup\'{e}r. (4)},
  FJOURNAL = {Annales scientifiques de l'\'{E}cole normale sup\'{e}rieure},
    VOLUME = {58},
      YEAR = {2025},
    NUMBER = {2},
     PAGES = {275--388},
      ISSN = {0012-9593, 1873-2151},
       DOI = {10.24033/asens.2606},
       URL = {https://doi.org/10.24033/asens.2606},
}

@misc{FYZ3,
    title={Modularity of higher theta series {I}: cohomology of the generic fiber},
    author={Feng, Tony and Yun, Zhiwei and Zhang, Wei},
    year={2023},
    eprint={2308.10979},
    archivePrefix={arXiv},
    primaryClass={math.NT},
    note={To appear in Publ. Math. Inst. Hautes \'{E}tudes Sci.; arXiv:2308.10979}
}

@incollection {FH,
    AUTHOR = {Feng, Tony and Harris, Michael},
     TITLE = {Derived structures in the {L}anglands correspondence},
 BOOKTITLE = {The {L}anglands Program},
    SERIES = {Proceedings of Symposia in Pure Mathematics},
    VOLUME = {112.2},
     PAGES = {471--568},
 PUBLISHER = {American Mathematical Society, Providence, RI},
      YEAR = {2025},
       DOI = {10.1090/pspum/112.2/02068},
       URL = {https://doi.org/10.1090/pspum/112.2/02068},
}

@misc{FHM,
    title={Higher {S}iegel--{W}eil formula for unitary groups {II}: corank one terms},
    author={Feng, Tony and Howard, Benjamin and Mkrtchyan, Mikayel},
    year={2025},
    eprint={2507.13473},
    archivePrefix={arXiv},
    primaryClass={math.NT}
}

@article{HM22,
      title={Kudla's modularity conjecture on integral models of orthogonal {S}himura varieties}, 
      author={Howard, Benjamin and Madapusi, Keerthi},
      journal={Compos. Math.},
      fjournal={Compositio Mathematica},
      volume={161},
      year={2025},
      number={12},
      pages={3380--3454},
      doi={10.1017/S0010437X26102966},
      url={https://doi.org/10.1017/S0010437X26102966}
}

@Article{Jin24,
 Author = {Jin, Fangzhou},
 Title = {Trace maps in motivic homotopy and local terms},
 FJournal = {Transactions of the American Mathematical Society. Series B},
 Journal = {Trans. Am. Math. Soc., Ser. B},
 ISSN = {2330-0000},
 Volume = {11},
 Pages = {215--247},
 Year = {2024},
 Language = {English},
 DOI = {10.1090/btran/169},
 zbMATH = {7803080}
}

@article {KM90,
    AUTHOR = {Kudla, Stephen S. and Millson, John J.},
     TITLE = {Intersection numbers of cycles on locally symmetric spaces and
              {F}ourier coefficients of holomorphic modular forms in several
              complex variables},
   JOURNAL = {Inst. Hautes \'{E}tudes Sci. Publ. Math.},
  FJOURNAL = {Institut des Hautes \'{E}tudes Scientifiques. Publications
              Math\'{e}matiques},
    NUMBER = {71},
      YEAR = {1990},
     PAGES = {121--172},
      ISSN = {0073-8301},
   MRCLASS = {11F32 (11F30 11F46 11F67 32N10 32N15)},
       URL = {http://www.numdam.org/item?id=PMIHES_1990__71__121_0},
}

@article {Kud21,
    AUTHOR = {Kudla, Stephen S.},
     TITLE = {Remarks on generating series for special cycles on orthogonal
              {S}himura varieties},
   JOURNAL = {Algebra Number Theory},
  FJOURNAL = {Algebra \& Number Theory},
    VOLUME = {15},
      YEAR = {2021},
    NUMBER = {10},
     PAGES = {2403--2447},
      ISSN = {1937-0652},
   MRCLASS = {14C25 (11F27 11F46 11G18 14G35)},
       DOI = {10.2140/ant.2021.15.2403},
       URL = {https://doi.org/10.2140/ant.2021.15.2403},
}

@misc{KhanI,
      title={Virtual fundamental classes of derived stacks {I}}, 
      author={Adeel A. Khan},
      year={2019},
      eprint={1909.01332},
      archivePrefix={arXiv},
      primaryClass={math.AG}
}

@article {Kudla1997a,
    AUTHOR = {Kudla, Stephen S.},
     TITLE = {Algebraic cycles on {S}himura varieties of orthogonal type},
   JOURNAL = {Duke Math. J.},
  FJOURNAL = {Duke Mathematical Journal},
    VOLUME = {86},
      YEAR = {1997},
    NUMBER = {1},
     PAGES = {39--78},
      ISSN = {0012-7094},
   MRCLASS = {11F32 (11F30 11G18 14C25 14G35)},
MRREVIEWER = {Dipendra Prasad},
       DOI = {10.1215/S0012-7094-97-08602-6},
       URL = {https://doi.org/10.1215/S0012-7094-97-08602-6},
}

@incollection {Kud04,
    AUTHOR = {Kudla, Stephen S.},
     TITLE = {Special cycles and derivatives of {E}isenstein series},
 BOOKTITLE = {Heegner points and {R}ankin {$L$}-series},
    SERIES = {Math. Sci. Res. Inst. Publ.},
    VOLUME = {49},
     PAGES = {243--270},
 PUBLISHER = {Cambridge Univ. Press, Cambridge},
      YEAR = {2004},
   MRCLASS = {11G18 (11F37 11F46 11F67 14G40)},
MRREVIEWER = {Alexey A. Panchishkin},
       DOI = {10.1017/CBO9780511756375.009},
       URL = {https://doi.org/10.1017/CBO9780511756375.009},
}

@article {KRI,
    AUTHOR = {Kudla, Stephen and Rapoport, Michael},
     TITLE = {Special cycles on unitary {S}himura varieties {I}.
              {U}nramified local theory},
   JOURNAL = {Invent. Math.},
  FJOURNAL = {Inventiones Mathematicae},
    VOLUME = {184},
      YEAR = {2011},
    NUMBER = {3},
     PAGES = {629--682},
      ISSN = {0020-9910},
   MRCLASS = {14G35 (11G18 14L05)},
MRREVIEWER = {Jeffrey D. Achter},
       DOI = {10.1007/s00222-010-0298-z},
       URL = {https://doi.org/10.1007/s00222-010-0298-z},
}

@article {KRII,
    AUTHOR = {Kudla, Stephen and Rapoport, Michael},
     TITLE = {Special cycles on unitary {S}himura varieties {II}: {G}lobal
              theory},
   JOURNAL = {J. Reine Angew. Math.},
  FJOURNAL = {Journal f\"{u}r die Reine und Angewandte Mathematik. [Crelle's
              Journal]},
    VOLUME = {697},
      YEAR = {2014},
     PAGES = {91--157},
      ISSN = {0075-4102},
   MRCLASS = {11G18 (14G35)},
MRREVIEWER = {Jeffrey D. Achter},
       DOI = {10.1515/crelle-2012-0121},
       URL = {https://doi.org/10.1515/crelle-2012-0121},
}

@book {KRY,
    AUTHOR = {Kudla, Stephen S. and Rapoport, Michael and Yang, Tonghai},
     TITLE = {Modular forms and special cycles on {S}himura curves},
    SERIES = {Annals of Mathematics Studies},
    VOLUME = {161},
 PUBLISHER = {Princeton University Press, Princeton, NJ},
      YEAR = {2006},
     PAGES = {x+373},
      ISBN = {978-0-691-12551-0; 0-691-12551-1},
   MRCLASS = {11G18 (11-02 11F66 14G35 14G40)},
MRREVIEWER = {Jens Funke},
       DOI = {10.1515/9781400837168},
       URL = {https://doi.org/10.1515/9781400837168},
}

@misc{luo2025kudlarapoportconjectureunramifiedmaximal,
      title={Kudla-Rapoport conjecture for unramified maximal parahoric level}, 
      author={Yu Luo},
      year={2025},
      eprint={2504.18528},
      archivePrefix={arXiv},
      primaryClass={math.NT},
      url={https://arxiv.org/abs/2504.18528}, 
}

@misc{Mad22,
      title={Derived special cycles on {S}himura varieties},
      author={Madapusi, Keerthi},
      year={2022},
      eprint={2212.12849},
      archivePrefix={arXiv},
      primaryClass={math.AG}
}

@article {Mae21,
    AUTHOR = {Maeda, Yota},
     TITLE = {The modularity of special cycles on orthogonal {S}himura
              varieties over totally real fields under the
              {B}eilinson-{B}loch conjecture},
   JOURNAL = {Canad. Math. Bull.},
  FJOURNAL = {Canadian Mathematical Bulletin. Bulletin Canadien de
              Math\'{e}matiques},
    VOLUME = {64},
      YEAR = {2021},
    NUMBER = {1},
     PAGES = {39--53},
      ISSN = {0008-4395},
       DOI = {10.4153/S000843952000020X},
       URL = {https://doi.org/10.4153/S000843952000020X},
}

@misc{Pollack,
      title={Automatic convergence for Siegel modular forms},
      author={Pollack, Aaron},
      year={2024},
      eprint={2408.16392},
      archivePrefix={arXiv},
      primaryClass={math.NT}
}

@misc{Raum,
      title={The Geometric Unitary Kudla Conjecture},
      author={Raum, Martin},
      year={2026},
      eprint={2603.04282},
      archivePrefix={arXiv},
      primaryClass={math.NT}
}

@misc{stacks-project,
    shorthand    = {Stacks},
    author       = {The {Stacks Project Authors}},
    title        = {\textit{Stacks Project}},
    howpublished = {\url{https://stacks.math.columbia.edu}},
    year         = {2020},
  }

@article {SSTT22,
    AUTHOR = {Shankar, Ananth N. and Shankar, Arul and Tang, Yunqing and
              Tayou, Salim},
     TITLE = {Exceptional jumps of {P}icard ranks of reductions of {K}3
              surfaces over number fields},
   JOURNAL = {Forum Math. Pi},
  FJOURNAL = {Forum of Mathematics. Pi},
    VOLUME = {10},
      YEAR = {2022},
     PAGES = {Paper No. e21, 49},
   MRCLASS = {14J28 (11G10 11G18 14G35 14J20)},
MRREVIEWER = {Sajad Salami},
       DOI = {10.1017/fmp.2022.14},
       URL = {https://doi.org/10.1017/fmp.2022.14},
}

@article {YZ,
    AUTHOR = {Yun, Zhiwei and Zhang, Wei},
     TITLE = {Shtukas and the {T}aylor expansion of {$L$}-functions},
   JOURNAL = {Ann. of Math. (2)},
  FJOURNAL = {Annals of Mathematics. Second Series},
    VOLUME = {186},
      YEAR = {2017},
    NUMBER = {3},
     PAGES = {767--911},
      ISSN = {0003-486X},
   MRCLASS = {11F67 (11F70 14G35 14H60)},
MRREVIEWER = {Shouwu Zhang},
       DOI = {10.4007/annals.2017.186.3.2},
       URL = {https://doi.org/10.4007/annals.2017.186.3.2},
}

@article {YZZ,
    AUTHOR = {Yuan, Xinyi and Zhang, Shou-Wu and Zhang, Wei},
     TITLE = {The {G}ross-{K}ohnen-{Z}agier theorem over totally real
              fields},
   JOURNAL = {Compos. Math.},
  FJOURNAL = {Compositio Mathematica},
    VOLUME = {145},
      YEAR = {2009},
    NUMBER = {5},
     PAGES = {1147--1162},
      ISSN = {0010-437X,1570-5846},
   MRCLASS = {11G18 (11F46)},
  MRNUMBER = {2551992},
MRREVIEWER = {Min\ Ho\ Lee},
       DOI = {10.1112/S0010437X08003734},
       URL = {https://doi-org.libproxy.mit.edu/10.1112/S0010437X08003734},
}

@book {Zh09,
    AUTHOR = {Zhang, Wei},
     TITLE = {Modularity of generating functions of special cycles on
              {S}himura varieties},
      NOTE = {Thesis (Ph.D.)--Columbia University},
 PUBLISHER = {ProQuest LLC, Ann Arbor, MI},
      YEAR = {2009},
     PAGES = {48},
      ISBN = {978-1109-34157-7},
   MRCLASS = {Thesis},
       URL =
              {http://gateway.proquest.com/openurl?url_ver=Z39.88-2004&rft_val_fmt=info:ofi/fmt:kev:mtx:dissertation&res_dat=xri:pqdiss&rft_dat=xri:pqdiss:3373585},
}

@article{Zh21,
    AUTHOR = {Zhang, W.},
     TITLE = {Weil representation and arithmetic fundamental lemma},
   JOURNAL = {Ann. of Math. (2)},
  FJOURNAL = {Annals of Mathematics. Second Series},
    VOLUME = {193},
      YEAR = {2021},
    NUMBER = {3},
     PAGES = {863--978},
      ISSN = {0003-486X},
   MRCLASS = {11F27 (11F67 11G40 14C25 14G35)},
MRREVIEWER = {Rolf Berndt},
       DOI = {10.4007/annals.2021.193.3.5},
       URL = {https://doi.org/10.4007/annals.2021.193.3.5},
}

@phdthesis{ZeffThesis,
    AUTHOR = {Zeff, Chaim Avram},
     TITLE = {Toward Generic Modularity of Higher Theta Series over Global Function Fields and {$p$}-Adic Local Fields},
    SCHOOL = {Columbia University},
      YEAR = {2025},
      NOTE = {ProQuest document ID 3201333388},
       URL = {https://www.proquest.com/docview/3201333388},
}

@misc{ZhouMotivicFourier,
    title={A motivic derived {F}ourier transform},
      author={Zhou, Tong},
      year={2026},
      eprint={2608.22501},
      archivePrefix={arXiv},
      primaryClass={math.AG}
}

\appendix

\end{document}